\pdfoutput=1
\RequirePackage{ifpdf}
\ifpdf % We~are running pdfTeX in pdf mode
\documentclass[pdftex]{sigma}%,draft
\else
\documentclass{sigma}
\fi

\usepackage{amsxtra}

\newtheorem{thm}{Theorem}[section]
\newtheorem{cor}[thm]{Corollary}
\newtheorem{lem}[thm]{Lemma}
\newtheorem{prop}[thm]{Proposition}
\newtheorem{conj}[thm]{Conjecture}

\theoremstyle{definition}
\newtheorem*{rem}{Remark}
\newtheorem{rems}[thm]{Remark}

\newtheorem{defn}[thm]{Definition}

\numberwithin{equation}{section}

\renewcommand{\P}{\mathbb P}
\newcommand{\Q}{\mathbb Q}
\newcommand{\Z}{\mathbb Z}
\newcommand{\N}{\mathbb N}
\newcommand{\C}{\mathbb C}
\newcommand{\R}{\mathbb R}
\newcommand{\A}{\mathbb A}
\newcommand{\G}{\mathbb G}
\renewcommand{\H}{\mathbb H}
\def\disj{\uplus}
\DeclareMathOperator{\Aut}{Aut}
\DeclareMathOperator{\GL}{GL}
\DeclareMathOperator{\GO}{GO}
\DeclareMathOperator{\PGL}{PGL}
\DeclareMathOperator{\SL}{SL}
\DeclareMathOperator{\Sp}{Sp}
\DeclareMathOperator{\Pic}{Pic}
\DeclareMathOperator{\Div}{Div}
\DeclareMathOperator{\NS}{NS}

\DeclareMathOperator{\Sym}{Sym}
\DeclareMathOperator{\Alt}{Alt}
\DeclareMathOperator{\End}{End}
\DeclareMathOperator{\Hom}{Hom}
\DeclareMathOperator{\Mat}{Mat}

\DeclareMathOperator{\rank}{rank}
\DeclareMathOperator{\coker}{coker}

\DeclareMathOperator{\ad}{ad}

\DeclareMathOperator{\Spec}{Spec}
\DeclareMathOperator{\Proj}{Proj}

\DeclareMathOperator{\Tor}{Tor}
\DeclareMathOperator{\im}{im}

\DeclareMathOperator{\coh}{coh}

\DeclareMathOperator{\Tw}{Tw}
\newcommand{\sO}{\mathcal O}

\newcommand{\tW}{\widetilde W}

\newcommand{\sHom}{\mathcal{H}{\rm om}}
\newcommand{\sEnd}{\mathcal E{\rm nd}}

\DeclareMathOperator{\Hilb}{Hilb}

\newcommand{\ratto}{\dashrightarrow}
\newcommand{\Gampq}{\Gamma_{\!\!p,q}}
\newcommand{\Gamq}{\Gamma_{\!\!q}}
\def\Gamm#1{\Gamma_{\!\! #1}}

\DeclareMathOperator{\Res}{Res}
\DeclareMathOperator{\Ind}{Ind}
\DeclareMathOperator{\Coind}{Coind}

\DeclareMathOperator{\lcm}{lcm}

\begin{document}

\allowdisplaybreaks

\newcommand{\arXivNumber}{1709.02989}

\renewcommand{\thefootnote}{}

\renewcommand{\PaperNumber}{111}

\FirstPageHeading

\ShortArticleName{Elliptic Double Affine Hecke Algebras}

\ArticleName{Elliptic Double Affine Hecke Algebras\footnote{This paper is a~contribution to the Special Issue on Elliptic Integrable Systems, Special Functions and Quantum Field Theory. The full collection is available at \href{https://www.emis.de/journals/SIGMA/elliptic-integrable-systems.html}{https://www.emis.de/journals/SIGMA/elliptic-integrable-systems.html}}}

\Author{Eric M.~RAINS}
\AuthorNameForHeading{E.M.~Rains}

\Address{Department of Mathematics, California Institute of Technology, USA}
\Email{\href{mailto:rains@caltech.edu}{rains@caltech.edu}}

\ArticleDates{Received December 19, 2019, in final form October 16, 2020; Published online November 05, 2020}

\Abstract{We give a construction of an affine Hecke algebra associated to any Coxeter group acting on an abelian variety by reflections; in the case of an affine Weyl group, the result is an elliptic analogue of the usual double affine Hecke algebra. As~an application, we use a variant of the $\tilde{C}_n$ version of the construction to construct a flat noncommutative deformation of the $n$th symmetric power of any rational surface with a smooth anticanonical curve, and give a further construction which conjecturally is a corresponding deformation of the Hilbert scheme of points.}

\Keywords{elliptic curves; Hecke algebras; noncommutative deformations}

\Classification{33D80; 39A70; 14A22}

\renewcommand{\thefootnote}{\arabic{footnote}}
\setcounter{footnote}{0}

\tableofcontents

\section{Introduction}\label{section1}

The origin of this paper was a question of P.~Etingof which was conveyed to
the author by A.~Okounkov at a 2011 conference,\footnote{Affine Hecke Algebras, the Langlands Program, Conformal Field Theory and Super Yang--Mills Theory, Centre International de Rencontres Math\'ematiques (Luminy), June 2011.} to wit whether the author knew of a way to
construct noncommutative deformations of symmetric powers of the complement
of a smooth cubic plane curve. Although the answer was ``no'' (at the
time, see below!), it seemed likely that it should be possible to extend
the approach of~\cite{sklyanin_anal} (which the author and S.~Ruijsenaars
had developed earlier that month) to multivariate difference operators;
although this would not answer the question as posed, it {\em would} give
analogous deformations associated to the complement of a~smooth biquadratic
curve in~$\P^1\times \P^1$, represented as algebras of elliptic difference
operators in~$n$ variables. The~author continued to develop this approach
while on a sabbatical that fall at~MIT, eventually coming up with a
construction for such deformations for {\em any} rational surface equipped
with a smooth anticanonical curve and a rational ruling.

There were, however, a couple of significant issues. One was that the
spaces of operators were cut out by a number of conditions, including in
particular certain residue conditions that only made sense for generic
values of the parameters. This would have been merely an annoying
technicality, except that the conditions specifically failed to make sense
in the commutative case, making it rather difficult to consider the family
as actually being a deformation. This could be worked around by
considering the family as a whole, or in other words only considering those
operators that could be extended to an open subset of parameter space.
However, although this would indeed give a well-defined family of algebras,
it would make the question of flatness even more difficult, and would in
principle even allow the representation in difference operators to fail to
be faithful!

Despite these difficulties, the definition was still well-behaved enough to
allow a fair amount of experimentation. One thing that became clear was
that the construction directly led to some spaces of operators that had
been considered in the literature, in particular those associated to the
difference equations for interpolation and biorthogonal functions
\cite{bctheta, xforms}, of particular interest in the latter case since no
satisfying construction was yet known for the full space of operators. In~addition, since these functions degenerate to more familiar functions, to
wit the Macdonald and Koornwinder polynomials, this suggested that the
algebras of elliptic difference operators should degenerate to algebras
related to those polynomials. In~particular, the latter algebras can be
constructed as spherical algebras of appropriate double affine Hecke
algebras, and P. Etingof suggested to the author that the same might hold
at the elliptic level, and give a possible approach to flatness.

\looseness=-1Indeed, the same approach to constructing the algebra of operators (as
operators preserving (locally) holomorphic functions and satisfying
appropriate vanishing conditions on the coefficients) could be fairly
easily extended to give a construction of elliptic double affine Hecke
algebras. The~resulting residue conditions turned out to be essentially
those of~\cite{GinzburgV/KapranovM/VasserotE:1997}, with again the caveat
that they only make sense when the noncommutative parameter $q$ is
non-torsion. In~fact, something slightly stronger is true: the residue
conditions are well-behaved as long as one only considers a sufficiently
small interval relative to an appropriate (Bruhat) filtration, with the
constraint on the interval being simply that it act faithfully. This led
the author to investigate that filtration more carefully, leading
eventually to the realization that (a) the residue conditions always make
sense on rank~1 subalgebras (which are very special cases of the
construction of~\cite{GinzburgV/KapranovM/VasserotE:1997}), and (b) those
rank~1 subalgebras always generate a flat algebra, even when~$q$ is
torsion. As a~result, one could avoid the residue conditions entirely and
simply consider the algebra generated by the rank~1 algebras. It is then
relatively straightforward to show that the resulting family of algebras is
flat (and the representation as difference-reflection operators is
faithful), and not too difficult to show that this flatness is inherited by
the spherical algebra. Moreover, much of the theory can be developed for
quite general actions of Coxeter groups on abelian varieties, so that the
DAHAs are just the special cases in which the Coxeter group is affine.

In the above discussion, we~have neglected a few technical issues. The~first is that the deformations of symmetric powers are not, in general,
{\em algebras} of operators. The~difficulty is that, with the exception of
complements of smooth anticanonical curves in del Pezzo surfaces, none of
the surfaces we wish to deform are actually affine! Since they are only
quasiprojective in general, it is easier to simply deform the symmetric
power of the original projective surface (and then take the appropriate
localization if desired). This still requires a choice of ample divisor,
and one then encounters the difficulty that twisting a noncommutative
variety by a line bundle tends to change the noncommutative variety. As~a
result, the object actually being deformed is the category of line bundles
on the symmetric power.

When it comes to the DAHA, however, the situation is even more complicated. The~problem is that the elliptic DAHA essentially arises by replacing one
of the two commutative subalgebras of the usual DAHA by the structure sheaf
of an abelian variety of the form $E^n$. In~the affine case considered in
\cite{GinzburgV/KapranovM/VasserotE:1997}, the commutative subalgebra is
finite over the center of the affine Hecke algebra, and thus one is
naturally led to consider sheaves over the center, or in other words
sheaves on the quotient $E^n/W$, where $W$ is the relevant finite Weyl
group. Unfortunately, in the double affine setting, $W$ is replaced by an
affine Weyl group, and there {\em is} no such quotient scheme! As~a~result, the trickiest part of our construction turns out to be simply
figuring out what kind of object we should be constructing. The~key idea,
coming from earlier work in noncommutative geometry
\cite{ArtinM/VandenBerghM:1990,VandenBerghM:1996} is that since the
elliptic DAHA should have $\sO_{E^n}$ as a subalgebra, it should have a
natural {\em bimodule} structure over $\sO_{E^n}$, and thus correspond to a
quasicoherent sheaf on the product $E^n\times E^n$. Subject to some
finiteness conditions (satisfied for any sheaf of meromorphic difference
operators), such bimodules form a monoidal category, and thus we can
construct the elliptic DAHA as a monoid object (``sheaf algebra'') in that
category.

A final technical issue is that we wish to deform symmetric powers of
arbitrary blowups of~$\P^1\times \P^1$ or the Hirzebruch surface $F_1$.
Each time we blow up a point, we~acquire a new parameter, and thus our
construction needs to admit an unbounded number of parameters. This is an
issue from the standpoint of traditional double affine Hecke algebras,
where one normally has precisely one parameter per root (which must be
constant on orbits), plus an overall parameter $q$. One partial exception
is the $C^\vee C_n$ case, where one has a total of 5+1 parameters. This is
normally explained by taking a nonreduced root system, so that one has~5~orbits of roots, but from the standpoint of the actual algebra, this is
rather artificial: there is an action of~$S_2\times S_2$ on the parameters
that has no effect on the algebra, but moves degrees of freedom between
corresponding short and long roots. It is thus much more natural to view
those four parameters as assigning an unordered pair to each orbit of short
roots of the {\em reduced} root system. This turns out to generalize
easily to the elliptic setting: we obtain an elliptic DAHA for every
assignment of an effective divisor on~$E$ to each orbit of roots. This
causes some difficulties in constructing the spherical algebra, as~the
usual construction via idempotents fails even generically, but one can show
that the spherical algebra still continues to inherit flatness in this
general case.

As we mentioned above, the construction of deformations of~$\Sym^n(X)$,
where $X$ is a projective rational surface with a choice of smooth
anticanonical curve, depends on a choice of rational ruling on~$X$. One
consequence is that we cannot directly obtain the case $X=\P^2$ from our
construction. This can be worked around by blowing up a point but then
only considering those line bundles coming from $\P^2$, but this approach
leads to a nontrivial question of showing that the result is independent of
the choice of point. Similarly, if $X\cong \P^1\times \P^1$, then there
are two choices of ruling on~$X$, but deformation theory suggests that the
resulting deformations should be the same (both have the maximum number of
parameters). Both questions turn out to~red\-uce to~the existence of a
certain generalized Fourier transform in the $\P^1\times \P^1$ case, which
is also key to proving the most general form of the flatness result. (The
DAHA only tells us that certain sheaves are flat, so gives only an
asymptotic flatness result for their global sections.) We~find that not
only is our deformation of the category of~symmetric powers of~line bundles
on~$X$ flat (modulo some possibility of bad parameters in codimension~$\ge
2$ not including the original symmetric-power-of-commutative-surface case),
but it is invariant under the action of~a~Coxeter group of~type
$W(E_{m+1})$. (In other words, to first approximation, the construction
only depends on~the~underlying surface $X$ and two points $q$ and $t$ of
the Jacobian of the anticanonical curve.) Note that both facts are
actually {\em not} true for the full original family of~commutative
categories, but hold for the subcategory in which we only allow those
morphisms that extend to~a~neighborhood in the family.

\looseness=-1 The plan of this paper is as follows. First, in Section~\ref{section2}, we~deal with a
largely notational issue that arises due to the fact that we wish to deal
with the construction from a purely algebraic standpoint. The~issue is
that we need in many cases to deal with {\em twisted} versions of our
algebras, in which the coefficients of the operators lie in nontrivial line
bundles. To make sense of this, one must not only describe those line
bundles but various maps between tensor products of pullbacks of those
bundles through elements of the Coxeter group, with associated concerns
about compatibility. In~the analytic setting, one can avoid most of those
issues by constructing the line bundle via an appropriate automorphy factor
on the universal cover, and our objective in Section~\ref{section2} is to do something
similar in the algebraic setting. The~key idea is to replace the
individual curve $E$ by the universal curve over the moduli stack of
elliptic curves; it turns out that one can compute the Picard group of the
corresponding stack ${\cal E}^n$ in general, with only very mild
rigidification required to make pullbacks and tensor products behave well. In~addition, certain of the line bundles come with natural global sections,
which one can use to construct sections of more general bundles; our most
significant result along these lines gives conditions for a function on the
analytic locus described via theta functions to extend to the full moduli
stack (i.e., when it extends in a nice way to elliptic curves over an
arbitrary base). We also partially consider the case of varieties which
are isogenous to powers of elliptic curves (i.e., over a moduli stack of
elliptic curves with a cyclic subgroup), and as an application give some
results on spaces of invariant sections of equivariant line bundles on~${\cal E}^n$. The~main result along those lines states that the dimension
of the space of invariants is independent of the curve $E$ with only
finitely many possible exceptions (supersingular curves of characteristic
dividing the order of the group).

In Section~\ref{section3}, we~give some structural results on the main scenario we
consider in the sequel, namely a~Coxeter group acting on an abelian variety
``by reflections''. In~particular, we~show that under reasonable
conditions every root of the Coxeter group can be assigned an associated
``coroot'' morphism to an elliptic curve, compatibly with the linear
relations between roots in the standard reflection representation (and
satisfying suitable notions of positivity!). This is a~key ingredient in
our construction, as~our parameters will correspond to effective divisors
on those curves. We also show that in the case of a finite Coxeter group,
the invariant theory is better behaved than suggested by the results of
Section~\ref{section2}: as long as a certain isogeny (which is an isomorphism in the
most natural cases) has diagonalizable kernel, the invariant theory
continues to behave well even for supersingular curves. This flatness of
invariants is a crucial ingredient in proving flatness of the spherical
algebra, and in particular means the $C^\vee C_n$ case will be flat over
any field.

Section~\ref{section4} is largely a~recapitulation of the construction of
\cite{GinzburgV/KapranovM/VasserotE:1997}, in which we associate to any
finite Coxeter group $W$ acting on~a~family $X$ of abelian varieties a~family (the ``elliptic Hecke algebra'' for concision, though it should be
thought of as affine, with the role of the commutative subalgebra being
played by the structure sheaf $\sO_X$) of sheaves of algebras on~$X/W$
parametrized by effective divisors on~the ``coroot'' curves. The~main
difference, apart from allowing arbitrary numbers of parameters and
slightly more general abelian varieties, is that we replace the residue
conditions of~\cite{GinzburgV/KapranovM/VasserotE:1997} by the (equivalent)
condition that the operators preserve the spaces of local sections of the
structure sheaf of~$X$ on~$W$-invariant open subsets, as~the latter is
easier to generalize from a~conceptual standpoint. Our main new tool for
studying these algebras is a~natural filtration by Bruhat order on~$W$,
which allows us to express various subsheaves as extensions of line bundles
in natural ways. In~particular, this makes it easy to show that the
algebra is generated by the subalgebras corresponding to simple roots,
which will be a~key ingredient of the extension to infinite Coxeter groups, as~well as giving a~construction which is easily seen to respect base
change. We also prove an important technical lemma on~the space of~``invariants'' of a~module over the elliptic AHA, giving fairly general
conditions on~a~family of~modules guaranteeing that the invariants are
well-behaved; this will be the key lemma in~proving flatness of~spherical
algebras. In~addition, we~give an analogue of Mackey's theorem for the
case of parabolic subgroups in which the usual sum over double cosets is
replaced by a~filtration.

Section~\ref{section5} begins with a discussion of sheaf algebras, which we use in place
of a sheaf of algebras on the quotient. Since the corresponding tensor
product of bimodules is somewhat tricky to deal with, we~discuss approaches
to dealing with this issue (and in particular ways to describe the maps
$A\otimes B\to C$ we need in order to discuss algebras or categories).
This makes it relatively straightforward to construct analogues of the
affine Hecke algebra in which $W$ is replaced by an infinite Coxeter group:
start with the sheaf algebra of meromorphic reflection operators and take
the sheaf subalgebra generated by the rank~1 Hecke algebras. (The only
tricky aspect comes when we consider twisted forms of the algebra, for~which we prove a fairly general result about ``orders'' in twisted forms of~$k(X)[W]$ containing $\sO_X[W]$.) Most of the results extend immediately
from the finite case (with the caveat that any parabolic subgroups
considered should be finite); in particular, the infinite analogue of the
Mackey result is precisely the remaining result we need to show that the
spherical algebras are flat.

Section~\ref{section6} discusses the special case in which $W$ is an affine Weyl group,
so that the sheaf algebras constructed in Section~\ref{section5} are analogues of double
affine Hecke algebras. Apart from some mild issues about viewing $q$ as a
parameter (not the case for the standard construction), this mainly
consists of some observations about the spherical algebra (relative to the
associated finite Weyl group): the fact that the fibers are (sheaf)
algebras of difference operators, and are thus in a natural sense domains,
and the fact that the elliptic DAHA is at least generically Morita
equivalent to its spherical algebra (as well as versions in which some of
the parameters have been shifted by $q$). In~addition, we~discuss the
consequences of the fact that the action of~$\tW$ fails to be faithful when~$q$ is torsion: not only does the sheaf algebra come from a sheaf of
algebras on the quotient by the image of~$\tW$, but it has a
$2n$-dimensional center (over which it is presumably finite). Moreover,
under mild conditions on the twisting, we~can identify the center as the
spherical algebra of an elliptic DAHA with~$q=0$ living on an isogenous
abelian variety.

Section~\ref{section7} considers in detail the case that $W$ is the affine Weyl group of
type $C_n$ and $X$ is a particularly nice action of that group, with a view
to constructing deformations of symmetric powers of rational surfaces. In~addition to the spherical algebras themselves, one must also consider
certain intertwining bimodules. We prove that these are always flat {\em
 as sheaf bimodules}, and show that in the case $t=0$ the result is indeed
a symmetric power of the univariate case (which was constructed in
\cite{noncomm1} without reference to DAHAs). This enables us to at least
partially extend the flatness as sheaves to flatness of global sections, by
giving a number of cases in which the sheaves are acyclic. In~addition to
these general results, our main result in this section is showing that if
we blow up $8$ points of~$\P^1\times \P^1$, then there is a hypersurface in
parameter space on which the ``anticanonical'' algebra is an integrable
system: it is generated by $n+1$ commuting (and self-adjoint) elliptic
difference operators in~$n$ variables. (One can verify that this is
precisely the integrable system of~\cite{vanDiejenJF:1994,vanDiejenJF:1995, KomoriY/HikamiK:1997}. In~addition, the geometry strongly suggests the existence of other integrable
systems with the same number of parameters, but higher-order operators.)

Section~\ref{section8} deals with the question of showing that the construction is
mostly independent of the way in which we represented our rational surface
as a blowup of a ruled surface. The~key ingredient is a certain ``Fourier
transform''. Analytically, this should be represented by the integral
operator with kernel constructed in~\cite{quadxforms}, but there are
difficulties in showing this is well-defined in general (as well as showing
that it respects the additional conditions associated to any points we have
blown up). As~a result, we~construct the transform in several steps.
First, we~give a construction that is manifestly well-defined (and a
homomorphism) as long as $q$ is not torsion, but lives on a certain
completion of the algebra of meromorphic difference operators. Although
the construction is in general quite complicated, it is sufficiently
well-behaved to allow us to compute a few special cases. The~result in
those special cases turns out to be well-defined even when~$q$ is torsion,
and this allows us to show that this formal transform extends in general.
Moreover, the special cases we understand are sufficiently close to
generating the full algebras that we can prove that the formal transform
restricts to an actual transform. This also gives us some ability to
explore how the algebras behave when we degenerate the elliptic curve, as~it is easy to take limits of the almost-generators. In~addition to
allowing us to prove a much stronger flatness result, the Fourier transform
also allows us to construct a large collection of~``quasi-integrable''
systems, in particular including the aforementioned operators associated to
interpolation and biorthogonal functions.

Section~\ref{section9} gives some partial results towards ``desingularizing'' the above
construction; i.e., giving a deformation of the Hilbert scheme of points of~$X$ rather than just the symmetric power. The~basic idea is fairly
straightforward: simply include additional bimodules that shift $t$ as well
as the points being blown up. The~argument for flatness breaks down in
general, even at the level of sheaves, but we can show both flatness and
agreement with Euler characteristics of line bundles on the Hilbert schemes
in a fair number of cases. Moreover, for~several of~those cases, not only
do we obtain the correct number of global sections, but one can identify
those global sections as global sections on the Hilbert scheme in such a
way that they satisfy precisely the same relations. In~particular, this
includes the line bundles corresponding to the embedding of~$\Hilb^n(X)$ in
a~Grassmannian that takes the ideal sheaf ${\cal I}$ to the subspace
$\Gamma(X;{\cal I}(D))\subset \Gamma(X;\sO_X(D))$. (One~can also show that
the Fourier transform extends, so that again the construction essentially
depends only on~$X$ and not the particular way it was obtained from a
Hirzebruch surface.)

We close with a summary of some of the various open problems that arose in
the course of this work. (Of course, this is only a small sampling of such
problems, as~nearly any existing result on double affine Hecke algebras
suggests the existence of a generalization to the elliptic case! We are
also omitting some questions discussed in Sections~\ref{section2} and~\ref{section3}, as~they are
peripheral to the main thrust of the work.) One big collection of
questions has to do with the fact that, although we show that the spherical
algebras of the elliptic DAHAs are flat in significant generality, we~can
prove almost nothing else about them in general. In~particular, we~cannot
even show that they are Noetherian (even in the specific $C^\vee C_n$-type
cases for which we have such strong flatness results). The~approach to
such questions in~\cite{ArtinM/TateJ/VandenBerghM:1990} suggests that one
should be able to reduce this to the case in which everything is defined
over a finite field, when one expects the spherical algebra to be finite
over its center (which should itself be Noetherian); unfortunately,
understanding the center (when~$q$ is torsion) is itself an open problem.
(This reduces to understanding the spherical algebra when~$q=0$.)

Another natural question is whether the Fourier transform of the usual DAHA
extends to~the elliptic level. The~construction of Section~\ref{section8} can be viewed
as a partial affirmative answer to~this question in the $C^\vee C_n$ case, as~it constructs a Fourier transform on the spherical algebra. For~$n=1$,
this at least implicitly leads to a Fourier transform on the elliptic DAHA
by using the appropriate Morita equivalence, but even there it is unclear
how to make the transformation explicit. (There is also a philosophical
question: our work suggests that the transform should really be viewed
as living on the spherical algebra, as~the analogous transform on the DAHA
is not as well behaved relative to the natural filtration.)

In fact, even in the $C_n$ case, there are still open questions about the
Fourier transform, as~the construction of Section~\ref{section8} only applies to the
case in which we have assigned precisely one parameter to the $D_n$-type
roots. There is evidence suggesting the existence of some other Fourier
transforms related to the versions of the DAHAs of types $A_n$ and $C_n$
that do not have this ``$t$'' parameter. Indeed, the paper
\cite{SpiridonovVP/WarnaarSO:2006} discusses several integral
transformations; one appears to relate two versions of the $A_n$ spherical
algebra, while the other two appear to relate the $A_n$ spherical algebra
to the $C_n$ spherical algebra. The~main obstruction to understanding
these cases is that the lack of a $t$ parameter makes it difficult to
control the spaces of global sections, as~we can no longer view the
spherical algebras as deformations of a symmetric power. In~any event,
this suggests that the existence of isomorphisms between spherical algebras
is a much more subtle question than one might have thought based on the
classical theory of double affine Hecke algebras. Another source of
questions about the Fourier transform is the analytic version constructed
in~\cite{quadxforms}. In~addition to the ``interpolation kernel'', that
paper constructed a few other functions with some similar properties (the
``Littlewood'', ``dual Littlewood'' and ``Kawanaka'' kernels), and it is
natural to ask whether those functions interact with the $C^\vee C_n$
spherical algebra in any interesting way. That such an interaction should
exist is strongly suggested by the fact that the quadratic transforms
proved in that paper were generalizations of results first proved using
the action of affine Hecke algebras on Laurent polynomials.

There are also a number of questions about our deformations of symmetric
powers having to do with taking global sections. One fundamental question
has to do with the fact that our construction, although (mostly) flat and
highly symmetrical, does not quite correspond directly to deformations of
symmetric powers of surfaces: we must still make a choice of ample line
bundle in order to obtain an actual projective deformation in the
commutative case or a deformation of the category of sheaves in the
noncommutative case. For~$n=1$, it was shown in~\cite{noncomm1} that all
such choices give the same result, but it will be difficult to extend those
techniques to~$n>1$. In~addition, one would like to show that the
corresponding family of commutative quasiprojective varieties is at least
generically smooth (which experiment suggests is the case). Of course, we~also expect that our conjectural deformation of Hilbert schemes should
always be smooth (i.e., in the noncommutative case, that the corresponding
category of coherent sheaves should satisfy Serre duality). In~addition, one
would like to have a proof that our family of algebras (or noncommutative
varieties) actually depends on all of the parameters (and is not, say,
simply a base change from a lower-dimensional family); this suggests trying
to understand how the family relates to the infinitesimal deformation
theory of the symmetric power. (Note, however, that the deformation
associated to the $t$ parameter is almost certainly not locally trivial.
Also, the fact that replacing $t$ by $q-t$ gives an isomorphic algebra
implies that the corresponding Kodaira--Spencer map will vanish unless we
first descend the family to the quotient by this symmetry, and it is
unclear how to do so without breaking the representation via difference
operators.) This would largely be settled if we could show that the
$q=0$ case of the Hilbert scheme deformation agreed with the deformation
constructed in~\cite{noncomm1} (as a moduli space of rank~1 sheaves on a
noncommutative rational surface).

In any event, a choice of ample divisor allows one to translate each
original line bundle into an actual sheaf on the (noncommutative) variety,
and thus produces a saturated version of the original $\Hom$ space (by
taking all morphisms between the sheaves). In~the univariate case, the
saturated morphisms were still difference operators, and there is a
primarily combinatorial algorithm for computing the resulting dimensions. The~question is more subtle in the multivariate setting, with two main
issues arising. For~some line bundles on~$F_1$ (including those coming
from line bundles on~$\P^2$), we~can only prove flatness away from a
possible bad locus of codi\-men\-sion~$\ge 2$, and some new approach to
constructing global sections is likely needed to eliminate this
possibility. The~other tricky case arises from the elliptic pencil on a
deformation of a non-Jacobian elliptic surface (i.e., in which the elliptic
fibration does not have a section). There, not only do we want to know how
many global sections there are (with the conjecture being that on the
hypersurface where the algebra of global sections is nontrivial, it is
flat), but also, by analogy with the Jacobian case, expect that the algebra
of global sections will give a~new integrable system, associated to the
same parameters as the van Diejen/Komori--Hikami system, with the addition
of a choice of nontrivial torsion point on the elliptic curve.

Finally, the condition that a projective rational surface have a {\em
 smooth} anticanonical curve is quite restrictive, and in particular
excludes a number of cases in which deformations were already known.
Although the strong version of flatness cannot be expected to extend in
general, one can still expect to have flatness for {\em ample} bundles.
There are already issues before blowing up any points, as~it appears one
must give an analogue of the elliptic DAHA in which the elliptic curve
becomes singular, reducible, or even nonreduced, and this can cause issues
with the Bruhat filtration as well as the generation in rank~1. Beyond
that, although it should be fairly straightforward (especially if the base
curve remains integral) to consider blowups in {\em smooth} points of the
base curve, this is again a pretty restrictive condition, while blowing up
singular points quickly leads to a combinatorial explosion without some
more conceptual approach.

\section[Line bundles on $E^n$ and their sections]{Line bundles on~$\boldsymbol{E^n}$ and their sections}\label{section2}

In the sequel, we~will quite frequently need to specify line bundles on a
power $E^n$ of an elliptic curve $E$ or (meromorphic) sections thereof.
(Here $E$ will typically be defined over an algebraically closed field,
though much of the discussion works over a general base.) At first glance,
the problem of specifying a line bundle appears nearly trivial. Indeed,
given any principally polarized abelian variety $A$, we~have a short exact
sequence
\begin{gather*}
0\to A\to \Pic(A)\to \NS(A)\to 0,
\end{gather*}
where the N\'eron--Severi group $\NS(A)$ is naturally isomorphic to the
group of endomorphisms of~$A$ which are symmetric under the Rosati
involution. If~$\End(E)=\Z$ (which holds generically), then this becomes
\begin{gather*}
0\to E^n\to \Pic(E^n)\to \NS(E^n)\to 0,
\end{gather*}
where $\NS(E^n)$ is the group of symmetric $n\times n$ integer matrices.
Moreover, it turns out (as we will discuss in more detail below) that this
short exact sequence splits, giving us canonical labels for line bundles
{\em up to isomorphism}.

This last caveat is quite significant, however; if ${\cal L}_1$ and ${\cal
 L}_2$ represent given classes in~$\Pic(E^n)$ and ${\cal L}_3$ represents
their sum, then there exists an isomorphism ${\cal L}_1\otimes {\cal
 L}_2\cong {\cal L}_3$, but this is only determined up to an overall
scalar multiple. Another consequence is that if we have (as we will below)
a group $G$ acting on~$E^n$, a $G$-invariant class in~$\Pic(E^n)$ need not
specify an {\em equivariant} line bundle.

If $E$ is an analytic curve $\C/\langle 1,\tau\rangle$, there is a standard
way to avoid these difficulties, namely the theory of theta functions.
Indeed, given a cocycle $z\in Z^1(\pi_1(E^n);\A(\C^n)^*)$, we~can construct
a corresponding line bundle ${\cal L}_z$, and these bundles satisfy ${\cal
 L}_{z_1}\otimes {\cal L}_{z_2}\cong {\cal L}_{z_1z_2}$ and $g^*{\cal
 L}_z\cong {\cal L}_{g^*z}$. Moreover, since $H^1(\Z,\A(\C)^*)=0$, we~can
arrange for our cocycles to have trivial restriction along $\Z^n\subset
\langle 1,\tau\rangle^n\cong \pi_1(E^n)$. We thus obtain the following
description of line bundles on~$E^n$: given a symmetric integer matrix $Q$
and constants $C_1,\dots,C_n\in \C^*$, we~could consider the line bundle on~$(\C/\langle 1,\tau\rangle)^n$ with local sections given by functions on~$\C$
satisfying
\begin{gather*}
 f(z_1,\dots,z_{i-1},z_i+1,z_{i+1},\dots,z_n)= f(z_1,\dots,z_n),
 \\
 f(z_1,\dots,z_{i-1},z_i+\tau,z_{i+1},\dots,z_n)
 = C_i e \bigg({-}\sum_j Q_{ij}z_j\bigg)
 f(z_1,\dots,z_n),
\end{gather*}
where $e(x):=\exp\big(2\pi\sqrt{-1}x\big)$. These line bundles behave well under
tensor product, but in slightly odd ways under pullback; it turns out that
we should not quite trivialize the cocycle along $\Z^n$ in general, but
instead merely insist that it restrict to an appropriate morphism $\Z^n\to
\{\pm 1\}$. This leads us to define the line bundle ${\cal L}_{Q;\vec{C}}$ on~$(\C/\langle 1,\tau\rangle)^n$ as the sheaf with local sections consisting of
holomorphic functions satisfying
\begin{gather*}
f(\vec{z}+\vec{x}\tau+\vec{y})
=
\prod_{1\le i\le n} (-1)^{Q_{ii}(x_i+y_i)} C_i^{x_i}
\prod_{1\le i,j\le n} e(-Q_{ij} x_i(z_j+x_j\tau/2))f(\vec{z}),
\end{gather*}
for~$\vec{x},\vec{y}\in \Z^n$. Note that this does not quite solve the
problem as stated, since it gives multiple representatives for each line
bundle (multiplying $f$ by the nowhere vanishing entire function~$e(z_i)$
multiplies $C_i$ by $e(\tau)$), but still makes it straightforward to
control equivariant structures on~line bundles, as~well as some of the more
gerbe-like structures we need to consider below.

Although this suffices for many purposes, we~would like to have an {\em
 algebraic} solution to this problem. Not only is our construction below
essentially algebraic in nature, there are also some indications that it
may prove useful in later work to be able to consider versions defined over
finite fields. (In particular, see the use in
\cite{ArtinM/TateJ/VandenBerghM:1990} of finite field instances of the
Sklyanin algebra in proving the latter is Noetherian, of particular
interest given that we do not yet have a proof that our algebras are
Noetherian; see also unpublished work of Bezrukavnikov and Okounkov.) A
first step towards this is to observe that we are not {\em really}
interested in constructing things over a particular curve; rather, we~wish
to have constructions that apply to {\em all} curves. In~other words, what
we truly want to understand are line bundles on the $n$th fiber power
${\cal E}^n$ of the universal curve ${\cal E}$ over the moduli stack ${\cal
 M}_{1,1}$ of elliptic curves.

\begin{rem}
For those readers who may be unfamiliar with stacks, it is worth noting
that there is an equivalence between statements about line bundles on~${\cal E}^n$ (or ${\cal M}_{1,1}$) and statements about equivariant line
bundles on an appropriate scheme-with-group-action. This comes from the
fact that each of these stacks is a quotient stack. Indeed, a choice of a
nonzero holomorphic differential~$\omega$ on an elliptic curve on which $6$
is invertible\footnote{Something similar holds over $\Z$, but with~$\G_m$
 replaced by the group $(y,x)\mapsto \big(a^3y+bx+c,a^2x+d\big)$.} determines a
unique expression of the curve as a~Weierstrass curve: the compactification
of a curve $y^2=x^3+a_4x+a_6$, with differential~${\rm d}x/2y$. It follows
immediately that for any family of elliptic curves $E/S$ with~$6$
invertible on~$S$, there is a $\G_m$-torsor $T$ over $S$ such that $E_T$ is
given by such an equation: simply take $T$ to be the torsor of nonzero
holomorphic differentials! It follows that there is a~natural isomorphism
between the (symmetric monoidal) category of line bundles on~${\cal
 M}_{1,1}[1/6]$ (or ${\cal E}^n[1/6]$) and the category of~$\G_m$-equivariant line bundles on~$\Spec(\Z[1/6,a_4,a_6,1/\Delta])$ (or
the corresponding family of abelian varieties). Indeed, if we are given a
$\G_m$-equivariant line bundle on the family of~Weierstrass curves (or, for~${\cal E}^n$, on the appropriate family of abelian varieties), then for any
family of elliptic curves over $S$ (resp.~with $n$ additional sections), we~obtain an~induced $\G_m$-equivariant bundle on the $\G_m$-torsor $T$, and
this descends naturally to a line bundle on~$S$ (satisfying the requisite
compatibility conditions). Conversely, any line bundle on the stack
induces a line bundle on the corresponding family of Weierstrass curves,
and the compatibility conditions make this line bundle naturally
equivariant. This equivalence would allow one to~enti\-rely eliminate stacks
from the discussion below, at the cost of requiring some additional
bookkeeping in the arguments to keep track of equivariance.
\end{rem}

The theory of Jacobi forms
\cite{EichlerM/ZagierD:1985,GritsenkoVA:1994,GritsenkoVA:1988, KriegA:1996}
gives us an approach to this at the analytic level. Analytically, ${\cal
 E}^n$ is the quotient of~$\C^n\times \H$ by the appropriate action of~$\Z^{2n}\rtimes \SL_2(\Z)$, and thus again we may specify line bundles via
cocycles. This leads to the following definition: Given a~symmetric
integer matrix $Q$ (the ``level'') and an integer $w$ (the ``weight''), we~define the line bundle ${\cal L}_{Q,w}$ on the complex locus of~${\cal
 E}^n$ to be the sheaf with local sections consisting of functions
$f(z_1,\dots,z_n;\tau)$ such that
\begin{gather*}
\begin{array}{l}
f(\vec{z}+\vec{x}\tau+\vec{y};\tau)=
(-1)^{\vec{x}^tQ\vec{x}+\vec{y}^tQ\vec{y}} e(-\vec{x}^t Q (\vec{z}+\vec{x}\tau/2))f(\vec{z};\tau),
\\[.5ex]
f(\vec{z}/(c\tau+d);(a\tau+b)/(c\tau+d))=
(c\tau+d)^w e(c\vec{z}^{\,t}Q\vec{z}/2(c\tau+d))f(\vec{z};\tau),
\end{array}
\end{gather*}
with $\vec{x},\vec{y}\in \Z^n$ and $\left(\begin{smallmatrix}
 a&b\\c&d\end{smallmatrix}\right)\in \SL_2(\Z)$. (This differs from the usual notion
 of Jacobi form by virtue of our not imposing any condition at the cusp;
 we have also allowed $Q$ to have odd diagonal.)

\begin{lem}
 The line bundle ${\cal L}_{1,-1}(-[0])$ on the complex locus of~${\cal
 E}$ is trivial, where the divisor~$[0]$ is the image of the identity
 section~$0\colon {\cal M}_{1,1}\to {\cal E}$.
\end{lem}

\begin{proof}
 We first observe that the function
 \begin{gather*}
 \vartheta(z;\tau) :=
 \frac{(e(z/2)-e(-z/2))\prod_{1\le j} (1-e(j\tau+z))(1-e(j\tau-z))}
 {\prod_{1\le j} (1-e(j\tau))^2}
 \\ \phantom{ \vartheta(z;\tau){:}}
{} = \frac{\sum_{k\in 1/2+\Z} (-1)^{k-1/2} e(k z+k^2\tau/2)}
 {\sum_{k\in 1/2+\Z} (-1)^{k-1/2} k e(k^2\tau/2)}
 \end{gather*}
 is a global section of~${\cal L}_{1,-1}$. This follows immediately from
 the standard transformation law for Jacobi theta functions together with
 the transformation law for the Dedekind eta function~$e(\tau/24)\prod_{1\le j} (1-e(j\tau))$. (Note that in each case, the
 overall transformation law involves complicated arithmetic characters,
 but these turn out to cancel.) This is holomorphic on~$\C\times \H$, and
 for each $\tau\in \H$ is a nonzero function vanishing only on the lattice
 $\langle 1,\tau\rangle$; thus the corresponding section of~${\cal
 L}_{1,-1}$ has divisor~$[0]$, establishing triviality as required.
\end{proof}

\begin{rem}
 The function~$\vartheta(z;\tau)$ may be expressed in the standard
 multiplicative notation for theta functions in elliptic special function
 theory as $\vartheta(z;\tau) =
 \frac{-x^{-1/2}\theta_p(x)}{(p;p)_\infty^2}$, where $x = e(z)$,
 $p=e(\tau)$.
\end{rem}

It follows in particular that we can extend ${\cal L}_{1,-1}$ (or, rather,
an algebraic representative of its isomorphism class!) to the entirety of~${\cal E}$: simply take the line bundle $\sO_{\cal E}([0])$. Pulling this back
through a homomorphism $g\colon {\cal E}^n\to {\cal E}$, $(x_1,\dots,x_n)\mapsto
\sum_i g_i x_i$, gives an algebraic version of~${\cal L}_{g^t g,-1}$, and
thus by taking tensor products an algebraic version of~${\cal L}_{Q,w}$ in
general. Of course, there is a potential issue here of uniqueness. Up to
isomorphism, this is settled by the following.

\begin{prop}
 The Picard group of~${\cal E}^n$ is a canonically split extension of~$\Pic({\cal M}_{1,1})\cong \Z/12\Z$ by the N\'eron--Severi group of the
 generic fiber.
\end{prop}

\begin{proof}
 That $\Pic({\cal M}_{1,1})\cong \Z/12\Z$ is standard (see~\cite{FultonW/OlssonM:2010} for an extension to fairly general base
 changes), as~is the fact that we can represent the restrictions of such
 line bundles to the complex locus via sheaves of modular forms of given
 weight. (The reduction mod 12 then comes from the fact that the
 discriminant $\Delta(\tau)$ is a holomorphic form of weight 12, nowhere
 vanishing on the smooth locus.) Moreover, the identity section~${\cal
 M}_{1,1}\to {\cal E}^n$ gives rise to a splitting of the natural
 pullback morphism $\Pic({\cal M}_{1,1})\to \Pic({\cal E}^n)$. The~above
 construction moreover shows that the natural map from $\Pic({\cal E}^n)$
 to the N\'eron--Severi group of the generic fiber is surjective. (Note
 that since ${\cal M}_{1,1}$ is a stack, ``the generic fiber'' does not
 quite make sense, but it suffices to base change to some smooth curve
 covering ${\cal M}_{1,1}$, say by imposing a full level $3$ or full level
 $4$ structure.)

 It thus remains only to show that if ${\cal L}$ is a line bundle on~${\cal E}^n$ which on the generic fiber is algebraically equivalent to
 the trivial divisor, then every fiber of~${\cal L}$ is trivial, and thus
 ${\cal L}$ is the pullback of a line bundle on~${\cal M}_{1,1}$. Since
 an algebraically trivial line bundle on an abelian variety gives rise to
 a point of the dual variety, and ${\cal E}^n$ is principally polarized
 via the product polarization, we~find that ${\cal L}$ induces a section~${\cal M}_{1,1}\to {\cal E}^n$, and we need merely show that this is the
 identity section. The~elliptic curve $y^2+txy=x^3+t^5$ over $\C(t)$
 has trivial Mordell--Weil group, and thus the pullback of any section~${\cal M}_{1,1}\to {\cal E}^n$ to this elliptic curve is trivial. Since
 the corresponding map $\Spec(\C(t))\to {\cal M}_{1,1}$ is dominant, it
 follows that any section~${\cal M}_{1,1}\to {\cal E}^n$ is generically
 trivial, and thus everywhere trivial since this is a closed condition.
\end{proof}

Thus each line bundle on~${\cal E}^n$ is determined by its weight (the
restriction to the zero section as an element of~$\Pic({\cal M}_{1,1})\cong
\Z/12\Z$) and polarization (the class in the N\'eron--Severi group of~the
generic fiber). It will be convenient going forward to represent the
polarization as a~symmetric integer matrix $Q$ or as the corresponding quadratic
polynomial $\vec{z} Q \vec{z}^{\,t}/2$; the latter will be~particularly
convenient when we have assigned names to the coordinates in~${\cal E}^n$.

Let us thus choose for each symmetric integer matrix $Q$ and integer $w\in
\Z$ a line bundle ${\cal L}_{Q,w}$ (unique up to isomorphism) which
restricts to the bundle of weight $w$ on the identity section and induces
the symmetric endomorphism $Q$ on the generic fiber of~${\cal E}^n$. This
very nearly solves our problem of constructing line bundles with consistent
isomorphisms under tensor product and pullbacks, by virtue of the fact that
there are very few global units on~${\cal E}^n$. Indeed, since~${\cal
 E}^n$ is proper over ${\cal M}_{1,1}$, the global units on~${\cal E}^n$
are just the global units on~${\cal M}_{1,1}$, and these are easily seen to
consist precisely of elements of the form $\pm \Delta^l$ for~$l\in \Z$.
But in fact the same thing that gives us a splitting of the Picard group
gives us a natural way to rigidify things completely.

\begin{defn}
 For an integer $w$, a {\em weight $w$ trivialization} of a line bundle
 ${\cal L}$ on~${\cal E}^n$ is an~isomorphism
 \begin{gather*}
 0^*{\cal L}\cong (0^*\omega_{{\cal E}/{\cal M}_{1,1}})^w.
 \end{gather*}
\end{defn}

We now enhance our data as follows: ${\cal L}_{Q,w}$ is not just a line
bundle in the appropriate iso\-mor\-phism class; rather, it is such a line
bundle together with a choice of weight $w$ trivialization. We also insist
that ${\cal L}_{0,0}=\sO_{{\cal E}^n}$ with the obvious trivialization.
(Here we now take $w$ an integer rather than a class mod 12, as~trivializing $0^*\omega_{{\cal E}/{\cal M}_{1,1}}^{12}$ requires choosing
one of~$\Delta$ or $-\Delta$. This is not a significant issue, however,
and in fact should allow us to extend everything to the cusps of~${\cal
 M}_{1,1}$.)

\begin{thm}
 There is a family of natural isomorphisms
 \begin{gather*}
 {\cal L}_{Q_1,w_1}\otimes {\cal L}_{Q_2,w_2}
 \cong
 {\cal L}_{Q_1+Q_2,w_1+w_2},
 \end{gather*}
 agreeing with the obvious isomorphisms when $(Q_1,w_1)=(0,0)$ or
 $(Q_2,w_2)=(0,0)$, and, for~any linear transformation~$g\colon \Z^n\to \Z^m$
 $($with induced homomorphism $g\colon {\cal E}^n\to {\cal E}^m)$, a family of~natu\-ral isomorphisms
 \begin{gather*}
 g^*{\cal L}_{Q,w}\cong {\cal L}_{g^t Q g,w},
 \end{gather*}
 agreeing with the obvious isomorphism when~$g=1$. Moreover, these are
 compatible in the sense that any isomorphism
 \begin{gather*}
 \bigotimes_{1\le i\le n} g_i^*{\cal L}_{Q_i,w_i}
 \to
 {\cal L}_{\sum_i g_i^t Q_i g_i,\sum_i w_i}
 \end{gather*}
 constructed from these ingredients is invariant under permutations of the
 tensor factors and independent of the order in which the above
 isomorphisms are applied.
\end{thm}

\begin{proof}
 In either case, not only are the bundles isomorphic, but so are their
 pullbacks through~$0$ (e.g., since $0^*g^*=0^*$), and the respective
 trivializations actually induce a specific choice of isomorphism of the
 pullbacks. Since both isomorphisms are determined up to a global unit on~${\cal M}_{1,1}$, rigidifying the pullback suffices to rigidify the
 desired isomorphisms. Similarly, the compatibility conditions certainly
 hold up to a scalar factor, and pulling back through~$0^*$ shows that
 that scalar must be 1.
\end{proof}

To specify (meromorphic) sections of such line bundles, we~would like to
have an analogue of~the~Jacobi theta function. Here we have the following.

\begin{lem}
 The line bundle $\sO_{\cal E}([0])$ on~${\cal E}$ has a natural weight $-1$
 trivialization.
\end{lem}

\begin{proof}
 We need an isomorphism
 \begin{gather*}
 0^*\sO_{\cal E}([0])\otimes 0^*\omega_{{\cal E}/{\cal M}_{1,1}}
 \cong
 \sO_{{\cal M}_{1,1}},
 \end{gather*}
 or equivalently
 \begin{gather*}
 0^*\big(\omega_{{\cal E}/{\cal M}_{1,1}}([0])\big)
 \cong \sO_{{\cal M}_{1,1}},
 \end{gather*}
 but this is just adjunction.
\end{proof}

\begin{rem}
 Note that the isomorphism coming from adjunction simply takes a
 differential with simple pole at $0$ to its residue. In~particular,
 the natural isomorphism $[-1]^*\sO_{\cal E}([0])\cong \sO_{\cal E}([0])$ negates
 the trivialization.
\end{rem}

\begin{defn}
 Let ${\cal L}_{1,-1}$ be the chosen line bundle with trivialization on~${\cal E}$. Then the global section~$\vartheta\in \Gamma({\cal
 L}_{1,-1})$ is the image of~$1$ under the isomorphism $\sO_{\cal E}([0])\cong
 {\cal L}_{1,-1}$ respecting the trivialization.
\end{defn}

Note that the function~$\vartheta$ considered above was normalized so that
\begin{gather*}
\Res_{z=0} \frac{{\rm d}z}{\vartheta(z;\tau)} = 1,
\end{gather*}
and thus the two definitions essentially agree on the analytic locus.
(This is not {\em quite} correct, since the analytic definition of~${\cal
 L}_{1,-1}$ only gives something {\em isomorphic} to an algebraic bundle,
but the residue condition implies that there is a system of isomorphisms
between the algebraic versions of~${\cal L}_{Q,w}$ and their analytic
versions satisfying appropriate compatibility conditions along tensor
products and pullbacks and taking $\vartheta$ to~$\vartheta(;\tau)$.) This
gives us the following consequence, a very powerful way of constructing
functions on powers ${\cal E}^n$.

\begin{thm}\sloppy
 Let $c_{ij}$, $1\le i\le l$, $1\le j\le n$ and $m_i$, $1\le i\le l$ be
 integers such that \mbox{$\sum_{1\le i\le l} m_i = 0$} and $\sum_{1\le i\le l}
 m_i c_{ij}c_{ik}=0$ for~$1\le j, k\le n$. Then there is a $($unique$)$
 meromorphic function on~${\cal E}^n$, defined on every fiber, which on
 the complex locus restricts to~$ \prod_{1\le i\le l} \vartheta\big(\sum_j c_{ij}z_j;\tau\big)^{m_i}$.
\end{thm}

\begin{proof}
 We may view each $c_i$ as a morphism ${\cal E}^n\to {\cal E}$, and
 find that
 $
 \prod_{1\le i\le l} (c_i^*\vartheta)^{m_i}
 $
 is a meromorphic section of the line bundle
 \begin{gather*}
 \bigotimes_{1\le i\le l} (c_i^* {\cal L}_{1,-1})^{m_i}
 \cong
 {\cal L}_{\sum_i m_i c_i^t c_i,-\sum_i m_i}
 \cong
 {\cal O}_{{\cal E}^n},
 \end{gather*}
 with both maps canonical. In~other words, the given product of sections
 of line bundles determines a meromorphic function on~${\cal E}^n$, the
 divisor of which manifestly has no vertical components.
\end{proof}

\begin{rems}
 By a very mild abuse of notation, we~will write the functions constructed
 in~this way as
 $
 \prod_{1\le i\le l} \vartheta(\sum_j c_{ij}z_j)^{m_i}$,
 and similarly for the analogous meromorphic sections of line bundles
 ${\cal L}_{Q,w}$.
\end{rems}

\begin{rems}
 Note that since $[-1]^*$ negates the natural trivialization of~$\sO_{\cal
 E}([0])$, we~have the identity $\vartheta(-z)=-\vartheta(z)$.
\end{rems}

\begin{rems}
 Of course, if we multiply by a suitable power of~$\Delta$, we~can
 construct similar functions under the weaker assumption that $\sum_i m_i$
 is a multiple of~$12$. On~the other hand, with the constraints as given,
 there is a natural limit at the cusp of~${\cal M}_{1,1}$, namely the
 rational function
 \begin{gather*}
 \prod_j x_j^{-\sum_i m_ic_{ij}/2}
 \prod_{1\le i\le l} \bigg(1-\prod_j x_j^{c_{ij}}\bigg)^{m_i}
 \end{gather*}
 on~$(\C^*)^n=e(\C)^n$ (or more generally $\G_m^n$). Note that $\sum_i
 m_ic_{ij}$ is even since $\sum_i m_i c_{ij}^2=0$, so this is indeed
 well-defined. This gives a useful method for sanity-checking
 calculations, by~verifying that the result is correct in this limit. One
 can also take a limit as all of the variables approach~0; this is
 ill-defined, but blowing up the~0 section gives the function
 \begin{gather*}
 \prod_{1\le i\le l} \bigg(\sum_j c_{ij}y_j\bigg)^{m_i}
 \end{gather*}
 on the exceptional $\P^{n-1}$ (as long as we avoid the finitely many
 characteristics in which one of the factors is identically 0).
\end{rems}

It is worth noting that there is an alternate approach to the theorem
which, while it does not help us deal with line bundles, is more powerful
in one important respect: it gives a reasonable algorithm for {\em
 evaluating} such functions at specific points (i.e., on~$n$-tuples of
points of specific elliptic curves, say over a finite field). We choose
(to make the appropriate induction work) a~differential $\omega$ on~$E$,
allowing us to eliminate the constraint $\sum_i m_i=0$. The~corresponding
trivialization of the line bundle $\sO_{\cal E}(-d[0])$ can then be computed as
follows: choose any uniformizer $u$ at 0 such that $\omega/u$ has residue
1, and then for any function~$f$ with multiplicity $d$ at 0, express
$f=c(f,\omega) u^d+O(u^{d+1})$ to obtain a ``leading coefficient''
$c(f,\omega)$. It is easy to see that this leading coefficient is
independent of the choice of~$u$, and that $c(f,\omega)$ scales in the
appropriate way with~$\omega$.

Now, we~can characterize the function~$ \prod_{1\le i\le l}
\vartheta\big(\sum_j c_{ij}z_j\big)^{m_i} $ up to a scalar multiple by vie\-wing it
as a function of~$z_n$ with specified divisor. If~there are no factors
depending only on~$z_n$, then the value of the function along $z_n=0$ is
then a function of the same form, which can be computed by induction,
giving us the requisite scale factor and letting us plug in the specific
value of~$z_n$ desired. If~the only factors depending only on~$z_n$ are
powers of~$\vartheta(z_n)$, then we use the leading coefficient instead and
proceed as before. (The correct leading coefficient is computed by
removing powers of~$\vartheta(z_n)$ and setting $z_n=0$.) Note that in
general we can compute the leading coefficient as long as every factor
$\vartheta(kz_n)^m$ that arises has $k$ invertible (since~$\vartheta(kz)$
has leading term $ku$ in characteristic 0). Thus if we can construct
functions of the form $\frac{\vartheta(k z)}{\vartheta(z)^{k^2}}$ directly, we~could eliminate those factors before computing the leading coefficient.
Since $\vartheta(kz)=-\vartheta(-kz)$, we~reduce to the case $k>1$.
Divisor considerations then tell us that such a function must be
proportional to the $k$-division polynomial (as defined in
\cite[Example~3.7]{SilvermanJH:2009}), and we find in fact that the $k$-division
polynomial has the correct leading term in characteristic 0, so gives the
desired function on general $E$.

Note that although we are writing $\vartheta(z)$ as a function, the fact
that it is only a section of a~line bundle means that we must be somewhat
cautious when specializing. In~particular, consider the ratio
$\vartheta(u+z)/\vartheta(u)$, a section of the line bundle with
polarization~$z^2/2+uz$ and weight 0. The~restriction of this line bundle
to the hypersurface $z=0$ is trivial (and in a natural way given our global
choices of trivializations), and thus $\vartheta(u+z)/\vartheta(u)$
restricts to a function on this hypersurface, which we can verify is equal
to 1. There is, however, an issue in the analytic setting, as~the
algebraic hypersurface $z=0$ corresponds to the analytic hypersurfaces
$z=x+y\tau$, $x,y\in \Z$, and $\vartheta(u+x+y\tau;\tau)/\vartheta(u;\tau)$
is of course nontrivial for most values of~$x$, $y$. (For $z=0$, there is
still no issue, naturally.)

In contrast, if we restrict the ratio $\vartheta(u+mz)/\vartheta(u)$ to the
hypersurface $mz=0$, then there are difficulties even algebraically. The~difficulty here is that although the restriction of the line bundle is
trivial on every fiber over ${\cal E}[m]$, it is not {\em canonically}
trivial, and thus the restriction could be a nontrivial line bundle on~${\cal E}[m]$, and even if trivial, the given section need not restrict~to~1.

Indeed, let ${\mathfrak q}\in \mu_2(E[2])$ be the character of the natural
$[-1]$-equivariant structure on~${\cal L}_{1,-1}$ (i.e., such that the
action on the fiber at $0$ is trivial) on the fixed subscheme $E[2]$. Note
that ${\mathfrak q}(0)=1$ and when the characteristic is not 2 is $-1$ at the
remaining $2$-torsion points; in general, it is a $\mu_2$-valued quadratic
form such that the induced bilinear form is the Weil pairing.

\begin{prop}\label{prop:pullback_simplification_i}
 The restriction of the function
 \begin{gather*}
 \frac{\vartheta(u-z)\vartheta(v-z)\vartheta(u+v+z)}
 {\vartheta(u+z)\vartheta(v+z)\vartheta(u+v-z)}
 \end{gather*}
 to the hypersurface $2z=0$ is given by ${\mathfrak q}(z)$.
\end{prop}

\begin{proof}
 To compute ${\mathfrak q}$, we~need merely express ${\cal L}_{1,-1}$ as
 ${\cal O}(D)$ for a suitable divisor $D$ disjoint from the $2$-torsion,
 at which point ${\mathfrak q}$ is the restriction to~$E[2]$ of the unique
 function~$f$ with divisor $D-[-1]^*D$ such that $f(0)=1$. Taking
 $D=[u]+[v]-[u+v]$ gives the desired result.
\end{proof}

We note the following fact useful for simplifying products of
such factors.

\begin{prop}\label{prop:pullback_simplification_ii}
 For any integer $m\ge 1$, the restriction of the function
 \begin{gather*}
 \frac{\vartheta(u+mz)\vartheta(u+w)\vartheta(v)\vartheta(v+w+mz)}
 {\vartheta(u)\vartheta(u+w+mz)\vartheta(v+mz)\vartheta(v+w)}
 \end{gather*}
 to the hypersurface $mz=0$ is $1$.
\end{prop}

\begin{proof}
 Indeed, the original function is a rational function on~${\cal E}^4$,
 which may be described as the unique elliptic function of~$u$ with the
 appropriate divisor and taking the value $1$ at $u=v$. On~the
 hypersurface $mz=0$, this elliptic function of~$u$ has divisor $0$, and
 is thus the function~$1$.
\end{proof}

\begin{rem}
 It follows that the restriction to~$mz=0$ of
 \begin{gather*}
 \frac{\vartheta(u)\vartheta(u+v+mz)}{\vartheta(u+mz)\vartheta(u+v)}
 \end{gather*}
 is independent of~$u$ and in a suitable sense a homomorphism in~$v$. Of
 course, again, it is only a~section of a line bundle which is isomorphic
 to the pullback of a line bundle on~${\cal E}[m]$, but not canonically
 so. In~the analytic setting, the corresponding ratio as a function of~$v$ indeed gives a homomorphism from the universal cover; the specific
 homomorphism depends on the choice of representative of the given
 $m$-torsion point, and for some $m$-torsion points cannot be made
 trivial.
\end{rem}

Note that since both lemmas are stated in terms of actual functions on the
moduli stack, we~may replace the respective variables by arbitrary
homomorphisms ${\cal E}^n\to {\cal E}$.

In addition to line bundles and their sections, we~will also need some
gerbe-ish structures. The~objects we want to consider are ``equivariant
gerbes'' (more precisely, equivariant structures on the trivial gerbe; all
equivariant gerbes considered in the present work will have trivial
underlying gerbe): a system of line bundles ${\cal Z}_g$ associated to~$g\in G\subset \GL_n(\Z)$ along with morphisms $\zeta_{g,h}\colon {\cal
 Z}_g\otimes \big(g^{-1}\big)^*{\cal Z}_h\cong {\cal Z}_{gh}$ making all diagrams
\begin{gather*}
\begin{CD}
{\cal Z}_{g_1}\otimes \big(g_1^{-1}\big)^*{\cal Z}_{g_2}
 \otimes \big((g_1g_2)^{-1}\big)^*{\cal Z}_{g_3}
@>>>
{\cal Z}_{g_1g_2}
 \otimes \big((g_1g_2)^{-1}\big)^*{\cal Z}_{g_3}\\
@VVV @VVV\\
{\cal Z}_{g_1}\otimes \big(g_1^{-1}\big)^*{\cal Z}_{g_2g_3}
@>>>
{\cal Z}_{g_1g_2g_3}
\end{CD}
\end{gather*}
commute. This, of course, is easy to construct: given any 1-cocycle for
$G$ valued in pairs $(Q,w)$, we~may take ${\cal Z}_g$ to be ${\cal
 L}_{Q_g,w_g}$ and $\zeta_{g,h}$ to be the natural isomorphism
\begin{gather*}
\zeta_{g,h}\colon\ {\cal L}_{Q_g,w_g}\otimes \big(g^{-1}\big)^*{\cal L}_{Q_h,w_h}
\cong
{\cal L}_{Q_g+g^{-t}Q_hg^{-1},w_g+w_h}
=
{\cal L}_{Q_{gh},w_{gh}}.
\end{gather*}

The situation becomes more complicated if we want to also construct
meromorphic sections of such gerbes: i.e., a system of meromorphic sections
$z_g\in {\cal Z}_g$ such that $\zeta_{g,h}\big(z_g\otimes
\big(g^{-1}\big)^*z_h\big)=z_{gh}$. In~this case, we~do not have a completely general
solution, but there is an important special case.

Consider first the case $G=\Z$, with generator acting on~${\cal E}^2$ by
$(z,q)\mapsto (z+q,q)$. There is a natural equivariant gerbe with
meromorphic section such that ${\cal Z}_1={\cal L}_{z^2/2,-1}$ and $z_1 =
\vartheta(z)$. Indeed, we~then find more generally that ${\cal Z}_k$ is
the natural line bundle of weight $-k$ with polarization
\begin{gather*}
(z,q) Q (z,q)^t/2 = k\frac{z^2}{2} + \frac{k(k-1)}{2} qz +
\frac{k(2k-1)(k-1)}{6} \frac{q^2}{2},
\end{gather*}
and
\begin{gather*}
z_k =\vartheta(z;q)_k:=
\begin{cases} \prod\limits_{0\le i<k} \vartheta(iq+z),& k\ge 0,
\\
\prod\limits_{1\le i\le -k} \vartheta(-iq+z)^{-1}, & k\le 0.
\end{cases}
\end{gather*}
We will denote this meromorphic section of an equivariant gerbe by
$\Gamq(z)$, which we refer to as an ``elliptic Gamma function''. Of
course, this is even less a function than $\vartheta$, but in the analytic
setting one can indeed replace $\Gamq(z)$ by a suitable meromorphic
solution of the functional equation~$\Gamq(q+z) = \vartheta(z)\Gamq(z)$.
Of course, this only determines $\Gamq$ up to multiplication by invertible
\mbox{$q$-periodic} functions, so the resulting meromorphic function is far from
unique. One~such solution (for $q$ in the upper half-plane) is
\begin{gather*}
\Gamq(z;\tau)
:=
\bigg({-}\prod_{1\le j} (1-e(j\tau))^2\bigg)^{-z/q}
e(-z(z-q)/4q)
\prod_{0\le j,k}
 \frac{1-e((j+1)\tau+(k+1)q-z)}
 {1-e(j\tau+kq+z)};
\end{gather*}
this depends on a choice of~$\log\big({-}\prod_{1\le j} (1-e(j\tau))^2\big)$, but this
will not matter for our purposes. Here the double product is just the
usual elliptic Gamma function~\cite{RuijsenaarsSNM:1997}.

More generally, for~given morphisms $q,z\colon {\cal E}^n\to {\cal E}$, we~may
pull this formal symbol back to~${\cal E}^n$. This is tricky to deal with
in complete generality, but if we fix $q$ and vary $z$, we~may consider
general products
\begin{gather*}
\prod_{1\le i\le l} \Gamq(\vec{\alpha}_i\cdot \vec{z})^{m_i}.
\end{gather*}
If the corresponding element $\sum_i m_i [\vec{\alpha}_i]$ of~$\Z[\Hom({\cal E}^n,{\cal E})/q]$
is trivial, then this formal product may be resolved into a product of~$\vartheta$ functions using the functional equation. Moreover, since for
each congruence class the corresponding subproduct may be pulled back from
${\cal E}^2$, we~find that the resulting product of~$\vartheta$ functions
is independent of any choices that may have been made.

More generally, if $G\subset \GL_n(\Z)$ fixes $q$ and the element $\sum_i
m_i [\vec{\alpha}_i]\in \Z[\Hom({\cal E}^n,{\cal E})/q]$, then the ratio of
the formal product and its pullback under $g\in G$ will always resolve to a
product of~$\vartheta$ functions, and thus gives an
equivariant-gerbe-with-meromorphic-section which we refer to as the {\em
 coboundary} of the formal symbol. Note in particular that the hypothesis
{\em always} holds for the (translation) subgroup of~$\GL_n(\Z)$ that acts
trivially on both $q$ and the quotient $\Hom({\cal E}^n,{\cal E})/q$; in~the cases of interest (affine Weyl groups), the intersection of~$G$ with
this subgroup will be cofinite in both groups, and it will be relatively
straightforward to check that individual instances give rise to sections of
gerbes.

We should caution the reader that there is a mild subtlety when it comes to
the reflection principle. Although the product $\Gamq(z)\Gamq(q-z)$
corresponds to the trivial equivariant gerbe (it~has polarization~$0$ and
weight $-1$ (see below), both of which are invariant under $z\mapsto z+q$),
the corresponding meromorphic section is not quite trivial: one finds that
$\Gamq(z)\Gamq(q-z)$ is negated by the translation~$z\mapsto z+q$.

On the other hand, the multiplication principle $\Gamq(z) = \prod_{0\le
 j<k} \Gamm{kq}(z+jq)$ for integer $k>0$ does work, as~both sides
truly do correspond to the same gerbe-with-section. This also works for
negative $k$: $\Gamq(z) = \prod_{1\le j\le -k} \Gamm{kq}(z-jq)^{-1}$.
Using this, one can make sense of products of~elliptic Gamma functions with
varying $q$ as long as the different $q$ that appear have a common
multiple. Note in particular the special case $\Gamq(z) =
\Gamm{-q}(z-q)^{-1}$.

\looseness=-1 We will of course want to know the polarizations and weights of the line
bundles associated to a given such gerbe section, which reduces to knowing
the polarization and weight when such a product resolves to a product of~$\vartheta$ functions. There is, in fact, a very simple bookkeeping
procedure for determining this. Define the polarization of~$\Gamq(z)$ to
be $\frac{z(z-q)(2z-q)}{12q}$, and the weight of~$\Gamq(z)$ to be
$-\frac{z}{q}$, extended to products of pullbacks in the obvious way. By
this definition, the polarization of the formal product
$\Gamq(q+z)/\Gamq(z)$ is $z^2/2$ and the weight is $-1$, agreeing with the
polarization and weight of~$\vartheta(z)$. It follows more generally that
the polarization and weight of a product of powers of elliptic Gamma
functions is consistent with the usual notion whenever the product resolves
to a product of~$\vartheta$ functions. (In particular, if the product
resolves and the polarization and weight are trivial, then it resolves to
an honest function on~${\cal E}^n$.) Of course, when considering the
associated gerbe-with-section, only the $z$-dependent portion of the
polarization matters.

One should note here that not every cocycle valued in pairs $(Q,w)$ comes
from a product of~$\Gamq$ symbols; for instance, any product of symbols
$\Gamq(az+bq)$ with trivial polarization has weight of the form $12cz/q+d$,
so that every line bundle in the coboundary has weight a multiple of~$12$.
(There are also some additional parity issues for $n>1$.) Of course, it is
conceivable that the cocycles violating these congruence conditions do not
have {\em any} consistent family of meromorphic sections.

Note that if one wishes to convert a product of standard elliptic Gamma
functions into a~product of symbols $\Gamq$, one must first use the
reflection principle to eliminate appearances of~$p$ from the arguments
(which, if there is a balancing condition that involves $p$ will
require shifting some of the variables by $p$ first). If the result has
$\sum_i m_i \big(\sum_i \vec{\alpha}_i\cdot z_i\big)^2=\sum_i m_i \big(\sum_i
\vec{\alpha}_i\cdot z_i\big)=0$, then replacing the standard Gamma functions by
the explicit meromorphic solution given above for the functional equation
of~$\Gamq$ will have no effect on the resulting meromorphic function.

Since we now have a method for constructing equivariant line bundles on~${\cal E}^n$, it is natural to ask about the corresponding representations
of~$G$ on global sections; in particular, we~will want to understand the
space of~$G$-invariant global sections. It turns out that if we want to
extend the standard analytic approach to this question to algebraic curves, we~will need to extend the above construction to cover certain abelian
varieties {\em isogenous} to~${\cal E}^n$.

To be precise, for~a positive integer $N$ let ${\cal X}_0(N)$ denote the
moduli stack of elliptic curves~$E_1$ equipped with a cyclic $N$-isogeny
$\phi\colon E_1\to E_N$, with universal curves denoted by ${\cal E}={\cal E}_1,
{\cal E}_N$. (Here ``cyclic'' is in the sense of
\cite{KatzNM/MazurB:1985}; in particular, note that in characteristic $p$,
{\em any} isogeny of degree $p^k$ is cyclic.) For each divisor $d|N$,
there is a corresponding factorization~$\phi=\phi_{N,d}\circ \phi_{d,1}$,
where $\phi_{d,1}\colon{\cal E}_1\to {\cal E}_d$, $\phi_{N,d}\colon{\cal E}_d\to
{\cal E}_N$ are cyclic isogenies of degrees $d$ and $N/d$ respectively. For~$d=p$ prime, this is constructed as follows: on the locus where the
$p$-part of~$\ker\phi$ is \'etale, ${\cal E}_p$ is the quotient by the
$p$-torsion of~$\ker\phi$; on the complementary locus, where the $p$-part
of~$\ker\phi$ is nonreduced, $\phi_{1,p}$ is the Frobenius isogeny. (Per
\cite{KatzNM/MazurB:1985}, this rule gives the correct limit to make~${\cal
 E}_p$ a flat family.) In either case, $\ker\phi_{1,p}\subset \ker\phi$,
and thus $\phi$ factors as required, and we find that $\phi_{p,N}$ is again
cyclic. We thus have induced factorizations for every $d|N$, and it is
easy to see that they are all compatible. More precisely, for~any pair
$d_1|d_2|N$, we~obtain an isogeny $\phi_{d_2,d_1}\colon{\cal E}_{d_1}\to {\cal
 E}_{d_2}$, and $\phi_{d_3,d_2}\circ \phi_{d_2,d_1}\cong \phi_{d_3,d_1}$.
We also, of course, have similarly compatible isogenies
$\phi_{d_1,d_2}\colon{\cal E}_{d_2}\to {\cal E}_{d_1}$ obtained by dualizing
$\phi_{d_2,d_1}$.

\begin{lem}
 For any $d_1,d_2|N$, we~have $\Hom({\cal E}_{d_1},{\cal E}_{d_2})\cong \Z$, generated
 by the composition
 \begin{gather*}
 \phi_{d_2,d_1}:=\phi_{d_2,\gcd(d_1,d_2)}\circ \phi_{\gcd(d_1,d_2),d_1}
 =\phi_{d_2,\lcm(d_1,d_2)}\circ \phi_{\lcm(d_1,d_2),d_1}.
 \end{gather*}
\end{lem}

\begin{proof}
 Let $d_0=\gcd(d_1,d_2)$ and $d_3=\lcm(d_1,d_2)$, and let $f\colon{\cal
 E}_{d_1}\to {\cal E}_{d_2}$ be any homomorphism. Then
 $\phi_{d_0,d_2}\circ f\circ \phi_{d_1,d_0}$ is an endomorphism of~${\cal
 E}_{d_0}$, so must be multiplication by some integer. By degree
 considerations, that integer must be a multiple of~$d_2/d_0$ and
 $d_1/d_0$, and thus (since these are relatively prime) of~$d_1d_2/d_0^2=d_3/d_0$. Since $\phi_{d_0,d_2}\circ \phi_{d_2,d_1}\circ
 \phi_{d_1,d_0}=\big[d_1d_2/d_0^2\big]$, the first claim follows. We moreover
 find that
 \begin{gather*}
 \phi_{d_0,d_2}\circ \phi_{d_2,d_3}\circ \phi_{d_3,d_1}\circ \phi_{d_1,d_0}
 =
 \phi_{d_0,d_3}\circ \phi_{d_3,d_0}
 =
 [d_3/d_0],
 \end{gather*}
 from which the other factorization follows.
\end{proof}

For any sequence of divisors $d_i$ of~$N$, we~may consider the
corresponding fiber product of~cur\-ves ${\cal E}_{d_i}$ over ${\cal X}_0(N)$;
we will generally omit ${\cal X}_0(N)$ from the product notation.
As in the $N=1$ case, morphisms between such products may be expressed as
matrices, with the only difference being that the $ij$ entry is now a multiple
of the natural isogeny $\phi_{d_i,d_j}$.

\begin{prop}
 For any $d_1,d_2|N$, there is an isomorphism
 \begin{gather*}
 {\cal E}_{d_1}\times {\cal E}_{d_2}\cong {\cal E}_{\gcd(d_1,d_2)}\times {\cal E}_{\lcm(d_1,d_2)}.
 \end{gather*}
\end{prop}

\begin{proof}
 Let $d_0=\gcd(d_1,d_2)$, $d_3=\lcm(d_1,d_2)$, and choose $a,b$ so that
 $ad_1-bd_2=d_0$. We then easily verify that the morphisms
 \begin{gather*}
 \begin{pmatrix}
 a\phi_{d_0,d_1} &b\phi_{d_0,d_2}\\
 \phi_{d_3,d_1} &\phi_{d_3,d_2}
 \end{pmatrix}
\! \colon\ {\cal E}_{d_1}\times {\cal E}_{d_2}\to {\cal E}_{d_0}\times {\cal E}_{d_3}
 \end{gather*}
 and
 \begin{gather*}
 \begin{pmatrix}
 \phi_{d_1,d_0} & -b(d_2/d_0)\phi_{d_1,d_3}\\
 -\phi_{d_2,d_0}& a(d_1/d_0)\phi_{d_2,d_3}
 \end{pmatrix}
\! \colon\ {\cal E}_{d_0}\times {\cal E}_{d_3}\to {\cal E}_{d_1}\times {\cal E}_{d_2}
 \end{gather*}
 are inverses, giving the desired isomorphism.
\end{proof}

It follows immediately that any product of curves ${\cal E}_d$ is isomorphic to
one of the form $\prod_i {\cal E}_{d_i}$, where $1|d_1|\cdots|d_n|N$.

If we attempt to repeat our $N=1$ construction for general $N$, we~encounter two difficulties. The~first is that for each $d|N$, we~may
obtain a line bundle on~${\cal X}_0(N)$ by taking the fiber at 1 of the
sheaf of differentials on~${\cal E}_d$, but these line bundles are not quite the
same. If $E_1\to E_d$ is an~{\em \'etale} isogeny, there is no problem:
$\phi_{1,d}^*$ induces an isomorphism of~$\omega_{E_d}$ and
$\omega_{E_1}$, so in particular of their fibers over 1. However, if
$E_1\to E_d$ is inseparable, then $\phi_{1,d}^*$ actually annihilates
$\omega_{E_d}$. As~a~result, $\omega_{E_d}|_0$ and $\omega_{E_1}|_0$
differ by a linear combination of components of the fibers of~${\cal
 X}_0(N)$ over primes dividing $d$. Of course, if all we want to do is
construct {\em line bundles}, this is not an~issue, but this does mean that
the construction of {\em functions} in this way will be nontrivial.

The more serious difficulty is that it is no longer the case that $\prod_i
{\cal E}_{d_i}\to {\cal X}_0(N)$ has only the trivial section. Indeed, we~have the following.

\begin{lem}
 Let $N$ be a positive integer, and $d|N$. Then the group of
 sections of~${\cal E}_d$ over~${\cal X}_0(N)$ consists entirely of~$2$-torsion. If both $d$ and $N/d$ are odd, the group is trivial; if
 precisely one is even, it has rank~$1$, and if both are even, it has rank~$2$.
\end{lem}

\begin{proof}
 Since the stack ${\cal X}_0(N)$ has nontrivial stabilizers, any section
 of~${\cal E}_d$ must be invariant under the action of the stabilizer, and is
 thus preserved by $[-1]$, so is $2$-torsion. The~structure of the
 $2$-torsion of~${\cal E}_d$ then follows by considering the image of the
 corresponding congruence subgroup in~$\SL_2(\Z/2\Z)$.
\end{proof}

\begin{rem}
 For more general level structures, it follows from~\cite[Theorem~5.5]{ShiodaT:1972} that for any subgroup $\Gamma\subset
 \SL_2(\Z/N\Z)$, the group of sections over the corresponding quotient
 stack ${\cal X}(N)/\Gamma$ is $N$-torsion, and isomorphic to the subgroup
 $\big((\Z/N\Z)^2\big)^\Gamma$.
\end{rem}

If $N$ is odd, we~may thus conclude as before that the pullback of~$\sO_{{\cal E}_{d_1}}([0])$ through~$\phi_{d_1,d_2}$ is indeed fiberwise
isomorphic to~\smash{$\sO_{{\cal E}_{d_2}}([0])^{\deg(\phi_{d_1,d_2})}$}. However,
if $N$ is even, this is no longer the case; indeed, the pullback of~$\sO_{E}([0])$ through a $2$-isogeny $E'\to E$ is the tensor product of~$\sO_{E'}(2[0])$ by the corresponding $2$-torsion line bundle.

It turns out that we can fix this, at the cost of imposing some additional
level structure. Let ${\cal X}_0(2N,2)$ be the slight reinterpretation of
the stack ${\cal X}_0(4N)$ obtained by dividing all of the subscripts of
the isogenous curves by $2$; that is ${\cal X}_0(2N,2)$ classifies cyclic
$4N$-isogenies ${\cal E}_{1/2}\to {\cal E}_{2N}$ in terms of the curve~${\cal E}_1$. Now, it follows from the lemma that for each integer~$d|N$,
the curve~${\cal E}_d$ has full $2$-torsion over ${\cal X}_0(2N,2)$; in
addition to the generators of the kernels ${\cal E}_d\to {\cal E}_{d/2}$
and ${\cal E}_d\to {\cal E}_{2d}$, it also has their sum, which we denote
by $\sigma_d$. We thus obtain a~line bundle $\hat{\cal L}_{1,{\cal
 E}_d}:=\sO_{{\cal E}_d}([\sigma_d])\otimes 0^*\sO_{{\cal
 E}_d}([\sigma_d])^{-1}$ on~${\cal E}_d$ with trivial fiber over $0$.
Note that again~$0^*\sO_{{\cal E}_d}([\sigma_d])$ is a nontrivial line
bundle, though it is actually far better-behaved than $0^*\sO_{{\cal
 E}_d}([0])$. Indeed, $0^*\sO_{{\cal E}_d}([\sigma_d])$ is trivial away
from the locus where $\sigma_d=0$. This can only happen in characteristic~2,
and only when the $4$-isogeny ${\cal E}_{d/2}\to {\cal E}_{2d}$ is
multiplication by 2 (corresponding to one of the three components of~${\cal
 X}_0(2,2)$ in characteristic~2).

The merit of using this $2$-torsion point is that it makes everything
compatible.

\begin{lem}
 If $d_1|d_2|N$ then there are natural isomorphisms
 \begin{gather*}
 \phi_{d_2,d_1}^*\hat{\cal L}_{1,{\cal E}_{d_2}}\cong \hat{\cal L}_{1,{\cal E}_{d_1}}^{d_2/d_1},
 \\
 \phi_{d_1,d_2}^*\hat{\cal L}_{1,{\cal E}_{d_1}}\cong \hat{\cal L}_{1,{\cal E}_{d_2}}^{d_2/d_1}.
 \end{gather*}
\end{lem}

\begin{proof}
 The two claims are equivalent via the (Atkin-Lehner) involution~${\cal
 X}_0(2N,2)\cong {\cal X}_0(2N,2)$ that replaces ${\cal E}_{1/2}\to {\cal E}_{2N}$ by
 its dual, so it suffices to prove the first. Since the fibers at the
 origin of the two bundles are trivial, it suffices to show that the line
 bundles are isomorphic on each fiber over ${\cal X}_0(2N,2)$, at which
 point we may take the isomorphism respecting the fibers at~the origin.
 We may also, for~convenience, reduce to the case $d_1=1$, $d_2=N$.

 Being isomorphic is a closed condition, so we may exclude primes dividing
 $4N$, and in particular assume that the isogenies are all separable. On~a given such curve, the line bundle $\phi_{N,1}^*\hat{\cal
 L}_{1,{\cal E}_N}$ is represented by the divisor
 \begin{gather*}
 \sum_{x\in \phi_{N,1}^{-1}(\sigma_N)} [x],
 \end{gather*}
 and thus we need to show
 \begin{gather*}
 \sum_{x\in \phi_{N,1}^{-1}(\sigma_N)} x = N\sigma_1,
 \end{gather*}
 or equivalently
 \begin{gather*}
 \sum_{x\in \phi_{N,1}^{-1}(\sigma_N-\phi_{N,1}(\sigma_1))} x = 0.
 \end{gather*}
 If $N$ is odd, then $\phi(\sigma_N)=\sigma_1$, and this sum becomes
 \begin{gather*}
 \sum_{x\in \ker\phi_{N,1}} x = 0
 \end{gather*}
 as required. Otherwise, we~may factor through~${\cal E}_2$ and thus reduce to
 the case $N=2$. Then we find that $\phi_{2,1}\sigma_1$ is the
 generator of the kernel of~$\phi_{1,2}$, and thus
 $\sigma_2-\phi_{2,1}\sigma_1$ is the generator of~the~kernel of~$\phi_{2,4}$. It follows that the preimage of~$\sigma_2-\phi_{2,1}\sigma_1$ consists of~the~two generators of the
 kernel of~$\phi_{1,4}$, and these add to 0 as required.
\end{proof}

\begin{rem}
 It is worth noting here that there are additional curves in the isogeny
 class of~${\cal E}_1$, since after all each of the $2$-torsion points
 $\sigma_d$ itself determines a $2$-isogeny. It is unclear whether we can
 extend the above family of line bundles to the other curves arising in
 this way.
\end{rem}

\begin{cor}
 For any integer $a$, $[a]^*\hat{\cal L}_{1,{\cal E}_d}\cong \hat{\cal L}_{1,{\cal E}_d}^{a^2}$.
\end{cor}

\begin{proof}
 This is clearly true for $a=0$ and $a=-1$, so we may assume $a>0$. Over
 ${\cal X}_0(2aN,2)$, we~have $[a]=\phi_{d,ad}\circ \phi_{ad,d}$, so that
 the claim follows immediately from the lemma. Since the isomorphism is
 natural, it descends to~${\cal X}_0(2N,2)$ as required.
\end{proof}

Since $\hat{\cal L}_{1,{\cal E}_d}$ represents the standard principal
polarization on~${\cal E}_d$, we~also have the following.

\begin{lem}
 For any $d|N$, the bundle
 \begin{gather*}
 [x_1+x_2]^*\hat{\cal L}_{1,{\cal E}_d}
 \otimes
 [x_1]^*\hat{\cal L}_{1,{\cal E}_d}^{-1}
 \otimes
 [x_2]^*\hat{\cal L}_{1,{\cal E}_d}^{-1}
 \end{gather*}
 is naturally isomorphic to the Poincar\'e bundle ${\cal P}_{{\cal E}_d}$ on~${\cal E}_d\times {\cal E}_d$.
\end{lem}

\begin{thm}
 For any sequence $1|d_1|\cdots|d_n|N$, suppose ${\cal L}_1$ and ${\cal
 L}_2$ are two line bundles on~$\prod_i {\cal E}_{d_i}$ obtained as
 tensor products of pullbacks of bundles $\hat{\cal L}_{1,{\cal E}_d}$
 through morphisms \linebreak \mbox{$\prod_i {\cal E}_{d_i}\to {\cal E}_d$. If ${\cal
 L}_1$} and ${\cal L}_2$ represent the same polarization, then they are
 naturally isomorphic.
\end{thm}

\begin{proof}
 By the theorem of the cube, it suffices to prove this for $n=2$. In~this
 case, it is clear that the images of the bundles $(1\times 0)^*\hat{\cal
 L}_{1,{\cal E}_{d_1}}$, $(0\times 1)^*\hat{\cal L}_{1,{\cal E}_{d_2}}$ and $(1\times
 \phi_{d_1,d_2})^*{\cal P}_{{\cal E}_{d_1}}$ span the group of symmetric
 endomorphisms of~${\cal E}_{d_1}\times {\cal E}_{d_2}$, and thus it will suffice to
 show that any pullback of~$\hat{\cal L}_{1,{\cal E}_d}$ is isomorphic to the
 appropriate product of these bundles (again using the triviality at~0 to
 fix the isomorphism).

 Thus consider a morphism
 $\psi:=a\phi_{d,d_1}+b\phi_{d,d_2}\colon {\cal E}_{d_1}\times {\cal E}_{d_2}\to {\cal E}_d$.
 We then have
 \begin{gather*}
 \psi^*\hat{\cal L}_{1,{\cal E}_d}
 \cong
 (a\phi_{d,d_1})^*\hat{\cal L}_{1,{\cal E}_d}
 \otimes
 (b\phi_{d,d_2})^*\hat{\cal L}_{1,{\cal E}_d}
 \otimes
 (a\phi_{d,d_1}\times b\phi_{d,d_2})^*{\cal P}_{{\cal E}_d}.
 \end{gather*}
 Now,
 \begin{gather*}
 (a\phi_{d,d_1})^*\hat{\cal L}_{1,{\cal E}_d}
 =
 [a]^* \phi_{\gcd(d,d_1),d_1}^* \phi_{d,\gcd(d,d_1)}^*\hat{\cal L}_{1,{\cal E}_d}
 =
 \hat{\cal L}_{1,{\cal E}_{d_1}}^{a^2 \gcd(d,d_1)^2/d_1d},
 \end{gather*}
 and similarly for the second term. We also have
 \begin{gather*}
 (a\phi_{d,d_1}\,{\times} \,b\phi_{d,d_2})^*{\cal P}_{{\cal E}_d}\, {\cong}\,
 (1\,{\times} \, a\phi_{d_1,d}b\phi_{d,d_2})^*{\cal P}_{{\cal E}_d}{ =}
 (1\,{\times} \, abc\phi_{d_1,d_2})^*{\cal P}_{{\cal E}_d}
\,{\cong}\,
 \big((1\,{\times} \, \phi_{d_1,d_2})^*{\cal P}_{{\cal E}_d}\big)^{abc}
 \end{gather*}
 for a suitable integer $c$, so that the claim follows. (The first step
 here is essentially the definition of the dual isogeny.)
\end{proof}

Thus for each symmetric $Q\in \End(\prod_i {\cal E}_{d_i})$, we~obtain a line bundle $\hat{\cal L}_{Q;d_1,\dots,d_n}$ on~$\prod_i
{\cal E}_{d_i}$, and these line bundles satisfy the same compatibility
relations as for our earlier construction on~${\cal E}^n$. In~general, our two constructions do not agree, but there is one
important special case.

\begin{prop}
 Suppose $Q\in \Mat_n(\Z)$ is a symmetric matrix with even diagonal
 entries. Then $\hat{\cal L}_{Q;1,\dots,1}$ descends to~${\cal M}_{1,1}$,
 where it is canonically isomorphic to~${\cal L}_{Q,0}$.
\end{prop}

\begin{proof}
 Any symmetric integer matrix with even diagonal is not only in the span of
 pullbacks of~$1$, but in fact is in the span of pullbacks
 of~$H={\left(\begin{smallmatrix}0&1\\1&0\end{smallmatrix}\right)}$, the symmetric endomorphism
 associated to the Poincar\'e bundle
 $
 {\cal P}
 \cong
 \hat{\cal L}_{H;1}
 \cong
 {\cal L}_{H,0}$.
\end{proof}

\begin{rem}
 More generally, if $Q_{ii}$ is even whenever $d_i$ is odd, then
 $\hat{\cal L}_{Q;\vec{d}}$ descends to~${\cal X}_0(2N)$; similarly, if
 $Q_{ii}$ is even whenever $N/d_i$ is odd, then $\hat{\cal L}_{Q;\vec{d}}$
 descends to~${\cal X}_0(N,2)$.
\end{rem}

It will be convenient to have a somewhat more functorial version of the
constructions of~${\cal E}^n$ or~the products $\prod_i {\cal E}_{d_i}$
above. For~the first, if $B$ is a finitely generated free abelian group,
then we may consider the family of group schemes ${\cal E}\otimes B$. For~a specific curve $E$, we~may also con\-st\-ruct~$E\otimes B$, which is simply
the corresponding fiber of~${\cal E}\otimes B$.

This clearly extends to a functor, which is exact on short exact sequences
of free abelian groups. Note, however, that if we extend it in the obvious
way to a functor on the category of~{\em all} finitely generated abelian
groups, then it is no longer exact. We readily compute the special cases
\begin{gather*}
 E\otimes \Z\cong E,
 \\[.5ex]
 \Tor_p(E,\Z) = 0,\qquad p>0,
 \\[.5ex]
 E\otimes \Z/N\Z= 0,
 \\[.5ex]
 \Tor_1(E,\Z/N\Z) = E[N],
 \\[.5ex]
 \Tor_p(E,\Z/N\Z) = 0,\qquad p>0\notag
\end{gather*}
using the obvious projective resolution of~$\Z/N\Z$. (This, of course, is
the expected behavior for~tensoring with a divisible group $E$.)

\begin{prop}
 Let $\phi\colon B\to C$ be a morphism of finitely generated free abelian
 groups. Then $E\otimes \phi$ is surjective iff $\coker(\phi)$ is
 finite, and injective iff $\ker(\phi)=0$ and $\coker(\phi)$ is free.
\end{prop}

\begin{proof}
 The four-term sequence
 \begin{gather*}
 0\to \ker\phi\to B\to C\to \coker\phi\to 0
 \end{gather*}
 is a free resolution, and thus we have an isomorphism $\coker(E\otimes
 \phi)\cong E\otimes \coker\phi$ and a short exact sequence
 \begin{gather*}
 0
 \to
 E\otimes \ker(\phi)
 \to
 \ker(E\otimes\phi)
 \to
 \Tor_1(E,\coker(\phi))
 \to
 0.
 \end{gather*}
 The claim follows immediately.
\end{proof}

\begin{prop}
 If $E$ does not have complex $($or quaternionic$)$ multiplication $($in
 particular if $E={\cal E})$, then $\Hom(B,C)\to \Hom(E\otimes B,E\otimes
 C)$ is an isomorphism.
\end{prop}

\begin{proof}
 This reduces to the case $B=\Z^n$, $C=\Z^m$, and thus to the case
 $B=C=\Z$, where it is essentially by definition.
\end{proof}

{\sloppy\begin{cor}
 If $E$ does not have complex multiplication, then the natural map
 $B\to$ $\Hom(E,E\otimes B)$ is an isomorphism, as~is the natural map
 $E\otimes \Hom(E,E\otimes B)\to E\otimes B$.
\end{cor}

}

The dual variety is then easy to compute.

\begin{prop}
 For $B$ free, there is a natural isomorphism $(E\otimes
 B)^\vee\cong E\otimes B^*$.
\end{prop}

\begin{proof}
 Indeed, we~have $E\otimes B\cong E^n$, so dually $(E\otimes B)^\vee\cong
 E^n$. Thus for $E$ without complex multiplication, the natural map
 $E\otimes \Hom\big(E,(E\otimes B)^\vee\big)\to (E\otimes B)^\vee$ is an
 isomorphism. Duality gives $\Hom(E,(E\otimes B)^\vee)\cong \Hom(E\otimes
 B,E)\cong \Hom(B,\Z)$, and thus $E\otimes \Hom(B,\Z)\cong (E\otimes
 B)^\vee$ as desired. The~isomorphism holds for $E={\cal E}$, and thus
 for all fibers $E$.
\end{proof}

This gives the following description of the N\'eron--Severi group of~${\cal
 E}\otimes B$: $NS({\cal E}\otimes B)$ consists of symmetric morphisms
${\cal E}\otimes B\to ({\cal E}\otimes B)^\vee\cong {\cal E}\otimes B^*$,
and thus of symmetric pairings $Q\colon B\otimes B\to \Z$. We then find as above
that any such symmetric pairing (and any weight) induces a line bun\-dle~${\cal L}_{Q,w}$ on~${\cal E}\otimes B$.

Now, suppose $B\to C$ is an injective morphism with finite cokernel. Then we
have a short exact sequence
\begin{gather*}
0\to \Tor_1({\cal E},C/B)\to {\cal E}\otimes B\to {\cal E}\otimes C\to 0,
\end{gather*}
where the kernel is a product of groups of the form $E[d_i]$. If $C/B$ has
exponent $N$, then we have $\Tor_1({\cal E},C/B)\cong {\cal
 E}[N]\otimes_{\Z/N\Z} C/B$, which in turn suggests that we consider the
subgroup $\kappa_N\otimes_{\Z/N\Z} C/B$, where $\kappa_N$ is the kernel of a
the cyclic $N$-isogeny corresponding to a~point of~${\cal X}_0(N)$. This,
it turns out, does not quite behave correctly in characteristic dividing
$N$, but we do have the following.

\begin{prop}
 Let $N$ be a positive integer, and let $B$ be a finitely generated
 abelian group of~expo\-nent~$N$. Then the group scheme
 $\kappa_N\otimes_{\Z/N\Z} B$ on~${\cal X}_0(N)\times \Spec(\Z[1/N])$
 extends in a~natural way to a flat group scheme on~${\cal X}_0(N)$.
\end{prop}

\begin{proof}
 If we choose an isomorphism $B\cong \bigoplus \Z/d_i\Z$, then we
 certainly have such an extension: away from $N$, the group is just
 $\prod_i \kappa_{d_i}$, where $\kappa_{d_i}$ is the kernel of the isogeny
 $E_1\to E_{d_i}$, and this product makes sense in all characteristics.
 If $B\to B'$ is a morphism of~$N$-torsion groups, then the morphism
 $\kappa_N\otimes_{\Z/N\Z} B\to \kappa_N\otimes_{\Z/N\Z}B'$ is simply the
 restriction of the morphism $E[N]\otimes_{\Z/N\Z} B\to
 E[N]\otimes_{\Z/N\Z}B'$. The~latter morphism is defined in all
 characteristics, and the requirement that it restrict to a specific
 morphism is a closed condition, so is inherited from the generic case.
\end{proof}

This allows us to define families of abelian varieties over ${\cal X}_0(N)$
as follows: given abelian groups $B$, $C$ such that $NC\subset B\subset C$, we~define ${\cal E}_{B,C}$ to be the quotient of~${\cal E}\otimes B$ by the
subgroup scheme extending $\kappa_N\otimes_{\Z/N\Z} C/B$. (We will denote
this extension by the same tensor product notation, but caution the reader
that it is {\em not} the actual tensor product in general.)

\begin{prop}
 If $NC_1\subset B_1\subset C_1$, $NC_2\subset B_2\subset C_2$ are pairs
 of finitely generated free abelian groups and $\phi\colon C_1\to C_2$ is a
 morphism such that $\phi(B_1)\subset B_2$, then there is an induced
 morphism ${\cal E}_{B_1,C_1}\to {\cal E}_{B_2,C_2}$, making the
 construction functorial.
\end{prop}

\begin{proof}
 The condition on~$\phi$ implies that it induces a morphism $C_1/B_1\to
 C_2/B_2$, and thus we have a commutative diagram
 \begin{gather*}
 \begin{CD}
 0 @>>> \kappa_N\otimes_{\Z/N\Z} C_1/B_1 @>>> {\cal E}\otimes B_1 @>>> {\cal
 E}_{B_1,C_1}@>>> 0\\
 @. @VVV @VVV @. @.\\
 0 @>>> \kappa_N\otimes_{\Z/N\Z} C_2/B_2 @>>> {\cal E}\otimes B_2 @>>>
 {\cal E}_{B_2,C_2} @>>> 0.
 \end{CD}
\end{gather*}
 Since the rows are exact and the given vertical arrows are functorial,
 the claim follows.
\end{proof}

Of course, given an isomorphism $C/B\cong \prod_i \Z/d_i\Z$, there is a
corresponding isomorphism ${\cal E}_{B,C}\cong \prod_i {\cal E}_{d_i}$.
This makes the following straightforward to verify.

\begin{prop}
 If $E_1\to E_N$ is a cyclic $N$-isogeny between curves with no complex
 multiplication, then the morphisms between corresponding fibers of~${\cal
 E}_{B_1,C_1}$ and ${\cal E}_{B_2,C_2}$ are precisely those coming from
 the previous proposition.
\end{prop}

\begin{cor}
 For $NC\subset B\subset C$, we~have
 \begin{gather*}
 \Hom({\cal E},{\cal E}_{B,C})\cong B,
 \\
 \Hom({\cal E}_N,{\cal E}_{B,C})\cong C
 \end{gather*}
 in such a way that composition with the dual isogeny ${\cal E}_N\to {\cal
 E}$ induces the inclusion~$B\subset C$.
\end{cor}

\begin{prop}
 There is a natural isomorphism ${\cal E}_{B,C}^{\vee} \cong {\cal
 E}_{C^*,B^*}$.
\end{prop}

\begin{proof}
 The previous corollary allows us to canonically identify anything
 isomorphic to a pro\-duct of curves ${\cal E}_{d_i}$ with a variety of the
 form ${\cal E}_{B,C}$. Since ${\cal E}_{B,C}$ is isomorphic to such a
 product, its dual is also of that form, and thus it remains only to
 compute $\Hom({\cal E},{\cal E}_{B,C}^{\vee})$ and $\Hom({\cal E}_N,{\cal
 E}_{B,C}^{\vee})$.
\end{proof}

It follows that the N\'eron--Severi group of~${\cal E}_{B,C}$ (or of any
fiber without complex multiplication) consists of pairings $B\otimes C\to
\Z$ which become symmetric when restricted to~$B\otimes B$; equivalently,
it~consists of symmetric pairings $Q\colon B\otimes B\to \Z$ such that
$Q(B,NC)\subset N\Z$. From our construction above, we~find that if we base
change to~${\cal X}_0(2N,2)$, then any such symmetric pairing induces a
natural line bundle $\hat{\cal L}_{Q;B,C}$ on~${\cal E}_{B,C}$, satisfying
the appropriate compatibility relations.

Suppose now that $Q$ is a positive definite pairing of~``level'' dividing
$N$ (i.e., such that \mbox{$NB^*\subset QB$}) Then we have a chain of free abelian
groups $Q^{-1}NB^*\subset B\subset C\subset Q^{-1}B^*$, giving rise to an
isogeny $\pi\colon {\cal E}_{B,C}\to {\cal E}_{B,Q^{-1}B^*}$. The~pairing $Q$
still induces an element of the N\'eron--Severi group of the quotient, which
by degree considerations is now a principal polarization, represented by
the line bundle $\hat\Theta_B:=\hat{\cal L}_{Q;B,Q^{-1}B^*}$. We moreover
find that $\pi^*\hat\Theta_B\cong \hat{\cal L}_{Q;B,C}$.

We also note that the action of~${\cal E}[N]\otimes \big(Q^{-1}B^*/B\big)\cong
\Tor_1\big({\cal E},Q^{-1}B^*/B\big)$ on~${\cal E}\otimes B$ des\-cends to an action
of the quotient $\big({\cal E}[N]\otimes \big(Q^{-1}B^*/B\big)\big) / \big(\kappa_N\otimes
\big(Q^{-1}B^*/B\big)\big)$ on~${\cal E}_{B,Q^{-1}B^*}$, which by the Weil pairing may
be identified with an action of the Pontryagin dual $\Hom\big(\kappa_N\otimes
\big(Q^{-1}B^*/B\big),\mu_N\big)$. If $\kappa_N$ is diagonalizable, then this dual is
discrete, and may be identified with~$\Hom\big(Q^{-1}B^*/B,\Hom(\kappa_N,\mu_N)\big)$.

\begin{lem}
 On the locus of~${\cal X}_0(2N,2)$ where the $N$-isogeny $E_1\to E_N$ has
 diagonalizable kernel $\kappa_N$, we~have natural
 identifications
 \begin{gather*}
 \Gamma\big(E_{B,C};\hat{\cal L}_{Q;B,C}\big)
 \cong
 \bigoplus_{g\in \Hom(Q^{-1}B^*/C,\Hom(\kappa_N,\mu_N))}
 g^* \Gamma\big(E_{B,Q^{-1}B^*};\hat\Theta_B\big),
 \end{gather*}
 in which each $(1$-dimensional!$)$ summand on the right is an eigenspace for
 the induced action of~$\ker(E_{B,C}\to E_{B,Q^{-1}B^*})$.
\end{lem}

\begin{proof}
 We have
 \begin{gather*}
 \Gamma\big(E_{B,C};\hat{\cal L}_{Q;B,C}\big)
 \cong
 \Gamma\big(E_{B,C};\pi^*\hat\Theta_B\big)
 \cong
 \Gamma\big(E_{B,Q^{-1}B^*};\pi_*\pi^*\hat\Theta_B\big).
 \end{gather*}
 The natural map $\hat\Theta_B\to \pi_*\pi^*\hat\Theta_B$ selects a
 particular eigenspace of the kernel of the isogeny, and the decomposition
 as claimed then follows by the structure theory of representations of
 Heisenberg groups, see~\cite{SekiguchiT:1977} as well as the exposition
 in~\cite{PolishchukA:2003}. Here we use the fact that
 $\Gamma\big(E_{B,C};\hat{\cal L}_{Q;B,C}\big)$ is the unique irreducible
 representation of the Heisenberg group ${\cal G}\big(\hat{\cal L}_{Q;B,C}\big)$
 on which the central~${\mathbb G}_m$ acts with weight 1, together with
 the fact that the commutator pairing on the Heisenberg group is precisely
 the Weil pairing, so the different isotypic components for the
 diagonalizable kernel are related via the complementary translation
 subgroup.
\end{proof}

\begin{cor}
 Suppose that the finite group $G$ acts on~$C$, preserving the subgroup
 $B$ and the polarization~$Q\colon B\to B^*$ so that $G$ acts on~${\cal E}_{B,C}$
 as automorphisms fixing the identity, with an indu\-ced equivariant
 structure on~$\hat{\cal L}_{Q;B,C}$. On~the locus of~${\cal X}_0(2N,2)$
 where $E_1\to E_N$ has diagonalizable kernel, the $G$-module
 $\Gamma\big(E_{B,C};\hat{\cal L}_{Q;B,C}\big)$ is isomorphic to the permutation
 module arising from the action of~$G$ on~$\Hom\big(Q^{-1}B^*/C,\Q/\Z\big)$.
\end{cor}

\begin{rem}
 Here we note that for general free abelian groups $B\subset C$ with
 finite quotient,
 \begin{gather*}
 \Hom(C/B,\Q/\Z)\cong \Tor_1(B^*/C^*,\Q/\Z) \cong
 B^*/C^*,
 \end{gather*}
 and thus
 \begin{gather*}
 \Hom\big(Q^{-1}B^*/C,\Q/\Z\big)\cong C^*/QB\cong
 Q^{-1}C^*/B.
 \end{gather*}
\end{rem}

\begin{cor}
 Let $B$ be a finitely generated free abelian group and $Q\colon B\otimes B\to
 \Z$ an {\em even} symmetric pairing of level dividing $N$, and let $E$ be
 an elliptic curve which, if supersingular, has characteristic prime to~$N$. Then $\Gamma(E\otimes B;{\cal L}_{Q,0})$ is isomorphic as a
 $G$-module to the permutation module coming from the action of~$G$ on~$Q^{-1}B^*/B$.
\end{cor}

\begin{proof}
 The condition on~$E$ ensures that we may choose a~point of~${\cal X}_0(2N,2)$
 lying over it such that the cyclic $N$-isogeny $E\cong E_1\to E_N$ has
 diagonalizable kernel. We may then identify the $G$-module
 $\Gamma(E\otimes B;{\cal L}_{Q,0})$ with~$\Gamma\big(E_{B,B};\hat{\cal
 L}_{Q;B,B}\big)$ and thus apply the previous corollary.
\end{proof}

\begin{cor}
 With the same hypotheses, the dimension~$\dim\big(\Gamma(E\otimes B;{\cal
 L}_{Q,0})^G\big)$ is equal to the number of orbits of~$G$ in~$Q^{-1}B^*/B$.
\end{cor}

\begin{rem}
 Both conclusions remain valid for supersingular curves of characteristic
 prime to~$|G|$; indeed, any 1-parameter family of~$G$-modules over an
 algebraically closed field containing $1/|G|$ is trivial, so the claims
 follow from the case of ordinary curves.
\end{rem}

It turns out that the exclusion of certain supersingular curves above is
indeed necessary. For~instance, suppose that $B=\Z^8$ and $Q$ is the Gram
matrix of the lattice $Q_8(1)$ of~\cite{ConwayJH/SloaneNJA:1988}. This is
a symmetric matrix with even diagonal and elementary divisors
1, 1, 1, 1, 5, 5, 5, 5, and the corresponding automorphism group $\GO_4^+(5)$
acts in the natural way on~$\coker(Q)$. If $E$ is the (geometrically
unique) supersingular curve of characteristic 5, then one finds that the
induced $\GO_4^+(5)$-module structure on~$\Gamma\big(E^8;{\cal L}_{Q,0}\big)$ is
{\em not} isomorphic to the given permutation representation. In~this
case, the actual invariant subspace does not jump, but we find that
\begin{gather*}
\dim\big(\Gamma\big(E^{16};{\cal L}_{Q\oplus Q,0}\big)^{\GO_4^+(5)}\big)\cong
\dim\big(\big(\Gamma\big(E^8;{\cal L}_{Q,0}\big)^{\otimes 2}\big)^{\GO_4^+(5)}\big)=160,
\end{gather*}
while for all other curves, the invariant space has dimension 156. To~compute the action of~$G$ in such supersingular cases, we~may use the
fact that ${\cal L}_{Q,0}$ always descends to the appropriate quotient and
gives an equivariant isomorphism
$\Gamma\big(E^n;{\cal L}_{Q,0}\big)\cong\Gamma\big(\ker\psi_Q^\vee;{\cal L}'\big)$;
what fails is that we no longer have a natural basis (which in the case of
the theorem are an eigenbasis for the action of the diagonalizable group~$\ker\psi_Q$) of the latter. There is, in fact, an
isomorphism
\begin{gather*}
\Gamma\big(\ker\psi_Q^\vee;{\cal L}'\big)\cong
\Gamma\big(\ker\psi_Q^\vee;\sO_{\ker\psi_Q^\vee}\big),
\end{gather*}
since $\ker\psi_Q^\vee$ is $0$-dimensional, but this is not in general
equivariant; in general, the action of~$G$ is twisted by some cocycle with
values in the unit group of the coordinate ring. When $\det(Q)$ is odd,
however, we~can use the description of~${\cal L}'$ in terms of pullbacks of~$\sO_{{\cal E}_d}([\sigma_d])$ to compute this cocycle. In~the $Q_8(1)$
case, this further simplifies, since we only need to know what happens on~$\alpha_5^4$, allowing us to reduce to an evaluation of functions on~$E_5^4$ in an appropriate formal neighborhood of the identity.

{\sloppy
A similar calculation applies to the Gram matrix of~$\sqrt{3}\Lambda_{E_6}^\perp$, with its automorphism group~$O_5(3)$; in
this case, there are also subgroups of~$O_5(3)$ for which the corollary
fails on the supersingular curve of characteristic 3. We can also obtain a
characteristic~2 counterexample from~$\sqrt{2}\Lambda_{E_7}^\perp$; in~this
case, it is unclear how to compute the cocycle, but we can simply check
that none of~the~16~elements of~$H^1\big(\Sp_6(2);\mu_2\big(\alpha_2^6\big)\big)$ give rise
to the permutation module. (There is also a~subgroup with too many
invariants, namely the preimage in~$W(D_6)$ of the transitive
$\Alt_5\subset S_6$, which has too many invariants in each of the 16
possible cases.)

}

We should further note that the requirement that $Q$ have even diagonal is
also necessary; indeed, otherwise the claim already fails for the case
$Q=1$, $G=\GL_1(\Z)$ for {\em any} curve of characteristic not 2.

Even when~$E$ is supersingular of characteristic dividing $N$, there may
still be isogenies of~the~form
\begin{gather*}
E_{B,C}\to E_{B,C'}
\end{gather*}
with diagonalizable kernel, which as an abstract group scheme can be
(geometrically) identified with~$\mu_N\otimes_{\Z/N\Z} C'/C$. Indeed, the only
requirement is that $|C'/C|$ be prime to the characteristic of~$E$.
Making this a canonical identification is somewhat trickier, as~the kernel
is only naturally described as the quotient
\begin{gather*}
(\kappa_N\otimes_{\Z/N\Z} C'/B)/(\kappa_N\otimes_{\Z/N\Z} C/B).
\end{gather*}
Moreover, the translations moving between the different eigenspaces are
only defined up to the kernel of the descended polarization. We find in
general that
\begin{gather*}
\Gamma(E_{B,C};{\cal L}_{Q;B,C})
\cong
\bigoplus_g g^*\Gamma(E_{B,C'};{\cal L}_{Q;B,C'}),
\end{gather*}
where $g$ ranges over the {\em quotient} group
\begin{gather*}
\Hom\big(Q^{-1}B^*/C,\Hom(\kappa_N,\mu_N)\big)/\Hom\big(Q^{-1}B^*/C',\Hom(\kappa_N,\mu_N)\big).
\end{gather*}
We can thus only use this decomposition in understanding group actions when
this quotient group has an equivariant splitting. Luckily, there is an
important case when this happens: if~$C'/C$ is the $l$-part of~$Q^{-1}B^*/C$ for some prime $l$ which is invertible on~$E$, then
the quotient may be identified with the $l$-part of~$\Hom\big(Q^{-1}B^*/C,\Hom(\kappa_N,\mu_N)\big)$.

\begin{lem}\label{lem:not_prime_power}
 With $B,C,Q$ as above and $l$ a prime invertible on~$E$, there is a
 $G$-equivariant isomorphism
\begin{gather*}
\Gamma(E_{B,C};{\cal L}_{Q;B,C}) \cong \bigoplus_H \Ind_H^G\Res^G_H
\Gamma(E_{B,C'};{\cal L}_{Q;B,C'}),
\end{gather*}
where $H$ ranges over the point stabilizers in the different orbits of the
action of~$G$ on the $l$-part of~$Q^{-1}C^*/B$ and
\begin{gather*}
C' = \bigcup_k Q^{-1}B^*\cap l^{-k}C.
\end{gather*}
\end{lem}

Here we should note that we have such a reduction for every prime dividing
$N$; in particular, if the polarization does not have prime power degree,
then we can always choose a prime dividing the degree of the polarization
which is invertible on~$E$, and use the corresponding reduction.

Although we have seen that there can indeed be (finitely many) bad curves
for such invariant theory questions, it turns out that our hypotheses are
in fact slightly more restrictive than they need to be. Suppose we have a
finite group $G$ acting on an abelian variety $A$ (fixing the identity).
There are two natural induced abelian subvarieties. The~subgroup scheme
$A^G$ is still projective, and thus (up to a possible inseparable base
change) we may take its reduced identity component $A^{G0}$. Equivalently
(and without need for base change), we~could instead define $A^{G0}$ to be
the image of the endomorphism $\sum_{g\in G} g\in \End(A)$. There is also
an almost complementary subvariety $A_G$ giving by the image of the
endomorphism $|G|-\sum_{g\in G} g$. Both subvarieties are clearly
preserved by $G$, and since the sum of the endomorphisms is an isogeny, it
follows that we have a natural $G$-equivariant isogeny $A_G\times A^{G0}\to
A$. Any $G$-equivariant line bundle on~$A$ (with trivial action on the
fiber over the identity) pulls back to a $G$-equivariant line bundle on~$A_G\times A^{G0}$, namely ${\cal L}|_{A_G}\boxtimes {\cal L}|_{A^{G0}}$.
Moreover, the action of~$G$ on the second factor is trivial, since it is
trivial at the identity.

It turns out that even though this isogeny can fail to have diagonalizable
kernel, we~can still use it to reduce questions about $G$-module structures
to~$A_G$.

\begin{lem}\label{lem:G-mod-factor}
 With $A$, $G$, ${\cal L}$ as above, assume that ${\cal L}$ is ample. Then
 there is a $G$-module iso\-morphism
 \begin{gather*}
 \Gamma(A;{\cal L})^{d}\cong \Gamma(A_G;{\cal L}|_{A_G})^{e},
 \end{gather*}
 where $d=\big|A_G\cap A^{G0}\big|$ and $e=\dim\Gamma\big(A^{G0};{\cal L}|_{A^{G0}}\big)$.
\end{lem}

\begin{proof}
 Let $K=A_G\cap A^{G0}$ be the kernel of the isogeny $A_G\times A^{G0}\to
 A$. Then we have a natural isomorphism
 \begin{gather*}
 \Gamma(A;{\cal L})
 \cong
 \Gamma\big(A_G\times A^{G0};{\cal L}|_{A_G}\boxtimes {\cal L}|_{A^{G0}}\big)^K
 \cong
 \big(\Gamma(A_G;{\cal L}|_{A_G})\otimes \Gamma\big(A^{G0};{\cal L}|_{A^{G0}}\big)\big)^K.
 \end{gather*}
 Let $H$ be the preimage of~$K$ in the Heisenberg group ${\cal G}\big({\cal
 L}_{A^{G0}}^{-1}\big)$. This certainly acts naturally on~$\Gamma\big(A^{G0};{\cal
 L}|_{A^{G0}}\big)^*$, but the fact that the line bundle descends to~$A$ implies
 that it also acts on~$\Gamma(A_G;{\cal L}|_{A_G})$. We thus have a
 natural isomorphism
 \begin{gather*}
 \big(\Gamma(A_G;{\cal L}|_{A_G})\otimes \Gamma\big(A^{G0};{\cal L}|_{A^{G0}}\big)\big)^K
 \cong \Hom_{H}\big(\Gamma\big(A^{G0};{\cal L}|_{A^{G0}}\big)^*,
 \Gamma(A_G;{\cal L}|_{A_G})\big),
 \end{gather*}
 which by Frobenius reciprocity further becomes
 \begin{gather*}
 \Hom_{H}\big(\Gamma\big(A^{G0};{\cal L}|_{A^{G0}}\big)^*,\Gamma(A_G;{\cal L}|_{A_G})\big)
 \\ \qquad
 \cong
 \Hom_{{\cal G}({\cal L}|_{A^{G0}}^{-1})}\big(\Gamma\big(A^{G0};{\cal L}|_{A^{G0}}\big)^*,
 \Ind_H^{{\cal G}({\cal L}_{A^{G0}}^{-1})} \Gamma(A_G;{\cal L}|_{A_G})\big).
 \end{gather*}
 By the structure of Heisenberg representations, we~may then conclude that
 there is a functorial isomorphism
 \begin{gather*}
 \Gamma(A;{\cal L})\otimes \Gamma\big(A^{G0};{\cal L}|_{A^{G0}}\big)
 \cong
 \Ind_H^{{\cal G}({\cal L}_{A^{G0}}^{-1})}
 \Gamma(A_G;{\cal L}|_{A_G})
 \end{gather*}
 of~${\cal G}\big({\cal L}_{A^{G0}}^{-1}\big)$-modules. Moreover, the splitting
 $K\to H$ is $G$-invariant, since it could be computed inside~$A^{G0}$,
 and thus we may rewrite this as a $G\times {\cal G}\big({\cal
 L}_{A^{G0}}^{-1}\big)$-module isomorphism:
 \begin{gather*}
 \Gamma(A;{\cal L})\otimes \Gamma\big(A^{G0};{\cal L}|_{A^{G0}}\big)
 \cong
 \Ind_{G\times H}^{G\times {\cal G}({\cal L}_{A^{G0}}^{-1})}
 \Gamma(A_G;{\cal L}|_{A_G}).
 \end{gather*}
 Moreover, the induction
 functor is exact (since the homogeneous space is affine), as~is
 restriction to~$G$, and thus if we forget the action of the Heisenberg group, we~obtain a $G$-module isomorphism
 \begin{gather*}
 \Gamma(A;{\cal L})^{e}
 \cong
 \Gamma(A_G;{\cal L}|_{A_G})^{e^2/d},
 \end{gather*}
 from which the claim follows.
\end{proof}

\begin{rems} Note that although $d$ always divides $e^2$, it need not divide
 $e$, and vice versa $e$ can fail to divide $d$.
\end{rems}

\begin{rems}
 The reader should note that the notion of an induced module for
 representations of group schemes corresponds to what would normally be
 called a {\em co}induced module.
\end{rems}

\section{Coxeter group actions on abelian varieties}\label{section3}

One of the major ingredients in the construction of elliptic analogues of
double affine Hecke algebras is a suitable action of an affine Weyl group
on a power of an elliptic curve (or more generally on a variety isogenous
to such a power). It will be convenient to work somewhat more abstractly,
and begin with the finite case.

With this in mind, let $A$ be an abelian variety, and suppose the finite
Weyl group $W$ acts faithfully on~$A$ (fixing the identity) in such a way
that for any reflection~$r\in R(W)$, the corresponding fixed subgroup scheme
has codimension~1. We will naturally refer to such an action as an action
``by reflections''.

For each reflection~$r\in R(W)$, the subgroup $\langle r\rangle$ splits $A$
(up to isogeny) as discussed above; in this case, we~have natural
subvarieties $A^{r0}:=\im(1+r)$ and $A_r:=\im(1-r)$, and an~indu\-ced isogeny
$A^{r0}\times A_r\to A$ (with kernel contained in~$A_r[2]$). Since
$A^{r0}$ by assumption has codimension~1, we~see that $A_r$ is a
$1$-dimensional abelian variety. In~other words, each reflection in~$W$
induces a corresponding elliptic curve $E_r=A_r$ contained in~$A$, as~the
image of the endomorphism $1-r$. We call such a curve the ``root curve''
associated to~$r$. Applying the same construction to the dual variety
$A^\vee$ gives root curves $E'_r\subset A^\vee$, the duals of which we
refer to as ``coroot curves''. Note that the coroot curve associated to~$r$ can be described directly as the cokernel of the endomorphism $1+r$. In~particular, the endomorphism $1-r$ factors through~$E'_r$, giving rise
to a natural map $E'_r\to E_r$ such that the composition~$A\to E'_r\to
E_r\to A$ is $1-r$ and the composition~$E_r\to A\to E'_r\to E_r$ is
multiplication by $2$.

Fix a system of simple roots $S=\{\alpha_1,\dots,\alpha_n\}$ in~$W$,
and let $E_1,\dots,E_n$; $E'_1,\dots,E'_n$ be the corresponding root
and coroot curves, with induced maps $\iota_i\colon E_i\to A$, $\iota'_i\colon A\to E'_i$,
$\nu_i\colon E'_i\to E_i$. The~action of~$s_i$ on~$E_j$ can be described quite
simply:
\begin{gather*}
s_i\circ\iota_j=\iota_j + (s_i-1)\circ\iota_j=
\iota_j - \iota_i\circ \nu_i\circ \iota'_i\circ\iota_j.
\end{gather*}
This suggests that we should define a morphism $\mu_{ij}\colon E_j\to E_i$
as the composition~$-\nu_i\circ \iota'_i\circ\iota_j$; that is,
it is the morphism $E_j\to E_i$ induced by $s_i-1$.

\begin{lem}
 The curves $E_1, \dots, E_n$ are distinct.
\end{lem}

\begin{proof}
 Suppose otherwise, and reorder the simple roots so that $E_1=E_2$.
 Then $s_1s_2\ne 1$, but
 \begin{gather*}
 (s_1s_2-1) = (s_1-1)(s_2-1)+(s_1-1)+(s_2-1),
 \end{gather*}
 so that $s_1s_2-1$ has image $E_1=E_2$. Since $s_1s_2$ fixes $E_1=E_2$, we~find that $(s_1s_2)^k-1 = k(s_1s_2-1)$ for all $k\ge 1$, and thus
 $s_1s_2$ has infinite order, contradicting finiteness of~$W$.
\end{proof}

\begin{lem}
 For $i\ne j$, the composition~$\mu_{ji}\circ\mu_{ij}$ is multiplication
 by $k\in \{0,1,2,3\}$, and if the composition is $0$, then
 $\mu_{ij}=\mu_{ji}=0$.
\end{lem}

\begin{proof}
 Since $E_i\ne E_j$, the product $E_i\times E_j$ is isogenous with its
 image in~$A$. Define an action of~$s_i$, $s_j$ on~$E_i\times E_j$ by
 \begin{gather*}
 s_i(x_i,x_j) = (-x_i + \mu_{ji}(x_j),x_j),
 \\
 s_j(x_i,x_j) = (x_i,-x_j + \mu_{ij}(x_i)).
 \end{gather*}
 The elements $s_i$, $s_j$ clearly act as involutions on the product, and
 the actions are compatible with~the actions on~$A$, so that the action of~$(s_is_j)^{m_{ij}}-1$ induces a homomorphism from
 $E_i\times E_j$ to the kernel of the map $E_i\times E_j\to A$. Since
 $E_i\times E_j$ is proper, reduced, and connected, and said kernel is
 finite, we~see that this description actually gives an action of the rank
 2 Weyl group~$\langle s_i,s_j\rangle$. Moreover, since this group is
 finite, there is a $W$-invariant polarization on~$E_i\times E_j$, of the
 form
 \begin{gather*}
 \begin{pmatrix}
 2r_i & -\psi_{ji}
 \\
 -\psi_{ji}^\vee & 2r_j
 \end{pmatrix}\!,
 \end{gather*}
 with~$4r_ir_j-\deg(\psi_{ji})>0$. We then find that $\psi_{ji} =
 r_i\mu_{ji}=r_j\mu_{ij}^\vee$, so that
 \begin{gather*}
 \deg(\psi_{ji}) = \psi_{ji}\psi_{ji}^\vee
 = r_i \mu_{ji} r_j \mu_{ij},
 \end{gather*}
 and thus $\mu_{ji}\mu_{ij}$ is multiplication by a nonnegative integer
 less than $4$. Moreover, since $r_i\mu_{ji}=r_j\mu_{ij}^\vee$, we~see
 that if one of~$\mu_{ij}$, $\mu_{ji}$ vanishes, then so does the other.
\end{proof}

\begin{rem}
 We then readily see that the order of~$s_is_j$ is equal to~$2$, $3$, $4$,
 $6$ when $\mu_{ij}\mu_{ji}=\mu_{ji}\mu_{ij}$ is equal to~$0$, $1$, $2$,
 $3$ respectively.
\end{rem}

\begin{prop}
 Let $(W,S)$ be a finite Weyl group, and suppose that $E_1,\dots,E_n$ is a
 system of elliptic curves and $\mu_{ij}\colon E_i\to E_j$, $i\ne j$, a system
 of morphisms such that $\mu_{ij}\mu_{ji} = 4\cos(\pi/m_{ij})^2$, with~$\mu_{ij}=\mu_{ji}=0$ whenever $m_{ij}=2$. Then there is a faithful
 action of~$W$ on~$\prod_i E_i$ such that
 \begin{gather*}
 s_i(x_1,\dots,x_n) =
 \bigg(x_1,\dots,x_{i-1},-x_i+\sum_{j\ne i} \mu_{ji}(x_j),x_{i+1},\dots,x_n\bigg).
 \end{gather*}
\end{prop}

\begin{proof}
 The action of~$s_i$ is clearly an involution, and the braid relations are
 straightforward to verify. (This is easy when $m_{ij}=2$, and for
 $m_{ij}>2$ it suffices to show that
 $\big((s_is_j)^2+$ $(2-\mu_{ij}\mu_{ji})(s_is_j)+1\big)(s_is_j-1)$ vanishes, which
 reduces to a computation in~$E_i\times E_j$.) So this certainly gives an
 action of~$W$, and it remains only to show that it is faithful.

 Since the construction clearly respects products, we~may as well assume
 that $W$ is irreducible. For~any path in the Coxeter diagram of~$W$, we~may take the corresponding composition of morphisms $\mu_{ij}$; since
 $\mu_{ij}=0$ iff $m_{ij}=2$, any such composition will be an isogeny.
 Moreover, since the Coxeter diagram is a tree by finiteness of~$W$, we~see that any two isogenies $E_i\to E_j$ arising in this way will differ
 by a {\em positive} factor (any time the path backtracks introduces a~factor $\mu_{ij}\mu_{ji}>0$). In~particular, for~any element $w\in W$,
 the induced map $E_i\to E_j$ (apply~$w$ then project onto the $j$th
 factor) is an integer linear combination of such isogenies, and in
 particular has a corresponding notion of positivity. This allows us to
 turn any element of~$W$ into a {\em real} matrix by taking each such
 morphism to the appropriately signed square root of its degree. The~consistency of sign ensures that this will give rise to an actual
 representation of~$W$, and we can then verify that up to a diagonal
 change of basis, this is precisely the standard reflection representation
 of~$W$.
\end{proof}

\begin{cor}
 Let the finite Weyl group $W$ act faithfully on the abelian variety $A$
 by reflections, with simple root curves $E_1,\dots,E_n$. Then the
 induced morphism $\prod_i E_i\to A$ is made $W$-equivariant by the above
 action, and its kernel is finite and fixed by $W$.
\end{cor}

\begin{proof}
 The equivariance is obvious by construction, so it remains only to show
 that the kernel~$K$ is fixed by $W$. Otherwise, some simple reflection~$s_i$ will act nontrivially on~$K$, and thus $(s_i-1)K$ $\subset K$ is
 nonzero. But $(s_i-1)K\subset E_i$, and $K\cap E_i=0$ since $E_i$ is
 defined as a subscheme of~$A$.
\end{proof}

\looseness=1
The proof of the proposition suggests an extension of this construction to
more general crystallographic Coxeter groups (in particular to affine Weyl
groups). Certainly, one could consider an action of the above form
associated to any system of morphisms $\mu_{ij}$, but to relate it to the
standard reflection representation of a~Coxeter group, we~make the
following assumptions:
\begin{itemize}\itemsep=0pt
 \item The composition~$\mu_{ij}\mu_{ji}$ is multiplication by $k_{ij}\in
 \{0,1,2,3,4\}$.
 \item There is a system of positive integers $r_i$ such that
 $r_i\mu_{ji}=r_j\mu_{ij}^\vee$ for each $i\ne j$.
 \item Any composition~$\mu_{i_1i_2}\mu_{i_2i_3}\cdots\mu_{i_mi_1}$ is
 multiplication by a nonnegative integer.
\end{itemize}
We call such a system of curves and morphisms an ``elliptic root datum''.

\begin{thm}
 Any elliptic root datum gives rise to a faithful action on~$\prod_i E_i$
 of the Coxeter group with multiplicities $m_{ij}$ given by $k_{ij} =
 4\cos(\pi/m_{ij})^2$, such that $s_i$ acts as above.
\end{thm}

\begin{proof}
 The conditions on the morphisms ensure that we can faithfully translate
 the action into one on a real vector space, taking each morphism to the
 appropriately signed square root of its degree. Conjugating by the
 diagonal matrix with entries $\sqrt{r_i}$ turns this into the standard
 reflection representation of the given Coxeter group.
\end{proof}

\begin{cor}
 For each conjugacy class $C$ of reflections, there is a corresponding
 elliptic curve~$E_C$ equipped with isomorphisms $E_C\cong E_r$, $r\in C$
 such that the action of~$W$ on the set of compositions $E_C\cong
 E_r\to A$ and their negatives can be identified with the action of~$W$ on the corresponding set of roots, with the compositions $E_C\cong
 E_r\to A$ corresponding to the positive roots.
\end{cor}

\begin{proof}
 Indeed, for~each simple root $\alpha_i$, we~may consider the set of
 compositions $w\circ \iota_i$ for~$w\in W$, and find that each such
 composition has the form $(\beta_{i1},\beta_{i2},\dots,\beta_{in})$ in
 which either every entry is a~nonnegative element of the relevant $\Hom$
 space (i.e., corresponding to a nonnegative real number) or every entry
 is nonpositive. If $r = w s_i w^{-1}$, then the image of~$1-r$ is $w$
 times the image of~$1-s_i$, and thus $w\circ \iota_i$ identifies $E_i$
 with~$E_r$. Each $E_r$ is afforded with precisely two such
 identifications, of which we naturally choose the one corresponding to a~positive root. We~furthermore see that conjugate simple reflections give
 rise to equivalent systems of identifications of root curves.
\end{proof}

\begin{cor}
 Relative to the action of~$W$ on~$\prod_i E_i$ arising in this way, any
 two reflections have distinct root curves.
\end{cor}

\begin{proof}
 Assuming without loss of generality that $W$ is irreducible, we~find that each $E_i$ is
 isogenous to~$E_1$ in an essentially canonical way (choose the isogeny
 of smallest degree among the ``positive'' isogenies), and then see that
 two root curves agree iff the corresponding maps $E_1\to \prod_i E_i$
 correspond to proportional real vectors, making the two reflections
 agree.
\end{proof}

More generally, if $B$ is an abelian variety with trivial action of~$W$, we~could consider the image $B\times E_1\times\cdots\times E_n$ under a
$W$-equivariant isogeny. We will say that the abelian variety $A$ arising
in this way has an action of~$W$ of~``root type''. Note that as in the
finite case, we~may always arrange for the kernel of the isogeny to be not
just preserved by $W$ but fixed elementwise by $W$, as~otherwise there will
be kernel elements contained in root curves. We will also need the dual
notion: an abelian variety with an action of~$W$ is of~``coroot type'' if
its dual is of root type. These are equivalent for finite groups, or more
generally for Coxeter groups with nondegenerate Cartan matrices, but in the
affine case the two notions do not agree. Note that in the coroot type
case, rather than having well-behaved positivity for roots, we~have
well-behaved positivity for coroots: for each conjugacy class of
reflections, we~can choose isomorphisms between the corresponding coroot
curves and a fixed curve $E$ in such a way that the resulting set of maps
to~$E$, together with their negatives, are in equivariant, sign-preserving
bijection with the corresponding set of root vectors.

Consider the case of the affine Weyl group of type $\tilde{A}_2$. Since
$m_{ij}=3$, $k_{ij}=1$, we~see that each~$\mu_{ij}$ is an isomorphism, and
the positivity assumption forces the isomorphisms to be consistent. We
thus obtain the following faithful action on~$E^3$:
\begin{gather*}
s_0(x_0,x_1,x_2) = (x_1+x_2-x_0,x_1,x_2),
\\
s_1(x_0,x_1,x_2) = (x_0,x_0+x_2-x_1,x_2),
\\
s_2(x_0,x_1,x_2) = (x_0,x_1,x_0+x_1-x_2).
\end{gather*}
This action fixes the diagonal copy of~$E$, but does not fix any morphism
{\em to} $E$. It follows that the corresponding action on the dual variety
fixes a morphism to~$E$, but does not fix any curve. In~fact, we~see that
the dual action takes the form
\begin{gather*}
 s_0(x_0,x_1,x_2) = (-x_0,x_0+x_1,x_0+x_2),
 \\
 s_1(x_0,x_1,x_2) = (x_0+x_1,-x_1,x_1+x_2),
 \\
 s_2(x_0,x_1,x_2) = (x_0+x_2,x_1+x_2,-x_2),
\end{gather*}
from which we may see that the corresponding root curves do not even
generate $E^3$.

For our purposes, we~will in fact prefer actions of~coroot type. The~main
issue with actions of~root type in the affine case is that since there are
only finitely many distinct coroots, the kernel of~any given coroot map is
fixed by infinitely many reflections. For~instance, in the above~$\tilde{A}_2$ example, both $s_0$ and $s_1s_2s_1$ fix the hypersurface
$x_1+x_2=2x_0$ pointwise. However, the dual of~the standard model, though
of~coroot type, is badly behaved for other reasons: the product of~root
curves corresponding to the finite Weyl group does not inject, and the
image of~the product is fixed by the translation subgroup.

Suppose $\tW=\langle s_0,\dots,s_n\rangle$ is an affine Weyl group
(with associated finite Weyl group $W=\langle s_1,\dots,s_n\rangle$), and
that the abelian variety $A$ is equipped with an action of~$\tW$ of
coroot type. The~$\tW$-invariant subvariety of~$A^\vee$ has
codimension~$n$, and induces by duality a universal equivariant morphism
$A\to B$ such that $\tW$ acts trivially on~$B$ and the fibers have
dimension~$n$. (In other words, we~may interpret the original action as a
family of actions on~$n$-dimensional varieties.) In contrast, the
invariant subvariety of~$A$ has codimension~$n+1$, and thus its image in~$B$ has codimension~$1$. Thus if we base change by a suitable isogeny
$B'\to B$, we~may arrange for $B$ to be the product of~$A^{\tW0}$ by
an elliptic curve, allowing us to split off that factor and reduce to the
case that $B$ is an elliptic curve $E$. Now, since $W$ is finite, $A^{W0}$
has codimension~$n$, and is thus itself an elliptic curve, which
necessarily surjects onto~$E$. Although this curve $A^{W0}$ is not
preserved by $\tW$, we~may still base change by it, and thus find
that the natural action of~$W$ on~$A^{W0}\times A_W$ extends to an action
of~$\tW$ in such a way that the isogeny $A^{W0}\times A_W\to A$ is
equivariant. (Note, however, that the factorization itself is {\em not}
equivariant; the projection to~$A_W$ is not an equivariant map.)

We can describe this action explicitly on generators. Of course, for~$1\le
i\le n$, $s_i(z,x)=(z,s_i(x))$, so only $s_0$ is nontrivial. The~root
curve associated to~$s_0$ is the same as the root curve associated to the
reflection~$r$ in the relevant root of~$W$, and the action on~$0\times A_W$
is the same as that of~$r$. We thus see that $s_0(z,x)=(z,r(x)+\zeta(z))$
for some (nonzero) morphism $\zeta\colon A^{W0}\to E_r$. Conversely, it is easy
to see that any action of this form has coroot type.

We may view this action as a family of actions of~$\tW$ on~$A_W$
parametrized by $z$, with the one caveat being that the action no longer
preserves the identity; indeed, the translation subgroup of~$\tW$
acts (unsurprisingly) as translations of~$A_W$. Of course, the action on a
given fiber depends only on the point $q:=\zeta(z)$, and we easily see that
it is faithful precisely when~$q$ is non-torsion. (It follows from the
above considerations that this is the typical form of an action of coroot
type, up to base change and twisting by a $(A_W)^W$-torsor.)

Note that in this construction, we~may as well start with a given action of~$W$ and then adjoin~$s_0$. In~non-simply-laced cases, one must choose an
orbit of roots and then obtain~$s_0$ by shifting the action of the
reflection in the highest root of that orbit by an element $q$ of the
corresponding root curve. This produces an action of the affine Weyl group
$W\ltimes \Lambda$, where $\Lambda$ is the free abelian group generated by
the root maps in that orbit. It is worth noting that the choice of orbit is
entirely orthogonal to the direction (if any) of the arrows in the finite
Dynkin diagram: e.g., each of the three versions of the finite diagram
$B_n=C_n$ gives rise to both an action of~$\tilde{B}_n$ and an action of~$\tilde{C}_n$, and similarly in the $G_2$ and $F_4$ cases, there are two
natural extensions to actions of the corresponding affine groups. It is
also worth noting that the various base changes required to put the action
in this form can eliminate some of the information present in the original
coroot model; in particular, in the $\tilde{C}_n$ case (including
$\tilde{C}_1=\tilde{A}_1$), the special node is connected to the rest of
the Dynkin diagram by an arrow, and thus there is a~choice of isogeny in
the coroot model. This includes some exotic coroot models in which~$E_0$
and~$E_n$ are merely $4$-isogenous and the corresponding family of abelian
varieties with~$\tilde{C}_n$ action has no section.

Returning to the finite case, suppose that $A/S$ is a family of abelian
varieties (over an integral base $S$) equipped with a faithful action of
the finite Weyl group $W$ by reflections, and suppose moreover that we are
given a $W$-invariant ample line bundle ${\cal L}$ on~$A$. This can be
made equivariant by taking the action on the fiber at $1$ to be trivial,
and we may then ask when the map $s\mapsto \dim\Gamma(A_s;{\cal L})^W$ is
constant on~$S$. By the reductions of the previous section, this reduces
to considering the corresponding question for $A_W$, which is very nearly a
variety of the form we considered above. To~be precise, the root curves in
each irreducible component of~$W$ are isogenous, and since indecomposable
finite Weyl groups have at most two conjugacy classes of reflections, we~see that each component is associated to a~point of~${\cal X}_0$, ${\cal
 X}_0(2)$, or ${\cal X}_0(3)$. To~ensure that we can apply our previous
results, we~must insist that the induced line bundles on the root curves be
suitable; to wit, we~insist that for each reflection, ${\cal L}|_{E_r}\cong
{\cal L}_{2d_r,0;E_r}$ for some positive integer $d_r$, clearly constant on
conjugacy classes of involutions. We thus see that the only possible
issues arise when (a) one of the curves $E_r$ is supersingular of
characteristic dividing~$W$, or (b) the ``root kernel'', i.e., the kernel
of~$\prod_i E_i\to A$, fails to be diagonalizable. In~fact, the first
condition turns out not to be necessary.

\begin{lem}
 Let ${\cal L}$ be a $W$-invariant ample line bundle on an abelian variety
 $A$ such that there are positive integers $d_i$ such that ${\cal
 L}|_{E_i}\cong {\cal L}_{2d_i,0;E_i}$ for each $i$. Then
 for any section~$f\in \Gamma(A;{\cal L})$, the antisymmetrization~$\sum_{w\in W} \sigma(w) wf$ vanishes along the divisor $\sum_{r\in R(W)}
 [\ker(r-1)]$.
\end{lem}

\begin{proof}
 We may write $\sum_{w\in W} \sigma(w) w = (1-r)\sum_{w\in W_0} w$, where
 $W_0$ is the even subgroup of~$W$. Since the divisors $\ker(r-1)$ are
 transverse for distinct $r$, it thus suffices to show that $(1-r)f$
 vanishes along the divisor $[\ker(r-1)]$. We thus reduce to the case
 that $W$ has rank~1. In~other words, $A$ is a quotient of a variety
 $B\times E$ (with $r$ acting trivially on~$B$) obtained by identifying
 some subgroup $K\subset E[2]$ with a subgroup of~$B$. The~action of~$r$
 lifts to~$B\times E$, and we conclude (by considering how $[-1]$ acts on
 sections of~${\cal L}_{2d}$) that the antisymmetrization of any section
 of the pulled back line bundle must vanish on the divisor $B\times E[2]$.
 This is the preimage of the divisor $(B\times E[2])/K$, which in turn is
 precisely the kernel of~$r-1$ as required.
\end{proof}

This gives us the following possible approach to controlling invariants in
such bundles. Let ${\cal L}_\Delta$ be the line bundle $\sO_A\big(\sum_{r\in
 R(W)} [\ker(r-1)]\big)$, but equipped with the equivariant structure which is
trivial at the identity. If there is a section~$g\in \Gamma(A_W;{\cal
 L}_{\Delta})$ with nontrivial antisymmetrization, then the operation~$f\mapsto \frac{\sum_{w\in W} \sigma(w) w(gf)} {\sum_{w\in W} \sigma(w)
 w(g)}$ induces an idempotent on any $\Gamma(A;{\cal L})$ which projects
onto the symmetric subspace. More generally, if we have a family of such
varieties such that such a section~$g$ exists locally (or, equivalently, on
every fiber), then we obtain such idempotents locally, and thus the spaces
$\Gamma(A_s;{\cal L})^W$ are fibers of a vector bundle, implying that their
dimensions are constant.

\begin{thm}\label{thm:W-invariants}
 Suppose $A/S$ is a family of abelian varieties equipped with a faithful
 action by reflections of the finite Weyl group $W$, and let ${\cal L}$ be
 a $W$-invariant ample line bundle on~$A$ such that the restriction to
 every root curve of every fiber is isomorphic to an even power of~${\cal
 L}_1$. If the root kernel of~$A$ is diagonalizable, then the functions
 $s\mapsto \dim\Gamma(A_s;{\cal L})^W$ and $s\mapsto \dim\big(\big(\sum_{w\in W}
 \sigma(w) w\big)\Gamma(A_s;{\cal L})\big)$ are constant on~$S$.
\end{thm}

\begin{proof}
 We first observe that by Lemma~\ref{lem:G-mod-factor}, the claims hold
 for $A$ iff they hold for $A_W$, a.k.a. the image of~$\prod_i E_i\to A$.
 We may thus without loss of generality assume that the morphism $\prod_i
 E_i\to A$ is an isogeny such that $W$ acts trivially on the root kernel~$K$. By assumption, $K$~is diagonalizable, so that $\Gamma\big(\prod_i
 E_i;{\cal L}\big)$ decomposes into~$K$-eigenspaces, and this decomposition is~compatible with the action of~$W$. It then follows by semicontinuity
 that the claims hold for~$A$ if they hold for $\prod_i E_i$. Since this
 is a product over the components of~$W$, we~may assume without loss of
 generality that $W$ is indecomposable.

 We may then reduce as discussed to showing that for any elliptic root
 datum corresponding to an indecomposable finite Weyl group, the
 corresponding line bundle ${\cal L}_\Delta$ contains a section with
 nontrivial antisymmetrization. (There is also the technical, but easy to
 verify, condition that ${\cal L}_\Delta|_{E_r}\cong {\cal
 L}_{2d_r,0;E_r}$ for suitable positive integers.)

 Here we may use the classification of finite Weyl groups. The~simplest
 case is $W=A_n$, in~which case we may identify $\prod_i E_i$ with the
 subvariety of~$E^{n+1}$ on which $\sum_i z_i=0$. By induction in~$n$
 (with trivial base case $n=0$), the result holds for $n-1$, and thus any
 $S_n$-anti-invariant section can be obtained by antisymmetrization over
 $S_n$. It thus suffices to show that there is an $S_n$-anti-invariant
 section that when summed over coset representatives of~$S_{n+1}/S_n$ with
 appropriate sign gives a nonzero result. Equivalently, by dividing by
 the appropriate product of~$\vartheta$ functions, we~need to find an
 $S_n$-invariant function with suitable poles that symmetrizes to a~nonzero constant. For~auxiliary parameters $y_1,\dots,y_{n+2}$, we~may
 consider the function
 \begin{gather*}
 \frac{\prod_{1\le i\le n+2} \vartheta(z_{n+1}-y_i)
 \prod_{1\le i\le n} \vartheta(Y-z_i)}
 {\prod_{1\le i\le n} \vartheta(z_{n+1}-z_i)},
 \end{gather*}
 where $Y=\sum_{1\le i\le n+2} y_i$. Summing this over $S_{n+1}/S_n$
 gives a function with no poles, which must therefore be constant; on the
 other hand, of the $n+1$ terms that result, all but one vanishes when~$z_{n+1}=Y$. We thus find that
 \begin{gather*}
 \sum_{w\in S_{n+1}/S_n}\!\!\!w\cdot
 \frac{\prod_{1\le i\le n+2} \vartheta(z_{n+1}-y_i)
 \prod_{1\le i\le n} \vartheta(Y-z_i)}
 {\prod_{1\le i\le n} \vartheta(z_{n+1}-z_i)}
 =
 \prod_{1\le i\le n+2} \vartheta(Y-y_i),
 \end{gather*}
 which is generically nonzero. (Note that the case $n=1$ is a version of
 the standard addition law for theta functions.) This identity is a
 disguised form of a classical theta function identity; see the discussion
 around~\cite[equation~(1.22)]{RosengrenH:2016}.

 For types $B/C/D$, we~may similarly reduce to lower rank cases, noting
 that $D_2$ and $B_1$ both follow from the result for $A_1$. There are
 four cases to consider: the action of~$D_n$ on~$E\otimes \Lambda_{D_n}$,
 the action of~$B_n$ on the same variety, the action of~$C_n$ on~$E\otimes
 \Z^n$ (following~\cite{LooijengaE:1976}, we~label the cases by the dual
 root system), and the action of~$BC_n$ on the variety
 $E_{\Lambda_{D_n},\Z^n}$ associated to a~point of~${\cal X}_0(2)$ lying
 over $E$. In~each case, there is a natural isogeny to~$E\otimes \Z^n$,
 and it turns out we can choose the function being symmetrized to be the
 pullback of a function on~$E\otimes \Z^n$. The~simplest identity
 corresponds to the $C_n$ case, valid for all $n\ge 1$:
 \begin{gather*}
 \sum_{w\in C_n/C_{n-1}}\!\!\!w\cdot
 \frac{\prod_{1\le i\le 2n+1} \vartheta(z_n-y_i)
 \prod_{1\le i\le n} \vartheta(Y+z_i)
 \prod_{1\le i<n} \vartheta(Y-z_i)}
 {\vartheta(2z_n)\prod_{1\le i<n} \vartheta(z_n+z_i)\vartheta(z_n-z_i)}
 \\ \hphantom{\sum_{w\in C_n/C_{n-1}}\!\!\!}
{} = \prod_{1\le i\le 2n+1} \vartheta(Y-y_i),
 \end{gather*}
 with~$Y=\sum_{1\le i\le 2n+1} y_i$; if we expand this out as a sum of~$2n$
 terms, we~find that it is simply the special case
 $(z_1,\dots,z_{2n})\mapsto (-z_n,\dots,-z_1,z_1,\dots,z_n)$ of the
 $S_{2n}/S_{2n-1}$ identity. In~characteristic not 2, we~may set
 $y_{2n-2},\dots,y_{2n+1}$ to be the four points of~$E[2]$ to obtain an
 identity for $D_n$, $n>2$:
 \begin{gather*}
 \sum_{w\in D_n/D_{n-1}}\!\!\! w\cdot
 \frac{\prod_{1\le i\le 2n-3} \vartheta(z_n-y_i)
 \prod_{1\le i\le n} \vartheta(Y+z_i)
 \prod_{1\le i<n} \vartheta(Y-z_i)}
 {\prod_{1\le i<n} \vartheta(z_n+z_i)\vartheta(z_n-z_i)}
 \\ \hphantom{\sum_{w\in D_n/D_{n-1}}\!\!\!}
 {}= \vartheta(2Y) \prod_{1\le i\le 2n-3} \vartheta(Y-y_i),
 \end{gather*}
 where $Y=\sum_{1\le i\le 2n-3} y_i$. Since this identity is expressed
 entirely in terms of~$\vartheta$, it continues to hold in characteristic~2. If we only specialize three parameters to the nonzero $2$-torsion
 points, we~instead obtain (for $n\ge 2$):
 \begin{gather*}
 \sum_{w\in B_n/B_{n-1}}\!\!\!w\cdot
 \frac{\prod_{1\le i\le 2n-2} \vartheta(z_n-y_i)
 \prod_{1\le i\le n} \vartheta(Y+z_i)
 \prod_{1\le i<n} \vartheta(Y-z_i)}
 {\vartheta(z_n)\prod_{1\le i<n} \vartheta(z_n+z_i)\vartheta(z_n-z_i)}
 \\ \hphantom{\sum_{w\in B_n/B_{n-1}}\!\!\!}
 {}= \frac{\vartheta(2Y)}{\vartheta(Y)} \prod_{1\le i\le 2n-2} \vartheta(Y-y_i).
 \end{gather*}
 We omit the analogous identity for $BC_n$ (obtained from the $C_n$
 identity by setting two of the~$y_i$ to be the $2$-torsion points not in
 the kernel of~$\phi$) due to notational difficulties with using
 $\vartheta$ in the presence of isogenies, but note that for purposes of
 the reduction there is no reason we cannot simply use the $B_n$ identity.

 For the seven exceptional cases (two each of~$G_2$ and $F_4$ along with
 the simply laced cases~$E_6$, $E_7$, $E_8$), we~observe that the relevant
 line bundle comes from a polarization of degree a~multiple of~6, and we
 may thus use Lemma~\ref{lem:not_prime_power} to reduce to a smaller
 group. That is, we~obtain an~equivariant isomorphism
 \begin{gather*}
 \Gamma(A;{\cal L}_{\Delta}) \cong \bigoplus_H \Ind_H^W \Res^W_H
 \Gamma(A';{\cal L}_{\Delta}),
 \end{gather*}
 where $H$ ranges over the point stabilizers in the different orbits of
 the action of~$G$ on an appro\-priate diagonalizable $2$- or $3$-group. The~image of antisymmetrization on the left is thus the direct sum of
 terms
 \begin{gather*}
 \bigg(
 \sum_{h\in H} \sigma(h) h\bigg) \Res^W_H \Gamma(A';{\cal L}_{\Delta}) ,
 \end{gather*}
 and since $H$ is a reflection group in each case, we~may apply induction.
 (In fact, only that term which is nonzero in characteristic 0 has any
 hope of contributing).

 For instance, for~$E_8$, the variety is $E^8$ with polarization given by
 $30$ times the Cartan matrix of~$E_8$. In~characteristic not 2, we~may
 use the $2$-part of~$\Lambda_{E_8}/30\Lambda_{E_8}$ to split into
 eigenspaces, of which only one term survives. We thus reduce to showing
 that the descended line bundle on~$E_{\Lambda_{E_8},2\Lambda_{E_8}}\cong
 E_2^8$ (with nonproduct polarization) has a nontrivial antisymmetrization
 under the stabilizer $W(D_8)$. This is $2$-isogenous to the standard
 $D_8$ model, and thus (since the characteristic is not~2) the image of
 antisymmetrization has the same dimension as in characteristic~0. A~similar reduction using the $3$-part reduces to antisymmetrization for~$W(A_8)$ and proves the result for any characteristic other than~3.
\end{proof}

\begin{rems}
Results of~\cite{LooijengaE:1976,SaitoK:1990} in characteristic 0 actually
compute the structure of the invariant ring $\big($i.e., $\bigoplus_d
\Gamma\big(E_{B,C};{\cal L}_Q^d\big)$, where $Q$ is the minimal invariant
polarization satisfying the evenness requirement$\big)$ and find that in each
case the result is a free polynomial ring in generators of degrees that can
be read off of the coefficients of the highest short (co)root. This
suggests that something similar should hold in arbitrary characteristic.
It would be natural in this context to also consider actions of complex
reflection groups on varieties isogenous to~$E^n$ where $j(E)\in
\{0,1728\}$, or even quaternionic reflection groups in the case of
supersingular curves of characteristic~2 or 3.
\end{rems}

\begin{rems}
 The diagonalizability hypothesis is necessary, at least as far as the
 antisymmetrization claim is concerned. For~example, suppose $x\in E[3]$
 is a nontrivial 3-torsion point, and consider the quotient $A$ of the sum
 0 subvariety of~$E^3$ by the subgroup generated by $(x,x,x)$. The~image
 in~$A$ of the point $(0,x,-x)$ is negated by every reflection, and is
 thus not contained in~any reflection hypersurface, but still has
 nontrivial stabilizer $\Z/3\Z\subset S_3$. It follows that in~characteristic 3, the antisymmetrization of any section of an ample line
 bundle will vanish at this point, and thus for {\em no} ample line bundle
 is the dimension of the image of antisymmetrization the same as in~characteristic 0 (except, of course, when there are no antisymmetric
 elements in~characteristic 0). The~invariants remain~well-behaved,
 however, as~the invariant ring is the same as that of~$\langle x\rangle$
 on~$\P^2$; this is a permutation representation, so its Hilbert series
 is independent of the characteristic.
\end{rems}

In the sequel, we~will need some understanding of equivariant gerbes on
abelian schemes, and thus in particular will want to understand
$H^1(W;\Pic(A))$. This turns out to be {\em nearly} trivial when~$W$ is
finite and the map $\prod_i E_i\to A$ is an isomorphism, in that any
nontrivial cohomology class remains nontrivial when restricted to some
simple reflection. (In particular, $H^1(W;\Pic(A))$ is 2-torsion.)

For inductive purposes, we~will need to consider a slightly larger class of
actions. Given an~elliptic root datum of finite type and a~point $u\in
\prod_i E_i(S)$, we~may define an action of~$W$ on~$\prod_i E_i$ by
\begin{gather*}
 s_i(x_1,\dots,x_n) =
\bigg(x_1,\dots,x_{i-1},-x_i+\sum_{j\ne i} \mu_{ji}(x_j)+u_i,x_{i+1},\dots,x_n
\bigg);
\end{gather*}
this is of course no longer a {\em pointed} action of~$W$, but is instead
the twist by a class in~$H^1\big(S;A^W\big)$ determined by $u$. (More precisely,
the cocycle corresponding to~$u$ is fppf locally a coboundary, and the
nonuniqueness of this representation gives a class in~$H^1\big(S;A^W\big)$.) (And,
of course, a~class in~$H^1\big(S;A^W\big)$ is representable in this form iff the
corresponding $A$-torsor has a section.)

\begin{lem}
 Let $W$ be a finite Weyl group of rank $2$, and let $X=E_1\times E_2$ be
 equipped with an action of the above form such that the root datum has
 $r_1\ge r_2$. Then the kernel of~$\sum_{w\in W} (-1)^{\ell(w)} w$ on~$\Pic(X)$ is spanned by $\pi_1^*\Pic(E_1)$ and $\Pic(X)^{\langle
 s_1\rangle}$.
\end{lem}

\begin{proof}
 It suffices to prove the corresponding claims for $\Pic^0(X)$ and
 $\NS(X)$ (modulo the requ\-ir\-e\-ment in the latter case that the classes lift
 to actual bundles still of the required form). Since $\Pic^0(X)$ is
 generated by $\pi_1^*\Pic^0(E_1)$ and $\pi_2^*\Pic^0(E_2)$, with the
 latter $s_1$-invariant, we~find that the antisymmetrizer vanishes on~$\Pic^0(X)$, and the claim is immediate. The~N\'eron--Severi group may be
 identified with the space of matrices of the form
 \begin{gather*}
 Q=\begin{pmatrix} a & b \\ b^\vee & c\end{pmatrix}\!,
 \end{gather*}
 with antisymmetrization
 \begin{gather*}
 \begin{pmatrix} 0 & 4b \\ 4b^\vee & 0\end{pmatrix}
 \end{gather*}
 when $m_{12}=2$ and
 \begin{gather*}
 \begin{pmatrix}
 0&m_{12}\big(b\mu_{21}^{-1}-\big(b\mu_{21}^{-1}\big)^\vee\big)\mu_{21}\\
 -m_{12}\mu_{21}^\vee\big(b\mu_{21}^{-1}-\big(b\mu_{21}^{-1}\big)^\vee\big) & 0
 \end{pmatrix}\!,
 \end{gather*}
 when $m_{12}\in \{3,4,6\}$. It follows in either case that the vanishing
 of the antisymmetrization implies $b\in \Z\mu_{21}$. (Here for
 $m_{12}\in \{3,4,6\}$, we~use the fact that $\deg(\mu_{21})\le
 \deg(\mu_{12})$ and $\mu_{12}\mu_{21}\in \{0,1,2,3\}$, so that
 $\deg(\mu_{21})$ is squarefree and an isogeny is in~$\Q\mu_{21}$ iff it
 is in~$\Z\mu_{21}$.) In particular, the kernel of the antisymmetrizer is
 spanned by the matrices
 \begin{gather*}
 \begin{pmatrix} 1 & 0 \\ 0 & 0\end{pmatrix}\!,\qquad
 \begin{pmatrix} 2 & -\mu_{21} \\ -\mu_{21}^\vee &
 \mu_{21}^\vee\mu_{21}\end{pmatrix}\!,\qquad
 \begin{pmatrix} 0 & 0 \\ 0 & 1\end{pmatrix}\!.
 \end{gather*}
 The first matrix is in the pullback of~$\pi_1^*$, while the second and
 third are $s_1$-invariant (and, in fact, images of~$s_1$-invariant
 bundles).
\end{proof}

\begin{prop}\label{prop:coboundary_if_in_rank_1}
 Let $X/S=\prod_i E_i$ be equipped with an action of the finite Weyl group
 $W$ of the above form, and let $z\in Z^1(W;\Pic(X))$ be such that
 $z_{s_i}$ is a coboundary for every $i$. Then $z$ is a coboundary.
\end{prop}

\begin{proof}
 Since $W$ is finite, its diagram is a forest, and we may thus order the
 roots in such a~way that $s_1$ corresponds to a leaf, and thus without
 loss of generality commutes with~$s_3,\dots,s_n$. Moreover, since
 every component has at most two values of~$r_i$, we~may arrange to have
 $r_1\ge r_i$ for all $i$. If we view $X$ as a family of abelian
 varieties over $E_1$, then we see that the action of~$W_{\{2,\dots,n\}}$
 is still of the above form, and thus the corresponding restriction of~$z$
 is a coboundary by induction. We may thus reduce to the case that
 $z_{s_2}=\cdots=z_{s_n}=0$.

 Now, consider the action of~$\langle s_1,s_2\rangle$ on~$X$ viewed as a
 family over $\prod_{3\le i\le n} E_i$, and let $v\in \Pic(X)$ be such
 that $z_{s_1} = v-{}^{s_1}v$. The~cocycle condition on~$z_{s_1}$ implies
 that $\sum_{w\in \langle s_1,s_2\rangle} (-1)^{\ell(w)} {}^w v$ $=0$, and
 thus the lemma implies that we may replace $v$ by the pullback of a class
 in~$E_1\times \prod_{3\le i\le n} E_i$ without changing $z_{s_1}$. Since
 the cocycle is trivial on~$W_{\{3,\dots,n\}}$ and $s_1$ commutes with
 this subgroup, $z_{s_1}$ must be $W_{\{3,\dots,n\}}$-invariant, and thus
 a computation in the N\'eron--Severi group shows that $v$ is in the
 subgroup generated by $\Pic(E_1)$ and $\Pic\big(\prod_{3\le i\le n}E_i\big)$.
 Since elements of the latter have no effect on~$z_{s_1}$, we~see that we
 may take $v\in \pi_1^*\Pic(E_1)$. This is $W_{\{2,\dots,n\}}$-invariant,
 and thus $z=\partial v$ as required.
\end{proof}

\begin{rem}
 Note that this can fail if we quotient by a subgroup of~$A^W$. Indeed,
 if $A=E^2$ is given the standard action of~$W(A_2)$, then the result
 fails for the induced action on~$A^\vee\cong E^2$. The~cocycles in which
 all bundles have degree 0 (and which are coboundaries in rank~1) are
 themselves classified by $A^\vee(S)$, while a coboundary must be in the
 image of~$A(S)$ under the invariant polarization, and these are different
 unless $E(S)$ is $3$-divisible. There is no difficulty with the
 polarization, however, so although the induction breaks down, it may
 still be the case that the claim holds on the N\'eron--Severi group,
 which would imply an fppf local version of the proposition.
\end{rem}

\section{Elliptic analogues of affine Hecke algebras}\label{section4}

Before proceeding to the construction of Hecke algebras associated to
general elliptic root data, it will be helpful to consider the finite case,
a generalization of the construction of
\cite{GinzburgV/KapranovM/VasserotE:1997}. Note that although we work with
a finite group, the resulting Hecke algebras are most naturally thought of
as elliptic analogues of {\em affine} Hecke algebras, as~they include
multiplication operators in addition to reflection operators. (In
particular,~\cite{GinzburgV/KapranovM/VasserotE:1997} constructs affine
Hecke algebras as degenerations of~a~special case of the construction given
below.)

Although the approach in~\cite{GinzburgV/KapranovM/VasserotE:1997} via
residue conditions can (mostly) be extended to the infinite case, there are
two alternate approaches for which the generalization is more
straightforward: one as a space of operators preserving appropriate
holomorphy conditions, the other as the subalgebra of the algebra of
operators generated by the rank~1 subalgebras. Since we will need to
understand the rank~1 case to give the second construction, we~begin with
the first.

In our application to noncommutative rational varieties below, we~will need
to be able to attach an arbitrary finite set of parameters to the endpoint
roots of the affine $C_n$ diagram; we~will thus give a version of the
general construction in which each conjugacy class of reflections can be
given arbitrarily many parameters. This, of course, includes a case
without any parameters at all, which we consider first.

This ``master'' Hecke algebra has a third description which does not
generalize well to the infinite case, but is simplest of all to give (and
extend to actions of arbitrary finite groups). Let~$X$ be a regular
integral scheme, and let $G$ be a finite group acting faithfully on~$X$.
Then we define the master Hecke algebra ${\cal H}_G(X)$ to be the sheaf of
algebras on~$X/G$ given by ${\cal H}_G(X):=\sEnd(\pi_*\sO_X)$, where
$\pi\colon X\to X/G$ is the quotient map.

If $\pi$ is flat, then ${\cal H}_G(X)$ is the endomorphism ring of a vector
bundle, so is in particular an Azumaya algebra on~$X/G$, and the category
of quasicoherent ${\cal H}_G(X)$-modules is equivalent to the category of
quasicoherent sheaves on~$X/G$. However, this condition holds only rarely;
even in the case when~$G$ is a finite Weyl group acting on an abelian
variety, this morphism can easily fail to be flat. For~instance, consider
the case $G=G_2$ acting on the sum zero subvariety~$X$ of~$E^3$ (as
permutations and global negation). In~characteristic not 2, consider the
point $(\tau_1,\tau_2,\tau_1+\tau_2)\in E^3$, where $\tau_1$, $\tau_2$
generate $E[2]$. This point has stabilizer $Z(G)$, and is isolated in the
subvariety fixed by its stabilizer, and thus we see that its image in~$X/G$
is a~singular point (of type $A_1$), and that the quotient morphism fails
to be flat in a neighborhood of that orbit.

There are two prominent cases in which we do have flatness, namely the
action of~$A_n$ on the sum 0 subvariety of~$E^{n+1}$ and the action of~$C_n$ on~$E^n$; in each case, the quotient morphism is flat because it is a
quasi-finite morphism between regular schemes (from $E^n$ to~$\P^n$, to be
precise).

In general, although ${\cal H}_G(X)$ may not be an Azumaya algebra, we~at
least know that it is torsion-free, and thus may be viewed as contained in
its generic fiber $\End_{k(X/G)}(k(X))$. Since~$k(X)$ is Galois over
$k(X/G)$, the generic fiber has an alternate description as a twisted group
algebra $k(X)[G]$, giving rise to the following description of~${\cal
 H}_G(X)$. Denote the natural action of~$\Aut(X)$ on~$k(X)$ by ${}^g f:=
(g^{-1})^*f$; we will also use a similar notation for the actions on~line
bundles and divisors. It will also be convenient to define (for finite
$G$) ${}^G f = \pi_G^*f$, where $f\in k(X/G)$ and $\pi\colon X\to X/G$ is the
natural quotient.

\begin{prop}
 The master Hecke algebra ${\cal H}_G(X)$ is the subsheaf of the
 twisted group algebra $k(X)[G]$ such that for any $G$-invariant open
 subset $U$, $\Gamma(U/G;{\cal H}_G(X))$ consists of the operators $\sum_i
 c_i g$ such that for any $G$-invariant open $V$ and $f\in \Gamma(V;\sO_X)$,
 $\sum_i c_i {}^g f\in \Gamma(U\cap V;\sO_X)$.
\end{prop}

\begin{proof}
 Indeed, ${\cal H}_G(X)$ is the subsheaf of~$\End_{k(X/G)}(k(X))$ which on~$U$ consists of endomorphisms preserving $\sO_X|_U$, or equivalently
 preserving global sections of~$\sO_X|_V$ for all invariant $V\subset U$. The~claim then follows by using the twisted group algebra description of
 the endomorphism ring and observing that $\Gamma(V;\sO_X)\subset
 \Gamma(U\cap V;\sO_X)$.
\end{proof}

One consequence is that if $H\subset G$, then there is a natural inclusion~${\cal H}_H(X)\subset {\cal H}_G(X)$ (where we conflate ${\cal H}_H(X)$
with its direct image under $X/H\to X/G$). In~addition, if $\alpha$ is
an~automorphism of~$X$ that normalizes $G$, then there is a corresponding
automorphism of~${\cal H}_G(X)$, either by pulling back through the induced
automorphism of~$X/G$, or on operators as $\sum_g c_g g \mapsto \sum_g
{}^\alpha c_{\alpha^{-1} g\alpha} g$.

Note that ${\cal H}_G(X)$ clearly contains a copy of the structure sheaf
$\sO_X$ as well as the operators~$g$ for each $g\in G$, and thus contains a
copy of the twisted group algebra $\sO_X[G]$. It can, however, be bigger
than the twisted group algebra. Consider the case of~$G=\mu_n=\langle
s\rangle$ acting on~$X=\A^1$ in characteristic prime to~$n$ by rescaling
the variable $x$. Applying the operator $1+\zeta_n s+\zeta_n^2 s^2+\cdots+
\zeta_n^{n-1} s^{n-1}$ to any function which is holomorphic at the origin
gives a function which vanishes to order $n-1$ at the origin, and thus
${\cal H}_{\mu_n}\big(\A^1\big)$ contains the operator $x^{1-n}\big(1+\zeta_n
s+\zeta_n^2 s^2+\cdots +\zeta_n^{n-1} s^{n-1}\big)$ not contained in~$\sO_X[G]$.

It turns out that this is the typical case in which the coefficients may
have poles; more precisely, the only poles are associated to ``complex
reflections'' in~$G$ (relative to the action on~$X$). It will be useful to
consider a more general setting. We begin with a couple of local results.

\begin{lem}
 Suppose $R$, $S$ are discrete valuation rings, and suppose
 $\psi_1,\dots,\psi_n\colon S\to R$ are distinct finite homomorphisms. Then the
 module of operators $\sum_i c_i \psi_i$ mapping $S$ to~$R$ is a free
 $R$-module of rank~$n$.
\end{lem}

\begin{proof}
 Consider the subalgebra $R\vec\psi S\subset R^n$ generated by the image
 of~$S$ under $(\psi_1,\dots,\psi_n)$. The~monoid characters
 $\psi_i\colon S\setminus\{0\}\to R^*$ are linearly independent over $K_R$ and
 thus over $R$, so that $R\vec\psi S$ is free of rank $n$ as an
 $R$-module. An operator $\sum_i c_i \psi_i$ maps $S$ to~$R$ iff the
 $K_R$-linear functional $\vec{c}$ maps $R\vec\psi S$ to~$R$, and thus the
 space of such operators is isomorphic to the dual~$R^n$.
\end{proof}

\begin{rem}
 It follows that the quotient by the submodule in which the coefficients
 are in~$R$ has finite length, equal to the colength of~$R\vec\psi S$ as a
 submodule of~$R^n$.
\end{rem}

For $n=1$, $R\vec\psi S=R$, and thus $(R\vec\psi S)^*=R$. We can also give
an explicit description for $n=2$.

\begin{cor}\label{cor:local_order_two_residue_conditions}
 Let $\psi_1,\psi_2\colon S\to R$ be finite homomorphisms of discrete valuation
 rings inducing the same action on residue fields. Then the operator
 $c_1\psi_1+c_2\psi_2\colon k(S)\to k(R)$ maps $S$ to~$R$ iff~$c_1,c_2\in
 \det(R\vec\psi S)^{-1}$ and $c_1+c_2\in R$.
\end{cor}

\begin{proof}
 The module $R\vec\psi S$ is spanned by elements of the form
 $(\psi_1(h),\psi_2(h))$, or equivalently by the element $(1,1)$ and
 elements of the form $(0,\psi_2(h)-\psi_1(h))$, and thus can be expressed
 as~$R(1,1)+\det\big(R\vec\psi S\big)(0,1)$. The~corresponding condition on the
 operator is that $c_1+c_2\in R$ and $c_2\in \det\big(R\vec\psi S\big)^{-1}$.
 These conditions imply that $c_1\in \det\big(R\vec\psi S\big)^{-1}$ as well.
\end{proof}

\begin{rems}
 For any $f\in S^*$, since $c_1\psi_1(f)+c_2\psi_2(f)-\psi_1(f)(c_1+c_2) =
 (\psi_2(f)-\psi_1(f)) c_2$ and $\psi_2(f)-\psi_1(f)\in
 \det\big(R\vec\psi\big)$ we see that when $c_2\in \det\big(R\vec\psi\big)^{-1}$, the
 conditions $c_1+c_2\in R$ and $c_1\psi_1(f)+c_2\psi_2(f)\in R$ are equivalent.
 This is useful in global situations in which one has twisted by a line bundle.
\end{rems}

\begin{rems}
 It follows that
 $(R(\psi_1,\psi_2)S)^*=R\psi_1 + \det\big(R\vec\psi S\big)^{-1}(\psi_1-\psi_2)$.
\end{rems}

For $n>2$, it is difficult to give an explicit description in general, but
in the case of Coxeter groups, we~can generally reduce to the $n=2$ case,
using the following result. Given a homomorphism $\psi\colon S\to R$ of local
rings, let $\bar\psi\colon k_S\to k_R$ denote the corresponding homomorphism of
residue fields. Also, note that for any subset $I\subset \{1,\dots,n\}$,
there is a natural morphism $R\vec\psi S\to R(\psi_i\colon i\in I)S$.

\begin{lem}\label{lem:holomorphy_preserving_splits}
Suppose $R$, $S$ are discrete valuation rings, and suppose
$\psi_1,\dots,\psi_n\colon S\to R$ are~distinct finite homomorphisms. Then the
algebra $R\vec\psi S$ splits as the direct sum $\bigoplus_\sigma
R(\psi_i\colon \bar\psi_i=\sigma)S$ of local rings, where $\sigma$ ranges over
all morphisms $k_S\to k_R$.
\end{lem}

\begin{proof}
The radical of~$R\vec\psi S$ is equal to its intersection with~${\mathfrak
 m}_R^n$, and the quotient is a product of copies of~$k_R$, with one for
each distinct reduction~$\bar\psi_i$ (in particular, the radical is maximal
and $R\vec\psi S$ local iff there is only one such reduction). If $R$ is
complete, then we can lift the idempotents from the reduction to obtain the
desired splitting. Moreover, the lifts are unique, and thus agree with the
lifts one would have obtained if working inside the larger algebra $R^n$,
i.e., the projections onto the given sets of coordinates. It thus follows
that even if $R$ is not complete, the lifts in the completion of~$R\vec\psi
S$ agree with the projections, and thus said projections lie in~$R\vec\psi
S$.
\end{proof}

\begin{rem}
By duality, the same splitting applies to the module of operators taking
$S$ to~$R$.
\end{rem}

This leads to the following global result.

\begin{lem}\label{lem:holomorphy_preserving}
 Let $X$, $Y$ be normal integral schemes and let $\phi_1,\dots,\phi_n\colon X\to
 Y$ be a collection of distinct finite morphisms. Let
 ${\cal M}_{\vec{\phi}}$ be the subsheaf of~$k(X)^n$ which on an open subset
 $U\subset X$ consists of those $n$-tuples $(c_1,\dots,c_n)\in k(X)^n$
 such that for any open subset $V\subset Y$ and any function~$f\in k(Y)$
 holomorphic on~$V$, the function~$\sum_i c_i \phi_i^* f$ is holomorphic
 on~$U\cap \bigcap_i \phi_i^{-1}V$. Then there is an $n$-tuple of divisors
 $\Delta_i\in \Div(X)$ such that
 ${\cal M}_{\vec{\phi}}\subset
 \bigoplus_i \sO_X(\Delta_i)$, and each $\Delta_i$ is supported on those
 hypersurfaces on which $\phi_j=\phi_i$ for some $j\ne i$.
\end{lem}

\begin{proof}
Let $D\subset X$ be a reduced irreducible hypersurface; we need to
understand the possible singularities of the coefficients along $D$. Note
that if $U_1$, $U_2$ are two open subsets meeting $D$, then $U_1\cap U_2$ also
meets $D$, and thus any bound on singularities holding on~$U_1\cap U_2$
also holds for global sections along $U_1$, $U_2$. We may thus take a limit
along those open subsets meeting $D$. Similarly, since we are only
considering holomorphy along $D$, we~may take a limit over those $V$
meeting $\phi_i(D)$ for every $i$. In~other words, the condition for
$(c_1,\dots,c_n)$ to be a section of the base change of~${\cal
M}_{\vec{\phi}}$ to the local ring of~$k(D)$ is that for any function~$f$
which is holomorphic along $\phi_1(D),\dots,\phi_n(D)$, the image $\sum_i
c_i \phi^*_i f$ is holomorphic along $D$.

Now, given such an $n$-tuple, let $d$ be the maximum order of pole of a
coefficient $c_i$ along~$D$. The~condition that $\sum_i c_i \phi^*_i f$ be
holomorphic only depends on the value of~$f$ modulo the inter\-sec\-ti\-ons of
the $d$-th powers of the maximal ideals at the divisors $\phi_i(D)$, and by
the Chinese remainder theorem, the reductions corresponding to distinct
divisors may be chosen independently. Taking those reductions to be 0
on all but $\phi_j(D)$ tells us that if the operator $\sum_i c_i
\phi^*_i$ preserves holomorphy, then so does
$\sum_{i:\phi_i(D)=\phi_j(D)} c_i \phi^*_i$.

We may as well assume, therefore, that the divisors $\phi_i(D)$ are all equal
to the same divisor~$D'$ in~$Y$, and thus reduce to the local case.
\end{proof}

\begin{rem}
 By mild abuse of notation, if $g_1,\dots,g_n\in \Aut(X)$, then ${\cal
 M}_{\vec{g}}$ will denote the sheaf corresponding to the $n$-tuple
 $\vec{\phi}=(g_1^{-1},\dots,g_n^{-1})$; i.e., the sheaf of operators
 $\sum_i c_i g_i$ that preserve holomorphy. This generalizes to the case
 of operators $\sum_i c_i g_i G\colon k(X/G)\to k(X)$ for a finite subgroup
 $G\subset \Aut(X)$, which act as $f\mapsto \sum_i c_i {}^{g_iG}f$.
\end{rem}

The local result for $n=2$ leads to the following result in the global
setting. Given two finite morphisms $\phi,\psi$ of normal integral schemes,
let $[\phi=\psi]$ denote the Cartier divisor obtained by~removing all
codimension~$>1$ components from the subscheme on which $\phi$ and $\psi$
agree.

\begin{cor}\label{cor:global_order_two_residue_conditions}
 With notation as above, suppose that for each $i$, the $n-1$ divisors
 $[\phi_i=\phi_j]$ are mutually transverse. Then ${\cal M}_{\vec{\phi}}$
 may be identified with the subsheaf of
 \begin{gather*}
 \bigoplus_i \sO_X
\bigg(
\sum_{j\ne i} [\phi_j=\phi_i]\bigg)
 \end{gather*}
 in which the local sections $(c_1,\dots,c_n)$ satisfy the additional
 ``residue'' conditions that $c_i+c_j$ is holomorphic along
 $[\phi_i=\phi_j]$ for all $i\ne j$.
\end{cor}

\begin{proof}
 The condition on the divisors ensures that when we perform the various
 reductions, we~will end up with local problems involving at most $2$
 morphisms. The~case with one morphism is trivial (the coefficient is
 forced to be holomorphic, and this makes the operator preserve
 holomorphy), and the case with two is just Corollary~\ref{cor:local_order_two_residue_conditions} above.
\end{proof}

\begin{rem}
 If $[\phi_i=\phi_j]$ is reduced, then the residue conditions can of
 course be stated in terms of the residues of~$c_i$ and $c_j$ with respect
 to any differential holomorphic and nonvanishing along $[\phi_i=\phi_j]$.
\end{rem}

Note that we can also right-multiply operators by local sections of~$\sO_Y$, and thus obtain
an~$(\sO_X,\sO_Y)$-bimodule structure on~${\cal
 M}_{\vec{\phi}}$. Such a structure is equivalent to an $\sO_X\otimes
\sO_Y$-module structure, or equivalently an $\sO_{X\times Y}$-module
structure. We will consider this structure in more detail when discussing
the infinite case, but for the moment we observe the following.

\begin{cor}
 The induced sheaf ${\cal M}_{\vec{\phi}}$ on~$X\times Y$ is a coherent
 subsheaf of~$\bigoplus_i (1,\phi_i)_*\sO_X(\Delta_i)$.
\end{cor}

\begin{rem}
 It is worth noting that this sheaf depends only on the {\em set} of
 morphisms $\{\phi_1,\dots,\phi_n\}$, and not on the ordering, and the map
 from global sections to operators can be reconstructed from the bimodule
 structure. As~a result, when considering sheaves on~$X\times Y$, we~allow the subscript to be a set rather than a sequence.
\end{rem}

\begin{cor}
 The algebra ${\cal H}_G(X)$, viewed as an $\sO_{X\times X}$-module, is
 coherent, and contained in the sum $\bigoplus_{g\in G} (1,g)_*\sO_X(\Delta)$,
 where $\Delta$ is an effective divisor supported on the ``reflection
 hypersurfaces'' of~$G$ on~$X$: the irreducible hypersurfaces which are
 fixed pointwise by some $g\in G$.
\end{cor}

In the simplest elliptic case $A_1$ acting on~$E$, the divisor $\Delta$ is
precisely the divisor correspon\-ding to the subscheme $E[2]$. In~characteristic not 2, this subscheme is reduced, and thus the coefficients
have at most simple poles at the $2$-torsion points, but in characteristic~2, the coefficients can have double poles at the $2$-torsion points of an
ordinary curve, and a quadruple pole at the origin of a supersingular
curve. This, of course, is an artifact of wild ramification; without that,
a ``reflection'' of order $n$ will admit poles of order at most $n-1$ along
the corresponding reflection hypersurfaces.

There is an important variation arising from the interpretation as an
$\sO_{X\times X}$-module: simply consider the sheaf ${\cal H}_{G;{\cal
 L}_1,{\cal L}_2}(X):={\cal H}_G(X)\otimes_{X\times X} {\cal
 L}_2\boxtimes {\cal L}_1^{-1}$ for invertible sheaves ${\cal L}_2$,
${\cal L}_1$ on~$X$. Since both left- and right-multiplication by sections
of~$\sO_{X/G}$ agree in~${\cal H}_G(X)$, this twisted version still
descends to a sheaf on~$X/G$, and the result moreover has induced
compositions
\begin{gather*}
 {\cal H}_{G;{\cal L}_1,{\cal L}_2}(X)
 \otimes_{X/G}
 {\cal H}_{G;{\cal L}_2,{\cal L}_3}(X)
 \to
 {\cal H}_{G;{\cal L}_1,{\cal L}_3}(X).
\end{gather*}
In fact, as~a sheaf on~$X/G$, we~have ${\cal H}_{G;{\cal L}_1,{\cal
 L}_2}(X)\cong \sHom_{G/X}(\pi_*{\cal L}_1,\pi_*{\cal L}_2)$, with the
obvious induced composition; this is most easily seen by representing
${\cal L}_1$, ${\cal L}_2$ by Cartier divisors, and obser\-ving that this
turns ${\cal H}_{G;{\cal L}_1,{\cal L}_2}(X)$ into the subsheaf of the
meromorphic twisted group algebra taking the subsheaf of~$k(X)$
corresponding to~${\cal L}_1$ into the subsheaf corresponding to~${\cal
 L}_2$. When ${\cal L}_1={\cal L}_2$, we~omit the second copy of~${\cal
 L}_2$, and note that the result is again a sheaf of~alge\-bras, with each
${\cal H}_{G;{\cal L}_1,{\cal L}_2}(X)$ a bimodule that induces a Morita
equivalence between ${\cal H}_{G;{\cal L}_1}(X)$ and ${\cal H}_{G;{\cal
 L}_2}(X)$.

In addition to the isomorphisms of such sheaves arising from isomorphisms
of~${\cal L}_1$, ${\cal L}_2$, there are also isomorphisms coming from
twisting both sheaves by a suitable invertible sheaf. Let ${\cal L}$ be a
$G$-equivariant invertible sheaf on~$X$ which ``descends in codimension~1''; that is, there is an open subscheme of~$X/G$ containing every
codimension~1 point over which ${\cal L}$ descends to a line bundle. (Note
that this is a local condition and is automatically satisfied on the
complement of~the reflection hypersurfaces.) Then for each reflection
hypersurface $H$ with inertia group (i.e.,~pointwise stabilizer) $I_H$,
there is a $I_H$-invariant neighborhood $U_H$ of the generic point of~$H$
such that the $H$-equivariant sheaf ${\cal L}_{I_H}$ is equivariantly
isomorphic to~$\sO_{U_H}$. It follows that there is a natural isomorphism
${\cal H}_{G;{\cal L}}(X) \cong {\cal H}_{G}(X)$. Indeed, since ${\cal L}$
is $G$-equivariant, both algebras are naturally contained in~$k(X)[G]$, and
the conditions for any given reflection hypersurface are the same on both
sides. This condition is automatically satisfied by the pullback of an
invertible sheaf on~$X/G$, but this is not necessary. For~instance, in the
case of~$W(G_2)$ acting on the sum~0 subvariety of~$E^3$, the quotient is a
weighted projective space with generators of degree 1, 1, 2. The~pullback
of~$\sO_{\P^2}(1)$ (from $X/W(A_2)\cong \P^2$) does not descend to a line
bundle on~$X/W(G_2)$ (since $\sO_{X/W(G_2)}(1)$ is not invertible on such a
weighted projective space), but does so if we~remove the singular point of~$X/W(G_2)$ and the three points of~$X$ lying over it.

A particularly important instance of twisting by line bundles arises when
we consider the natural involution on operators: $\sum_g c_g g\mapsto
\sum_g g^{-1} c_g$.

\begin{prop}
 This gives a contravariant isomorphism ${\cal H}_G(X)^{\rm op}\cong {\cal
 H}_{G;\omega_X}(X)$.
\end{prop}

\begin{proof}
 We need to show that $\sum_g c_g g$ preserves holomorphic functions iff
 $\sum_g g^{-1}c_g$ preserves holomorphic $n$-forms. By Lemma~\ref{lem:holomorphy_preserving}, it suffices to prove that the conditions
 on individual hypersurfaces are the same, and thus we may fix a
 hypersurface $D$ and restrict our attention to the case that every term
 in the sum gives the same divisor $g^{-1}D=D'$. By duality, a~function~$f$ is holomorphic along $D$ iff $\Res_D f\omega=0$ for all $n$-forms
 $\omega$ which are holomorphic along $D$, and we may similarly detect
 holomorphy of~$n$-forms by taking residues against test functions. In~particular, for~any function~$f$ holomorphic along $D$, and any $n$-form
 $\omega$ holomorphic along $D'$, we~have
 \begin{gather*}
 \Res_{D'} \bigg(\!
\sum_g c_g {}^gf\bigg) \omega
 =
 \sum_g \Res_{D'} c_g {}^gf \omega
 =
 \sum_g \Res_{D} {}^{g^{-1}}c_g f {}^{g^{-1}}\omega
 =
 \Res_D f\cdot\bigg(\!
\sum_g {}^{g^{-1}}(c_g\omega)\!\bigg).
 \end{gather*}
 It follows that $\sum_g c_g {}^g f$ is holomorphic along $D'$ for all $f$
 iff $\sum_g {}^{g^{-1}}(c_g\omega)$ is holomorphic along~$D$ for all
 $\omega$.
\end{proof}

Let us now turn to the ``elliptic'' case, in which $G$ is a finite Weyl
group $W$ acting by reflections on an abelian torsor $X/S$ over a normal
integral base $S$. That is, the flat family $X/S$ is a torsor over an
abelian variety $A/S$, and $W$ acts on~$X$ in such a way that the induced
action on~$A$ is an action by reflections. Note that since the subvariety
$X^W$ is the intersection of the simple reflection hypersurfaces, it has
codimension at most $n$, so is nonempty (the corresponding intersection is
nonempty and transverse in~$A$, and thus the corresponding intersection
number is positive) and thus a torsor over $A^W$; conversely, any
$A^W$-torsor induces a corresponding family~$X/S$ by twisting. The~action
of~$W$ is faithful on every fiber, and this remains true if we view~$X$ as
a family over the quotient $S':=X/A_W$. We will see that the fibers of~${\cal H}_W(X)$ over~$S$ are identified with the master Hecke algebras of
the fibers, in a fairly strong way.

Given a reflection~$r$, let $[X^r]$ denote the effective Cartier divisor
cut out by the equ\-a\-tion~\mbox{$rx=x$}. Also, let $C_r$ denote the quotient of~$X$
by the abelian subvariety $(r+1)A$; this, of course, is a torsor over the
corresponding coroot curve $E'_r$. The~morphism $(r-1)\colon X\to A$ factors
through~$C_r$ and has image $E_r$, so that there is a morphism $C_r\to E_r$
compatible with the isogeny $E'_r\to E_r$, and thus $C_r$ corresponds to a
class in~$H^1(S;\ker(E'_r\to E_r))$. Note that in contrast to the coroot
curve, we~cannot expect to have a natural torsor over the root curve inside
$X$.

For the rank~1 case, we~have the following immediate consequence of Lemma~\ref{lem:holomorphy_preserving}. Here a~``hyperelliptic curve of genus~1''
is a smooth genus~1 curve $C$ with a marked involution such that the
quotient is (geometrically) rational; note that the torsor arising in the
rank~1 case is always a family of such curves-with-involutions.

\begin{lem}
 Let $C/S$ be a flat family of hyperelliptic curves of genus 1 with~$G=A_1=\langle s\rangle$ acting by the marked involution. Then for
 any $G$-invariant open set $U$, $\Gamma(U;{\cal H}_{A_1}(C))$ consists of
 operators $f_0+f_1(s-1)$ such that $f_0\in \Gamma(U;\sO_C)$ and $f_1\in
 \Gamma(U;\sO_C([C^s]))$.
\end{lem}

\begin{rems}
 Note that
 \begin{gather*}
 f_0 + f_1(s-1) = (f_0-f_1-{}^sf_1)+(1+s) {}^sf_1,
 \end{gather*}
 and thus (since $f_1+{}^sf_1\in \Gamma(U;\sO_C)$) we may also describe
 $\Gamma(U;{\cal H}_{A_1}(C))$ as the space of operators $f'_0 +
 (1+s)f'_1$ such that $f'_0\in \Gamma(U;\sO_C)$ and $f'_1\in
 \Gamma(U;\sO_C([C^s]))$. This also follows from the above
 description of the adjoint once we realize that $\omega_C$ is the trivial
 line bundle with equivariant structure such that $s$ acts as $-1$.
\end{rems}

\begin{rems}
 It will be useful in the sequel to know when an $A_1$-equivariant line
 bundle on~$C$ descends in codimension~1. If ${\cal L}$ is equivariantly
 isomorphic to~$\sO_C(D)$ for some symmetric Cartier divisor, then we may
 write $D$ as a linear combination of divisors $D'$ and $D''+{}^sD''$ with~$D'$,~$D''$ irreducible and $D'={}^sD'$. The~latter case is the pullback
 of the image of~$D''$ in~$C/A_1$, and thus certainly has no effect on
 twisting, while if $D'$ is not a component of~$[C^s]$, then its image in~$C/A_1$ is twice a divisor, so that again~$D'$ is a pullback. We are
 thus left to consider the linear combinations of reflection
 hypersurfaces, and thus determine that the condition on~$D$ is precisely
 that the valuations along reflection hypersurfaces must be even (or no
 condition at all in characteristic~2 if the reflection hypersurface is
 inseparable over $S$). This, of course, is for the standard equivariant
 structure on~$\sO_C(D)$; in odd characteristic, we~may twist by the sign
 character of~$A_1$, in which case the condition becomes that the
 valuations along reflection hypersurfaces are odd. Note that in any
 event, a line bundle that descends in codimension~1 will restrict on the
 generic fiber to a power of the hyperelliptic bundle.
\end{rems}

For rank $n$, we~have the following.

\begin{lem}
 The $\sO_{X/W}$-algebra ${\cal H}_W(X)$ is generated by the
 $\sO_{X/W}$-subalgebras ${\cal H}_{\langle s_i\rangle}(X)$ for $1\le i\le
 n$.
\end{lem}

\begin{proof}
 It follows from Lemma~\ref{lem:holomorphy_preserving} that ${\cal
 H}_W(X)$ is generated as a left $\sO_X$-module by the twisted group
 algebra $\sO_X[W]$ along with the subsheaves arising from operators of
 the form $c_w w + c_{rw} rw$ for some reflection~$r\in R(W)$. Now, by
 the above explicit description, each subalgebra ${\cal H}_{\langle
 s_i\rangle}(X)$ contains $\sO_X[s_i]$, and thus the algebra they
 generate contains $\sO_X[W]$. But then if we express $r$ above as
 $w_1^{-1} s_i w_1$ for some simple reflection~$s_i$, we~have
 \begin{gather*}
 c_w w + c_{rw} rw
 =
 c_w w + c_{w_1^{-1} s_i w_1w} w_1^{-1} s_i w_1w
 =
 w_1^{-1} \big({}^{w_1} c_w + {}^{w_1}c_{w_1^{-1} s_i w_1w} s_i\big) w_1 w,
 \end{gather*}
 and find that $c_w w + c_{rw} rw$ is a local section of~${\cal
 H}_W(X)$ iff
 \begin{gather*}
 {}^{w_1}c_w
 +
 {}^{w_1}c_{w w_1s_iw_1^{-1}} s_i
 \end{gather*}
 is a local section of~${\cal H}_W(X)$, iff it is a local section of~${\cal H}_{\langle s_i\rangle}(X)$. The~claim follows.
\end{proof}

\begin{rem}
 Of course, the same argument shows that if $G$ is generated by a
 collection of cyclic groups meeting every conjugacy class of
 (generalized) reflections, then ${\cal H}_G(X)$ is generated by~the
 corresponding subalgebras.
\end{rem}

We can actually say a great deal more in the case of interest; not only is
${\cal H}_W(X)$ a flat sheaf in general, but we can in fact express it as
an extension of invertible sheaves on~$X$. The~key ingredient is the fact
that there is a natural partial order on~$W$ (the Bruhat order), the
weakest partial order such that if $w'w^{-1}$ is a reflection, then $w$ and
$w'$ are comparable and ordered according to their length. Note that
omitting any set of reflections from a reduced word for~$w$ gives an
element $w'\le w$, and classical results on Coxeter groups give the
converse: for any reduced word for $w$, $w'\le w$ iff some word for~$w'$
(iff some {\em reduced} word for~$w'$) can be obtained by omitting
reflections from the chosen reduced word.

Given an order ideal $I$ with respect to Bruhat order (i.e., a subset
$I\subset W$ such that if $w\in I$ and $w'\le w$, then $w'\in I$), we~may
consider the subsheaf ${\cal H}_W(X)[I]$ of~${\cal H}_W(X)$ consisting of
those operators in which the coefficient of~$w$ is 0 for~$w\notin I$. Any
chain of order ideals induces in this way a filtration, and we will show
that in the case of a maximal chain, the subquotients of the filtration are
invertible sheaves on~$X$. Let $[\le w]$ denote the order ideal consisting
of elements~$\le w$.

For any element $w\in W$, define a divisor $D_w:=\sum_{r\in
 R(W),rw<w} [X^r]$.

\begin{lem}\label{lem:Bruhat_fin_noparm}
 Let $I$ be a Bruhat order ideal, and suppose that $w$ is a maximal
 element of~$I$. Then there is a short exact sequence
 \begin{gather*}
 0\to {\cal
 H}_W(X)[I\setminus\{w\}]\subset {\cal H}_W(X)[I]\to \sO_X(D_w)\to 0.
 \end{gather*}
\end{lem}

\begin{proof}
 By definition, ${\cal H}_W(X)[I\setminus\{w\}]$ is the kernel of the
 ``coefficient of~$w$'' map on~${\cal H}_W(X)[I]$, so we first need to
 show that the coefficient of~$w$ is contained in~$\sO_X(D_w)$. It
 follows from Corollary~\ref{cor:global_order_two_residue_conditions} that
 the coefficient has polar divisor bounded by $\sum_{w'\in (I\setminus
 \{w\})\cap R(W)w} \big[X^{w'w^{-1}}\big]$. If~$w'w^{-1}$ is a reflection, then
 $w'$ is comparable to~$w$, which since $w$ is maximal in~$I$ implies that
 $w'<w$, and thus the bound on the divisor is $D_w$ as required.

 It remains only to show that the map is surjective. Choose a reduced
 word $w=s_1\cdots s_n$, and consider the multiplication map
 \begin{gather*}
 {\cal H}_{\langle s_1\rangle}(X)\otimes\cdots\otimes
 {\cal H}_{\langle s_n\rangle}(X)
 \to
 {\cal H}_W(X).
 \end{gather*}
 Every term in the resulting expansion corresponds to an element in which
 some (possibly empty) subset of the simple reflections have been omitted,
 and thus the image of this multiplication map is contained in~${\cal
 H}_W(X)[\le w]\subset {\cal H}_W(X)[I]$. The~image under the leading
 coefficient map can then be determined by replacing each factor by its
 corresponding leading coefficient line bundle. It thus remains only to
 verify that
 \begin{gather*}
 D_w = \sum_{1\le i\le n} {}^{s_1\cdots s_{i-1}} [X^{s_i}]
 = \sum_{1\le i\le n} [X^{s_1\cdots s_{i-1}s_is_{i-1}\cdots s_1}].
 \end{gather*}
 But this follows from the strong exchange property: the reflections $r$
 such that $rw<w$ are precisely those of the form $s_1\cdots
 s_{i-1}s_is_{i-1}\cdots s_1$, and these are distinct since such
 reflections are~naturally bijective with the $n=\ell(w)$ positive roots
 that become negative under~$w$.
\end{proof}

\begin{cor}\label{cor:Bruhat_fin_noparm}
 For any reduced word $w=s_1\cdots s_n$, the multiplication map
 \begin{gather*}
 {\cal H}_{\langle s_1\rangle}(X)\otimes\cdots\otimes {\cal H}_{\langle
 s_n\rangle}(X)
 \to
 {\cal H}_{W}(X)[\le w]
 \end{gather*}
 is surjective. Moreover, any product of rank~$1$ subalgebras is equal to
 some Bruhat interval.
\end{cor}

\begin{proof}
 It suffices to show that if $sw<w$, then
 \begin{gather*}
 {\cal H}_{\langle s\rangle}\otimes {\cal H}_W(X)[\le sw]
 \to
 {\cal H}_W(X)[\le w]
 \end{gather*}
 is surjective. The~image clearly contains the subsheaf ${\cal
 H}_W(X)[\le sw]$, so it suffices to show surjectivity to the quotient
 ${\cal H}_W(X)[\le w]/{\cal H}_W(X)[\le sw]$. This, in turn, is an
 iterated extension of invertible sheaves $\sO_X(D_{w'})$ on~$X$, and
 thus it suffices to show surjectivity for each subquotient. That is, if
 $[\le sw]\subset I\subset[\le w]$ is an order ideal and $w'$ is a maximal
 element of~$I$ not contained in~$[\le sw]$, then we need to show that the
 intersection of the image with~${\cal H}_W(X)[I]$ surjects onto~${\cal
 L}_{D_{w'}}$.

 Since $sw<w$, we~may choose a reduced word for~$w$ beginning with~$s$,
 and the subword description of the Bruhat order then tells us that $[\le
 w]=[\le sw]\cup s[\le sw]$. Since $w'\not\le sw$ and $w'\ne w$, it
 follows that $sw'<sw$. We may thus consider the composition
 \begin{gather*}
 {\cal H}_{\langle s\rangle}(X)\otimes {\cal H}_W(X)[\le sw']
 \to
 {\cal H}_W(X)[\le w']
 \to
 {\cal H}_W(X)[I].
 \end{gather*}
 The proof of the lemma shows that the composition with the ``coefficient
 of~$w'$'' map is surjective as required.

 The second claim follows immediately from the fact that for any word
 $s_1\cdots s_l$, the products of~subwords still form a Bruhat interval
 (so multiplication maps into the corresponding subsheaf), and the maximum
 of that interval can be represented by a reduced subword (so
 multiplication surjects).
\end{proof}

\begin{rem}
 In particular, the closest thing to an analogue of the braid relations in
 this setting is the fact that if $(s_is_j)^{m_{ij}}=1$, then the products
 \begin{gather*}
 {\cal H}_{\langle s_i\rangle}(X)
 {\cal H}_{\langle s_j\rangle}(X)
 \cdots
 =
 {\cal H}_{\langle s_j\rangle}(X)
 {\cal H}_{\langle s_i\rangle}(X)
 \cdots
 \end{gather*}
 (with $m_{ij}$ terms on each side) agree as subsheaves of~${\cal
 H}_W(X)$. Indeed, both sides are equal to the order ideal generated by
 the longest element of~$\langle s_i,s_j\rangle$, and thus equal ${\cal
 H}_{\langle s_i,s_j\rangle}(X)$.
\end{rem}

\begin{cor}
 The construction~${\cal H}_W(X)$ respects base change $T\to S$.
\end{cor}

\begin{proof}
 Let $\pi_1\colon X\times_S T\to X$ be the natural projection; we need to show
 that $\pi_1^*{\cal H}_W(X)\cong {\cal H}_W(X\times_S T)$. In~the rank~1
 case, this is immediate from the explicit description and the fact that
 $\pi_1^*\big(\sO_X([X^s])\big)\cong \sO_X\big([(X\times_S T)^s]\big)$. Since the rank~1
 subalgebras generate the full algebra, this induces a morphism
 $\pi_1^*{\cal H}_W(X)\to {\cal H}_W(X\times_S T)$. (Normally one would
 need to check relations, but this is simply the restriction of the
 corresponding isomorphism for the algebra of meromorphic operators such
 that the common polar divisor does not contain any fiber; thus generators
 suffice.)

 It remains only to show that this morphism is an isomorphism, but this
 follows from the existence of compatible filtrations (i.e., coming from a
 chain of Bruhat order ideals) such that the induced maps on subquotients
 are isomorphisms.
\end{proof}

We can also give an alternate description of the adjoint involution.
Let $w_0$ be the longest element of~$W$.

\begin{prop}
 There is a contravariant isomorphism ${\cal H}_W(X)^{\rm op}\cong
 \sO_X(-D_{w_0})\otimes {\cal H}_W(X)\otimes \sO_X(D_{w_0})$.
\end{prop}

\begin{proof}
 It suffices to show that the na\"{\i}ve adjoint $\sum_w c_w w\mapsto \sum_w
 w^{-1} c_w$ on the meromorphic twisted group algebra restricts to an
 isomorphism as claimed. Since this is an involution, it~redu\-ces to
 showing the analogous claim for each rank~1 subalgebra. Let $U$ be an
 open subset on~which the effective Cartier divisor $D_{w_0}$ is cut out
 by an equation~$h=0$. We thus need to show (using the two descriptions
 of the rank~1 Hecke algebra and taking the adjoint on the left)
 \begin{gather*}
 \Gamma(U;\sO_X) + (s_i-1) \Gamma\big(U;\sO_X([C^s])\big)
 \subset
 h \big(\Gamma(U;\sO_X) + (s_i+1) \Gamma(U;{\cal L}[C^s])\big) h^{-1}.
 \end{gather*}
 Given an instance $f_0+(s_i-1)f_1$ on the left, conjugating by $h$ gives
 \begin{gather*}
 f_0-(1+(h/{}^{s_i}h))f_1
 +
 (s_i+1) (h/{}^{s_i}h) f_1.
 \end{gather*}
 Since ${}^{s_i}D_{w_0}=D_{w_0}$, we~find that $h/{}^{s_i}h$ is a unit,
 and local considerations near $[X^{s_i}]$ tell us that
 $1+(h/{}^{s_i}h)$ vanishes on~$[X^{s_i}]$.
\end{proof}

\begin{rems}
 We could also show that $\omega_X\otimes \sO_X(D_{w_0})$ descends in
 codimension~1. The~divi\-sor~$D_{w_0}$ is certainly invariant under every
 reflection, and thus it remains only to verify the conditions along
 reflection hypersurfaces. In~characteristic not 2, $\omega_X$ is
 equivariantly isomorphic to the twist of~$\sO_X$ by the sign character,
 and thus the condition is that the divisor must have odd valuation along
 the reflection hypersurfaces, while in characteristic~2, there is no need
 to twist, and the valuations of separable reflection hypersurfaces must
 be even. In~either case, the condition is automatically satisfied.
\end{rems}

\begin{rems}
 It is worth noting that this operation is triangular, in the sense that
 the image of the subsheaf corresponding to an order ideal is always the
 subsheaf corresponding to an order ideal. This follows immediately from
 the fact that $w\mapsto w^{-1}$ is an order-preserving automorphism of
 the Bruhat poset.
\end{rems}

The proof of Theorem~\ref{thm:W-invariants} has the following consequence
for our algebras. Here and below, by the root kernel of~$X$, we~mean the
root kernel of the corresponding abelian scheme $A$.

\begin{prop}
 Let $X/S$ be a flat family of abelian torsors equipped with a faithful
 action by ref\-lections of the finite Weyl group $W$, with~$\dim(X/S)=\rank(W)$. If~the root kernel of~$X$ is diagonalizable on~$S$, then $S$ may be covered by open subsets on which ${\cal H}_W(X)$ has
 a symmetric idempotent: an idempotent global section which on each fiber
 has image $\Gamma(X_x;\sO_X)^W$.
\end{prop}

\begin{proof}
 For any fiber $x$, choose a global section~$h$ of~$\sO_X(D_{w_0})$ with
 nonzero symmetrization (ref\-lections negate the fiber over $0$, so this
 corresponds to the antisymmetrization of Theorem~\ref{thm:W-invariants}),
 extend it to a neighborhood of~$x\in S$, and observe that the
 antisymmetrization will remain nonzero in a possibly smaller, but
 nonempty, neighborhood. (By the proof of Theorem~\ref{thm:W-invariants},
 this is guaranteed to exist for any geometric fiber (over which we can
 equivariantly trivialize the torsor), but the existence of an element
 with nontrivial antisymmetrization in an extension field implies the
 existence of such an element over the ground field.) Dividing by the
 symmetrization gives a function~$f$ with poles at most $\Delta$ such
 that $\sum_{w\in W} {}^w f=1$. The~proof of Theorem~\ref{thm:W-invariants} shows that the idempotent operator $\sum_{w\in W}
 w f$ preserves the space of functions holomorphic on any given invariant
 open subset of~$X$, and thus is a section of~${\cal H}_W(X)$ over the
 given open subset of~$S$.
\end{proof}

\begin{rem}
 Of course, we~can always ensure the $\dim=\rank$ condition holds by
 taking the parameter space to be $X/A_W$. This condition is needed so
 that $\sum_{r\in R(W)} [X^r]$ is ample, allowing Theorem~\ref{thm:W-invariants} to be applied. For~instance, in the case of~$A_1$
 acting on~$E^2$ by swapping the factors, every global section of~${\cal
 H}_{A_1}(E^2)$ has holomorphic coefficients, and thus in characteristic~2 there is no symmetric idempotent. The~same argument gives a weaker
 statement without the dimension condition: for any ample divisor $D$ on~$X/W$, there is (locally on~$S$) a symmetric idempotent in~$\Gamma(X/W;{\cal H}_W(X)\otimes \sO_{X/W}(D))$.
\end{rem}

Since ${\cal H}_W(X)$ contains the twisted group algebra of na\"{\i}vely
holomorphic operators, there is in particular an action of the group on any
${\cal H}_W(X)$-module, and thus for any such module~$M$ which is
(quasi)coherent as an $\sO_{X/W}$-module, we~could consider the
$W$-invariant subsheaf of~$M$. This as it stands may not be well-behaved,
say for torsion sheaves supported on the reflection hypersurfaces. To~obtain a better notion, we~note that if we view $\sO_X$ as a module over
${\cal H}_W(X)$, then there is a surjective ${\cal H}_W(X)$-module morphism
$\sum_w c_w w\mapsto \sum_w c_w$ from ${\cal H}_W(X)$ to~$\sO_X$, with
kernel containing the kernel of the natural morphism $\sO_X[W]\to\sO_X$.
Thus for modules which are torsion-free as $\sO_X$-modules, the sheaf
$\sHom_{{\cal H}_W(X)}(\sO_X,M)$ will agree with the sheaf of~$W$-invariant
 sections of~$M$. With this in mind, we~define $M^W$ as the image of~$\sHom_{{\cal H}_W(X)}(\sO_X,M)$ under the ``evaluate at $1$'' morphism.

\begin{cor}
 If the root kernel of~$X$ is diagonalizable, then the functor ${-}^W$ on~${\cal H}_W(X)$-modu\-les is exact and commutes with base change.
\end{cor}

\begin{proof}
 If ${\cal H}_W(X)$ has an idempotent of the form $\big(\sum_w w\big)h$, then the
 map ${\cal H}_W(X)\to \sO_X$ $\big($taking $\sum_w c_w w$ to~$\sum c_w\big)$ splits
 as $f\mapsto f\big(\sum_w w\big)h$. Since such idempotents exist locally on~$S'=X/A_W$, it follows that $\sO_X$ is locally projective, and thus the
 corresponding sheaf Hom functor is exact. Moreover, it follows that
 $M^W=\big(\sum_w w\big)h M$, and this operation clearly commutes with base
 change.
\end{proof}

Of course, as~it stands, the algebra ${\cal H}_W(X)$ does not bear a
terribly strong resemblance to the more familiar Hecke algebras, due to the
lack of any parameters associated to the roots. Classically, one generally
has one parameter for each orbit of roots, but in the classical $C_n$ case
(viewing the affine Hecke algebra as being specified by an action of the
finite Hecke algebra on the space of Laurent polynomials), one effectively
has two parameters associated to the endpoint of the Dynkin diagram. This
is traditionally interpreted as arising from the nonreduced root system
$BC_n$, in which the endpoint is associated to two orbits of roots
(differing by a factor of~2). If~one looks at the actual action on Laurent
polynomials, however, one finds that there is more symmetry in the
parameters than is suggested by this interpretation, making it far more
natural to associate an unordered pair of parameters to the given simple
reflection. In~fact, as~we mentioned above, for~our application, we~will
need a place to put an unbounded number of parameters; since there is
already an example in which one can assign two parameters to a~root without
breaking things, this suggests that we should be able to assign arbitrarily
many parameters to each orbit of roots.

Let us first consider the case of rank~1, so that~$X$ is a flat family
$C/S$ of hyperelliptic curves of genus~1. By consideration of the
classical $A_1$ and $C_1$ cases, we~are led to consider the following
algebra.

\begin{defn} Let $C/S$ be a flat family of hyperelliptic curves of genus 1
 on which $A_1=\langle s\rangle$ acts as the marked involution, and let
 $T$ be an effective Cartier divisor on~$C$ not containing any fiber of~$C$ over $S$. The~{\em rank~1 Hecke algebra} ${\cal H}_{A_1,T}(C)$ is the
 subsheaf of~${\cal H}_{A_1}(C)$ such that the coefficient of~$s$ in a
 local section of~${\cal H}_{A_1,T}(C)$ is a local section of~$\sO_C([C^s]-T)$.
\end{defn}

To see that this is an algebra, we~note that the local sections of~${\cal
 H}_{A_1,T}(C)$ are precisely the operators of the form $f_0+f_1(s-1)$
with $f_0\in \Gamma(U;\sO_C)$, $f_1\in \Gamma(U;\sO_C([C^s]-T)$, or
equivalently the operators of the form $g_0+(s+1){}^s g_1$ with~$g_0\in
\Gamma(U;\sO_C)$, $g_1\in \Gamma(U;\sO_C([C^s]-T))$. Thus the general
product of two local sections can be expressed as
\begin{gather*}
(f_0+f_1(s-1))(g_0+(s+1){}^sg_1)=
f_0g_0 + f_1({}^sg_0-g_0) + f_0(g_1+{}^sg_1)
\\ \hphantom{(f_0+f_1(s-1))(g_0+(s+1){}^sg_1)=}
{}+(f_1 {}^s g_0 + f_0 g_1)(s-1),
\end{gather*}
so that the coefficient of~$s$ in the product is again a section of~$\sO_C([C^s]-T)$.

There is an alternate description which makes the algebra property clearer,
at the cost of a~mild loss of generality.

\begin{prop}\label{prop:generic_rank_one_as_intersection}
 The algebra ${\cal H}_{A_1,T}(C)$ is contained in the subalgebra of~${\cal H}_{A_1}(C)$ which preserves the subsheaf $\sO_C(-T)\subset
 \sO_C$. If~the divisors $T$ and ${}^sT$ have no component in common,
 then this subalgebra is equal to~${\cal H}_{A_1,T}(C)$.
\end{prop}

\begin{proof}
 Let $U$ be an invariant open subset on which $T$ is cut out by a single
 equation~$h=0$. Then the space $\Gamma(U;{\cal H}_{A_1,T}(C))$ can be
 described as the space of operators $f_0+(s+1)f_1{}^sh$ such that $f_0\in
 \Gamma(U;\sO_C)$, $f_1\in \Gamma(U;\sO_C([C^s]))$. Similarly,
 $\Gamma(U;\sO_C(-T))=h\Gamma(U;\sO_C)$, so we need to show that
 $f_0+(s+1)f_1{}^sh$ preserves $h\Gamma(U;\sO_C)$, or equivalently that
 $h^{-1}(f_0+(s+1)f_1{}^sh)h\in \Gamma(U;{\cal H}_{A_1}(C))$. Since
 \begin{gather*}
 h^{-1}(f_0 + (s+1) f_1 {}^sh)h
 =
 f_0 + h^{-1}(s+1)f_1 ({}^s h h)
 =
 f_0 + {}^sh(s+1)f_1,
 \end{gather*}
 and $f_0,{}^sh,(s+1)f_1\in \Gamma(U;{\cal H}_{A_1}(C))$, the first claim
 follows.

 Conversely, if $f_0+sf_1\in \Gamma(U;{\cal H}_{A_1}(C))$ also preserves
 $\sO_C(-T)|_U$, then both $f_0+sf_1$ and $h^{-1}(f_0+sf_1)h$ are in~$\Gamma(U;{\cal H}_{A_1}(C))$. The~first condition implies $f_1\in
 \Gamma(U;\sO_C([C^s]))$, while the second implies $f_1\in
 \Gamma(U;\sO_C([C^s]-T+{}^sT)$. If~${}^sT$ has no component in common with~$T$, so that $\sO_C(T)\cap \sO_C({}^sT) = \sO_C$, then $\sO_C([C^s])\cap
 \sO_C([C^s]-T+{}^sT) = \sO_C([C^s]-T)$. In~other words, $f_1\in
 \Gamma(U;\sO_C([C^s]-T))$, so that $f_0+sf_1\in \Gamma(U;{\cal
 H}_{A_1,T}(C))$ as required.
\end{proof}

\begin{rem}
 It is likely that the condition on~$T$ here is slightly stronger than
 strictly necessary: the claim most likely continues to hold as long as
 $T\cap {}^sT$ is contained in~$(X^s)^{\text{red}}$.
\end{rem}

Note that we could have used this to prove the algebra property, in the
following way. For~each nonnegative integer $m$, let $S'$ be the relative
symmetric $m$-th power of~$C$ over $S$, and let $C'$ be the base change of~$C$ to~$S'$. There is a corresponding tautological divisor $T'$, and our
original data $(C/S,T)$ (assuming $T$ has degree $m$ over $S$) is the base
change of~$(C'/S',T')$ by the section~$S\to S'$ corresponding to~$T$. The~space of operators as described respects base change, and thus it suffices
to prove the algebra property in this larger family. Since $T'$ and
${}^sT'$ have no component in common, this follows from the above result.
There is one caveat here, though: although our original description
respects base change, the description from the proposition does not.
Indeed, if there is an effective divisor $T_0$ such that $T-T_0-{}^sT_0$ is
effective, then the subalgebra preserving $\sO_C(-T)$ is the same as that
preserving $\sO_C(-T+T_0+{}^sT_0)$, but the corresponding rank~1 Hecke
algebras are not the same.

In the above argument, we~used the fact that ${}^s h h$ is central. This
means we could also have described ${\cal H}_{A_1,T}(C)$ (subject to the
given condition on~$T$) as the subalgebra of~${\cal H}_{A_1}(C)$ preserving
the {\em supersheaf} $\sO_C({}^sT)$. This symmetry leads to the
following.

\begin{prop}
 There is a natural isomorphism
 \begin{gather*}
 \sO_C(T)\otimes {\cal H}_{A_1,T}(C)\otimes \sO_C(-T)
 \cong
 {\cal H}_{A_1,{}^sT}(C).
 \end{gather*}
\end{prop}

\begin{proof}
 Replacing $(C/S,T)$ by a larger family as necessary, we~may assume that
 $T$ and ${}^sT$ have no component in common. We then have
 \begin{gather*}
 {\cal H}_{A_1,T}(C)
 =
 {\cal H}_{A_1}(C)
 \cap
 \sO_C(-T)\otimes {\cal H}_{A_1}(C)\otimes \sO_C(T),
 \end{gather*}
 and thus, conjugating by $\sO_C(T)$,
 \begin{gather*}
 \sO_C(T)\otimes {\cal H}_{A_1,T}(C)\otimes \sO_C(-T)
 =
 {\cal H}_{A_1}(C)
 \cap
 \sO_C(T)\otimes {\cal H}_{A_1}(C)\otimes \sO_C(-T).
 \end{gather*}
 Replacing $T$ by ${}^sT$ in the alternate description
 \begin{gather*}
 {\cal H}_{A_1,T}(C)
 =
 {\cal H}_{A_1}(C)
 \cap
 \sO_C({}^sT)\otimes {\cal H}_{A_1}(C)\otimes \sO_C(-{}^sT)
 \end{gather*}
 tells us that
 \begin{gather*}
 {\cal H}_{A_1}(C)
 \cap
 \sO_C(T)\otimes {\cal H}_{A_1}(C)\otimes \sO_C(-T)
 =
 {\cal H}_{A_1,{}^sT}(C)
 \end{gather*}
 as required.
\end{proof}

\begin{prop}
 The adjoint isomorphism ${\cal H}_{A_1}(C)^{\rm op}\cong
 \sO_C(-[C^s])\otimes {\cal H}_{A_1}(C)\otimes \sO_C([C^s])$
 restricts to a contravariant isomorphism
 \begin{gather*}
 {\cal H}_{A_1,T}(C)^{\rm op}
 \cong
 \sO_C(-[C^s])\otimes {\cal H}_{A_1,{}^sT}(C)\otimes \sO_C([C^s]),
 \end{gather*}
 inducing a contravariant isomorphism
 \begin{gather*}
 {\cal H}_{A_1,T}(C)^{\rm op}
 \cong
 \sO_C(T-[C^s]) \otimes {\cal H}_{A_1,T}(C)\otimes \sO_C([C^s]-T).
 \end{gather*}
\end{prop}

With this construction in mind, let $\vec{T}$ be a system of effective
Cartier divisors $T_\alpha$ on~$X$ associated to the roots $\alpha\in
\Phi(W)$, such that $T_\alpha$ never contains a fiber of~$X$ and
$w(T_\alpha) = T_{w\alpha}$ for all $\alpha\in \Phi(W)$, $w\in W$.
Clearly, to specify such a system, it suffices to specify $T_\alpha$ for
one representative of each orbit of roots, subject to the condition that
$w(T_\alpha)=T_\alpha$ whenever $w\alpha=\alpha$. Although the
construction would work in this generality, we~will also impose the further
condition that $T_\alpha$ descends to a divisor on the corresponding coroot
curve (or, equivalently, is invariant under translation by any point in~$(1+r_\alpha)A$). This makes the stabilizer condition automatic, and thus
we may specify $\vec{T}$ by specifying effective divisors on the coroot
curves associated to a~set of inequivalent simple roots.

We will call such a system $\vec{T}$ of divisors a ``system of parameters
for~$W$ on~$X$''.

\begin{defn}
 Let $W$ be a finite Weyl group acting on an abelian torsor $X/S$ by
 reflections, and let $\vec{T}$ be a system of parameters for~$W$ on~$X$.
 Then the {\em Hecke algebra} ${\cal H}_{W;\vec{T}}(X)$ is the subalgebra
 of~${\cal H}_W(X)$ generated by the rank~1 algebras ${\cal H}_{\langle
 s_i\rangle,T_i}(X)$.
\end{defn}

We again have a filtration by Bruhat order, inherited from ${\cal H}_W(X)$,
and the subquotients are again explicit line bundles.

\begin{lem}\label{lem:Bruhat_fin_parm}
 Let $I$ be a Bruhat order ideal, and suppose that $w$ is a maximal
 element of~$I$. Then there is a short exact sequence
 \begin{gather*}
 0\to {\cal H}_{W;\vec{T}}(X)[I\setminus\{w\}]\subset {\cal
 H}_{W;\vec{T}}(X)[I]\to \sO_X(D_w(\vec{T}))\to 0,
 \end{gather*}
 where $D_w(\vec{T}):=\sum_{r\in
 R(W),rw<w} ([X^r]-T_{\alpha_r})$, with~$\alpha_r$ the positive root
 corresponding to~$r$.
\end{lem}

\begin{proof}
 Suppose first that $T_\alpha$ never has a component in common with the
 discriminant divisor~$D_{w_0}$. Then an easy induction tells us that the
 left coefficient of~$w$ in any local section of~${\cal H}_{W;\vec{T}}(X)$
 vanishes on~$T_{\alpha_r}$ for every reflection~$r$ such that $rw<w$;
 this is by a calculation as in Lemma~\ref{lem:Bruhat_fin_noparm} above,
 except that we must also argue that $s_1\cdots s_{i-1}\alpha_i$ is
 positive. But this is again standard Coxeter theory; if it were not
 positive, then $s_1\cdots s_i$ could not be a reduced word. The~claim
 then follows as in the no parameter case.

 To extend this to bad parameters, we~observe (as in the rank~1 case, as~we will discuss more precisely below) that we can always embed our family
 in a larger family which generically satisfies the condition on~$\vec{T}$. On~the one hand, since ${\cal H}_{W;\vec{T}}(X)$ is generated
 by a flat family of submodules, its Hilbert polynomial is lower
 semicontinuous and is thus bounded above by the sum of the Hilbert
 polynomials of the line bundles of the subquotients of the {\em generic}
 Bruhat filtration. Since we can construct elements of Bruhat intervals
 with the desired leading coefficients, it follows that this bound must be
 tight, and the claim follows in general.
\end{proof}

\begin{rem}
 We may also write the divisor as $D_w(\vec{T})=\sum_{\alpha\in
 \Phi^+(W)\cap w\Phi^-(W)} ([X^{r_\alpha}]-T_\alpha)$.
\end{rem}

This leads to an alternate description valid under fairly weak conditions
on the system of parameters.

\begin{cor}
 Suppose $\vec{T}$ is such that every $T_\alpha$ is transverse to every
 reflection hypersurface. Then ${\cal H}_{W;\vec{T}}(X)$ may be
 identified with the subalgebra of~${\cal H}_W(X)$ consisting of local
 sections $\sum_w c_w w$ such that $c_w$ vanishes on~$\sum_{\alpha\in
 \Phi^+(W)\cap w\Phi^-(W)} T_\alpha$ for every $w\in W$.
\end{cor}

\begin{proof}
 We showed in the proof of Lemma~\ref{lem:Bruhat_fin_parm} that every
 section of~${\cal H}_{W;\vec{T}}(X)$ satisfies the given vanishing
 conditions, so it remains only to show that every local section of~${\cal
 H}_{W}(X)$ satisfying the conditions is in fact a local section of~${\cal H}_{W;\vec{T}}(X)$. Let ${\cal D}=\sum_w c_w$ be such a section
 (on the open subset $U\subset X/W$), and $I$ be the smallest order ideal
 containing the support of~${\cal D}$, with~$w_1$ a maximal element of~$I$. Then $c_{w_1}$ is a section of~$\sO_X(D_{w_1}(\vec{T}))$, so that
 by Lemma~\ref{lem:Bruhat_fin_parm}, there is an open covering $U=\cup_i
 V_i$ such that on each $V_i$ there is a local section of~${\cal
 H}_{W;\vec{T}}(X)$ supported on~$I$ with the same left coefficient of~$w_1$. Subtracting this local section gives an element which by
 induction is itself a local section of~${\cal H}_{W;\vec{T}}(X)$. It
 follows that the restriction of~${\cal D}$ to each $V_i$ is a section of~${\cal H}_{W;\vec{T}}(X)$, and thus ${\cal D}$ is a local section of~${\cal H}_{W;\vec{T}}(X)$ as required.
\end{proof}

Similarly, the corollaries carry over immediately.

\begin{cor}\label{cor:Bruhat_intervals_product}
 For any reduced word $w=s_1\cdots s_n$, the multiplication map
 \begin{gather*}
 {\cal H}_{\langle s_1\rangle,\vec{T}}(X)\otimes\cdots\otimes {\cal H}_{\langle
 s_n\rangle,\vec{T}}(X)
 \to
 {\cal H}_{W;\vec{T}}(X)[\le w]
 \end{gather*}
 is surjective. Moreover, any product of rank~1 subalgebras is equal to
 some Bruhat interval.
\end{cor}

\begin{cor}
 The construction~${\cal H}_{W;\vec{T}}(X)$ respects base change.
\end{cor}

We also have an immediate extension of the adjoint isomorphism.

\begin{prop}
 The adjoint isomorphism ${\cal H}_W(X)^{\rm op}\cong
 \sO_X(-D_{w_0})\otimes {\cal H}_W(X)\otimes \sO_X(D_{w_0})$ restricts to a
 contravariant isomorphism
 \begin{gather*}
 {\cal H}_{W;\vec{T}}(X)^{\rm op}
 \cong
 \sO_X\big({-}D_{w_0}(\vec{T})\big) \otimes {\cal H}_{W;\vec{T}}(X)\otimes \sO_X\big(D_{w_0}(\vec{T})\big).
 \end{gather*}
\end{prop}

\begin{proof}
 Again, it suffices to prove that the adjoint identifies the corresponding
 rank~1 subalgebras, and one finds that twisting by
 $\sO_X\big(D_{w_0}(\vec{T})-D_{s_i}(\vec{T})\big)$ has no effect, so the claim
 follows from the rank~1 case.
\end{proof}

One important special case is when $T_\alpha=[X^{r_\alpha}]$ (which
descends to the coroot curve since it is the preimage of the identity under
the composition~$X\to E'_r\to E_r$). In~that case, we~find that the rank~1
subalgebras are just the twisted group algebras $\sO_X[\langle s\rangle]$,
and thus that the full algebra is itself simply equal to~$\sO_X[W]$.

One disadvantage of the approach via rank~1 subalgebras is that it is not
particularly convenient when trying to determine whether a given operator
is a (local) section of the Hecke algebra. For~this, it will be helpful to
have a generalization of Proposition~\ref{prop:generic_rank_one_as_intersection}.

\begin{prop}\label{prop:finite_hecke_as_intersection}
 The algebra ${\cal H}_{W;\vec{T}}(X)$ is contained in the subalgebra of~${\cal H}_W(X)$ preserving the subsheaf $\sO_X\big({-}\sum_{\alpha\in
 \Phi^+(W)} T_\alpha\big)\subset \sO_X$,
 with equality holding unless there is a root $\alpha$ such that~$T_\alpha$ and $T_{-\alpha}$ have a common component.
\end{prop}

\begin{proof}
 Containment reduces to showing that the rank~1 subalgebras preserve the
 given subsheaf. Since the simple reflection~$s_i$ permutes the positive
 roots other than $\alpha_i$, the divisor $T_i-\sum_{\alpha\in
 \Phi^+(W)}T_\alpha$ is $s_i$-invariant, and has trivial valuation along
 the components of~$[X^{s_i}]$. It~follows that on the corresponding rank~1 subalgebra, preserving $\sO_X\big({-}\sum_{\alpha\in \Phi^+(W)} T_\alpha\big)$
 is equivalent to preserving $\sO_X(-T_i)$, at which point the claim is
 just Proposition~\ref{prop:generic_rank_one_as_intersection}.

 Using the Bruhat filtration, we~see that equality holds whenever
 \begin{gather*}
 \sO_X\big(D_w(\vec{T})\big)
 =
 \sO_X(D_w)
 \cap
 \sO_X\bigg(D_w -\sum_{\alpha\in \Phi^+(W)} T_\alpha
 +w\bigg(
\sum_{\alpha\in \Phi^+(W)} T_\alpha\bigg)\!
\bigg).
 \end{gather*}
 Since
 \begin{gather*}
 \sum_{\alpha\in \Phi^+(W)} \!\!T_\alpha
 -w\bigg(
\sum_{\alpha\in \Phi^+(W)}\!\! T_\alpha\bigg)
 =
 \sum_{\alpha\in \Phi^+(W)}\!\! T_\alpha
 -\sum_{\alpha\in \Phi^+(W)}\!\! T_{w\alpha}
=
 \sum_{\alpha\in \Phi^+(W)\cap w\Phi^-(W)}\!\! (T_\alpha-T_{-\alpha}),
 \end{gather*}
 we have equality as long as there is no cancellation, i.e.,
 unless there is a positive root $\alpha$ and a~negative
 root $\beta$ such that $T_\alpha$ and $T_\beta$ have a common component. If~$\beta\ne -\alpha$, then the two divisors are pulled back through
 different coroot maps, and thus cannot have a common component, so~only
 the case $T_\alpha$, $T_{-\alpha}$ is relevant, and the claim follows.
\end{proof}

As in the rank~1 case, the restriction on the divisors is not particularly
serious, as~we can always obtain the algebra we want as the base change of
a more general family. In~particular, if $S'$ is an appropriate product of
relative symmetric powers of coroot curves, then there is a~corresponding
tautological system of parameters $\vec{T}'$ on the base change to~$S'$,
and the original system $\vec{T}$ is the pullback along a suitable section~$S\to S'$.

\begin{cor}
 There is a natural isomorphism
 \begin{gather*}
 \sO_X\bigg(
\sum_{\alpha\in \Phi^+(W)} T_\alpha\bigg)
\otimes {\cal H}_{W;\vec{T}}(X)\otimes \sO_X
\bigg({-}\sum_{\alpha\in \Phi^+(W)} T_\alpha\bigg)
 \cong {\cal H}_{W;{}^-\vec{T}}(X),
 \end{gather*}
 where ${}^-T_\alpha:=T_{-\alpha}$.
\end{cor}

Another source of isomorphisms is diagram automorphisms.

\begin{cor}
 Let $\delta$ be an automorphism of~$X$ over $S$ such that composition
 with~$\delta$ permutes the set of positive coroot maps. Then $\delta$
 normalizes $W$, and the induced action on~${\cal H}_W(X)$ preserves
 ${\cal H}_{W;\vec{T}}(X)$ for all $\vec{T}$.
\end{cor}

\begin{proof}
 The assumption on~$\delta$ (and finiteness of~$W$) implies that $\delta$
 preserves the set of simple coroot maps as well as the divisor
 $\sum_{\alpha\in \Phi^+(W)} T_\alpha$.
\end{proof}

\begin{cor}
 Let $w_0$ be the longest element of~$W$. Then the action of~$w_0$ on~${\cal H}_W(X)$ takes~${\cal H}_{W;\vec{T}}$ to~${\cal
 H}_{W;{}^-\vec{T}}$.
\end{cor}

\begin{proof}
 Since $[-1]w_0$ acts as a diagram automorphism, it suffices to show the
 corresponding fact for~$[-1]$, which clearly commutes with~$W$ and
 satisfies
 \begin{gather*}
 [-1]\bigg(
\sum_{\alpha\in \Phi^+(W)} T_\alpha\bigg) =
 \sum_{\alpha\in \Phi^+(W)} T_{-\alpha}.\tag*{\qed}
 \end{gather*}
\renewcommand{\qed}{}
\end{proof}

We now turn to modules over ${\cal H}_{W;\vec{T}}(X)$. Since we
constructed ${\cal H}_{W;\vec{T}}(X)$ as a space of operators, this gives
rise to a natural left module denoted $\sO_X$, which as a sheaf on~$X/W$ is the direct image of~$\sO_X$. This works more generally for any
$W$-equivariant line bundle ${\cal L}$ that descends in codimension~1, as~${\cal H}_{W;0}(X)$ still acts on such bundles. (We also have a
corresponding submodule ${\cal L}\big({-}\sum_{\alpha\in \Phi^+(W)} T_\alpha\big)$
coming from Proposition~\ref{prop:finite_hecke_as_intersection}, which we
will discuss more below.)

An important construction of modules comes from the fact that our algebras
are generated by the rank~1 subalgebras, and thus any parabolic subgroup
$W_I$ induces a corresponding parabolic subalgebra ${\cal
 H}_{W_I;\vec{T}|_{\Phi(W_I)}}(X)\subset {\cal H}_{W;\vec{T}}(X)$, which
by mild abuse of notation we denote by ${\cal H}_{W_I;\vec{T}}(X)$. As~a
result, given a (left) ${\cal H}_{W_I;\vec{T}}(X)$-module $M$, we~may
tensor with~${\cal H}_{W;\vec{T}}(X)$ to obtain an induced ${\cal
 H}_{W;\vec{T}}(X)$-module which we denote by $\Ind^{W;\vec{T}}_{W_I} M$,
or by $\Ind^{W;0}_{W_I}M$ when considering the analogous construction for
the master Hecke algebra; we also denote the corresponding restriction
functor as $\Res^{W;\vec{T}}_{W_I}$. Note that the restriction of the left
module associated to a line bundle is the left module associated to the
same line bundle.

If ${\cal L}$ is a $W$-equivariant line bundle on~$\sO_X$ that descends in
codimension~1, then we have a locally free module $\Ind^{W;\vec{T}}_1 {\cal
 L}$, which establishes a Morita autoequivalence of the category of~${\cal
 H}_{W;\vec{T}}(X)$-modules, which we denote by ${\cal L}\otimes{-}$. We
of course have a corresponding notion for right modules, with~$\Ind^{W;\vec{T}}_1 {\cal L}\cong {\cal L}\otimes {\cal H}_{W;\vec{T}}(X)
\cong {\cal H}_{W;\vec{T}}(X)\otimes {\cal L}$. Similarly, there is an
equivalence ${\cal L}\big({-}\sum_{\alpha\in \Phi^+(W)}T_\alpha\big)\otimes{-}$
taking ${\cal H}_{W;\vec{T}}(X)$-modules to~${\cal
 H}_{W;{}^-\vec{T}}(X)$-modules.

If we take the restriction of an induced module, we~would ordinarily expect
the result to split as a sum over double cosets. This fails even in the
case of the regular representation, as~${\cal H}_{W;\vec{T}}(X)$ does not
naturally split as a direct sum of~$\big({\cal H}_{W_I;\vec{T}}(X),{\cal
 H}_{W_J;\vec{T}}(X)\big)$-bimodules corresponding to double cosets. The~case
$I=J=\varnothing$ is suggestive however: although the Hecke algebra does not
split as a {\em sum} of line bundles indexed by $W$, our results on the
Bruhat filtration come fairly close. It turns out that there is a natural
Bruhat order on (parabolic) double cosets. Indeed, every double coset $W_I
w W_J$ has a unique minimal representative, and the restriction of Bruhat
order to the set of such representatives is well-behaved. (See, e.g.,
\cite{StembridgeJR:2005} and references therein.) In particular, for~any
order ideal in the set ${}^IW^J$ of minimal representatives, the
corresponding union of double cosets is an order ideal in~$W$. As~a
result, any order ideal in~${}^IW^J$ induces a~corresponding sub-bimodule
of~${\cal H}_{W;\vec{T}}(X)$, and thus a subfunctor of~$\Res_{W_I}^{W;\vec{T}}\Ind^{W;\vec{T}}_{W_J}$.

Given any element $w\in {}^IW^J$, the intersections $W_I\cap w W_J w^{-1}$
and $w^{-1} W_I w\cap W_J$ are both parabolic, giving subsets $I(w)\subset
I$, $J(w)\subset J$ such that $W_{I(w)}\cong W_{J(w)}$, extending in an
obvious way to an isomorphism of the corresponding Hecke algebras.

\begin{lem}
 For any $w\in {}^IW^J$, $D_w\big(\vec{T}\big)$ is $W_{I(w)}$-invariant, and has
 trivial valuation along the corresponding reflection hyperplanes.
\end{lem}

\begin{proof}
 We recall the expression
 \begin{gather*}
 D_w\big(\vec{T}\big) = \sum_{\alpha\in \Phi^+(W)\cap w\Phi^-(W)}
\big ([X^{r_\alpha}]-T_\alpha\big).
 \end{gather*}
 The fact that $w$ is $W_I$-minimal implies that no root of~$W_I$ appears
 in this sum, and thus in particular that no root of~$W_{I(w)}$ appears.
 It thus remains only to show $W_{I(w)}$-invariance, but this follows by
 comparing $D_{s_iw}\big(\vec{T}\big)$ and $D_{ws_j}\big(\vec{T}\big)$ for reflections
 $s_i\in W_I$, $s_j\in W_J$ such that $s_iw=ws_j$.
\end{proof}

This ensures that the twisting functor in the following Mackey-type result
is well-defined.

\begin{prop}\label{prop:Mackey_for_Hecke}
 Let $I,J\subset S$. Then for any ${\cal H}_{W_J;\vec{T}}(X)$-module $M$
 and any maximal chain in the Bruhat order on~${}^IW^J$, the subquotient
 corresponding to~$w\in {}^IW^J$ in the resulting filtration of~$\Res_{W_I}^{W;\vec{T}}\Ind^{W;\vec{T}}_{W_J}M$ is the ${\cal
 H}_{W_I;\vec{T}}(X)$-module
 \begin{gather*}
 \Ind^{W_I;\vec{T}}_{W_{I(w)}} \Big(\sO_X\big(D_w\big(\vec{T}\big)\big)\otimes
 w\Res^{W_J;\vec{T}}_{W_{J(w)}} M\Big),
 \end{gather*}
 where here $w$ represents the induced isomorphism from the category of~${\cal H}_{W_{J(w)},\vec{T}}(X)$-modules to the category of~${\cal
 H}_{W_{I(w)},\vec{T}}(X)$-modules.
\end{prop}

\begin{proof}
 Since the description of the subquotient is functorial, it suffices to
 consider the case that $M={\cal H}_{W_J;\vec{T}}(X)$, or in other words
 to consider the Bruhat filtration on~${\cal H}_{W;\vec{T}}(X)$ viewed as
 a bimodule. Let $O$, $O\cup \{w\}$ be the elements of the chosen maximal
 chain that differ by $w$, so that we need to understand the quotient of
 the subsheaf corresponding to~$W_I (O\cup \{w\}) W_J$ by the subsheaf
 corresponding to~$W_I O W_J$. Both of these are bimodules over the
 respective Hecke algebras, and the actions commute with projecting onto
 the vector space of meromorphic operators supported on~$W_I w W_J$. We
 thus immediately see from Corollary~\ref{cor:Bruhat_intervals_product}
 that the quotient is generated by the subsheaf supported on~$W_{I(w)}wW_{J(w)}=W_{I(w)}w$, and is in fact induced from the
 corresponding $\big({\cal H}_{W_{I(w)},\vec{T}}(X),{\cal
 H}_{W_{J(w)},\vec{T}}(X)\big)$-bimodule structure. Moreover, one easily
 verifies that this bimodule induces the Morita equivalence $M\mapsto
 \sO_X\big(D_w\big(\vec{T}\big)\big)\otimes wM$, from which the result follows. Note that
 the fact that $w\in {}^IW^J$ ensures that $D_w\big(\vec{T}\big)$ is
 $W_{I(w)}$-invariant and has trivial valuation along the reflection
 hypersurfaces corresponding to~$R(W_{I(w)})$, so this twisting is indeed
 well-defined.
\end{proof}

Taking $I=\varnothing$ gives the following, where we omit $\varnothing$ from
the notation in~${}^\varnothing W^J$.

\begin{cor}
 Let $I\subset S$. Then for any ${\cal H}_{W_I;\vec{T}}(X)$-module $M$
 and any maximal chain in the Bruhat order on~$W^I$, the subquotient
 corresponding to~$w\in W^I$ in the resulting filtration of~$\Ind^{W;\vec{T}}_{W_I}M$ is the $\sO_X$-module $\sO_X\big(D_w\big(\vec{T}\big)\big)\otimes
 {}^w M$.
\end{cor}

\begin{cor}
 The functor $\Ind_{W_I}^{W;\vec{T}}$ is exact.
\end{cor}

\begin{proof}
 Indeed, the proof of Proposition~\ref{prop:Mackey_for_Hecke}
 shows that the ${\cal H}_{W_I;\vec{T}}(X)$-module ${\cal
 H}_{W;\vec{T}}(X)$ has a filtration by (locally) free modules, so is
 itself locally projective.
\end{proof}

For finite groups, this exactness arises from the fact that restriction and
induction are adjoint in both directions. This is again not {\em quite}
true in our setting, but something fairly close is true. Define
\begin{gather*}
\Coind_{W_I}^{W;\vec{T}} M
:=
\sHom_{{\cal H}_{W_I;\vec{T}}(X)}\big({\cal H}_{W;\vec{T}}(X),M\big),
\end{gather*}
which is clearly right adjoint to~$\Res_{W_I}^{W;\vec{T}}$. Given a set
$I$ of simple roots, let $I'$ denote its \mbox{image} under the diagram
automorphism corresponding to~$w_0$, and note that the double coset $W_{I'}
w_0 W_I=w_0 W_I=W_{I'}w_0$. Let $w_I$ denote the longest element of~$W_I$.

\begin{lem}
 For any ${\cal H}_{W_I;\vec{T}}(X)$-module $M$, there is a natural isomorphism
 \begin{gather*}
 \Ind_{W_I}^{W;\vec{T}}(M)
 \cong
 \Coind_{W_{I'}}^{W;\vec{T}}\big(\sO_X\big(D_{w_0w_I}\big(\vec{T}\big)\big) \otimes w_0w_I M\big).
 \end{gather*}
\end{lem}

\begin{proof}
 There is a natural morphism
 \begin{gather*}
 \Res_{W_{I'}}^{W;\vec{T}}\Ind_{W_I}^{W;\vec{T}}(M)
 \to
 \sO_X\big(D_{w_0w_I}\big(\vec{T}\big)\big)\otimes w_0w_I M,
 \end{gather*}
 since the codomain is precisely the top subquotient in the Bruhat
 filtration of the domain. By~adjunction, this induces a morphism from
 the induced module to the coinduced module, and it remains only to show
 that this is an isomorphism. The~image of a local section~$x\in
 \Ind_{W_I}^{W;\vec{T}}(M)$ in the coinduced module is the map taking
 local sections $y\in \Res_{W_{I'}}^{W;\vec{T}}{\cal H}_{W;\vec{T}}(X)$
 to the top subquotient of~$yx$ in the $W_{I'}\setminus W/W_I$
 filtration. The~map $w\mapsto w_0w_Iw^{-1}$ induces an order-reversing
 isomorphism from $W^I$ to~${}^{I'}W$, and thus our putative
 isomorphism is triangular, and it suffices to show that it is an
 isomorphism on the diagonal. This reduces to showing that
 \begin{gather*}
 D_{w_0w_Iw^{-1}}\big(\vec{T}\big)+{}^{w_0w_Iw^{-1}} D_w\big(\vec{T}\big)
 =
 D_{w_0W_I}\big(\vec{T}\big)
 \end{gather*}
 for any $w\in W^I$, which in turn reduces to~$\ell(w_0w_Iw^{-1})+\ell(w)=\ell(w_0W_I)$ and thus to~$\ell(ww_I)=\ell(w)+\ell(W_I)$.
\end{proof}

\begin{rem}
 Since the transformation being applied to~$M$ is invertible, we~also have
 an expression
 \begin{gather*}
 \Coind_{W_I}^{W;\vec{T}}M
 \cong
 \Ind_{W_{I'}}^{W;\vec{T}}\big(\sO_X\big({-}D_{w_0w_I}\big({}^-\vec{T}\big)\big) \otimes w_0w_IM\big).
 \end{gather*}
\end{rem}

As in the master Hecke algebra case, we~again have a module $\sO_X$ coming
from the action on operators, and the restriction to~${\cal
 H}_{W;\vec{T}}(X)$ of the natural map ${\cal H}_{W}(X)\to \sO_X$ is still
surjective. In~particular, we~may again define $M^W$ to be the image of
the natural injective morphism $\sHom_{{\cal H}_{W;\vec{T}}(X)}(\sO_X,M)\to
M$.

\begin{prop}
 The kernel of the natural morphism ${\cal H}_{W;\vec{T}}(X)\to \sO_X$ is
 generated as a left ideal sheaf by the subsheaves of the form
 $\sO_X([X^{s_i}]-T_i)(s_i-1)$.
\end{prop}

\begin{proof}
 Let ${\cal I}$ be the left ideal sheaf so generated. This is clearly
 contained in the kernel, so it remains to show that it contains the
 kernel. Let $\sum_w c_w w$ be a local section of the kernel, and suppose
 $w_1$ is Bruhat-maximal among the elements of~$W$ for which $c_w\ne 0$.
 Since by definition~$\sum_w c_w=0$, $w_1$ cannot be the identity, and
 thus has a reduced expression of the form $w_1 = s_1\cdots s_m$ with~$m>0$. We thus have a (surjective) multiplication map
 \begin{gather*}
 {\cal H}_{\langle s_1\rangle,\vec{T}}(X)\cdots {\cal H}_{\langle
 s_m\rangle,\vec{T}}(X)
 \to
 {\cal H}_{W;\vec{T}}(X)[\le w_1].
 \end{gather*}
 Restricting the last tensor factor to~$\sO_X([X^{s_m}]-T_m)(s_m-1)$
 gives an image in~${\cal I}$ without changing the leading coefficient
 line bundle, and thus there is an element $\sum_{w\le w_1} c'_w w$
 of~${\cal I}$ with~$c'_w=c_w$. Subtracting this element makes the order
 ideal generated by the support of the operator smaller, and thus the
 result follows by induction.
\end{proof}

\begin{cor}
 There is an exact sequence of~$\sO_{X/W}$-modules
 \begin{gather*}
 0\to M^W\to M\to \bigoplus_{1\le i\le n} \sO_X\big(T_i-[X^{s_i}]\big)\otimes M.
 \end{gather*}
\end{cor}

\begin{proof}
 The proposition gives a presentation of~$\sO_X$, and this is just
 the sheaf $\sHom$ from that presentation to~$M$.
\end{proof}

If $M$ is $S$-flat, then this tells us that $M^W$ is the kernel of a
morphism of~$S$-flat sheaves.

\begin{lem}
 Let $Y/S$ be a projective morphism with relatively ample line bundle
 $\sO_Y(1)$, and suppose $\phi\colon M\to N$ is a morphism of~$S$-flat coherent
 sheaves on~$Y$. If~the Hilbert polynomial of~$\ker(\phi_s)$ is
 independent of the point $s\in S$, then the kernel, image, and
 cokernel of~$\phi$ are all $S$-flat, and the natural map $\ker(\phi)_s\to
 \ker(\phi_s)$ is an isomorphism for all $s$.
\end{lem}

\begin{proof}
 If the Hilbert polynomial of~$\ker(\phi_s)$ is independent of~$s$, then
 so is the Hilbert polynomial of~$\coker(\phi_s)\cong \coker(\phi)_s$.
 It follows that $\coker(\phi)$ is $S$-flat, implying immediately that the
 image and kernel are also $S$-flat (as kernels of surjective morphisms of~$S$-flat sheaves). The~final claim follows using the four-term sequence
 \begin{gather*}
 0\to \Tor_2(\coker(\phi),k(s))\to \ker(\phi)_s\to \ker(\phi_s)\to
 \Tor_1(\coker(\phi), k(s))\to 0
 \end{gather*}
 arising by comparing the two spectral sequences for tensoring the complex
 $M\to N$ with~$k(s)$.
\end{proof}

\begin{cor}
 If $M$ is a coherent $S$-flat ${\cal H}_{W;\vec{T}}(X)$-module such that
 the Hilbert polynomial of~$(M_s)^W$ is independent of~$s$, then $M^W$\!,
 $M/M^W$ are flat and the natural map $\big(M^W\big)_s\to (M_s)^W$ is an
 isomorphism for all $s$.
\end{cor}

If $M$ satisfies the hypothesis, we~say that $M$ has ``strongly flat
invariants''.

Before introducing parameters, we~could show that this functor respected
base change and flatness by observing that (subject to diagonalizability of
the root kernel) the Hecke algebra had (locally on~$X/A_W$) idempotents
projecting onto~$M^W$. Unfortunately, this fails, and quite badly, in
cases with parameters. Indeed, if the divisors $T_\alpha$ are of
sufficiently large degree, then the subquotient corresponding to~$w$ in the
Bruhat filtration will have negative degree unless $w$ is the identity, and
thus in such a case the fibers of the Hecke algebra cannot have {\em any}
nonscalar global sections, let alone symmetric idempotents.

Luckily, $S$-flatness is local on the source, not the base, and thus the
correct condition is not that there be {\em global} symmetric idempotents,
but merely that there be {\em local} symmetric idempotents.

\begin{lem}
 Let $U$ be a $W$-invariant open subset. If~$h\in
 \Gamma\big(U;\sO_X\big(D_{w_0}\big(\vec{T}\big)\big)\big)$, then
 \begin{gather*}
\bigg(\sum_w w\bigg){}^{w_0}h\in
 \Gamma\big(U;{\cal H}_{W;\vec{T}}(X)\big).
 \end{gather*}
\end{lem}

\begin{proof}
 Suppose first that $T_\alpha$ and $T_{-\alpha}$ have no common component
 for any root $\alpha$, so that we may use Proposition~\ref{prop:finite_hecke_as_intersection} to characterize ${\cal
 H}_{W;\vec{T}}(X)$. The~given operator clearly maps $\Gamma(V;\sO_X)$
 to~$\Gamma(U\cap V;\sO_X)^W$ for any invariant open $V$, so that $\big(\sum_w
 w\big){}^{w_0}h\in {\cal H}_W(X)$. Similarly, if $f\in
 \Gamma\big(V;-\sum_{\alpha\in\Phi^+(W)}T_\alpha\big)$, then ${}^{w_0}h f\in
 \Gamma\big(V;D_{w_0}-\sum_{\alpha\in \Phi(W)}T_\alpha\big)$, so that the
 symmetrization vanishes along $\sum_{\alpha\in\Phi(W)}T_\alpha$. The~claim follows in this case.

 For the general case, we~base change from the family with universal
 $\vec{T}$, and observe that since~$D_{w_0}(\vec{T})$ is a flat family of
 divisors, $h$ extends to a local section of~$\sO_X\big(D_{w_0}\big(\vec{T}\big)\big)$
 on a~neighborhood of the original base. We thus find that there is a
 local section of the larger Hecke algebra that restricts to the desired
 local section, from which the result follows.
\end{proof}

We say that ${\cal H}_{W;\vec{T}}(X)$ {\em has a local symmetric
 idempotent} at a~point $x\in X/W$ if the restriction of~${\cal
 H}_{W;\vec{T}}(X)$ to the local ring at $x$ contains an idempotent of the
form \smash{$\big(\sum_w w\big)h$}. By the lemma, this is equivalent to asking for the
restriction of~\smash{$\sO_X\big(D_{w_0}\big(\vec{T}\big)\big)$} to the local ring at the orbit
corresponding to~$x$ to contain an element $h$ with~$\sum_{w\in W}
{}^wh=1$. Moreover, if there is an element for which this sum is a unit,
then we can divide by the sum to obtain an element symmetrizing to 1. It
follows that if the condition holds on the fiber containing $x$, then we
still have a local symmetric idempotent at $x$; that is, the condition of
having a local symmetric idempotent respects base change.

Similarly, we~say that ${\cal H}_{W;\vec{T}}(X)$ is {\em covered by
 symmetric idempotents} if it has a local symmetric idempotent at every
point $x\in X/W$. This is too much to hope for even without imposing
parameters, as~the $A_2$ example we considered at the end of Section~\ref{section3} gives an explicit point where the master Hecke algebra
fails to have a local symmetric idempotent. In~general, the most we can
say is that there is a (possibly empty) open subset of~$S$ such that the
base change is covered by symmetric idempotents. Indeed, the condition to
have a symmetric idempotent at $x$ is open, and $X/S$ is proper, so we can
simply take the complement of the image of the complement of the locus with
local symmetric idempotents.

\begin{lem}\label{lem:has_local_idempotents}
 Suppose that the root kernel of~$X$ is diagonalizable, and that for any
 nonnegative linear dependence $\sum_i k_i \alpha_i=0$ of roots, the
 intersection~$\cap_i T_{\alpha_i}$ is empty. Then ${\cal
 H}_{W;\vec{T}}(X)$ is covered by symmetric idempotents.
\end{lem}

\begin{proof}
 This is local in~$S$, so we may restrict to an open subset over which
 ${\cal H}_{W,0}(X)$ has a~glo\-bal symmetric idempotent $\sum_w w h$;
 diagonalizability of the root kernel ensures that these open subsets
 cover $S$. For~any point $x\in X$, let $D_x$ be the corresponding
 decomposition group, and observe that
 \begin{gather*}
 1 = \sum_{w\in W} {}^w h =
 \sum_{g\in D_x}
 {}^g\bigg(
\sum_{w\in D_x\setminus W}
 {}^w h\bigg),
 \end{gather*}
 and thus there is a section of~$\sO_X(D_{w_0})$ in the local ring at
 $x$ for which the sum over the decomposition group is 1.

 Now, suppose that $x$ is not contained in~$T_\alpha$ for any positive
 $\alpha$. Then this local section of~$\sO_X(D_{w_0})$ at $x$ is in
 fact a section of~$\sO_X\big(D_{w_0}-\sum_{\alpha} T_\alpha\big)$ near $x$,
 and we can add a section that vanishes at $x$ in such a way that the
 resulting section is holomorphic on the orbit $Wx$ and vanishes at the
 points of the orbit other than $x$. It follows that the resulting
 function symmetrizes to a unit in the relevant local ring, and thus gives
 rise to a local symmetric idempotent at $Wx$.

 In general, let $\Phi_x$ be the set of roots such that $x\in T_\alpha$.
 Since $x$ is contained in the corresponding intersection of divisors
 $T_\alpha$, we~conclude that the elements of~$\Phi_x$, viewed as real
 vectors, cannot satisfy any nonnegative linear dependence. This implies
 that there is a real linear functional which is negative on~$\Phi_x$, and
 thus (since all systems of positive roots in a finite Weyl group are
 equivalent) that $w\Phi_x\subset \Phi^-(X)$ for some $w\in W$.
 This implies that $wx$ satisfies the conditions for the construction of
 the previous paragraph to apply, and thus that there is a local symmetric
 idempotent in a neighborhood of the orbit $Wx$.
\end{proof}

It is not too hard to see that the empty intersection condition is
satisfied on the generic fiber of the family with universal $\vec{T}$;
indeed, if we further base change to express each $T_i$ as a sum of points,
then the values of those parameters at a~point of such an intersection must
themselves satisfy a nonnegative linear dependence, and there are only
finitely many minimal such dependences to consider. It follows that any
family of Hecke algebras is the base change of a family which is
generically covered by symmetric idempotents.

\begin{lem}
 If ${\cal H}_{W;\vec{T}}(X)$ is covered by symmetric idempotents, then
 ${-}^W$ is exact and any coherent ${\cal H}_{W;\vec{T}}(X)$-module
 $M$ has strongly flat invariants.
\end{lem}

\begin{proof}
 Indeed, $\sO_X$ is locally a direct summand of~\smash{${\cal H}_{W;\vec{T}}(X)$},
 and is therefore locally projective as before. Exactness follows
 immediately. Any local section of~$(M_s)^W$ on an open subset supporting
 a symmetric idempotent extends to a local section of~$M$ which can then
 be projected to a section of~$M^W$ restricting to the given section of~$(M_s)^W$. It follows that the natural map $\big(M^W\big)_s\to (M_s)^W$ is an
 isomorphism. Since $M^W$ is locally a direct summand of~$M$, it is flat,
 and thus its fibers have constant Hilbert polynomial as required.
\end{proof}

It turns out that if the generic fiber is covered by symmetric idempotents,
then this has consequences even on those fibers without such a covering.

\begin{lem}
 Suppose that there is a dense open subset of~$S$ over which ${\cal
 H}_{W;\vec{T}}(X)$ is covered by symmetric idempotents, and suppose
 that the module $M$ admits a filtration such that each subquotient is
 $S$-flat with strongly flat invariants. Then $M$ has strongly flat
 invariants.
\end{lem}

\begin{proof}
 Fix a relatively ample bundle $\sO_{X/W}(1)$ on~$X/W$. By
 semicontinuity, for~any point $s\in S$ and $d\gg 0$, we~have
\begin{gather*}
\dim\big(\Gamma\big(X/W,(M_{k(S)})^W(d)\big)\big)\le \dim\big(\Gamma\big(X/W,(M_s)^W(d)\big)\big)
\\ \hphantom{\dim\big(\Gamma\big(X/W,(M_{k(S)})^W(d)\big)\big)}
{}\le \sum_i \dim\big(\Gamma\big(X/W,(M^i_s)^W(d)\big)\big)
\\ \hphantom{\dim\big(\Gamma\big(X/W,(M_{k(S)})^W(d)\big)\big)}
{}= \sum_i \dim\big(\Gamma\big(X/W,\big(M^i_{k(S)}\big)^W(d)\big)\big),
\end{gather*}
 where the $M^i$ are the subquotients of the given filtration on~$M$.
 Since the generic fiber of~${\cal H}_{W;\vec{T}}(X)$ is covered by
 symmetric idempotents, ${-}^W$ is exact on the generic fiber and thus
 \begin{gather*}
 \dim\big(\Gamma\big(X/W,\big(M_{k(S)}\big)^W(d)\big)\big)
 =
 \sum_i
 \dim\big(\Gamma\big(X/W,\big(M^i_{k(S)}\big)^W(d)\big)\big)
 \end{gather*}
 for~$d\gg 0$, implying
 \begin{gather*}
 \dim\big(\Gamma\big(X/W,(M_{k(S)})^W(d)\big)\big)
 = \dim\big(\Gamma\big(X/W,(M_s)^W(d)\big)\big)
 \end{gather*}
 as required.
\end{proof}

To apply this, we~will need a family of modules for which we can prove
strongly flat invariants without resorting to idempotents. For~$I\subset
S$, let $w_I$ denote the maximal element of~$W_I$.

\begin{prop}\label{prop:T-ish-invariants-of-L}
 Let $I\subset S$ and let ${\cal L}$ be a $W_I$-equivariant line bundle
 that descends in codimension~1. Then
 we have a natural isomorphism
 $\big(\Ind^W_{W_I}{\cal L}\big)^W \cong \big({\cal L}\big(
 D_{w_0}\big({}^-\vec{T}\big)-D_{w_I}\big({}^-\vec{T}\big)
 \big)\big)^{W_I}$
 of~$\sO_{X/W}$-modules, where the right-hand side denotes the sheaf of~$W_I$-invariant sections of the given line bundle.
\end{prop}

\begin{proof}
 For any ${\cal H}_{W_I;\vec{T}}(X)$-module $M$, we~have
 \begin{gather*}
 \Big(\Ind^{W;\vec{T}}_{W_I} M\Big)^W\cong
 \Big(\Coind_{W_{I'}}^{W;\vec{T}}\big(\sO_X\big(D_{w_0w_I}\big(\vec{T}\big)\big) \otimes w_0w_I M\big)\Big)^W
 \\ \hphantom{ \big(\Ind^{W;\vec{T}}_{W_I} M\big)^W}
 \cong \big(\sO_X\big(D_{w_0w_I}\big(\vec{T}\big)\big) \otimes w_0w_I M\big)^{W_{I'}}
 \\ \hphantom{ \big(\Ind^{W;\vec{T}}_{W_I} M\big)^W}
{}\cong \big(\sO_X\big({}^{w_Iw_0}D_{w_0w_I}\big(\vec{T}\big)\big) \otimes M\big)^{W_I}
\\ \hphantom{ \big(\Ind^{W;\vec{T}}_{W_I} M\big)^W}
\cong \big(\sO_X\big(D_{w_0}\big({}^-\vec{T}\big)-D_{w_I}\big({}^-\vec{T}\big)\big) \otimes M\big)^{W_I}.
 \end{gather*}
 We may thus reduce to the case $W=W_I$, where the result is immediate.
\end{proof}

\begin{cor}
 Suppose that the root kernel of~$X$ is diagonalizable. Then any
 module of the form $M=\Ind^W_{W_I}{\cal L}$ has strongly flat invariants.
\end{cor}

\begin{cor}
 Let $I,J\subset S$ and let ${\cal L}_I$, ${\cal L}_J$ be $W_I$,
 $W_J$-equivariant line bundles that descend in codimension~1. If~the
 root kernel of~$X$ is diagonalizable, then
 \begin{gather*}
 \sHom_{{\cal H}_{W;\vec{T}}(X)}\Big(\Ind^{W;\vec{T}}_{W_I} {\cal L}_I,
 \Ind^{W;\vec{T}}_{W_J} {\cal L}_J\Big)
 \end{gather*}
 is $S$-flat and respects base change.
\end{cor}

\begin{proof}
 By extending the family as appropriate (noting that ${\cal L}_I$ and
 ${\cal L}_J$ themselves extend), we~may assume that there is a dense open
 subset $U\subset S$ which is covered by symmetric idempotents. We
 observe that
 \begin{gather*}
 \sHom_{{\cal H}_{W;\vec{T}}(X)}\Big(\Ind^{W;\vec{T}}_{W_I} {\cal L}_I,
 \Ind^{W;\vec{T}}_{W_J} {\cal L}_J\Big) \cong
 \sHom_{{\cal H}_{W_I;\vec{T}}(X)}\Big({\cal L}_I,
 \Res^{W;\vec{T}}_{W_I}\Ind^{W;\vec{T}}_{W_J} {\cal L}_J\Big)
 \\ \hphantom{\sHom_{{\cal H}_{W;\vec{T}}(X)}\Big(\Ind^{W;\vec{T}}_{W_I} {\cal L}_I,
 \Ind^{W;\vec{T}}_{W_J} {\cal L}_J\Big)}
 \cong \Big({\cal L}_I^{-1}\otimes
 \Res^{W;\vec{T}}_{W_I}\Ind^{W;\vec{T}}_{W_J} {\cal L}_J\Big)^{W_I}\!.
 \end{gather*}
 Each subquotient of the Bruhat filtration for
 \begin{gather*}
 {\cal L}_I^{-1}\otimes
 \Res^{W;\vec{T}}_{W_I}\Ind^{W;\vec{T}}_{W_J} {\cal L}_J
 \end{gather*}
 has strongly flat invariants, and thus the same holds for the ${\cal
 H}_{W_I;\vec{T}}(X)$-module itself. It follows in particular that the
 module is $S$-flat and that the construction commutes with base change.
\end{proof}

\begin{prop}
 Let ${\cal L}_I$, ${\cal L}_J$ be $W_I$, $W_J$-equivariant line bundles
 that descend in codimension~1. If~the root kernel of~$X$ is
 diagonalizable, then there is a natural isomorphism
 \begin{gather*}
 \sHom_{{\cal H}_{W}(X)}\Big(\Ind^{W;0}_{W_J}{\cal L}_J, \Ind^{W;0}_{W_I}{\cal L}_I\Big)
 \cong \sHom\big((\pi_*{\cal L}_I)^{W_I},(\pi_*{\cal L}_J)^{W_J}\big)
 \end{gather*}
 which is (contravariantly) compatible with composition.
\end{prop}

\begin{proof}
 Since the root kernel is diagonalizable, both ${\cal H}_{W_I}(X)$ and
 ${\cal H}_{W_J}(X)$ have symmetric idempotents $e_I$, $e_J$ on the
 complement of any $W$-invariant ample divisor, and these embed as local
 sections of~$\sEnd(\pi_*{\cal L}_I)$ and $\sEnd(\pi_*{\cal L}_J)$
 respectively. We may thus identify the sections of the left-hand side as
 the subspace $e_J \sHom(\pi_*{\cal L}_I,\pi_*{\cal L}_J) e_I$, and this
 is contravariant with respect to composition. As~a section of~$\sEnd(\pi_*{\cal L}_I)$, $e_I$ is a projection onto~$(\pi_*{\cal
 L}_I)^{W_I}$, and similarly for~$e_J$; the claim follows immediately.
\end{proof}

\begin{rems}
 Note that if ${\cal L}_I$ descends to a line bundle on~$X/W_I$, then
 $(\pi_*{\cal L}_I)^{W_I}$ may be identified with the direct image of that
 line bundle.
\end{rems}

\begin{rems}
 If ${\cal L}_I$ and ${\cal L}_J$ are trivial and $J=\varnothing$, then the
 right-hand side consists of ope\-ra\-tors locally taking $W_I$-invariant
 holomorphic functions to holomorphic functions, and may thus be
 identified with the sheaf ${\cal M}_{\vec{g}W_I}$, where $g_1,\dots,g_n$
 are a system of coset representatives for~$W/W_I$. More generally, if
 $\sum_{w W_I\in W/W_I} c_{wW_I} w W_I$ is such an operator, requiring
 that the image be \mbox{$W_J$-inva}\-riant is equivalent to requiring that
 $(g-1)\sum_{w W_I\in W/W_I} c_{wW_I} w W_I$ annihilate every
 $W_I$-invariant function, and thus that the operator itself is invariant
 under left multiplication by elements of~$W_J$. Thus, on any invariant
 open subset, the given space is determined by $W_J$-invariance along with
 (by Corollary~\ref{cor:global_order_two_residue_conditions}) the
 condition that the pole of~$c_{wW_I}$ is bounded by the sum
 $\sum_{r:rwW_I\ne wW_I} [X^r]$, and the condition that for any
 reflection, $c_{wW_I}+c_{rw W_I}$ is holomorphic along $[X^r]$. If~${\cal L}_I$ and ${\cal L}_J$ are nontrivial, but ${\cal L}_I$ descends,
 then the first two conditions remain the same, but the residue condition
 becomes the holomorphy of~${}^wfc_{wW_I}+{}^{rw}f c_{rwW_I}$ for any
 $W_I$-invariant meromorphic section of~${\cal L}_I$ such that ${}^wf$ is
 holomorphic on~$[X^r]$. Here, it suffices to check the condition for a
 single section~$f$ as long as both ${}^wf$ and ${}^wf^{-1}$ are
 holomorphic along $[X^r]$.
\end{rems}

There is an analogue of the adjoint in this setting.

\begin{cor}
 If the root kernel of~$X$ is diagonalizable, then there is an isomorphism
 \begin{gather*}
 \sHom_{{\cal H}_{W}(X)}\Big(\Ind^{W;0}_{W_J}\sO_X, \Ind^{W;0}_{W_I}\sO_X\Big)
 \\ \hphantom{\sHom_{{\cal H}_{W}(X)}}
 \cong \sHom_{{\cal H}_{W}(X)}\Big(\Ind^{W;0}_{W_I}\sO_X(D_{w_0}-D_{w_I}),
 \Ind^{W;0}_{W_J}\sO_X(D_{w_0}-D_{w_J})\Big),
 \end{gather*}
 contravariant with respect to composition.
\end{cor}

\begin{proof}
 Embedding the left-hand side in~$\sEnd(\pi_*\sO_X)$ and taking the
 adjoint there gives an isomorphism to~$\sHom(e_J^*\pi_*\sO_X(D_{w_0}),e_I^*\pi_*\sO_X(D_{w_0}))$, where
 $e_I^*$, $e_J^*$ are the adjoints of the corresponding idempotents. If~$e_I = \big(\sum_{w\in W_I} w\big)h_I$, then $e_I^*=h_I\big(\sum_{w\in W_I} w\big)$, and
 we then find that
 \begin{gather*}
 e_I^* \pi_*\sO_X(D_{w_0}) = h_I \bigg(
\sum_{w\in W_I} w\bigg) \pi_*\sO_X(D_{w_0})
 = h_I (\pi_*\sO_X(D_{w_0}-D_{w_I}))^{W_I},
 \end{gather*}
 where the second equality follows from Theorem~\ref{thm:W-invariants}.
 We thus have
 \begin{gather*}
 \sHom\big(e_J^*\pi_*\sO_X(D_{w_0}),e_I^*\pi_*\sO_X(D_{w_0})\big)
 \\ \hphantom{\sHom}
 \cong \sHom\big(h_J(\pi_*\sO_X(D_{w_0}-D_{w_J}))^{W_J},h_I(\pi_*\sO_X(D_{w_0}-D_{w_I}))^{W_I}\big)
 \\ \hphantom{\sHom}
 \cong \sHom\big((\pi_*\sO_X(D_{w_0}-D_{w_J}))^{W_J},(\pi_*\sO_X(D_{w_0}-D_{w_I}))^{W_I}\big),
 \end{gather*}
 from which the claim follows.
\end{proof}

\begin{rem}
 There is, of course, a version with a pair of line bundles; we omit the
 details.
\end{rem}

Define
\begin{gather*}
 {\cal H}_{W,W_I,W_J;\vec{T}}(X) :=
 \sHom_{{\cal H}_{W;\vec{T}}(X)}\Big( \Ind^{W;\vec{T}}_{W_J}\sO_X, \Ind^{W;\vec{T}}_{W_I}\sO_X\Big),
\end{gather*}
with composition law given by
\begin{gather*}
 {\cal H}_{W,W_J,W_K;\vec{T}}(X) \otimes
 {\cal H}_{W,W_I,W_J;\vec{T}}(X) \to {\cal H}_{W,W_I,W_K;\vec{T}}(X),
\end{gather*}
contravariant to the standard composition on Hom sheaves.

\begin{prop}\label{prop:spherical_as_intersection}
 If the root kernel of~$X$ is diagonalizable, then ${\cal
 H}_{W,W_I,W_J;\vec{T}}(X)$ may be identified with a subsheaf of~${\cal
 H}_{W,W_I,W_J}(X)$, compatibly with composition. Moreover, the
 corresponding operators take \mbox{$W_I$-inva}\-riant sections of the line bundle
 $\sO_X\big(\sum_{\alpha\in \Phi^-(W)\setminus\Phi^-(W_I)} T_\alpha\big)$
 to~\mbox{$W_J$-inva}riant sections of the line bundle $\sO_X\big(\sum_{\alpha\in
 \Phi^-(W)\setminus \Phi^-(W_J)}T_\alpha\big)$. Conversely, if there are no
 roots such that $T_{\alpha}$ and $T_{-\alpha}$ have a common component,
 then ${\cal H}_{W,W_I,W_J;\vec{T}}(X)$ is precisely the subsheaf of~${\cal H}_{W,W_I,W_J}(X)$ cut out by this condition.
\end{prop}

\begin{proof}
 Suppose first that ${\cal H}_{W_I;\vec{T}}(X)$ and ${\cal
 H}_{W_J;\vec{T}}(X)$ are covered by symmetric idempotents. This allows
 us to (locally) embed ${\cal H}_{W,W_I,W_J;\vec{T}}(X)$ in~${\cal
 H}_{W;\vec{T}}(X)$ as in the parameter-free case. It follows that
 ${\cal H}_{W,W_I,W_J;\vec{T}}(X)$ acts on the $W_I$-invariant sections of~$k(X)$ in such a way as to take $W_I$-invariant sections of~$\sO_X$ to~$W_J$-invariant sections of~$\sO_X$ and $W_I$-invariant sections of~$\sO_X\big(\sum_{\alpha\in \Phi^-(W)}T_\alpha\big)$ to~$W_J$-invariant
 sections of~$\sO_X\big(\sum_{\alpha\in \Phi^-(W)}T_\alpha\big)$.

 Since $\sum_{\alpha\in\Phi^-(W)}T_\alpha$ is not $W_I$-invariant, a
 $W_I$-invariant section of~$\sO_X\big(\sum_{\alpha\in\Phi^-(W)}T_\alpha\big)$
 must lie in the intersection of the images of this bundle under $W_I$,
 so in particular (taking the intersection with the image under $w_I$) in
 \begin{gather*}
 \sO_X\bigg(
\sum_{\alpha\in \Phi^-(W)}T_\alpha\bigg)
 \cap \sO_X\bigg(\sum_{\alpha\in \Phi^+(W_I)\cup \Phi^-(W)\setminus \Phi^-(W_I)}T_\alpha\bigg),
 \end{gather*}
 which by hypothesis is $\sO_X\big(\sum_{\alpha\in\Phi^-(W)\setminus
 \Phi^-(W_I)} T_\alpha\big)$. The~same calculation for~$J$ tells us that the
 elements of~${\cal H}_{W,W_I,W_J;\vec{T}}(X)$ act as required.

{\sloppy
For a $W_I$-invariant section~$\sum_{w\in W^I} c_w w W_I$ of~${\cal
 H}_{W,W_I,W_J;\vec{T}}(X)$ to be contained in ${\cal H}_{W,W_I,W_J}(X)$
 is a closed condition, and thus holds in general (extending to the family
 with universal parameters as necessary). That it respects the given
 supersheaves is also a closed condition, and thus the first claim follows
 for general parameters.

 }

 To show equality under the conditions on~$T_\alpha$, it suffices to
 compare subquotients in the respective Bruhat filtrations, and thus to
 compute the intersection
 \begin{gather*}
 \sO_X \cap \sO_X\bigg(
 \sum_{\alpha\in \Phi^-(W)\setminus \Phi^-(W_J)}T_\alpha
 -\sum_{\alpha\in \Phi^-(W)\setminus \Phi^-(W_I)}T_{w\alpha}
\bigg).
 \end{gather*}
 (More precisely, we~want to intersect the $W_{J(w)}$-invariant
 subsheaves, but since both sheaves are equivariant, we~may as well take
 their intersection before passing to invariants.) We may write
 \begin{gather*}
 \sum_{\alpha\in \Phi^-(W)\setminus \Phi^-(W_J)}\!\!\!\!T_\alpha
 -\!\!\!\sum_{\alpha\in \Phi^-(W)\setminus \Phi^-(W_I)}\!\!\!\!T_{w\alpha}
 = \sum_{\alpha\in \Phi^-(W)}(T_\alpha-T_{w\alpha})
 - \sum_{\alpha\in \Phi^-(W_J)}\!\!\!\!T_\alpha
 + \sum_{\alpha\in \Phi^-(W_I)}\!\!\!\!T_{w\alpha}.
 \end{gather*}
 Here
 \begin{gather*}
 \sum_{\alpha\in \Phi^-(W)}(T_\alpha-T_{w\alpha})
 =\!\! \sum_{\alpha\in \Phi^+(W)\cap w\Phi^-(W)}\!\!\!\!(T_{-\alpha}-T_\alpha),
 \end{gather*}
 while
 \begin{gather*}
 \sum_{\alpha\in \Phi^-(W_I)}T_{w\alpha} - \sum_{\alpha\in \Phi^-(W_J)}T_\alpha
 =\!\! \sum_{\alpha\in \Phi^-(W_I)\setminus \Phi^-(W_I\cap w^{-1}W_J w)}\!\!\!\!T_{w\alpha}
 - \!\!\sum_{\alpha\in \Phi^-(W_J)\setminus \Phi^-(W_J\cap w W_I w^{-1})}\!\!\!\!T_\alpha.
 \end{gather*}
 The hypotheses ensure that there is no further cancellation, so the
 intersection is
 \begin{gather*}
 \sO_X\bigg({-}\sum_{\alpha\in \Phi^+(W)\cap w\Phi^-(W)}\!\!\!\!T_\alpha
 -\sum_{\alpha\in \Phi^-(W_J)\setminus \Phi^-(W_J\cap w W_I w^{-1})}\!\!\!\!T_\alpha\bigg),
 \end{gather*}
 agreeing with the ($T$-dependent part of the) line bundle arising in the
 Bruhat filtration.
\end{proof}

\begin{cor}
 Let $\vec{T}$ be a system of parameters such that every $T_\alpha$ is
 transverse to every reflection hypersurface. If~the root kernels of~$W_I$ and $W_J$ are diagonalizable, then ${\cal
 H}_{W,W_I,W_J;\vec{T};\gamma}(X)$ may be identified with the
 subsheaf of~${\cal H}_{W,W_I,W_J;\gamma}(X)$ consisting of operators
 $\sum_w c_w w W_I$ such that for every $w$, $c_w$ vanishes on the divisor
 \begin{gather*}
 \sum_{\alpha\in\Phi^+(W)\cap w\Phi^-(W)} T_{\alpha} +\!\!
 \sum_{\alpha\in\Phi^-(W_J)\setminus \Phi^-(W_J\cap w W_I w^{-1})}\!\!\!\! T_{\alpha}.
 \end{gather*}
\end{cor}

\begin{cor}
 There is an isomorphism
 \begin{gather*}
 {\cal H}_{W,W_I,W_J;\vec{T}}(X) \cong
 \sO_X\big(D_{w_0}\big({}^-\vec{T}\big)-D_{w_J}\big({}^-\vec{T}\big)\big)
 \otimes {\cal H}_{W,W_J,W_I;\vec{T}}(X)
 \\ \hphantom{ {\cal H}_{W,W_I,W_J;\vec{T}}(X) \cong}
 \otimes \sO_X\big(D_{w_I}\big({}^-\vec{T}\big)-D_{w_0}\big({}^-\vec{T}\big)\big)
 \end{gather*}
 which is contravariant for the natural composition.
\end{cor}

\begin{proof}
 Extend to universal parameters, write the left-hand side as an
 intersection, and take the adjoint of both twists of~${\cal
 H}_{W,W_I,W_J;\vec{T}}(X)$. The~resulting equality extends as usual to
 the full parameter space.
\end{proof}

When $I=J$, we~denote this by ${\cal H}_{W,W_I;\vec{T}}(X)$, and call the
resulting sheaf of algebras a {\em spherical algebra} of the Hecke algebra,
which is in general a subalgebra of the algebra \smash{$\sEnd\big((\pi_*{\cal
 L}_I)^{W_I}\big)$} corresponding to~$\vec{T}=0$.

Proposition~\ref{prop:spherical_as_intersection} leads in the usual way to
a description of the spherical algebra for general~$\vec{T}$ as an
intersection of two twists of the master spherical algebra:
\begin{gather*}
 {\cal H}_{W,W_I,W_J;\vec{T}}(X) = {\cal H}_{W,W_I,W_J}(X) \\ \qquad{}
 \cap \sO_X\bigg(\sum_{\alpha\in \Phi^-(W)\setminus\Phi^-(W_J)} \!\!\!\! T_\alpha\bigg)
 \otimes {\cal H}_{W,W_I,W_J}(X) \otimes
 \sO_X\bigg({-}\sum_{\alpha\in \Phi^-(W)\setminus\Phi^-(W_I)}\!\!\!\! T_\alpha\bigg).
\end{gather*}
In the case of the Hecke algebra, we~could twist this description to make
the second term untwisted, and found a relation between the original Hecke
algebra and the Hecke algebra for~${}^-\vec{T}$. In~this case, however,
the resulting twisted bimodule is not itself a spherical bimodule, as~the
divisors
\begin{gather*}
\sum_{\alpha\in \Phi^-(W)\setminus\Phi^-(W_I)}\!\!\!\! T_\alpha
\qquad\text{and}\qquad
-\!\!\!\sum_{\alpha\in \Phi^+(W)\setminus\Phi^+(W_I)}\!\!\!\! T_\alpha
\end{gather*}
do not differ by a $W$-invariant divisor. One can, however, give a
somewhat related description for the resulting twist, using the other
family of~${\cal H}_{W;\vec{T}}(X)$-modules associated to line bundles. We
find the following by tracing the various twists.

\begin{cor}
 We have a natural isomorphism
 \begin{gather*}
 \sHom_{{\cal H}_{W;\vec{T}}(X)}\bigg(
 \Ind^{W;\vec{T}}_{W_J}\sO_X\bigg(
 {-}\sum_{\alpha\in \Phi^+(W_J)} T_\alpha\bigg),
 \Ind^{W;\vec{T}}_{W_I}\sO_X\bigg(
 {-}\sum_{\alpha\in \Phi^+(W_I)} T_\alpha\bigg)\! \bigg)
 \\ \hphantom{\sHom_{{\cal H}}\,\,\,}
 \cong \sO\bigg(
 \sum_{\alpha\in \Phi^-(W)\cup \Phi^+(W_J)} T_\alpha\bigg)
 \otimes {\cal H}_{W,W_I,W_J;{}^-\vec{T}}(X)
 \otimes \sO\bigg(
 {-}\sum_{\alpha\in \Phi^-(W)\cup \Phi^+(W_I)} T_\alpha\bigg),
 \end{gather*}
 and in particular the left-hand side respects base change.
\end{cor}

This suggests looking at what happens when only one of the two bundles is
of the given form. Unfortunately, this behaves badly for special values of~$\vec{T}$, as~the piece associated to any given Bruhat subquotient involves
taking $W$-invariant sections of a line bundle which is not itself
equivariant, but instead a subbundle of an equivariant bundle cut out by
(generically) non-invariant vanishing conditions. This of course is not a
problem when there is a covering by symmetric idempotents, which suggests
looking for an alternate description that gives the same bimodule on such
fibers. This is not too difficult: when ${\cal H}_{W;\vec{T}}(X)$ has a
covering by symmetric idempotents, the module $\sO_X$ is locally projective
(and thus so is every module induced from~it), while the ${\cal
 H}_{W;\vec{T}}(X)$-module $\sO_X\big({-}\sum_{\alpha\in \Phi^+(W)}T_\alpha\big)$ is
locally projective whenever ${\cal H}_{W;{}^-\vec{T}}(X)$ has a covering by
symmetric idempotents. As~a result, if ${\cal L}$ is a twist of either of
these by an equivariant bundle, then the natural map
\begin{gather*}
\sHom_{{\cal H}_{W;\vec{T}}(X)}\big({\cal L},{\cal H}_{W;\vec{T}}(X)\big)
\otimes M\to \sHom_{{\cal H}_{W;\vec{T}}(X)}({\cal L},M)
\end{gather*}
is an isomorphism. The~tensor product {\em still} behaves badly in the
cases of interest, but in a~different way: for bad values of~$\vec{T}$, it
acquires torsion. Thus we may hope that the {\em image} of this natural
map will give a strongly flat extension of the $\Hom$ sheaf from the
symmetric idempotent locus. This image has a general description which is
closely related to the description coming from symmetric idempotents when
they exist.

\begin{prop}
 The image of the natural map\vspace{-.5ex}
 \begin{gather*}
 \sHom_{{\cal H}_{W;\vec{T}}(X)}\big(\sO_X,{\cal H}_{W;\vec{T}}(X)\big)
 \otimes_{{\cal H}_{W;\vec{T}}(X)} M \to M^W
 \end{gather*}
 is the same as the image of the map\vspace{-.5ex}
 \begin{gather*}
 m\mapsto \sum_{w\in W} w m\colon\ \sO_X\big(D_{w_0}\big({}^-\vec{T}\big)\big)\otimes M\to M^W.
 \end{gather*}
\end{prop}

\begin{proof}
 The natural map ${\cal H}_{W;\vec{T}}(X)^W\otimes M\to M^W$ is just the
 restriction of the map giving the action of the Hecke algebra on~$M$. The~local sections of~${\cal H}_{W;\vec{T}}(X)^W$ have the form
 $\sum_{w\in W} w h$, where $h$ is a local section of~$\sO_X\big(D_{w_0}\big({}^-\vec{T}\big)\big)$, and thus the image of the natural map
 consists of~elements $\sum_{w\in W} w h m$, immediately giving the
 desired description. Note here that $\sum_{w\in W} w$ is indeed
 contained in~${\cal H}_{W;\vec{T}}(X)\otimes
 \sO_X\big({-}D_{w_0}\big({}^-\vec{T}\big)\big)$, so this is well-defined.
\end{proof}

Taking the analogous result for~${}^-\vec{T}$ and reexpressing everything
in terms of the original Hecke algebra gives the following.

\begin{prop}
 The image of the natural map
 \begin{gather*}
 \begin{split}
 &\sHom_{{\cal H}_{W;\vec{T}}(X)}\bigg(\!
 \sO_X\bigg({-}\sum_{\alpha\in \Phi^+(W)}T_{\alpha}\bigg),
 {\cal H}_{W;\vec{T}}(X)\bigg) \otimes_{{\cal H}_{W;\vec{T}}(X)} M
 \\
 & \phantom{\sHom_{{\cal H}_{W;\vec{T}}(X)}}
 \to \sHom_{{\cal H}_{W;\vec{T}}(X)}\bigg(\!
 \sO_X\bigg({-}\sum_{\alpha\in\Phi^+(W)}T_\alpha\bigg), M\bigg)
 \end{split}
 \end{gather*}
 is naturally isomorphic to the image of the natural map
 \begin{gather*}
 m\mapsto \sum_{w\in W} w m\colon\
 \sO_X(D_{w_0})\otimes M \to
\bigg(\sO_X\bigg(
\sum_{\alpha\in \Phi^+(W)}T_\alpha\bigg)
 \otimes M\bigg)^W.
 \end{gather*}
\end{prop}

Using these, it is quite straightforward to obtain the desired strong
flatness results (in those cases where the $\Hom$ sheaf itself is not
strongly flat, that is): as images of maps of~$S$-flat sheaves, the sheaves
in question can only get smaller under specialization, so it suffices to
show that the~natural limits of the Bruhat subquotients are saturated.
Each such subquotient reduces to the image of~$\sum_{w\in W_J} w$ on a line
bundle of the form ${\cal L}(D_{w_J})$ where ${\cal L}$ is equivariant and
descends in~codimension~1; by the existence of symmetric idempotents when~$\vec{T}=0$, we~conclude that the image is just ${\cal L}^W$ independently
of~$\vec{T}$.

\begin{prop}\label{prop:shift_operators_finite}
 There is a natural strongly flat family of bimodules which for~$\vec{T}$
 in general position is given by
 \begin{gather*}
 \sHom_{{\cal H}_{W;\vec{T}}(X)}\bigg(\Ind^{W;\vec{T}}_{W_J}\sO_X
 \bigg({-}\sum_{\alpha\in \Phi^+(W_J)} T_\alpha\bigg),
 \Ind^{W;\vec{T}}_{W_I}\sO_X\bigg),
 \end{gather*}
 and may be expressed as the intersection
 \begin{gather*}
 {\cal H}_{W,W_I,W_J}(X) \cap \sO_X
\bigg(\sum_{\alpha\in \Phi^-(W)\cup\Phi^+(W_J)} \!\!\!\!
T_\alpha\bigg) \otimes {\cal H}_{W,W_I,W_J}(X)
 \otimes \sO_X\bigg({-}\sum_{\alpha\in \Phi^-(W)\setminus\Phi^-(W_I)}\!\!\!\! T_\alpha\bigg),
 \end{gather*}
 compatibly with composition.
\end{prop}

\begin{proof}
 We may interpret $\Ind^{W;\vec{T}}_{W_I}\sO_X$ as ${\cal
 H}_{W,W_I,1;\vec{T}}(X)$, and thus view it as a subsheaf of the
 parameter-free case ${\cal H}_{W,W_I,1}(X)$. The~operation~$m\mapsto
 \sum_{w\in W_J} w m$ gives a well-defined (surjective) map
 \begin{gather*}
 \sO_X(D_{w_J}) {\cal H}_{W,W_I,1}(X)
 \to
 {\cal H}_{W,W_I,W_J}(X),
 \end{gather*}
 and thus the (strongly flat) subsheaf discussed above of the $\Hom$ sheaf
 may itself be interpreted as a subsheaf of~${\cal H}_{W,W_I,W_J}(X)$.

 Twisting gives the analogous subsheaf of
 \begin{gather*}
 \sHom_{{\cal H}_{W;{}^-\vec{T}}(X)}\bigg(
 \Ind^{W;{}^-\vec{T}}_{W_J}\sO_X, \Ind^{W;{}^-\vec{T}}_{W_I}\sO_X
\bigg({-}\sum_{\alpha\in \Phi^-(W_I)} T_\alpha\bigg)\!\bigg).
 \end{gather*}
 Since this is itself a subsheaf of
 \begin{gather*}
 \sHom_{{\cal H}_{W;{}^-\vec{T}}(X)}\Big( \Ind^{W;{}^-\vec{T}}_{W_J}\sO_X,
 \Ind^{W;{}^-\vec{T}}_{W_I}\sO_X\Big) ={\cal H}_{W,W_I,W_J;{}^-\vec{T}}(X),
 \end{gather*}
 the other inclusion follows. Comparing Bruhat quotients gives the
 desired equality for~$\vec{T}$ in general position.
\end{proof}

\begin{cor}
 There is a natural strongly flat family of bimodules which for~$\vec{T}$
 in general position is given by
 \begin{gather*}
 \sHom_{{\cal H}_{W;\vec{T}}(X)}\bigg(
 \Ind^{W;\vec{T}}_{W_J}\sO_X, \Ind^{W;\vec{T}}_{W_I}\sO_X
\bigg({-}\sum_{\alpha\in \Phi^+(W_I)} T_\alpha\bigg)\!\bigg),
 \end{gather*}
 and may be expressed as the intersection
 \begin{gather*}
 {\cal H}_{W,W_I,W_J}(X) \cap
 \sO_X\bigg(
 \sum_{\alpha\in \Phi^-(W)\setminus \Phi^-(W_J)} \!\!\!\! T_\alpha\!\bigg)
 \otimes {\cal H}_{W,W_I,W_J}(X) \otimes
 \sO_X\bigg({-}\sum_{\alpha\in \Phi^-(W)\cup \Phi^+(W_I)} \!\!\!\! T_\alpha\!\bigg),
 \end{gather*}
 compatibly with composition.
\end{cor}

\section{Infinite groups}\label{section5}

Most of our arguments above regarding the structure of the Hecke algebra
boiled down to the combinatorics of (double) cosets in Coxeter groups and
the associated Bruhat order. Indeed, virtually everything in the above
discussion carries over immediately to the case of infinite Coxeter groups,
with one glaring exception: the Hecke algebra was defined as a sheaf of
algebras on the quotient $X/W$, and there is no such quotient scheme when~$W$ is infinite!

Thus the primary (and to first approximation only) issue in generalizing
the above construction is simply to determine what manner of object we will
be constructing. Luckily, a suitable generalization of sheaves of algebras
has already appeared in the literature on noncommutative geometry, namely
the notion of a ``sheaf algebra''.

We recall the definition from~\cite[Section~2]{VandenBerghM:1996}, generalizing
an earlier definition of~\cite[Section~2]{ArtinM/VandenBerghM:1990}. We first
need the notion of a sheaf bimodule: Let $X$, $Y$ be Noetherian $S$-schemes
of finite type: An {\em $\sO_S$-central $(\sO_X,\sO_Y)$-bimodule} is a
quasicoherent $\sO_{X\times_S Y}$-module $M$ such that the support of any
coherent subsheaf of~$M$ is finite over both $X$ and $Y$ (relative to the
projections). We~will sometimes shorthand this by saying that $M$ is a
sheaf bimodule on~$X\times_S Y$. Note that if~$X=\Spec(R_X)$,
$Y=\Spec(R_Y)$, then a sheaf bimodule on~$X\times_S Y$ is an
$(R_X,R_Y)$-bimodule such that $\Gamma(S;\sO_S)$ is central and such that
any finitely generated sub-bimodule is finitely generated both as a left
module and as a right module.

As with ordinary bimodules, there is a notion of tensor product for sheaf
bimodules. If~$M$ is a sheaf bimodule on~$X\times_S Y$ and $N$ is a sheaf
bimodule on~$Y\times_S Z$, then we can construct a sheaf bimodule on~$X\times_S Z$ by pulling back $M$ and $N$ to~$X\times_S Y\times_S Z$,
tensoring, and then taking the direct image to~$X\times_S Z$ to obtain a
sheaf $M\otimes_Y N$. Note that if $\Delta_{X/S}$ is the diagonal in~$X\times_S X$, then $\sO_{\Delta_{X/S}}$ is a sheaf bimodule on~$X\times_S
X$, and for any sheaf bimodule $M$ on~$X\times_S Y$, there is a natural
isomorphism $\sO_{\Delta_{X/S}}\otimes_X M\cong M$. Furthermore, this
tensor product operation is naturally associative and agrees with the usual
tensor product when the schemes are affine, see~\cite{VandenBerghM:1996}.

The tensor product provides the category of~$(\sO_X,\sO_X)$-bimodules with
a natural monoidal structure, thus allowing one to define a {\em sheaf
 algebra} on~$X/S$ to be a monoid object in that category; that is, a
sheaf bimodule $A$ equipped with morphisms $\sO_{\Delta_X}\to A$ and
$A\otimes_X A\to A$ satisfying the obvious axioms. More generally, one may
also consider {\em sheaf categories}, in which every object of the category
has an associated scheme and the $\Hom$ sets are replaced by sheaf
bimodules.

One difficulty in dealing with the above construction is that it is not
always easy to work with local sections of the tensor product of sheaf
bimodules. Indeed, since the tensor product is a direct image, we~in
general need to choose an affine open covering of~$Y$ and look for
compatible systems of elements of the corresponding na\"{\i}ve tensor
products. It turns out that for {\em coherent} sheaf bimodules, there is
a cleaner approach.

\begin{prop}\label{prop:computing_tp}
 Let $M$ be a sheaf bimodule on~$X\times_S Y$ and let $N$ be a sheaf
 bimodule on~$Y\times_S Z$. Let $\Spec(R)\cong V\subset Y$ be an affine
 open subset and let $U\subset X$, $W\subset Z$ be open subsets. If~$M$
 is coherent and the preimage of~$V$ in the support of~$M$ contains the
 preimage of~$U$, or if $N$ is coherent and the preimage of~$V$ in the
 support of~$N$ contains the preimage of~$W$, then there is a natural
 isomorphism
 \begin{gather*}
 \Gamma(U\times W;M\otimes_Y N)\cong \Gamma(U\times
 V;M)\otimes_{R}\Gamma(V\times W;N).
 \end{gather*}
\end{prop}

\begin{proof}
 By symmetry, we~may suppose that the constraint on~$M$ holds. The~sections of a~cohe\-rent sheaf bimodule on a product of open subsets
 depends only on the intersection of those open subsets on the support of
 the sheaf bimodule. It follows, therefore, that there is a natural
 isomorphism
 \begin{gather*}
 \Gamma(U\times V';M)\cong \Gamma(U\times (V\cap V');M)
 \end{gather*}
 for any open subset $V'$. Computing the tensor product via an affine
 open covering of~$Y$ con\-tai\-ning~$V$ gives a natural morphism
 \begin{gather*}
 \Gamma(U\times W;M\otimes_Y N)\to \Gamma(U\times V;M)\otimes_{R}\Gamma(V\times W;N),
 \end{gather*}
 and the compatibility conditions ensure that this is an isomorphism as
 required.
\end{proof}

\begin{rem}
 For coherent $M$, there is a maximal $U$ satisfying the hypothesis: take
 $X\setminus U$ to be the image of~$X$ of the preimage of~$Y\setminus V$,
 and observe that finiteness implies that this image is closed, so $U$ is
 open. For~quasicoherent $M$, it is tempting to consider the intersection
 of the $U$'s corresponding to the coherent subsheaves of~$M$, but of
 course this will rarely be open. Of~course, if it {\em is} open, then
 taking the limit tells us that the conclusion of the proposition
 continues to~hold.
\end{rem}

Of course, we~would like to know that this incorporates the usual notion of
a sheaf of algebras on the quotient.

\begin{prop}\label{prop:omitting_of}
 Let $f\colon X\to Y$ be a finite morphism of Noetherian $S$-schemes of finite
 type, and suppose that ${\cal A}$ is a quasicoherent sheaf
 of~$\sO_Y$-algebras on~$Y$ equipped with an algebra morphism $f_*\sO_X\to
 {\cal A}$. Then ${\cal A}$ induces a sheaf algebra ${\cal A}_X$ on~$X/S$
 such that for any open subsets $U,V\subset Y$, $\Gamma\big(f^{-1}(U)\times
 f^{-1}(V);{\cal A}_X\big)\cong \Gamma(U\cap V;{\cal A})$.
\end{prop}

\begin{proof}
 As~usual, we~may assume that $S$ is affine. For~any affine open subset
 $U\subset Y$, the sheaf-of-algebras morphism $f_*\sO_X\to {\cal A}$
 induces an algebra morphism $\Gamma(U;f_*\sO_X)\to \Gamma(U;{\cal A})$,
 and thus makes $\Gamma(U;{\cal A})$ a bimodule over
 $\Gamma(U;f_*\sO_X)\cong \Gamma\big(f^{-1}(U);\sO_X\big)$. More generally, if
 $U,V\subset Y$ are two affine open subsets, then we may use the morphisms
 $\Gamma(U;f_*\sO_X)\to \Gamma(U\cap V;f_*\sO_X)$ and
 $\Gamma(V;f_*\sO_X)\to \Gamma(U\cap V;f_*\sO_X)$ to make $\Gamma(U\cap
 V;{\cal A})$ a
 $\big(\Gamma\big(f^{-1}(U);\sO_X\big),\Gamma\big(f^{-1}(V);\sO_X\big)\big)$-bimodule.

 This in particular gives a family of~$\big(\Gamma\big(f^{-1}(U_i);\sO_X\big),\Gamma\big(f^{-1}(U_j);\sO_X\big)\big)$-bimodules
 associated to any affine open covering of~$Y$. Since the coefficient
 rings are commutative, we~may reinterpret this as a
 $\Gamma\big(f^{-1}(U_i);\sO_X\big)\otimes_{\sO_S}
 \Gamma\big(f^{-1}(U_j);\sO_X\big)$-module structure, and then observe that
 \begin{gather*}
 \Gamma\big(f^{-1}(U_i);\sO_X\big)\otimes_{\sO_S} \Gamma\big(f^{-1}(U_j);\sO_X\big)
 \cong
 \Gamma\big(f^{-1}(U_i)\times_S f^{-1}(U_j);\sO_{X\times_S X}\big).
 \end{gather*}
 The open subsets $f^{-1}(U_i)\times_S f^{-1}(U_j)$ cover $X\times_S X$,
 and their intersections have the same form, making it easy to see that one
 has natural and compatible restriction maps. It follows that these
 module structures glue together to give a sheaf on~$X\times_S
 X$. Moreover, the sheaf is supported on the preimage in~$X\times_S X$ of
 the diagonal in~$Y\times_S Y$, and thus satisfies the requisite
 finiteness condition to be a sheaf bimodule.

 It remains to see that this is a sheaf algebra. We note that for any
 open subset of~$Y$, the pair of open subsets $\big(f^{-1}(U),f^{-1}(U)\big)$
 satisfies the hypotheses of Proposition~\ref{prop:computing_tp} for any
 coherent subsheaf of~${\cal A}_X$ on either side, from which it easily
 follows that the morphisms $f_*\sO_X\to {\cal A}$ and~${\cal
 A}\otimes_{f_*\sO_X}{\cal A}\to {\cal A}$ induce a sheaf algebra
 structure on~${\cal A}_X$.
\end{proof}

There is a more general version of the approach to tensor products via open
subsets that works for fairly general quasicoherent sheaf algebras and
bimodules (including any case where the schemes are projective); this leads
to some simpler alternate arguments and constructions below, but will not
be strictly needed.

The construction of a sheaf algebra from a sheaf of algebras suggests
defining an ``invariant open subset'' of a sheaf algebra, as~an open subset
$U$ such that the two preimages of~$U$ in the support of any coherent
subsheaf of the sheaf algebra agree. If~$U$ is an affine open subset which
is invariant for a given sheaf algebra ${\cal A}$, we~immediately conclude
that the sections on~$U\times U$ of~${\cal A}$ form an algebra equipped
with a morphism from $\Gamma(U;\sO_U)$. The~difficulty, of course, is that
a typical sheaf algebra will have no invariant affine opens. Luckily, the
usual construction of a sheaf by gluing depends far more on the subsets
being affine than that they be open. Define an {\em affine localization}
of~$X$ to be a nonempty affine scheme which is the directed limit of
a~(possibly infinite) family of open subschemes of~$X$. As~with affine
opens, an affine localization is determined by its image in the underlying
topological space of~$X$, and if given a family of~such subsets covering
$X$ (in a locally finite way: every open subset of~$X$ needs to be
contained in a finite union of localizations), the corresponding family of
morphisms will be faithfully flat. It then follows by fpqc descent that we
may specify a sheaf (or morphisms of sheaves) using a~covering by affine
localizations in place of a covering by affine opens. One similarly finds
that there is a well-behaved notion of~``invariant'' affine localization,
and the restriction of a sheaf algebra to an invariant affine localization
is an algebra.

There is still a difficulty here, in that there is no guarantee that there
will always be a~locally finite covering by invariant affine localizations.
(For instance, the only invariant affine localization of the sheaf algebra
$\overline{k}\big(\P^1\big)\big[\PGL_2(\overline{k})\big]$ on~$\P^1_{\overline{k}}$ is the
field $\overline{k}\big(\P^1\big)$ itself.) If there were an~fpqc base change
$S'\to S$ such that the pullback to~$(X\times_S S')\times_{S'} (X\times_S
S')$ was covered by~invariant localizations, then we could work with those
localizations to understand the algebra structure and then use a further
application of fpqc descent to recover the morphisms on~$X\times_S X$.
Of~course, rather than make two separate applications of fpqc descent, we~could simply observe that $U_i\otimes_{S'} U_j\to X\times_S X$ give an fpqc
covering and do the descent directly. But this tells us that there was no
need for the $U_i$ to cover $X\times_S S'$; all we need is for them to
cover $X$.

Given a scheme $S$, let $\hat\A^1_S$ denote the localization of~$\A^1_S$
obtained as the limit of those open subschemes that are dense in every
geometric fiber (equivalently, that contain the generic point of every
fiber). Although the map $\hat\A^1_S\to \A^1_S$ is not an fpqc cover, the
composition~$\hat\A^1_S\to S$ is both fpqc and surjective, and thus an fpqc
cover.

This construction is functorial in~$S$ and if $S\to T$ is a finite morphism
with $T$ Noetherian, then $\hat\A^1_S\cong \hat\A^1_T\times_T S$. Note,
however, that this construction does not respect open embeddings, so the
obvious way to associate a sheaf of algebras on~$S$ to this construction
does not always produce a quasicoherent sheaf.

We observe that if $X$ is projective over a Noetherian ring $R$, then
$\hat\A^1_X$ is affine over $R$. Since the construction respects closed
embeddings, it suffices to consider the case $X=\P^n_R$. In~that case, we~note that the section~$\sum_i t^i x_i$ of the pullback of~$\sO_X(1)$ is
invertible, and thus the trivial line bundle is very ample on~$\hat\A^1_X$,
making it affine. Since this holds for any embedding of~$X$ in projective
space, it follows that any very ample line bundle on~$X$ becomes trivial on~$\hat\A^1_X$, and thus (since very ample bundles generate the Picard group)
that any line bundle on~$X$ becomes trivial on~$\hat\A^1_X$.

\begin{prop}\label{prop:natural_localization}
 Suppose $X$ and $Y$ are projective over the Noetherian affine scheme $S$,
 and let $M$ be a quasicoherent sheaf bimodule on~$X\times_S Y$. Let $M'$
 be the base change of~$M$ to~$\hat\A^1_S$. Then the fiber products of~$\hat\A^1_X$ and $\hat\A^1_Y$ with the support of any coherent subsheaf
 of~$M'$ are canonically isomorphic, and the affine localization~$\hat\A^1_X\times_{\hat\A^1_S}\hat\A^1_Y$ of~$X\times_S Y\times_S
 \hat{A}^1_S$ is an fpqc covering of~$X\times_S Y$.
\end{prop}

\begin{proof}
 The only thing to observe is that the support of any coherent subsheaf of~$M'$ is contained in the base change of the support of a coherent
 subsheaf of~$M$, and thus the claim reduces to~the fact that the
 construction~$\hat\A^1$ respects finite morphisms.
\end{proof}

\begin{rem}
 In particular, given bimodules on~$X\times_S Y$ and $Y\times_S Z$, we~can
 use the corresponding bimodules on~$\hat\A^1_X\times_{\hat\A^1_S}
 \hat\A^1_Y$ and $\hat\A^1_Y\times_{\hat\A^1_S}\hat\A^1_Z$ to control the
 tensor product.
\end{rem}

Thus when~$X$ is projective, or more generally when~$\hat\A^1_X$ is affine, we~always have the option to replace the sheaf algebra with the actual
algebra of global sections of~$M$ on~$\hat\A^1_X\times_{\hat\A^1_S}\hat\A^1_X$, and~des\-cent essentially reduces
to checking that the algebra has a description which is independent of the
auxiliary coordinate.

According to Proposition~\ref{prop:omitting_of}, the algebras ${\cal
 H}_G(X)$, ${\cal H}_W(X)$, ${\cal H}_{W;\vec{T}}(X)$ from the previous
section may all be interpreted as sheaf algebras on~$X/S$, as~can the
twisted group algeb\-ras~$k(X)[G]$, $\sO_X[G]$. The~latter are quite easy to
generalize to the infinite case.

\begin{prop}[{\cite[Lemma~2.8]{VandenBerghM:1996}}]
 Let $g$ be an automorphism of~$X$. Then for any quasicoherent sheaf $M$ on~$X$, $\big(1,g^{-1}\big)_*M$ is a sheaf bimodule, and for~$h\in \Aut(X)$ and
 $N\in\coh(X)$, we~have a natural isomorphism
 \begin{gather*}
 \big(1,g^{-1}\big)_*M\otimes_X \big(1,h^{-1}\big)_*N
 \to
 \big(1,(gh)^{-1}\big)_*(M\otimes {}^gN).
 \end{gather*}
\end{prop}

\begin{proof}
 Since $\big(1,g^{-1}\big)_*M$ is supported on the graph of an automorphism, the
 same applies to any coherent subsheaf, and thus it satisfies the
 requisite finiteness condition to be a sheaf bimodule.

 Now, let $U$ be any affine open subset of~$X$. Then $\big(U,g^{-1}(U)\big)$
 satisfy the hypotheses of~Pro\-position~\ref{prop:computing_tp}, and thus we
 have
 \begin{gather*}
 \Gamma\big(U\times X;\big(1,g^{-1}\big)_*M\otimes_X (1,h^{-1})_*N\big)
 \\ \hphantom{\Gamma(U\times X;}
 {}\cong \Gamma\big(U\times g^{-1}(U);\big(1,g^{-1}\big)_*M\big)
 \otimes_{\Gamma(g^{-1}(U);\sO_X)} \Gamma\big(g^{-1}(U)\times X;\big(1,h^{-1}\big)_*N\big)
 \\ \hphantom{\Gamma(U\times X;}
 \cong \Gamma(U;M) \otimes_{\Gamma(g^{-1}(U);\sO_X)} \Gamma\big(g^{-1}(U);N\big),
 \end{gather*}
 where $f\in \Gamma\big(g^{-1}(U);\sO_X\big)$ acts on~$\Gamma(U;M)$ as
 multiplication by ${}^gf$. With this action, there is a~natural isomorphism
 \begin{gather*}
 \Gamma(U;M) \otimes_{\Gamma(g^{-1}(U);\sO_X)} \Gamma\big(g^{-1}(U);N\big)
 \cong \Gamma(U;M) \otimes_{\Gamma(U;\sO_X)} \Gamma\big(U;{}^gN\big)
 \end{gather*}
 given by $m\otimes n\mapsto m\otimes {}^gn$, and thus the result follows.
\end{proof}

\begin{defn}
 Let $X/S$ be an Noetherian $S$-scheme of finite type with integral
 geometric fibers, and let $G$ be a group acting on~$X$. Then the
 ``twisted group sheaf algebra'' $k(X)[G]$ is the sheaf
 \begin{gather*}
 \bigoplus_{g\in G} \big(1,g^{-1}\big)_*k(X)
 \end{gather*}
 on~$X\times X$ (with $k(X)$ denoting the sheaf of meromorphic functions
 on~$X$ which are defined on~the generic point of every geometric fiber of~$X$ over $S$) with sheaf algebra structure induced by~the natural
 morphisms
 \begin{gather*}
 \big(1,g^{-1}\big)_*k(X)\otimes_X \big(1,h^{-1}\big)_*k(X)
 \to
 \big(1,(gh)^{-1}\big)_*k(X)
 \end{gather*}
 coming from the proposition.
\end{defn}

\begin{rem}
 We could also define this using an affine localization. The~affine
 scheme $\hat{\A}^1_X$ constructed in Proposition~\ref{prop:natural_localization} is functorial for~$\Aut_S(X)$, and thus
 has an induced action of~$G$. We~may thus define a twisted group algebra
 $k(X)\otimes_S\sO_{\hat{\A}^1_X}[G]$, and this has a natural associated
 descent datum. The~only nontrivial thing to verify is the fact that
 cyclic bimodules are finitely generated on both sides, but this follows
 easily from the fact that $\hat{\A}^1_X$ is Noetherian: the bimo\-dule
 $\sO_{\hat{\A}^1_X}c_g g\sO_{\hat{\A}^1_X}$ is cyclic on both sides, and
 any other cyclic bimodule is contained in a finite sum of such bimodules,
 so is finitely generated on both sides.
\end{rem}

We may similarly define $\sO_X[G]$ to be the sheaf subalgebra
$\bigoplus_{g\in G} \big(1,g^{-1}\big)_*\sO_X$, which we readily verify to contain
the image of the identity (a copy of~$\sO_X$) and be preserved by the
multiplication map.

Moreover, we~have the following.

\begin{prop}\sloppy
 Let $g_1,\dots,g_n$; $h_1,\dots,h_m$ be two finite subsets of~$\Aut(X/S)$,
 and let \linebreak ${\cal M}_{\{g_1,\dots,g_n\}}$, ${\cal M}_{\{h_1,\dots,h_m\}}$, ${\cal
 M}_{\{g_1h_1,\dots,g_nh_m\}}$ be interpreted as sheaf
 sub-bimodules of $k(X)[\Aut(X/S)]$. Then the multiplication on~$k(X)[\Aut(X/S)]$ restricts to a morphism
 \begin{gather*}
 {\cal M}_{\{g_1,\dots,g_n\}}\otimes {\cal M}_{\{h_1,\dots,h_m\}}
 \to
 {\cal M}_{\{g_1h_1,\dots,g_nh_m\}}.
 \end{gather*}
\end{prop}

\begin{proof}
 Given an open subset $V\subset X$, we~may associate an open subset
 $U_V=\bigcap_i g_i(V)$, and we claim that there is an affine open
 covering $V_i$ of~$X$ such that $U_{V_i}$ also covers $X$. Indeed, $x\in
 U_V$ iff $g_1^{-1}(x),\dots,g_n^{-1}(X)\in V$, and thus if $V_x$ is an
 affine open neighborhood of this set of points, we~have $x\in U_{V_x}$.
 It thus suffices to specify how the above morphism acts on local sections
 on~sets of the form $U_V\times X$, for~which we note
 \begin{gather*}
 \Gamma\big(U_V\times X;{\cal M}_{\{g_1,\dots,g_n\}}\otimes_X {\cal M}_{\{h_1,\dots,h_m\}}\big)
 \\ \hphantom{ \Gamma(U_V\times X;}
 \cong \Gamma\big(U_V\times X;{\cal M}_{\{g_1,\dots,g_n\}}\big)
 \otimes_{\Gamma(V;\sO_V)} \Gamma\big(V\times X;{\cal M}_{\{h_1,\dots,h_m\}}\big).
 \end{gather*}
 But the result then follows immediately from the definition of~${\cal
 M}_{\{g_1,\dots,g_n\}}$ as the space of ope\-ra\-tors preserving holomorphy.
\end{proof}

\begin{rem}
 When $X/S$ is projective, or more generally when we have nice affine
 localizations as~con\-st\-ructed above, then we could define ${\cal
 M}_{\{g_1,\dots,g_n\}}$ more simply as the subsheaf of \linebreak $k(X)[\Aut(X/S)]$ supported on~$\{g_1,\dots,g_n\}$ and preserving the
 subring $\sO_{\hat{\A}^1_X}$.
\end{rem}

In particular, if $G$ is any subgroup of the group of automorphisms of~$X/S$, we~may define a sheaf algebra ${\cal H}^+_G(X)$ as the union in~$k(X)[G]$ of the sheaves ${\cal M}_{\vec{g}}$ associated to finite subsets
of~$G$. More generally, we~will wish to only allow poles on a proper (but
$G$-invariant) subset of the reflection hypersurfaces, as~otherwise in the
affine Weyl group case we could acquire poles along the fibers where the
``$q$'' parameter is torsion.

Suppose now that $X/S$ is a family of abelian varieties and that $(W,S)$ is
a~Coxeter group (of~finite rank, but possibly infinite) equipped with an
action on~$X$ of coroot type. Then we define~${\cal H}_W(X)$ to be the
sheaf subalgebra of~${\cal H}^+_W(X)$ consisting of operators which are
holomorphic away from the reflection hypersurfaces corresponding to
conjugates of the simple reflections. Clearly, this agrees with our
previous notation, in that when~$W$ is finite, ${\cal H}_W(X)$ is the sheaf
algebra on~$X$ associated to the sheaf of algebras on~$X/W$ we previously
denoted by ${\cal H}_W(X)$.

Since $W$ still has a Bruhat order, we~may consider the subsheaf ${\cal
 H}_W(X)[I]$ for any order ideal~$I\subset W$, and if the order ideal is
finite, the result will be of the form ${\cal M}_I(X)$ and thus coherent. In~fact, we~have the following, by precisely the same argument as Lemma~\ref{lem:Bruhat_fin_noparm} and its corollaries.

\begin{prop}
 If $I$ is a finite order ideal in~$W$ and $w\in I$ is a maximal element,
 then there is a short exact sequence
 \begin{gather*}
 0\to {\cal H}_W(X)[I\setminus \{w\}]
 \to {\cal H}_W(X)[I]
 \to (1,w^{-1})_*\sO_X(D_w)
 \to 0
 \end{gather*}
 of sheaf bimodules, where $D_w=\sum_{r\in R(W),rw<w} [X^r]$.
\end{prop}

\begin{cor}
 For any reduced word $w=s_1\cdots s_n$, the multiplication map
 \begin{gather*}
 {\cal H}_{\langle s_1\rangle}(X)\otimes_X\cdots\otimes_X {\cal H}_{\langle
 s_n\rangle}(X)
 \to
 {\cal H}_{W}(X)[\le w]
 \end{gather*}
 is surjective.
\end{cor}

\begin{cor}
 The construction~${\cal H}_W(X)$ respects base change.
\end{cor}

Since any finite subset of~$W$ is contained in a finite order ideal, we~also obtain the following.

\begin{cor}
 The sheaf algebra ${\cal H}_W(X)$ is the sheaf subalgebra of~$k(X)[W]$
 generated by the sheaf subalgebras ${\cal H}_{\langle s\rangle}(X)$ for~$s\in S$.
\end{cor}

Since the action of~$W$ on~$X$ is of coroot type, we~still have a
well-defined association of coroot morphisms to the roots of~$X$,
respecting positivity, and thus the notion of~a~system of~parameters
carries over.

\begin{defn}
 The (untwisted) Hecke algebra ${\cal H}_{W;\vec{T}}(X)$ is the sheaf
 subalgebra of~${\cal H}_W(X)$ generated by the rank~1 sheaf algebras
 ${\cal H}_{\langle s_i\rangle,T_i}(X)$.
\end{defn}

Lemma~\ref{lem:Bruhat_fin_parm} and Corollary~\ref{cor:Bruhat_intervals_product} again carry over immediately, as~does
the fact that this construction respects base change. However, the
description of the adjoint and the description of Proposition~\ref{prop:finite_hecke_as_intersection} both founder on the fact that they
involve a sum over all positive roots. To~deal with this, we~will need to
generalize the construction further.

We first note that if we are given a $G$-equivariant gerbe ${\cal Z}$
(including, of course, the compatible explicit isomorphisms
$\zeta_{g,h}:{\cal Z}_g\otimes {}^g{\cal Z}_h\cong {\cal Z}_{gh}$), then as
in the finite case, there is a corresponding crossed product algebra: take
the sheaf bimodule $\bigoplus_g \big(1,g^{-1}\big)_*{\cal Z}_g$ with multiplication
induced by $\zeta$. Of course, this also gives a twisted version of~$k(X)[G]$ by replacing ${\cal Z}_g$ by its sheaf of meromorphic sections.
Note that if ${\cal Z}_g$ and ${\cal Z}'_g$ are meromorphically equivalent
(i.e., there is a system of nonzero meromorphic maps between the line
bundles which are compatible with the maps~$\zeta$), then this induces an
isomorphism between the corresponding meromorphic crossed product algebras. In~particular, the meromorphic crossed product algebra associated to a
given equivariant gerbe is isomorphic to the usual twisted group algebra
iff there is a consistent family of meromorphic sections of~${\cal Z}_g$,
iff the equivariant gerbe has the form ${\cal Z}_g=\sO_X(Z_g)$, where $Z_g$
is a cocycle valued in Cartier divisors. (More generally, if one chooses
arbitrary meromorphic sections, one obtains a meromorphic equivariant gerbe
in which the line bundles are trivial but the maps $\zeta_{g,h}$ are only
meromorphic, giving a class in~$Z^2(G;k(X)^*)$ and making the algebra
a~crossed product algebra in the usual sense. We will encounter an example
of such a gerbe in Theorem~\ref{thm:daha_for_torsion_q} below.) We denote
these sheaf algebras as $\sO_X[G]_{\cal Z}$ and $k(X)[G]_{\cal Z}$
respectively.

\looseness=-1
In this generality, we~cannot expect to have a well-defined analogue of~${\cal H}_G(X)$ without some additional data. For~any line bundle ${\cal
 L}$, there is an induced sheaf algebra structure on~${\cal
 H}_G(X)\otimes_{X\times X} {\cal L}\boxtimes {\cal L}^{-1}$, which is
sandwiched between $\sO_X[G]_{\partial{\cal L}}$ and
$k(X)[G]_{\partial{\cal L}}$, with~$\partial{\cal L}$ denoting the
coboundary gerbe ${\cal Z}_g={\cal L} \otimes {}^g{\cal L}^{-1}$. This
equivariant gerbe is isomorphic to the trivial equivariant gerbe when~${\cal L}$ is $G$-equivariant (more precisely, such an isomorphism
specifies a $G$-equivariant structure on~${\cal L}$), and this will in
general give a {\em different} sheaf algebra sandwiched between $\sO_X[G]$
and $k(X)[G]$. (Consider multiplying the equivariant structure by a
character of~$G$.)

We thus start with a more general construction. Given a subsheaf ${\cal
 A}\subset k(X)[G]_{\cal Z}$ and a~subset $S\subset G$, let ${\cal A}|_S$
denote the subsheaf of operators supported on~$S$ (which is a subalgebra if
$S$ is a subgroup). An {\em order} in~$k(X)[G]_{\cal Z}$ is defined to be
a torsion-free sheaf algebra ${\cal A}$ with generic fiber $k(X)[G]_{\cal
 Z}$ such that ${\cal A}|_S$ is coherent for every finite subset $S\subset
G$. (Note that if $G$ is finite, then $k(X)[G]_{\cal Z}$ may be viewed as
a central simple algebra over $k(X/G)$, and this is precisely the usual
notion of order.) In particular, $\sO_X[G]_{\cal Z}$ is an order in~$k(X)[G]_{\cal Z}$, and ${\cal H}_G(X)$ is an order in~$k(X)[G]$.

We are interested in a particular subclass of orders, namely those which
are left reflexive, in~that each sheaf ${\cal A}|_S$ is reflexive as a left
$\sO_X$-module. The~order $\sO_X[G]_{\cal Z}$ is clearly left reflexive,
and since an element of~$k(X)[G]$ is a (left) local section of~${\cal
 H}_G(X)$ iff it is a local section over every codimension~1 local ring,
${\cal H}_G(X)$ is also left reflexive. (Note that when~$G$ is finite but
$X\to X/G$ is not flat, ${\cal H}_G(X)$ is left reflexive as an $\sO_X$
module, but {\em not} as an $\sO_{X/G}$-module.) This suggests that we
should consider left reflexive orders more generally.

\begin{prop}
 Let ${\cal A}$ be a left reflexive order in~$k(X)[G]_{\cal Z}$ containing
 $\sO_X[G]_{\cal Z}$. For~any open subset $U\subset X$, an element $D\in
 k(X)[G]_{\cal Z}$ is in~$\Gamma(U\times X;{\cal A})$ iff for every
 codimension~$1$ point $x\in U$, with inertia group $I_x$, $D\in
 \sO_{X,x}\otimes_{\sO_X}({\cal A}|_{I_x} \sO_X[G]_{\cal Z})$.
\end{prop}

\begin{proof}
 Since ${\cal A}$ is left reflexive, we~certainly have that $D$ is a local
 section iff $D\in \sO_{X,x}\otimes_{\sO_X}{\cal A}$ for all $x\in U$ of
 codimension~1, as~this is true for any reflexive sheaf on~$X$, and ${\cal
 A}$ inherits it from its restrictions to finite subsets. We thus need
 to understand the modules $\sO_{X,x}\otimes_{\sO_X}{\cal A}|_S$ for
 finite $S$. But then we may proceed as in the proof of Lemmas~\ref{lem:holomorphy_preserving}
 and~\ref{lem:holomorphy_preserving_splits} to deduce that
 \begin{gather*}
 \sO_{X,x}\otimes_{\sO_X}{\cal A}|_S
 =
 \bigoplus_{g\in I_x\setminus G} \sO_{X,x}\otimes {\cal A}|_{I_xg\cap S},
 \end{gather*}
 and thus reduce to the case $S\subset I_x g$. Since ${\cal
 A}|_{I_x}{\cal Z}_g g\subset {\cal A}|_{I_xg}$ and ${\cal
 A}|_{I_x}\subset {\cal A}|_{I_xg} {\cal Z}_{g^{-1}} g^{-1},$ we~conclude that ${\cal A}|_{I_xg}={\cal A}|_{I_x} {\cal Z}_g g$, giving the
 desired result.
\end{proof}

\begin{prop}\label{prop:reflexive_order_construction}
 Let $X$ be a normal integral scheme equipped with an action of the group
 $G$ and an equivariant gerbe ${\cal Z}$, and for each codimension~$1$ point
 $x\in X$, let ${\cal A}_x$ be an subalgebra of~$k(x)[I_x]_{\cal Z}$
 containing $\sO_{X,x}[I_x]_{\cal Z}$ and such that ${\cal A}_x|_S$ is a
 free left $\sO_{X,x}$-module for every finite subset $S$. Suppose
 moreover that for any $g\in G$, we~have ${\cal A}_{gx} = {\cal Z}_g g
 {\cal A}_x {\cal Z}_{g^{-1}} g^{-1}$. Then there is a unique left
 reflexive order ${\cal A}$ in~$k(X)[G]_{\cal Z}$ such that
 $\sO_{X,x}\otimes_{\sO_X}{\cal A}|_{I_x} = {\cal A}_x$ for all $x$.
\end{prop}

\begin{proof}
 For each finite $S\subset G$, we~certainly obtain a well-defined left
 reflexive sheaf bimodule consisting of operators $\sum_{g\in S} c_g g$
 which are in~${\cal A}_x \sO_X[G]_{\cal Z}$ for every $x$; the point is
 that just as in the case of holomorphy-preserving operators, there are
 only finitely many $x$ such that $S$ meets some $I_xg$ in more than one
 element. By the previous proposition, any algebra as described must {\em
 contain} these sheaves, so it remains only to show that these sheaf
 bimodules are compatible under multiplication. That is, if $D_1$ is a
 section on~$U\times X$ and $D_2$ is a section on~$X\times V$, then~$D_1D_2$ is a section on~$U\times V$. We thus need to check that
 $D_1D_2$ is in~${\cal A}_x \sO_X[G]_{\cal Z}$ for every codimension~1
 point $x\in U$. Since this must hold for all $U$, we~may as well take
 the limit and thus take $D_1$ to be an element of the left stalk at $x$.
 Since this splits as a direct sum, we~may further suppose that $D_1$ is
 supported on a single coset of the inertia group, so that $D_1\in {\cal
 A}_x {\cal Z}_g g$. This is equivalent to~$D_1\in {\cal Z}_g g {\cal
 A}_{g^{-1}x}$, and we find
 \begin{gather*}
 {\cal A}_x {\cal Z}_g g D_2\subset {\cal A}_x \sO_X[G]_{\cal Z}\otimes \sO_V
 \end{gather*}
 iff
 \begin{gather*}
 {\cal A}_{g^{-1}x} D_2\subset {\cal A}_{g^{-1}x} \sO_X[G]_{\cal
 Z}\otimes \sO_V.
 \end{gather*}
 Since $D_2$ is a section of the sheaf, we~have
 \begin{gather*}
 D_2\in {\cal A}_y\sO_X[G]_{\cal Z}\otimes \sO_V
 \end{gather*}
 for every codimension~1 point $y\in X$, and thus for~$y=g^{-1}x$,
 and the claim follows from the fact that ${\cal A}_{g^{-1}x}$ is an
 algebra.
\end{proof}

We thus see that the construction of a reasonable analogue of~${\cal
 H}_G(X)$ reduces to constructing an analogue for each inertia group,
subject to the compatibility conditions under conjugation. There are two
approaches we might take to this. The~first is that if we are given a
trivialization of the gerbe (or, more precisely, its restriction to~$I_x$)
along the local ring at $x$, then this lets us pull ${\cal H}_{I_x}(X)$
back to an algebra containing $\sO_{X,x}[I_x]_{\cal Z}$ which we may use as
an ingredient in the above construction. Note that when~$I_x$ is finite, we~do not actually require $I_x$ to be a~trivialization of the gerbe for
this to work; it merely needs to be an {\em approximate} trivialization,
since what we are really determining is the quotient $\sO_{X,x}$-module
$\sO_{X,x}\otimes {\cal H}_{I_x}(X)/\sO_{X,x}[I_x]_{\cal Z}$.
Equivalently, we~may ask for a trivialization over the {\em complete} local
ring.

A second approach is to simply ask for an order of approximately the same
``shape''. This is particularly feasible in the case of order 2
reflections. Although the result is the same in most cases, this will be
particularly useful for us, as~it will be easy to extend to more general
base schemes. With this in mind, let $X/S$ be a normal scheme with an
action of an involution~$s$ such that the fixed subscheme $X^s$ is an
irreducible hypersurface, and let ${\cal Z}_s$, $\zeta_s$ be a gerbe, so
that $\zeta_s\colon {\cal Z}_s\otimes {}^s{\cal Z}_s\to \sO_X$ is an isomorphism
satisfying $\zeta_s={}^s\zeta_s$. An obvious ``shape'' to take for a
larger algebra would be to take operators $c_1+c_s s$ such that $c_1\in
\sO_X([X^s])$, $c_s\in {\cal Z}_s([X^s])$, and $c_1+hc_s\in \sO_{X,X^s}$
for some rational map $h\colon {\cal Z}_s\ratto \sO_X$ which is holomorphic along
$X^s$. We may then readily verify that the result is closed under
multiplication iff $\zeta_s-h{}^sh$ vanishes along $2[X^s]$. Note that if
$h-h'$ vanishes on~$X^s$, then replacing $h$ by $h'$ gives the same
algebra; the consistency condition is unchanged since $h{}^sh-h'{}^sh'$
vanishes along $2[X^s]$. (Note that we may view $h{}^sh$ as the pullback
of the norm of~$h$ down to~$X/\langle s\rangle$, and this interpretation
induces a natural norm from $\Hom({\cal Z}_s,\sO_{[X^s]})$ to~$\Hom({\cal
 Z}_s\otimes{}^s{\cal Z}_s,\sO_{2[X^s]})$.) This is very nearly the same
as asking for a trivialization over the complete local ring; indeed, the
two notions disagree only when the residue field is an inseparable
extension of the residue field of the $s$-invariant subring.

\looseness=1
In particular, given a line bundle ${\cal L}$ on~$X$, we~have a coboundary
gerbe ${\cal Z}_s = {\cal L}\otimes {}^s{\cal L}^{-1}$, and there is a
natural choice of~$h$, namely $h={}^sf f^{-1}|_{[X^s]}$ for some (any)
meromorphic section~$f$ of~${\cal L}$ which is holomorphic and not
identically 0 along $X^s$. This, of course, simply corresponds to the
usual notion of twisting by a line bundle. This coboundary operation is
functorial in a particularly strong sense: not only is it functorial, but
the functor takes any automorphism to the identity. Thus if instead of a
line bundle on~$X$ we are given a line bundle ${\cal L}_T$ on the base
change $X_T$ to some fppf cover, then all we need for the coboundary to
descend to~$S$ is for the two pullbacks of~${\cal L}_T$ to~$T\times_S T$ to
be isomorphic (that is, they need not be {\em compatibly} isomorphic!).
Indeed, if ${\cal L}_T$ does not descend, then the obstruction is given by
an automorphism of a pullback to~$T\times_S\times T\times_S T$, and the
coboundary functor turns this into the identity. In~particular, if $X/S$
is projective, any section of the relative Picard scheme gives rise to a
well-defined coboundary.

More generally, given $X/S$ with an action of an involution~$s$ such that
$X^s$ is nonempty and everywhere of codimension~1, define a ``twisting
datum'' to be an equivariant gerbe equipped with a morphism $h_s\colon {\cal
 Z}_s\to \sO_{[X^s]}$ of norm $\zeta_s$. Note that if $X/S$ is proper,
then $\zeta_s$ is the pullback of a function on~$S$, so is determined by
its restriction to~$2[X^s]$, and thus by $h_s$, which must merely satisfy
the requirement that its norm be invertible and constant. For~each fppf
$T/S$, let $\Tw^0(X/S)(T)$ be the group of twisting data with~${\cal
 Z}_s=\sO_X$; when~$X$ is proper, this is the set of global sections of~$\sO_{[X^s]\times_S T}$ with norm in~$\sO_T^*$, and thus $\Tw^0(X/S)$ is
represented by a group scheme. There is a natural morphism to this group
scheme from the sheaf of groups consisting of pairs $({\cal L},\psi)$ with~$\psi\colon {}^s{\cal L}\cong {\cal L}$: take the coboundary twisting datum and
use $\psi$ to make the line bundle trivial. When $X/S$ is projective, this
sheaf of groups is itself representable; it is a $\G_m$-bundle over the
$s$-invariant subscheme of the Picard scheme of~$X$. Moreover, there is a~natural homomorphism from $\Pic(X/\langle s\rangle)$ to this scheme given
by pulling back and letting $\psi$ be the natural equivariant structure.

\begin{lem}
 Let $C/S$ be a hyperelliptic curve of genus $1$. Then the above morphisms
 give rise to a short exact sequence $0\to \Pic(C/\langle s\rangle)\to
 \G_m.\Pic(C)^{\langle s\rangle} \to \Tw^0(C)\to 0$ of group schemes.
\end{lem}

\begin{proof}
 Since $\Pic(C/\langle s\rangle)\cong \Z$ (the free group generated by
 the isomorphism class $\sO(1)$) and the first map simply doubles degree, we~find that it is indeed injective. A pair $({\cal L},\psi)$
 induces the gerbe $(\sO_C,\psi {}^s\psi)$, so the gerbe is trivial iff
 $\psi$ makes ${\cal L}$ equivariant. We then find that $h$ is given by
 the restriction of the equivariant structure to~$C^s$, and thus is
 trivial iff ${\cal L}$ descends to~the~quotient; thus the sequence is
 exact in the middle.

 It remains to show that the second morphism is surjective. This is
 essentially a statement about fppf sheaves and thus we may feel free to
 base change so that $C/S$ is elliptic. In~that case, both groups map to~$\G_m$ by taking restrictions to the identity, and it will suffice to
 show that the remaining factors are isomorphic. We readily verify that
 the subgroup of~$\Tw^0(E)$ such that $h(0)=1$ is a finite group scheme of
 order 8 in every fiber; indeed, it may be identified with the complete
 intersection of 3 quadrics in a suitable $\P^3$. Similarly, the relevant
 quotient of~$\Pic(E)^{\langle s\rangle}$ may be identified with the
 disjoint union of two copies of~$E[2]$ (for~$C$, this is $\Pic^0(C)[2]$
 and the torsor of Weierstrass points), so is also a finite flat group
 scheme of order 8. The~image of~${\cal L}_1$ is (by definition) the
 quadratic function~${\mathfrak q}\in \mu_2(E[2])$ considered above, which
 in characteristic not 2 is $1$ at 0 and $-1$ at the nontrivial
 $2$-torsion points. The~action of~$E[2]$ by translation induces an
 action on both groups, and the homomorphism is equivariant. In~particular, we~may compute the image of~$x\in E[2](S)$ in~$\Tw^0(E)$ as
 the image of~$x^*{\cal L}_1\otimes {\cal L}_1^{-1}$. Modulo overall
 scalars, this is ${\mathfrak q}(z+x){\mathfrak q}(z)^{-1}$; applying the
 splitting gives ${\mathfrak q}(z+x){\mathfrak q}(z)^{-1}{\mathfrak
 q}(x)^{-1}$. Since this is precisely the Weil pairing of~$z$ and $x$, we~conclude that the restriction of the coboundary morphism to~$E[2]$ is
 the Weil pairing, and in particular is injective. Since the image of any
 element of this index 2 subgroup scheme is a homomorphism and ${\mathfrak
 q}$ is {\em not} a homomorphism, we~conclude that the morphism on~$\Pic^0(E)[2]\disj \Pic^1(E)^{\langle s\rangle}$ is injective, and thus
 by degree considerations is surjective.
\end{proof}

\begin{cor}
 Any twisting datum for the action of~$s$ on~$C$ is isomorphic to a
 coboundary.
\end{cor}

\begin{proof}
 Since the line bundle ${\cal Z}_s$ has norm $\sO_{C/s}$, it must have
 degree 0, and since multiplication by 2 is surjective on the group scheme
 $\Pic^0(C)$, there is an fppf covering $T\to S$ and a line bundle~${\cal
 L}$ of degree 0 on~$C_T$ such that ${\cal L}\otimes {}^s{\cal
 L}^{-1}\cong {\cal Z}_s$. If~we choose such an isomorphism, then the
 coboundary induces a twisting datum with line bundle ${\cal Z}_s$, and
 thus we may divide our original twisting datum by this new twisting datum
 to obtain a datum with trivial line bundle. By the lemma, this is the
 image of~$({\cal L}',\psi)$, where ${\cal L}'$ is an isomorphism class of
 line bundles of degree 0 or 1 and $\psi:{\cal L}'\cong {}^s{\cal L}'$.
 We thus conclude that the base change of the original twisting datum is
 isomorphic to the coboundary of~${\cal L}\otimes {\cal L}'$. The~two
 pullbacks of this bundle to~$T\times_S T$ have the same coboundary, so
 must differ by a pullback from $\sO_{C/\langle s\rangle}$, which by
 degree considerations must be trivial. Thus ${\cal L}\otimes {\cal L}'$
 corresponds to an $S$-point of the relative Picard scheme of~$C$, and the
 twisting datum is the coboundary of this point.
\end{proof}

We may extend the notion of twisting datum to more general groups by
assigning a twisting datum to each order 2 subgroup that fixes some
reflection hypersurface, and insist on the appropriate compatibility
relations. Again, any point of the relative Picard scheme has a
well-defined coboundary.

\begin{prop}
 Let $W$ be a finite Weyl group and let $X/S$ be an abelian torsor on
 which $W$ acts by reflections. If~the root kernel of~$X$ is
 trivial, then any twisting datum on~$X$ is isomorphic to a coboundary.
\end{prop}

\begin{proof}
 Let $T/S$ be an fppf cover over which $X$ has a $W$-invariant section. The~induced twisting datum on each rank~1 parabolic subgroup is the
 coboundary of a~point in the Picard scheme, but since $X_T$ has a
 section, it is in fact the coboundary of a line bundle on~$X_T$. We may
 then apply Proposition~\ref{prop:coboundary_if_in_rank_1} to express the
 induced class in~$Z^1(W;\Pic(X_T))$ as a coboundary. Moreover, we~may
 use the invariant section (which is contained in every $[X^s]$) to
 rigidify the various isomorphisms, and thus express the base changed
 twisting datum as the product of this coboundary and a twisting datum
 with trivial underlying gerbe.

 Choose representatives among the simple reflections of the conjugacy
 classes of reflections, and observe that for each such $s_i$, there is a
 unique bundle ${\cal L}_i$ on~$E_i$ of degree 0 or 1 such that the
 restriction of the latter twisting datum on~$s_i$ is the coboundary of~$\pi_i^*{\cal L}_i$. Tensoring the conjugates of those bundles gives a
 $W$-invariant (but not equivariant) line bundle with the desired twisting
 datum, and thus expresses the base change of the original twisting datum
 as the coboundary of a line bundle ${\cal L}_T$.

 Since the two pullbacks of~${\cal L}_T$ to~$T\times_S T$ have the same
 coboundary, they differ by an~equi\-variant bundle on~$X_T$ that descends in
 codimension~1. Since the polarization of a line bundle is locally
 constant, we~may arrange for the ratio to have trivial polarization, so
 be a~point of~$\Pic^0(X)(T\times_S T)\cong \prod_i E_i(T\times_S T)$. The~restriction to~$s_i$ of the coboundary of such a~point is essentially
 just the coboundary of the corresponding point of~$E_i(T\times_S T)$;
 since it must be trivial, we~conclude that the two pullbacks of~${\cal
 L}_T$ are in fact isomorphic, and thus that ${\cal L}$ descends to a~point of the relative Picard scheme of~$X/S$.
\end{proof}

\begin{rem}
 This can fail if the root kernel is nontrivial, even when~$X$ has a
 section. However, there is a flat finite cover $T\to S$ over which $X$
 can be expressed as the quotient of a torsor with trivial root kernel,
 so that we may describe a twisting datum on~$X$ as the coboundary
 of a line bundle on~$X_T$, subject to appropriate descent conditions.
\end{rem}

More generally, if $W$ is a~Coxeter group and $X/S$ is an abelian torsor
with an action of~$W$ of coroot type, then an assignment of twisting data
on the rank~1 parabolic subgroups extends to at most one twisting datum for
$W$. Indeed, we~may extend the collection of pairs $({\cal
 Z}_{s_i},\zeta_i)$ to an equivariant gerbe by choosing a reduced word for
each $w\in W$ and defining
\begin{gather*}
{\cal Z}_w :=
{\cal Z}_{s_1}\otimes {}^{s_1}{\cal Z}_{s_2}\otimes {}^{s_1s_2}{\cal
 Z}_{s_3}\otimes\cdots\otimes {}^{s_1\cdots s_{m-1}}{\cal Z}_{s_m}.
\end{gather*}
If all we had was the equivariant gerbe structure, we~would also need to
specify isomorphisms corresponding to the different braid relations, which
would themselves need to satisfy compatibility relations (coming from
finite parabolic subgroups of rank $\le 3$).\footnote{In general, one can
 define a $G$-equivariant gerbe by giving a line bundle for each generator
 of~$G$, a morphism for each relation, subject to a consistency condition
 for each $3$-cell of the classifying space $BG$. For~Coxeter groups,
 there is a model of~$BW$ with~$k$-cells corresponding to multisets of~$k$
 simple roots such that the corresponding parabolic subgroup is finite,
 and thus the $3$-cells come from finite parabolic subgroups of rank $\le
 3$.} Luckily, each braid relation may be restated as a conjugacy
relation between simple reflections, and the isomorphism corresponding to
the braid relation appears linearly in the corresponding consistency
condition on the twisting data. For~any rank $\le 3$ finite parabolic
subgroup, the different reflection hypersurfaces have a nonempty common
intersection, and thus the further compatbility conditions of~the~gerbe
will be automatically satisfied. Thus a collection of rank~1 twisting data
extends to a~full twisting datum iff it extends for every finite rank 2
subgroup.

Let $\gamma$ denote such a twisting datum, and write $k(X)[W]_\gamma$ for
$k(X)[W]_{{\cal Z}_\gamma}$, and let ${\cal H}_{W;\vec{T};\gamma}(X)$
denote the sheaf subalgebra generated by the rank~1 algebras ${\cal
 L}_i\otimes {\cal H}_{\langle s_i\rangle,T_i}(X)\otimes {\cal L}_i^{-1}$,
where ${\cal L}_i$ is any line bundle with coboundary $\gamma|_{s_i}$.
Then the usual arguments carry over from the finite case to give the
following.

\begin{prop}
 If $I$ is a finite order ideal in~$W$ and $w\in I$ is a maximal element,
 then there is a short exact sequence
 \begin{gather*}
 0\to {\cal H}_{W;\vec{T};\gamma}(X)[I\setminus \{w\}]
 \to {\cal H}_{W;\vec{T};\gamma}(X)[I]
 \to (1,w^{-1})_*
 \big({\cal Z}_{\gamma,w}\otimes \sO_X\big(D_w\big(\vec{T}\big)\big)\big)
 \to 0
 \end{gather*}
 of sheaf bimodules.
\end{prop}

\begin{cor}
 The construction of the sheaf algebra ${\cal H}_{W;\vec{T};\gamma}(X)$
 respects base change.
\end{cor}

{\sloppy\begin{cor}
 Let $\vec{T}$ be a system of parameters such that every $T_\alpha$ is
 transverse to every ref\-lection hypersurface. Then ${\cal
 H}_{W;\vec{T};\gamma}(X)$ may be identified with the sheaf subalgebra
 of~${\cal H}_{W;\gamma}(X)$ consisting locally of operators $\sum_w c_w
 w$ such that for every $w$, $c_w$ vanishes on the divisor
 $\sum_{\alpha\in\Phi^+(W)\cap w\Phi^-(W)} T_{\alpha}$.
\end{cor}

}

\begin{cor}\label{cor:infin_hecke_twisted_interval}
 Let $w\in W$ be given by the reduced word $w=s_1\cdots s_l$. Then the
 multiplication map
 ${\cal H}_{\langle s_1\rangle,\vec{T};\gamma}(X)
 \otimes_X
 \cdots
 \otimes_X
 {\cal H}_{\langle s_l\rangle,\vec{T};\gamma}(X)
 \to
 {\cal H}_{W;\vec{T};\gamma}(X)[\le w]$
 is surjective. Moreover, any product of rank~$1$ subalgebras is equal to
 some Bruhat interval.
\end{cor}

Suppose that $D_1, \dots, D_n$ are Cartier divisors such that
$Z_{s_i}=D_i-{}^{s_i}D_i$ extends to a cocycle valued in Cartier divisors,
and consider the case that $\gamma_i = \partial\sO_X(D_i)$. Since the
action of~$W$ on the group of Cartier divisors is a permutation module, its
restriction to any finite subgroup is induced from a trivial module, so has
trivial $H^1$. In~particular $Z|_{\langle s_i,s_j\rangle}$ is a coboundary
of some~$D_{ij}$ for any finite rank 2 parabolic subgroup. If~$D_i-D_{ij}$
has even valuation along any (separable) component of~$[X^{s_i}]$, and
similarly for~$D_j-D_{ij}$, then $\gamma_i=\partial\sO_X(D_{ij})$ and
similarly for $\gamma_j$, and thus we have a compatible extension of
twisting data. Note that since $D_{ij}$ is only determined up to~$\langle
s_i,s_j\rangle$-invariant divisors, we~can change its parity along each
orbit of~$\langle s_i,s_j\rangle$-reflection hypersurfaces independently,
and thus if $s_i$ and $s_j$ are not conjugate, this condition can always be
satisfied, and otherwise reduces to a simple parity constraint.

Given any other twisting datum $\gamma$, let $\gamma\big(\vec{D}\big)$ denote the
twisting datum obtained by tensoring with the above twisting datum. Since
$Z_{s_i}$ extends to a cocycle valued in Cartier divisors, the
resulting equivariant gerbe comes with a natural meromorphic equivalence to
the original equivariant gerbe, and thus we have an induced isomorphism
\smash{$k(X)[W]_{\gamma}\cong k(X)[W]_{\gamma(\vec{D})}$} for any~$\gamma$,
and may in this way view \smash{${\cal H}_{W;\vec{T};\gamma(\vec{D})}$} as
a subalgebra of~$k(X)[W]_{\gamma}$.

\looseness=-1 To understand such isomorphisms more generally, we~will need to understand
cocycles valued in Cartier divisors. The~fact that $\Hom(W,\Z)=0$ implies
that any {\em coinduced} module for $W$ has trivial $H^1$. Since Cartier
divisors are a sum of {\em induced} modules, there can be (and are)
cocycles valued in Cartier divisors which are not coboundaries. However,
since the induced modules are contained in the corresponding coinduced
modules, we~can always express such a cocycle as a~coboundary in the larger
module (of integer-valued functions on the set of irreducible Cartier
divisors). Note that since the typical element of a coinduced module will
not have coboundary in the induced submodule, we~need to add the condition
that any element of~$w$ only changes finitely many values of the function;
naturally, it suffices to verify the condition for the simple reflections.

For instance, if we interpret $\sum_{\alpha\in \Phi^+(W)}T_\alpha$
as giving an integer-valued function on irreducible Cartier divisors (i.e.,
the sum over $\alpha\in\Phi^+(W)$ of the valuation of~$T_\alpha$ along the
given divisor), then any element of~$W$ only changes finitely many values
of the function, and thus we obtain a~well-defined coboundary
$Z_w=\sum_{\alpha\in \Phi^+(W)\cap w\Phi^-(W)} (T_\alpha-T_{-\alpha})$.

We may also use such formal sums to define (meromorphically trivial)
twisting data; if $\Gamma$ is an integer-valued function on irreducible
Cartier divisors such that $\Gamma-{}^{s_i}\Gamma$ has finite support, then
we may obtain a divisor $D_i$ with the same coboundary on~$\langle
s_i\rangle$ by restricting $\Gamma$ to the union of the support of~$\Gamma-{}^{s_i}\Gamma$ and the components of the reflection hyperplanes.
Similarly, if $\langle s_i,s_j\rangle$ is finite, then we may obtain a
divisor $D_{ij}$ by restricting $\Gamma$ to the union of the supports of~$\Gamma-{}^w\Gamma$ for~$w\in \langle s_i,s_j\rangle$, and find that
$D_{ij}-D_i$ and $D_{ij}-D_j$ are pullbacks, so that
$\gamma_i=\partial\sO_X(D_i)$ gives a well-defined twisting datum. We
denote the twist of some other $\gamma$ by this meromorphically trivial
datum by $\gamma(\Gamma)$. (More precisely, a twisting datum is determined
by $\Gamma$ along with a choice of representation of each
$\Gamma-{}^{s_i}\Gamma$ as a coboundary; the above convention can behave
badly in families, but there is always a consistent way to take a limit of
the choices of representations as coboundaries in rank~1.)

We then introduce the notation
\begin{gather*}
\sO_X(\Gamma)\otimes {\cal H}_{W;\vec{T};\gamma}(X)\otimes \sO_X(-\Gamma)
\end{gather*}
for ${\cal H}_{W;\vec{T};\gamma(\Gamma)}(X)$ viewed as a subalgebra of~$k(X)[W]_\gamma$. Note that if $\Gamma'-\Gamma$ has finite support,
then
\begin{gather*}
\begin{split}
& \sO_X(\Gamma')\otimes {\cal H}_{W;\vec{T};\gamma}(X)\otimes \sO_X(-\Gamma')
\\ & \hphantom{\sO_X(\Gamma')}
\cong\sO_X(\Gamma'-\Gamma)\otimes (\sO_X(\Gamma)
\otimes {\cal
 H}_{W;\vec{T};\gamma}(X)\otimes \sO_X(-\Gamma))\otimes \sO_X(\Gamma-\Gamma'),\end{split}
\end{gather*}
where the outer twist on the right hand side is the usual twist by a line
bundle. We may also define a sheaf $\sO_X(\Gamma')\otimes {\cal
 H}_{W;\vec{T};\gamma}(X)\otimes \sO_X(-\Gamma)$ in this case by
$\sO_X(\Gamma'-\Gamma)\otimes (\sO_X(\Gamma)\otimes {\cal
 H}_{W;\vec{T};\gamma}(X)\otimes \sO_X(-\Gamma))$.

\begin{prop}\label{prop:split_system_of_parameters}
 Let $\vec{T}$, $\vec{T}'$ be two systems of parameters for~$W$ on~$X$.
 Then
 \begin{gather*}
 {\cal H}_{W;\vec{T}+\vec{T}';\gamma}(X)
 =
 \sO_X\bigg({-}\sum_{\alpha\in \Phi^+(W)} T'_\alpha \bigg)\otimes
 {\cal H}_{W;\vec{T}+{}^{-}\vec{T}';\gamma}(X)
 \otimes \sO_X\bigg(\sum_{\alpha\in \Phi^+(W)} T'_\alpha \bigg)
 \end{gather*}
 as subalgebras of~$k(X)[W]_\gamma$.
\end{prop}

\begin{proof}
 This reduces immediately to the corresponding claim in the rank~1 case,
 where (after twisting by a line bundle to make $\gamma$ trivial) it reads
 \begin{gather*}
 {\cal H}_{A_1,T+T'}(C) =
 \sO_C(-T')\otimes {\cal H}_{A_1,T+{}^sT'}(C)\otimes \sO_C(T').
 \end{gather*}
 For general parameters (such that no two of~$T$, ${}^sT$, $T'$, ${}^sT'$
 have a common component), this is straightforward: it is easy to see that
 that ${\cal H}_{A_1,T+T'}(C)$ preserves the subsheaf $\sO_C(-T')$, and~${\cal H}_{A_1,T+{}^sT'}(C)$ preserves the subsheaf $\sO_C(-{}^sT')$, and
 this gives both inclusions.
\end{proof}

The proof of Proposition~\ref{prop:finite_hecke_as_intersection}
carries over to give the following.

\begin{prop}\label{prop:infinite_hecke_as_intersection}
 Suppose that $T_\alpha$ and $T_{-\alpha}$ have no common component for
 any $\alpha\in \Phi(W)$. Then
 \begin{gather*}
 {\cal H}_{W;\vec{T};\gamma}(X) = {\cal H}_{W;\gamma}(X) \cap
 \sO_X\bigg({-}\sum_{\alpha\in \Phi^+(W)} T_\alpha\bigg)\otimes
 {\cal H}_{W;\gamma}(X)
 \otimes \sO_X\bigg( \sum_{\alpha\in \Phi^+(W)} T_\alpha\bigg).
 \end{gather*}
\end{prop}

We also note the following fact, which allows us to decouple the conditions
associated to different parameters.

\begin{prop}\label{prop:decoupling_of_parameters}
 Suppose that $\vec{T}$ and $\vec{T}'$ are such that $T_\alpha$ and
 $T'_\alpha$ have no common component for any $\alpha$. Then
 \begin{gather*}
 {\cal H}_{W;\vec{T}+\vec{T}';\gamma}(X)
 =
 {\cal H}_{W;\vec{T};\gamma}(X)
 \cap
 {\cal H}_{W;\vec{T}';\gamma}(X).
 \end{gather*}
\end{prop}

\begin{proof}
 The rank~1 subalgebras on the left are contained in the corresponding
 subalgebras on the right, so algebra on the left is certainly contained
 in the intersection on the right. To~see equality, we~use the Bruhat
 filtration and observe that each subquotient on the left is the
 intersection of the corresponding subquotients on the right.
\end{proof}

\begin{rem}
 This easily gives a version of Proposition~\ref{prop:infinite_hecke_as_intersection} in which ${\cal
 H}_{W;\vec{T}+\vec{T}';\gamma}(X)$ is given as an intersection of two
 twists of~${\cal H}_{W;\vec{T};\gamma}(X)$.
\end{rem}

The construction of the adjoint in the finite case carries over. Note that
the na\"{\i}ve adjoint $\sum_w c_w w\mapsto \sum_w w c_w$ induces a natural
isomorphism $k(X)[W]_\gamma^{\rm op}\cong k(X)[W]_{\gamma^{-1}}$. (In
terms of the sheaf algebra itself, all we are doing is swapping the two
factors of~$X\times_S X$.) To describe how this acts on the Hecke
algebras, it will be helpful to denote the formal sum $\sum_{\alpha\in
 \Phi^+(W)}\big([X^{r_\alpha}]-\vec{T}\big)$ by \smash{$D_{w_0}\big(\vec{T}\big)$}, and similarly
for~$D_{w_0}$. This of course agrees with the usual notation whenever the
longest element $w_0\in W$ actually exists.

\begin{prop}
 The na\"{\i}ve adjoint on~$k(X)[W]_\gamma$ induces an identity
 \begin{gather*}
 {\cal H}_{W;\vec{T};\gamma}(X)^{\rm op} =
 \sO_X\big(D_{w_0}\big(\vec{T}\big)\big)\otimes {\cal H}_{W;\vec{T};\gamma^{-1}}(X)
 \otimes \sO_X\big({-}D_{w_0}\big(\vec{T}\big)\big)
 \\ \hphantom{{\cal H}_{W;\vec{T};\gamma}(X)^{\rm op}}
{} = \sO_X(D_{w_0})\otimes {\cal H}_{W;{}^-\vec{T};\gamma^{-1}}(X)
 \otimes \sO_X(-D_{w_0}).
 \end{gather*}
 of subalgebras of~$k(X)[W]_{\gamma^{-1}}$.
\end{prop}

\begin{proof}
 Again, this reduces immediately to the rank~1 case.
\end{proof}

Diagram automorphisms of course work as well in the infinite case; the only
caveat is that unlike in the finite case, a diagram automorphism can fail
to preserve the parameters and twisting datum. More generally, if $H$ is a
group of automorphisms of~$X$ acting as diagram automorphisms of~$W$ and
preserving the parameters, and there is an $H$-equivariant gerbe ${\cal
 Z}_h$ such that ${}^h\gamma_i\sim \gamma_i\otimes \partial{\cal Z}_h$ for
each $i$, then the corresponding holomorphic crossed product algebra
normalizes the Hecke algebra, and we can combine them into a larger algebra
associated to the extended Coxeter Group $W\rtimes H$. (In the $C_n$ case
we consider in detail below, we~will see that even the requirement that the
parameters be invariant can be finessed.)

Suppose $A$ and $B$ are sheaf algebras, on~$X/S$ and $Y/S$ respectively.
An $(A,B)$-bimodule is then simply a sheaf bimodule $M$ on~$X\times_S Y$
equipped with multiplication maps $A\otimes_X M\to M$, $M\otimes_Y B\to M$
making the obvious diagrams commute. (Note that the restriction of~$M$ to
a~compatible pair of localizations is a bimodule over the corresponding
restrictions of~$A$ and $B$.) The tensor product is then defined in the
obvious way, so that we may define induced modules. Restriction is of
course also easy to define, though the sheaf form of Frobenius reciprocity
is somewhat tricky, as~there are difficulties with defining $\sHom$ on
sheaf bimodules in general. (The~difficulty is that the category of sheaf
bimodules is cocomplete, but not complete, and $\sHom$ from a direct limit
is an inverse limit. Thus the $\sHom$ of sheaf bimodules will still be a~quasicoherent sheaf on the relevant fiber product scheme, but may fail to
satisfy the finiteness requirement.)

This is not a problem for the analogue of Proposition~\ref{prop:Mackey_for_Hecke}; the only change is that $M$ should be replaced
by a suitable bimodule. In~the finite case, this is no difficulty: when~$W$ is finite, any~${\cal H}_{W;\vec{T}}(X)$-module in the usual sense
determines a corresponding $\big({\cal H}_{W;\vec{T}}(X),\sO_{X/W}\big)$-bimodule
structure.

{\sloppy\begin{prop}
 Suppose $I,J\subset S$ are such that the parabolic subgroups $W_I$, $W_J$
 are finite. Then for any $\big({\cal H}_{W_J;\vec{T};\gamma}(X),\sO_Y\big)$-bimodule $M$ and any maximal chain
 in the Bruhat order on~${}^IW^J$, the subquotient corresponding
 to~$w\in {}^IW^J$ in the resulting filtration of~$\Res^{W;\vec{T};\gamma}_{W_I}\Ind^{W;\vec{T};\gamma}_{W_J} M$ is the
 $\big({\cal H}_{W_I;\vec{T};\gamma}(X),\sO_Y\big)$-bimodule
 \begin{gather*}
 \Ind^{W_I;\vec{T};\gamma}_{W_{I(w)}} {\cal Z}_{\gamma,w}\big(D_w\big(\vec{T}\big)\big)
 \otimes w \Res^{W_J;\vec{T};\gamma}_{W_{J(w)}} M.
 \end{gather*}
\end{prop}

}

We also have a weaker form of Frobenius reciprocity.

\begin{lem}
 For any $Y$, induction and restriction are adjoint functors between the
 categories of~$\big({\cal H}_{W;\vec{T};\gamma},\sO_Y\big)$-bimodules and $\big({\cal
 H}_{W_I;\vec{T};\gamma},\sO_Y\big)$-bimodules.
\end{lem}

\begin{proof}
 Since both functors are constructed as tensor products, we~see that it
 suffices to construct compatible morphisms
 \begin{gather*}
 {\cal H}_{W;\vec{T};\gamma}\otimes_{{\cal H}_{W_I;\vec{T};\gamma}}
 {\cal H}_{W;\vec{T};\gamma}
 \to
 {\cal H}_{W;\vec{T};\gamma}
 \end{gather*}
 and
 \begin{gather*}
 {\cal H}_{W_I;\vec{T};\gamma}
 \to
 \Res_{W_I}^{W;\vec{T};\gamma}{\cal H}_{W;\vec{T};\gamma},
 \end{gather*}
 both of which (along with compatibility) follow directly from the fact
 that ${\cal H}_{W_I;\vec{T};\gamma}$ is a subalgebra of~${\cal
 H}_{W;\vec{T};\gamma}$.
\end{proof}

\begin{cor}
 Let $M$ be a coherent $\big({\cal H}_{W_I;\vec{T};\gamma},\sO_Y\big)$-bimodule
 and $N$ an $\big({\cal H}_{W;\vec{T};\gamma},\sO_Z\big)$-bimo\-dule. Then the
 quasicoherent sheaf $\sHom_{{\cal
 H}_{W;\vec{T};\gamma}}\big(\Ind_{W_I}^{W;\vec{T};\gamma} M,N\big)$ on~$Y\times_S Z$ is a sheaf bimodule.
\end{cor}

\begin{proof}
 Frobenius reciprocity gives (when~$Y$ and $Z$ are affine, which we may
 certainly reduce~to)
 \begin{gather*}
 \sHom_{{\cal H}_{W;\vec{T};\gamma}}\big(\Ind_{W_I}^{W;\vec{T};\gamma} M,N\big)
 \cong \sHom_{{\cal H}_{W_I;\vec{T};\gamma}}\big(M,\Res_{W_I}^{W;\vec{T};\gamma} N\big),
 \end{gather*}
 and the latter is a sheaf bimodule since $M$ is coherent.
\end{proof}

\begin{cor}
 Let $M_I$ be a coherent $\big({\cal H}_{W_I;\vec{T};\gamma},Y_I\big)$-bimodule,
 $M_J$ a coherent $\big({\cal H}_{W_J;\vec{T};\gamma},Y_J\big)$-bimodule, and
 $M$ an $\big({\cal H}_{W;\vec{T};\gamma},Y\big)$-bimodule. Then there is a
 natural composition morphism
 \begin{gather*}
 \sHom_{{\cal H}_{W;\vec{T};\gamma}}\big( \Ind_{W_I}^{W;\vec{T};\gamma} M_I,
 \Ind_{W_J}^{W;\vec{T};\gamma} M_J\big) \otimes_{Y_J}{}
 \sHom_{{\cal H}_{W;\vec{T};\gamma}}\big( \Ind_{W_J}^{W;\vec{T};\gamma} M_J, M\big)
 \\ \hphantom{\sHom_{{\cal H}_{W;\vec{T};\gamma}}}
 \to \sHom_{{\cal H}_{W;\vec{T};\gamma}}\big( \Ind_{W_I}^{W;\vec{T};\gamma} M_I, M\big)
 \end{gather*}
 of~$(Y_I,Y)$-bimodules, satisfying the obvious associativity relation.
\end{cor}

In particular, given an $\big({\cal H}_{W;\vec{T};\gamma},Y\big)$-bimo\-dule~$M$, we~may again define an $(X/W_I,Y)$-bimo\-dule $M^{W_I}$ as
\begin{gather*}
\sHom_{{\cal H}_{W;\vec{T};\gamma}}\big(\Ind_{W_I}^{W;\vec{T};\gamma} \sO_X,M\big)
\cong \sHom_{{\cal H}_{W_I;\vec{T};\gamma}}\big(\sO_X,\Res_{W_I}^{W;\vec{T};\gamma}M\big).
\end{gather*}
This, of course, is essentially just the extension of~$\big(\Res_{W_I}^{W;\vec{T};\gamma}M\big)^{W_I}$
to bimodules in the obvious way.

Now that we have reasonable definitions, most of the calculations we did in
the finite case carry over. We find that (assuming $\gamma$ is trivial on~$W_I$ and $W_J$) the submodule of~$\Res_{W_I}^{W;\vec{T};\gamma}\Ind^{W;\vec{T};\gamma}_{W_J}\sO_X$
corresponding to any finite Bruhat order ideal has strongly flat invariants
for~$W_I$, and thus
\begin{gather*}
 {\cal H}_{W,W_J,W_I;\vec{T};\gamma}(X):=
 \big(\Ind^{W;\vec{T};\gamma}_{W_J}\sO_X\big)^{W_I}
\end{gather*}
is an $S$-flat sheaf bimodule on~$X/W_I\times_S X/W_J$, and this
construction commutes with base change. The~subquotients in the
corresponding Bruhat filtration may all be described in the following way. For~each $w\in {}^IW^J$, there is a corresponding line bundle ${\cal L}_w$
on~$X/W_{I(w)}$ (constructed from $\vec{T}$ and $\gamma$) such that the
subquotient is the direct image in~$X/W_I\times_S X/W_J$ of the
$(X/W_{I(w)},X/W_{J(w)})$-bimodule $(1,w^{-1})_*{\cal L}_w$. More
precisely, the line bundle ${\cal L}_w$ is the descent to~$X/W_{I(w)}$ of
the ($W_{I(w)}$-equivariant!) line bundle
${\cal Z}_{\gamma,w}\big(D_w\big(\vec{T}\big)\big) \otimes
\sO_X\big(D_{w_I}\big({}^-\vec{T}\big)-D_{w_{I(w)}}\big({}^-\vec{T}\big)\big)$.

If $\Gamma_I$, $\Gamma_J$ are $W_I$, $W_J$-invariant functions which are
even on reflection hypersurfaces and have finitely supported difference,
then for any twisting datum $\gamma$ which is trivial on~$W_I$ and $W_J$,
\begin{gather*}
\sO_X(\Gamma_J)\otimes {\cal H}_{W;\vec{T};\gamma}\otimes \sO_X(-\Gamma_I)
\end{gather*}
becomes a left ${\cal H}_{W_J;\vec{T}}$-module and a right ${\cal
 H}_{W_I;\vec{T}}$-module, and thus has a corresponding spherical module
which we may denote by
\begin{gather*}
\sO_X(\Gamma_J)\otimes {\cal H}_{W,W_I,W_J;\vec{T};\gamma}\otimes
\sO_X(-\Gamma_I).
\end{gather*}

The adjoint takes the following form.

\begin{prop}
 If the root kernels of~$W_I$ and $W_J$ on~$X$ are diagonalizable and the
 twisting datum $\gamma$ is trivial on~$W_I$ and $W_J$, then there is an
 isomorphism
 \begin{gather*}
 {\cal H}_{W,W_I,W_J;\gamma}(X)
 \cong
 \sO_X(D_{w_0}-D_{w_I})
 \otimes
 {\cal H}_{W,W_J,W_I;\gamma^{-1}}(X)
 \otimes
 \sO_X(D_{w_J}-D_{w_0})
 \end{gather*}
 which is contravariant with respect to composition.
\end{prop}

\begin{proof}
 We may view the element $\sum_w w$ as a section of~${\cal H}_{W_I}(X)\otimes \sO_X(-D_{w_I})$, since it takes sections of~$\sO_X(D_{w_I})$ to sections of~$\sO_X$. We thus have an embedding
 \begin{gather*}
 {\cal H}_{W,W_I;\gamma}(X)\to {\cal H}_{W;\gamma}(X)\otimes \sO_X(-D_{w_I})
 \end{gather*}
 acting as
 \begin{gather*}
 \sum_{w\in W^I} c_{w W_I} w w_I\mapsto \sum_{w\in W} c_{w W_I} w.
 \end{gather*}
 Moreover, if the original operator is a local section of~${\cal
 H}_{W,W_I,W_J;\gamma}(X)$, then we have
 \begin{gather*}
 c_{w' w W_I}={}^{w'}c_{w W_I}
 \end{gather*}
 for $w'\in W_J$, so that we may write the image as
 \begin{gather*}
 \sum_{w'\in W_J} w' \sum_{w\in {}^JW} c_{W_J w W_I} w.
 \end{gather*}
 Taking the adjoint (including the twist by $\sO_X(D_{w_0})$) gives
 \begin{gather*}
 \sum_{w'\in W_J} \sum_{w\in {}^JW} (-1)^{\ell(w w')}
 w^{-1} c_{W_J w W_I} w^{\prime{-}1}
 =
 \sum_{w\in W^J} \sum_{w'\in W_J} (-1)^{\ell(w w')} w c_{W_J w^{-1} W_I} w'
 \end{gather*}
 in~$\sO_X(-D_{w_I})\otimes {\cal H}_{W;\gamma^{-1}}(X)$.
 Right dividing by $\sum_{w'\in W_J} (-1)^{\ell(w')} w'$
 gives a section of
 \begin{gather*}
 \sO_X(-D_{w_I}) \otimes {\cal H}_{W,W_J,W_I;\gamma^{-1}}(X) \otimes \sO_X(D_{w_J})
 \end{gather*}
 as required. Compatibility with composition follows by observing that in
 a composition, the factor $\sum_{w\in W_I} w$ needed in the middle is
 already present in the other operator, and the factor $\sum_{w\in W_J}
 (-1)^{\ell(w)} w$ that should be removed is needed in the other operator.
\end{proof}

\begin{prop}
 If the root kernels of~$W_I$ and $W_J$ on~$X$ are diagonalizable, then
 \begin{gather*}
 {\cal H}_{W,W_I,W_J;\vec{T};\gamma}(X)\subset{\cal H}_{W,W_I,W_J;\gamma}(X)
 \\ \hphantom{{\cal H}_{W,W_I,W_J;\vec{T};\gamma}(X)}
 \cap \sO_X\bigg(\sum_{\alpha\in \Phi^-(W)\setminus \Phi^-(W_J)}\!\!\!\!\!\!T_\alpha\bigg)
 \otimes {\cal H}_{W,W_I,W_J;\gamma}(X) \otimes
 \sO_X\bigg({-}\!\!\!\sum_{\alpha\in \Phi^-(W)\setminus \Phi^-(W_I)}\!\!\!\!\!\!T_\alpha\bigg),
 \end{gather*}
 with equality unless there is a root $\alpha$ such that $T_\alpha$ and
 $T_{-\alpha}$ have a common component.
\end{prop}

\begin{proof}
 Over the locus of~$S$ covered by symmetric idempotents, we~may use those
 idempotents to locally identify ${\cal H}_{W,W_I,W_J;\vec{T};\gamma}(X)$
 with a submodule of~${\cal H}_{W,W_I,W_J;\gamma}(X)$ (using the same
 idempotent to embed both in~${\cal H}_{W;\gamma}(X)$). This
 identification is compatible with the identification of meromorphic
 fibers, so extends to a global identification on each fiber covered by
 symmetric idempotents, and from there to the closure of the symmetric
 idempotent locus.

 Similarly, local idempotents embed ${\cal H}_{W,W_I,W_J;\vec{T};\gamma}(X)$ in
 \begin{gather*}
 \sO_X\bigg(\sum_{\alpha\in \Phi^-(W)} T_\alpha\bigg)
 \otimes {\cal H}_{W;\gamma}(X) \otimes
 \sO_X\bigg({-}\sum_{\alpha\in \Phi^-(W)} T_\alpha\bigg),
 \end{gather*}
 and the idempotents eliminate the contributions of~$T_\alpha$ for~$\alpha\in W_I$, $W_J$ respectively.

 To see that the inclusion is tight, we~need merely verify that both sides
 have the same Bruhat subquotients, which reduces to verifying that the
 negative part of
 \begin{gather*}
 \sum_{\alpha\in\Phi^+(W)\cap w\Phi^-(W)} (T_{-\alpha}-T_\alpha)
 +
 \sum_{\alpha\in\Phi^-(W_J)\setminus \Phi^-(W_J\cap w W_I w^{-1})}
 (T_{w\alpha}-T_\alpha)
 \label{eq:weird_divisor}
 \end{gather*}
 has no further cancellation, just as in the finite case.
\end{proof}

{\sloppy\begin{cor}
 Let $\vec{T}$ be a system of parameters such that every $T_\alpha$ is
 transverse to eve\-ry ref\-lec\-tion hypersurface. If~the root kernels of~$W_I$ and $W_J$ are diagonalizable, then ${\cal
 H}_{W,W_I,W_J;\vec{T};\gamma}(X)$ may be identified with the sheaf
 sub-bimodule of~${\cal H}_{W,W_I,W_J;\gamma}(X)$ consisting of operators
 $\sum_w c_w w W_I$ such that for every $w$, $c_w$ vanishes on the divisor
 \begin{gather*}
 \sum_{\alpha\in\Phi^+(W)\cap w\Phi^-(W)} T_{\alpha} +
 \sum_{\alpha\in\Phi^-(W_J)\setminus \Phi^-(W_J\cap w W_I w^{-1})} T_{\alpha}.
 \end{gather*}
\end{cor}

}

\begin{cor}\label{cor:spherical_t_symmetry}
 If the root kernels of~$W_I$ and $W_J$ on~$X$ are diagonalizable, then
 there is an isomorphism
 \begin{gather*}
 {\cal H}_{W,W_I,W_J;\vec{T};\gamma}(X) \cong \sO_X\big(D_{w_0}\big({}^-\vec{T}\big)-D_{w_I}\big({}^-\vec{T}\big)\big)
 \otimes {\cal H}_{W,W_J,W_I;\vec{T};\gamma^{-1}}(X)
 \\ \hphantom{{\cal H}_{W,W_I,W_J;\vec{T};\gamma}(X) \cong}
 \otimes \sO_X\big(D_{w_J}\big({}^-\vec{T}\big)-D_{w_0}\big({}^-\vec{T}\big)\big)
 \end{gather*}
 which is contravariant with respect to composition.
\end{cor}

\begin{rem}
 Of course, we~also have analogous results for the other three natural
 $\Hom$ sheaves discussed above.
\end{rem}

We close by considering the analogue in this setting of the residue
conditions of~\cite{GinzburgV/KapranovM/VasserotE:1997}. It suffices to
consider the algebras ${\cal H}_{W,W_I,W_J;\gamma}(X)$, since $\vec{T}$
simply imposes generic vanishing conditions on the coefficients, as~already
discussed. (And, of course, this includes the Hecke algebras themselves by
taking $I=J=\varnothing$.) We may further assume $J=\varnothing$, as~${\cal
 H}_{W,W_I,W_J;\gamma}(X)$ is the submodule of~${\cal
 H}_{W,W_I;\gamma}(X)$ consisting of~$W_J$-invariant operators.

Since we may embed ${\cal H}_{W,W_I;\gamma}(X)$ in~${\cal H}_{W;\gamma}(X)$
using a symmetric idempotent, the fact that the latter is spanned by
submodules
\begin{gather*}
 {\cal H}_{\langle r\rangle;\gamma}(X) \sO_X[W]_\gamma
\end{gather*}
implies something similar for the former: it is spanned by the
submodules
\begin{gather*}
 {\cal H}_{\langle r\rangle;\gamma}(X) \bigoplus_{w\in W^I} {\cal Z}_w w W_I.
\end{gather*}
If $rw W_I = w W_I$, then we may rewrite the corresponding summand as
\begin{gather*}
{\cal Z}_w w {\cal H}_{\langle w^{-1} r w\rangle;\gamma}(X) W_I
=
{\cal Z}_w w W_I,
\end{gather*}
and thus we may omit any such summand.

If we restrict to a finite subset $S\subset W/W_I$, then we conclude that
\begin{gather*}
c_{w W_I}\in {\cal Z}_w\bigg(
\sum_{r\in R(W)\colon rw W_I\in S\setminus \{w W_I\}} [X^r]\bigg)
\end{gather*}
and that there is a residue condition relating $c_{w W_I}$ and $c_{r w
 W_I}$ along $[X^r]$. Again moving ${\cal Z}_w w$ to the left lets us
express this condition in the form
\begin{gather*}
c_w + {}^w h_{w^{-1}rw} \zeta_{w,w^{-1}rw}^{-1} c_{rw} = 0\in {\cal
 Z}_w([X^r])|_{[X^r]},
\end{gather*}
where $h_{w^{-1}r w}$ comes from the root datum on~$\langle w^{-1}r
w\rangle$. Note that we could also determine this by right dividing by
$f_w w W_I$, where $f_w$ trivializes ${\cal Z}_w$ in a neighborhood of~$X^r$
to obtain the condition in the form
\begin{gather*}
c_w + f_w {}^r f_w^{-1} h_r \zeta_{r,w}^{-1} c_{rw}=0\in {\cal
 Z}_w([X^r])|_{[X^r]}.
\end{gather*}
Of course, these are equivalent (since otherwise $\gamma$ would violate the
compatibility conditions).

Naturally, in the untwisted case, the condition is just that $c_w+c_{rw}$
is holomorphic (essentially Corollary~\ref{cor:local_order_two_residue_conditions}), and the same holds (along
$[X^r])$ if we embed ${\cal H}_{W,W_I;\gamma}(X)$ in~$k(X)[W/W_I]$ via an
expression of the twisting datum as the coboundary of a formal divisor
which is transverse to the reflection hypersurfaces.

\section{The (double) affine case}\label{section6}

The most interesting case for our purposes is when the Coxeter group is an
affine Weyl group~$\tW$. We actually want to modify the construction
(very) slightly in that case, as~the abelian variety being acted on is
slightly larger than we would like. That is, rather than have an
$(n+1)$-dimensional variety with an invariant map to an elliptic curve, we~would prefer to act on the fibers of that map. If~we pull back the sheaf
bimodule ${\cal H}_{\tW;\vec{T};\gamma}(X)$ from $X\times_S X$ to~$X\times_{X/A_{\tW}}X$, then we find (by considering what happens on
invariant localizations, say) that the result is still naturally a sheaf
algebra. The~group no longer acts faithfully on every fiber, but the
various calculations involving the Bruhat filtration carry over without
difficulty, so that we still obtain a flat family of sheaf algebras
generated by the rank~1 subalgebras. One caveat is that $T_\alpha$ and~$T_\beta$ need not be transverse for~$\alpha\ne\pm \beta$; if they
correspond to the same root of the finite root system, then their divisors
differ only by a translation, which may act trivially on some fibers.

Still, we~have the following definition. First, if $X$ is a torsor over
the abelian scheme $A$, an action of~$\tW$ on~$X$ by affine reflections is
simply an action such that every simple reflection fixes a hypersurface and
the action on~$A$ factors through a faithful action of the corresponding
finite Weyl group. Any such action arises from an action of coroot type by
specializing the parameter $q$. (Recall that $q=\zeta(z)$ gives the image
of the origin of a fiber under the special reflection~$s_0$.) In addition,
every finite parabolic subgroup still acts faithfully (regardless of~$q$),
and thus in particular we still have good notions of systems of parameters
and twisting data. The~one caution is that when expressing twisting data
as a coboundary of some formal sum of divisors, one needs to assume $q$
non-torsion. This is not truly an issue, however, as~one can simply take
the limit of the twisting data from the non-torsion case (which simply adds
an extra level of formality to the formal sum). In~particular, the
cocycles in Cartier divisors associated to the formal sums $\sum_{\alpha\in
 \Phi^+(W)} T_\alpha$ and $\sum_{\alpha\in \Phi^+(W)} [X^{r_\alpha}]$
remain cocycles after specializing $q$.

\begin{defn}
 Let $\tW$ act on~$X$ by affine reflections, let $\vec{T}$ be a
 system of parameters, and let~$\gamma$ be a twisting datum.
 The~corresponding {\em elliptic double affine Hecke algebra}
 ${\cal H}_{\tW;\vec{T};\gamma}(X)$ is the sheaf subalgebra of~$k(X)\big[\tW\big]_\gamma$ generated by the rank~1 subalgebras ${\cal
 H}_{\langle s\rangle;\vec{T};\gamma}(X)$ for $s\in S$.
\end{defn}

\begin{prop}
 The subsheaf of~${\cal H}_{\tW;\vec{T};\gamma}(X)$ corresponding to
 any finite Bruhat order ideal is an $S$-flat coherent sheaf on~$X\times_S
 X$.
\end{prop}

\begin{prop}
 We have
 \begin{gather*}
 {\cal H}_{\tW;\vec{T};\gamma}(X) \subset {\cal H}_{\tW;\gamma}(X) \cap
 \sO_X\bigg({-}\sum_{\alpha\in \Phi^+(\tW)} \!T_\alpha\bigg)
 \otimes {\cal H}_{\tW;\gamma}(X) \otimes
 \sO_X\bigg(\sum_{\alpha\in \Phi^+(\tW)} \! T_\alpha\bigg),
 \end{gather*}
 with equality unless there are $\alpha\in \Phi^+\big(\tW\big)$, $\beta\in
 \Phi^-\big(\tW\big)$ such that $T_\alpha$ and $T_\beta$ have a common component.
\end{prop}

\begin{rem}
 When $q$ is torsion, then in fact each $\alpha$ has infinitely many
 $\beta$ (both positive and negative) such that $T_\alpha=T_\beta$, and
 thus the hypothesis for equality is never satisfied.
\end{rem}

{\sloppy\begin{cor}\label{cor:affine_vanishing_conditions}
 Let $\vec{T}$ be a system of parameters such that every $T_\alpha$ is
 transverse to every ref\-lec\-tion hypersurface. Then ${\cal
 H}_{\tW;\vec{T};\gamma}(X)$ may be identified with the sheaf
 subalgebra of~${\cal H}_{\tW;\gamma}(X)$ consisting locally of
 operators $\sum_w c_w w$ such that for every $w$, $c_w$ vanishes on the
 divisor $\sum_{\alpha\in\Phi^+(W)\cap w\Phi^-(W)} T_{\alpha}$.
\end{cor}

}

It will be helpful to understand the possible twisting data in this
scenario. If~we assume that $W$ has trivial root kernel, then we may
trivialize the twisting datum along $W$, and it thus remains to determine
the possibilities for the restriction to~$s_0$. If~$\tW\ne
\tilde{A}_n$, then the affine diagram is a tree with~$s_0$ as a leaf, and
thus $s_0$ commutes with a rank $n-1$ parabolic subgroup of~$W$. The~polarization of~${\cal Z}_{s_0}$ must be invariant under that subgroup, and
since it is also negated by $s_0$, we~find that ${\cal Z}_{s_0}$ must have
trivial polarization. To~fully specify the twisting datum, we~need to
choose a solution of~${\cal L}\otimes {}^{s_0}{\cal L}^{-1}={\cal
 Z}_{s_0}$. If~$\tW\ne \tilde{C}_n$, so that $s_0$ is connected to
the finite diagram via an ordinary edge, then we find that the polarization
of~${\cal L}$ is the sum of a~polarization on the connected kernel of the
coroot map and a polarization in the image of~$1+s_0$; it thus follows that
we may replace ${\cal L}$ itself by a line bundle with trivial polarization
without affecting the root datum.

In other words, for~$\tW$ not of type $A$ or $C$, any twisting datum
is (up to an overall twist by a line bundle) given by taking the twisting
datum along $s_0$ to be the coboundary of a~point in~$\Pic^0(X)$. For~type
$C$, there is at most one additional component which may be reached if it
exists by taking the coboundary of the pullback of a degree $1$ line bundle
on the coroot curve. For~type $A$, the situation is more complicated, as~it turns out that there are in fact twisting data with nontrivial
polarizations. Luckily, these can always be described via cocycles in
Cartier divisors. Type $A_n$ corresponds to the action of~$S_{n+1}$ on the
sum zero subvariety of~$E^{n+1}$, with~$s_0$ acting as
$(z_1,z_{n+1})\mapsto (q+z_{n+1},z_1-q)$. For~any point $u\in E$, we~may
consider the formal sum
\begin{gather*}
\sum_{j\ge 0} \sum_{1\le i\le n+1} [z_i=u+jq].
\end{gather*}
This is invariant under the finite Weyl group, while its coboundary under
$s_0$ is $[z_1=u]-[z_{n+1}=u+q]$. This gives a well-defined twisting datum
for generic $u$, $q$, and thus extends to all $u$, $q$. As~$u$ varies,
these cover the corresponding component of the group scheme classifying
twisting data trivial on~$W$, and one readily verifies that this component
generates the full group scheme.

In the $\tilde{A}_n$ case, we~could obtain every twisting datum via a
cocycle in Cartier divisors. This is unlikely to hold in general type
(albeit without a known counterexample), but something only slightly weaker
is true.

\begin{prop}
 Let $X/S$ be an abelian torsor equipped with an action of~$\tW$ of
 coroot type such that $W$ has trivial root kernel. Then any twisting
 datum on~$X/S$ can be fppf locally represented by a cocycle in Cartier
 divisors.
\end{prop}

\begin{proof}
 Certainly the twist by a line bundle may be represented as a cocycle in
 Cartier divisors, so we reduce to the case that the twisting datum is
 trivial along $W$. Let $\lambda_0\colon X\to E'$ be a~(nonconstant)
 homomorphism from $X$ to an elliptic curve, and consider the orbit
 $\tW\lambda_0$ of such maps for general $u\in E'$. There is a~point $q'\in E'$ (determined from $q$ and $\lambda_0$) such that
 $\lambda$ is in the orbit iff
 \begin{gather*}
 \lambda = w\lambda_0 + jq'
 \end{gather*}
 for some $w\in W$, $j\in \Z$. Now, consider the formal sum
 \begin{gather*}
 \sum_{j\ge 0} \sum_{\lambda\in W\lambda_0} [w\lambda_0=u+jq']
 \end{gather*}
 of divisors on~$X$. This is $W$-invariant, and its coboundary with
 respect to~$s_0$ has finite support, so we obtain a well-defined family
 of cocycles in Cartier divisors.

 Apart from type $A$ and $C$, it will suffice to show that it depends
 nontrivially on~$u$ (since then we have a nonconstant morphism from $E'$
 to the connected $1$-parameter group scheme parametrizing twisting data).
 Such dependence is clearly independent of~$q$, so we may assume $q$
 non-torsion, and thus $q'$ non-torsion. Then shifting $u$ by $q'$ above
 subtracts $\sum_{\lambda\in W\lambda_0} [w\lambda_0=u]$ from the formal
 sum. This divisor class is $W$-invariant and ample, so is {\em not}
 $s_0$-invariant, and thus shifting $u$ has a nontrivial effect on the
 twisting datum as required.

 In type $C$, the same argument shows that everything in the identity
 component of the scheme parametrizing twisting data is fppf locally
 represented by cocycles in Cartier divisors, and one can explicitly
 verify (indeed, see the discussion of the $\tilde{C}$ case below) that
 when there is another component, it can still be reached in this way.
\end{proof}

\begin{rems}
 The presence of exotic twisting data in type $A$ can be explained by
 considering the induced cocycle in polarizations. As~discussed in more
 detail in the $\tilde{C}_n$ case below, such a cocycle for arbitrary type
 may be obtained as the coboundary of a $W$-invariant rational homogeneous
 polynomial of the form $p_3(\vec{z},q)/q)$ (with appropriate
 modifications in the presence of nontrivial isogenies). The~only part
 that contributes to the polarization of the coboundary is the part of
 degree 3 in~$\vec{z}$, and the only indecomposable finite Weyl groups
 with invariants of degree 3 are those of type $A_n$ for~$n\ge 2$.
\end{rems}

\begin{rems}
 In the above argument, we~used the fact that shifting $u$ by $q'$ had the
 effect of twisting by a $W$-invariant line bundle, and that this had a
 nontrivial effect on the twisting datum. It follows that we do not have
 a well-defined scheme parametrizing twisting data modulo twists by line
 bundles. Indeed, it follows that for $q$ non-torsion, twists by line
 bundles are dense in the identity component of the group scheme of
 twisting data trivial on~$W$.
\end{rems}

It is unclear (but likely) if this result holds for actions with nontrivial
root kernel. This fact is useful enough, however, that we will include it
as an implicit assumption below; that is, we~will impose as an additional
condition that the twisting datum is fppf locally represented by a cocycle
in Cartier divisors, or equivalently (by the above argument) that the
restriction of the twisting datum to~$W$ is a coboundary.

The description of~${\cal H}_{\tW}(X)$ $\big($or ${\cal
 H}_{\tW,\tW_I,\tW_J}(X)\big)$ as holomorphy-preserving
operators continues to hold, as~long as $\tW$ acts faithfully, or
more generally for any Bruhat order ideal that injects in~$\Aut(X)$. So we
may again apply Corollary~\ref{cor:global_order_two_residue_conditions} to
obtain residue conditions analogous to those of
\cite{GinzburgV/KapranovM/VasserotE:1997} for the finite case, just as in
the non-affine case.

For the twisted case, we~note that when~$q$ is nontorsion, the algebra
${\cal H}_{\tW;\gamma}(X)$ is still an algebra of the type constructed in
Proposition~\ref{prop:reflexive_order_construction}, so there is no
difficulty in generalizing the proofs and we still have reasonable residue
conditions. (If we represent the twisting datum as a cocycle in Cartier
divisors transverse to the reflection hypersurfaces, then the corresponding
embedding in~$k(X)\big[\tW\big]$ is again holomorphy-preserving away from the
support of the cocycle, and thus the residue conditions can still be
obtained from Corollary~\ref{cor:global_order_two_residue_conditions}. Of
course, the argument we used in the non-affine case works equally well!)
Any Bruhat interval is flat as $q$ varies, and thus we can obtain
conditions for torsion~$q$ by taking limits. This can become quite
complicated when the Bruhat interval is large compared to the order of~$q$,
but in the case that the Bruhat interval acts faithfully, there is no
difficulty taking the limit; the poles are at most simple, and one simply
has the usual condition along each reflection hypersurface given by the
twisting datum. The~same applies to spherical modules; again, the sheaf on
each Bruhat interval can be obtained as the limit from general $q$, and the
limit is straightforward as long as there is no coefficient such that two
components of its allowed polar divisor coalesce in the limit. In~other
words, for~the $W_I$-invariant Bruhat interval $[\le w W_I]$, if the set of
reflections $\{r\in R(W)\colon r w' W_I\in [\le w W_I]\setminus \{w' W_I\}\}$
injects in~$\Aut(X)$ for every $w' W_I\le w W_I$, then the residue
conditions are precisely as expected from the non-affine case.

The most important case for the spherical algebra construction is when~$\tW_I=W$ is the corresponding finite Weyl group. In~that case, we~note that each coset $\tW/W$ contains a unique translation, and thus
we may interpret elements of the spherical algebra as (elliptic) difference
operators, with~${\cal H}_{\tW,W}(X)$ for non-torsion~$q$
corresponding to difference operators that (locally) preserve $W$-invariant
holomorphic functions. Note that the Bruhat order
on~$W\setminus\tW/W$ is simply the usual dominance order on dominant
weights~\cite{LusztigG:1983}.

This has the following consequence. We say that a sheaf algebra is a
domain if the product of a nonzero section on~$U\times V$ and a nonzero
section on~$V\times W$ is always a nonzero section on~$U\times W$.

\begin{prop}
 Suppose that the root kernel for~$W$ on~$X$ is diagonalizable. Then for
 any twisting datum $\gamma$ which is trivial on~$W$, every fiber of the
 spherical algebra ${\cal H}_{\tW,W;\vec{T};\gamma}(X)$ is a~domain.
\end{prop}

\begin{proof}
 We first note that the inclusion of~${\cal
 H}_{\tW,W;\vec{T};\gamma}(X)$ in~$k(X)[\Lambda]_\gamma$ is
 injective on fibers. This follows from the fact that we can compute
 ${\cal H}_{\tW,W;\vec{T};\gamma}(X)$ as the $W$-invariant submodule
 of an~$S$-flat module with strongly flat invariants, and thus the
 inclusion~${\cal H}_{\tW,W;\vec{T};\gamma}(X) \subset
 \Ind^{\tW;\vec{T};\gamma}_W \sO_X$ is injective on fibers; as the
 induced module injects in the induced module of~$k(X)$ and this equals
 $k(X)[\Lambda]_\gamma$, the desired injectivity follows.

 In particular, any local section of a fiber on a product of~$W$-invariant
 open subsets can be identified with a $W$-invariant section of~$k(X)[\Lambda]_\gamma$. Since this identification is compatible with the
 multiplication, it will suffice to show that $k(X)[\Lambda]_\gamma$ is a
 domain. Since $\Lambda$ is a finitely generated free abelian group,
 there exist injective homomorphisms $\Lambda\to \R$, allowing us to
 define a total ordering on~$\Lambda$ compatible with the group law. In~particular, for~any nonzero element $\sum_{\lambda\in\Lambda}
 c_\lambda[\lambda]$ of~$k(X)[\Lambda]_\gamma$, there is a corresponding
 notion of~``leading monomial'', defined as $c_\lambda[\lambda]$, where
 $\lambda$ is the largest element of the support. If~$f$ has leading
 monomial $f_\lambda[\lambda]$ and $g$ has leading monomial~$g_\mu[\mu]$,
 then $fg$ has leading monomial $\zeta_{\lambda\mu}(f_\lambda\otimes
 {}^{\lambda}g_\mu)$, and is therefore nonzero as required.
\end{proof}

Another important feature of the spherical algebra in the affine case is
that it is Morita equivalent to the DAHA itself, at least for generic
parameters. The~proof relies on the following result on two-sided ideals
in the DAHA.

\begin{lem}
 Let $k$ be an algebraically closed field, and suppose $\tW$ acts
 faithfully on the abelian torsor $Y/k$. Let $S$ be a product of
 symmetric powers of coroot curves of~$Y$, and let $\vec{T}$ be the
 corresponding universal system of parameters on~$X:=Y\times S$. Then for any
 ideal sheaf ${\cal I}\subset \sO_X$, there is a dense open subset of~$S$
 on which ${\cal I}$ generates ${\cal H}_{\tW;\vec{T};\gamma}(X)$ as a
 two-sided ideal.
\end{lem}

\begin{proof}
 Suppose otherwise. We may assume that
 \begin{gather*}
 {\cal I} = \bigl( {\cal H}_{\tW;\vec{T};\gamma}(X)
 {\cal I} {\cal H}_{\tW;\vec{T};\gamma}(X) \bigr)\big|_1,
 \end{gather*}
 since both sides generate the same two-sided ideal. If~we replace each
 instance of the DAHA by~the restriction of the DAHA to operators
 supported entirely on~$s_i$, it follows that
 \begin{gather*}
 {\cal I} \supset {}^{s_i}{\cal I}\otimes \sO_X\big(T_i+{}^{s_i}T_i\big).
 \end{gather*}
 This in turn induces a weak symmetry condition on the {\em set} $Z$ cut
 out by ${\cal I}$: if $x\in Z$, then either~$s_i(x)\in Z$ or $x\in
 T_i\cup {}^{s_i}T_I$, and in those terms our goal is to prove that $Z$
 does not meet the generic fiber.

 We claim that this is true for {\em any} proper closed subset of~$X$
 satisfying this condition. Since the weak symmetry condition is
 preserved under restriction to the generic fiber of~$S$ as well as under
 taking Zariski closure, we~may assume that every component of~$Z$ meets
 the generic fiber, and thus by properness that it meets every fiber. Now
 consider a fiber on which each $T_i$ is contained (as a set) in~$[Y^{s_i}]$. If~$x$ is a~point of~$Z$ in such a fiber, the weak symmetry
 condition implies that either $s_i(x)\in Z$ or $x\in [Y^{s_i}]$ and thus
 $x=s_i(x)$. Thus the restriction of~$Z$ to such a fiber must be
 $s_i$-invariant (as a set!) and thus $\tW$-invariant. Since
 $\tW$ acts faithfully on~$Y$, any~$\tW$-orbit in~$Y$ is
 Zariski dense, and thus $Z$ must contain every such fiber.

 It will thus suffice to prove that $Z$ cannot contain any fiber, and thus
 obtain a contradiction. Again, since $\tW$ acts faithfully, $Z_{k(S)}$
 cannot contain any $\tW$-orbits, and thus for any point $x\in Z_{k(S)}$,
 there exists $w\in \tW$ and $i\in \{0,\dots,n\}$ such that $wx\in
 Z_{k(S)}$ but $s_iwx\notin Z_{k(S)}$. It~follows that $wx\in T_i\cup
 {}^{s_i}T_i$, and thus $Z_{k(S)}$ is covered by the sets of the form
 ${}^wT_i$ for~$w\in \tW$, $i\in \{0,\dots,n\}$. Since every component of~$Z$ meets the generic fiber, $Z$ is covered by the same sets; since none
 of these sets contains a fiber, $Z$ cannot contain a fiber, and we are
 done.
\end{proof}

\begin{prop}\label{prop:spherical_is_Morita_equivalent}
 Suppose $\tW$ acts faithfully on~$X$ such that the root kernel of~$W$ on~$X$ is diagonalizable and $\gamma$ is trivial on~$W$. If~$\vec{T}$ is in
 sufficiently general position, then the categories of~${\cal
 H}_{\tW;\vec{T};\gamma}(X)$-modules and ${\cal
 H}_{\tW,W;\vec{T};\gamma}(X)$-modules are equivalent, with the inverse
 equivalences given by ${-}^W$ and $\Ind_W^{\tW;\vec{T};\gamma}
 \sO_X\otimes_{{\cal H}_{\tW,W;\vec{T};\gamma}(X)}{-}$.
\end{prop}

\begin{proof}
 Since we are assuming $\vec{T}$ is in sufficiently general position, we~may in particular assume that the finite Hecke algebra ${\cal
 H}_{W;\vec{T}}(X)$ has a covering by symmetric idempotents, so
 that both functors are exact. It suffices to check that they are
 inverses on the regular representation. One direction is by definition
 of the spherical algebra, so we reduce to showing that the natural map
 \begin{gather*}
 \phi\colon\ \Ind_W^{\tW;\vec{T};\gamma} \sO_X\otimes_{{\cal
 H}_{\tW,W;\vec{T};\gamma}(X)} {\cal H}_{\tW;\vec{T};\gamma}(X)^W
 \to {\cal H}_{\tW;\vec{T};\gamma}(X)
 \end{gather*}
 is surjective. The~image certainly contains the image of
 \begin{gather*}
 \sO_X\otimes_{{\cal H}_{W,W;\vec{T}}(X)}{\cal H}_{W;\vec{T}}(X)^W
 \to {\cal H}_{W;\vec{T}}(X).
 \end{gather*}
 This map is almost certainly {\em not} surjective, but its cokernel is
 torsion, since the analogous map for~$k(X)[W]_\gamma$ is surjective.
 It follows that the image meets $\sO_X$ nontrivially, and thus the same
 is true for~$\phi$. But then the image of~$\phi$ is a two-sided ideal in
 the DAHA meeting $\sO_X$ nontrivially, and thus the lemma tells us that
 $\phi$ is surjective.
\end{proof}

\begin{rem}
 More generally, if $\vec{T}$ and $\vec{T}'$ are two systems of parameters
 in sufficiently general position, the same argument tells us that ${\cal
H}_{\tilde{W};{}^-\vec{T}+\vec{T}';\gamma}(X)$
 is Morita equivalent to
 \begin{gather*}
 \End\bigg(\Ind_W^{\tW;\vec{T}+\vec{T}';\gamma}
 \sO_X\bigg({-}\sum_{\alpha\in \Phi^+(W)} T_\alpha\bigg)\! \bigg).
 \end{gather*}
 (Note that by Proposition~\ref{prop:split_system_of_parameters}, this is
 indeed a module, and is cut out by the image of local symmetric idempotents in~${\cal H}_{W;{}^-\vec{T}+\vec{T}'}(X)$.)
\end{rem}

In the affine case, the spherical algebra has an additional symmetry. As~we mentioned above, for~general Coxeter groups, the usual symmetry
replacing $\vec{T}$ by ${}^-\vec{T}$ has an issue in the spherical algebra
case. The~proof of that symmetry relied on the fact that $\sum_{\alpha\in
 \Phi(W)}T_\alpha$ has no effect on twisting, so that twisting by
$\sum_{\alpha\in \Phi^+(W)}T_\alpha$ and $-\sum_{\alpha\in
 \Phi^-(W)}T_\alpha$ have the same effect, letting one move the twist to
the other half of the intersection. For~the spherical algebra, this
operation instead turns $-\sum_{\alpha\in \Phi^-(W)\setminus
 \Phi^-(W_I)}T_\alpha$ into~$\sum_{\alpha\in\Phi^+(W)\cup
 \Phi^-(W_I)}T_\alpha$, and thus does not give an algebra of the same form
$\big($but rather involves the other natural representation
$\sO_X\big({-}\sum_{\alpha\in \Phi^+(W_I)}T_\alpha\big)\big)$. However, in the {\em
 affine} case, it turns out that there actually {\em is} a system of
parameters $\vec{T}'$ such that
\begin{gather*}
\sum_{\alpha\in\Phi^+(\tW)\cup \Phi^-(W)}T_\alpha
=
\sum_{\alpha\in\Phi^-(\tW)\setminus \Phi^-(W)}T'_\alpha.
\end{gather*}
In both cases, we~can break up the sum as a sum over roots of the finite
Weyl group $W$, and find that on the left-hand side we have a sum over
translates of~$T_\alpha$ by nonnegative multiples of some $q_\alpha$, while
on the right-hand side we have a sum over translates of~$T'_{\alpha}$ by
negative multiples of~$q_\alpha$. We may thus simply take $T'_\alpha$ to
be the translate of~$T_{-\alpha}$ by $q_\alpha$.

As a result, we~find that the algebra
\begin{gather*}
\sHom_{{\cal H}_{\tW;\vec{T}}(X)}\bigg(\Ind^{\tW;\vec{T};\gamma}_{W}
\sO_X\bigg({-}\sum_{\alpha\in \Phi^+(W)} T_\alpha\bigg),
 \Ind^{\tW;\vec{T};\gamma}_{W}\sO_X\bigg({-}\sum_{\alpha\in \Phi^+(W)} T_\alpha\bigg)\!\bigg)
\end{gather*}
may be identified with an instance of~${\cal H}_{\tW,W;\vec{T}}(X)$ in
which the parameters have been translated by $q_\alpha$. We also obtain a
pair of (strongly flat) intertwining bimodules as discussed above in the
finite case. For~generic parameters, each of the algebras is Morita
equivalent to the original DAHA by Proposition~\ref{prop:spherical_is_Morita_equivalent} and the remark following, and
thus the algebras are Morita equivalent to each other, with the
equivalences induced by the corresponding bimodules. More generally, if we
write $\vec{T}=\vec{T}'+\vec{T}''$, then the algebra obtained by
translating $\vec{T}''$ but leaving $\vec{T}'$ alone is still {\em
 generically} expressible as the endomorphism ring of an induced
representation $\big($coming from $\sO_X\big({-}\sum_{\alpha\in
 \Phi^+(W)}\vec{T}'_\alpha\big)\big)$, so the proof of Proposition~\ref{prop:spherical_is_Morita_equivalent} still gives a Morita equivalence
to the original DAHA for generic parameters, and thus Morita equivalences
between the spherical algebras with shifted parameters. Note that the
corresponding DAHAs are themselves Morita equivalent for generic parameters,
since their spherical algebras are Morita equivalent.

Another feature of the double affine case is that the inverse map acts on
the poset ${}^W\tW^W$ as a~diagram automorphism (which is often
trivial). In~particular, if we can arrange for the pullback of the adjoint
through the diagram automorphism to have isomorphic twist datum, then this
gives an actual involution of the Hecke algebra, which allows us to
consider self-adjoint operators (and even have a reasonable chance of
proving commutativity). For~a specific example of this phenomenon in the
$\tilde{C}_n$ case, see Theorem~\ref{thm:vandiejen} below.

One new phenomenon that arises in the affine case is that the group can
fail to act faithfully. Since the finite Weyl group acts faithfully on~$A$, the kernel is necessarily contained in the translation subgroup, and
we see that there is a kernel precisely when~$q$ is torsion. In~that case,
the action of~$\tW\cong W\ltimes \Lambda$ on~$X$ factors through a
semidirect product of the form $\tW_q:=W\ltimes (\Lambda/\Lambda_q)$,
where $\Lambda_q$ is the sublattice acting trivially (which is a multiple
of either $\Lambda$ or, in the non-simply-laced case, the lattice generated
by the other orbit of roots). One thus finds that ${\cal
 H}_{\tW,\vec{T};\gamma}(X)$ is the sheaf algebra associated to a
sheaf of algebras on the quotient $X/\tW_q$. The~centralizer of~$k(X)$ in the generic fiber is isomorphic to~$k(X)[\Lambda_q]_\gamma$, and
thus the center of the generic fiber is isomorphic to~$k\big(X/\tW_q\big)[\Lambda_q]_\gamma^W$. This agrees with the generic fiber
of the center, and thus the center of~${\cal
 H}_{\tW,\vec{T};\gamma}(X)$ for $q$ torsion is the coordinate sheaf
of an integral~$X$-scheme of relative dimension~$n$.

When $\vec{T}=0$, we~can make this quite precise; it turns out that for~$q$
torsion, \smash{${\cal H}_{\tW;\gamma}(X)$} is an algebra \`a la Proposition~\ref{prop:reflexive_order_construction} with~$G=\tW_q$ acting on an
explicitly described (but no longer projective) scheme $X^+$; its center is
thus the structure sheaf of~\smash{$X^+/\tW_q$}, and everything is
Noetherian. Clearly, if such a construction exists, $X^+$ must be the
relative $\Spec$ of the (abelian) restriction~${\cal
 H}_{\tW;\gamma}(X)|_{\Lambda_q}$. This is difficult to understand
directly (we will in fact give an alternate purely geometric construction
and then verify that its relative coordinate ring is as described), but
luckily it will largely suffice to deal with the $\tilde{A}_1$ case.

Thus let $C/S$ be a genus 1 curve equipped with a pair of hyperelliptic
involutions $s_0$, $s_1$, and let $\gamma$ be a twisting datum for the
corresponding action of~$\tilde{A}_1=\langle s_0,s_1\rangle$ on~$C$. We
may describe~$\gamma$ by giving a pair of line bundles ${\cal L}_0$ and
${\cal L}_1$ such that $\gamma|_{\langle s_i\rangle}$ is the coboundary of~${\cal L}_i$. It is then straightforward to determine the residue
conditions for ${\cal H}_{\tilde{A}_1;\gamma}(C)$ when~$s_1s_0$ has
infinite order. We find that $c_{(s_1s_0)^k}$ and $c_{(s_1s_0)^{k+l}s_1}$
have a potential pole along $\big[C^{(s_1s_0)^{2k+l}s_1}\big]$, subject to the
condition that
\begin{gather*}
c_{(s_1s_0)^k} {}^{(s_1s_0)^k}f
+
c_{(s_1s_0)^{k+l}s_1} {}^{(s_1s_0)^{k+l}s_1}f
\end{gather*}
is holomorphic on that reflection hypersurface for any local section
\begin{gather*}
f\in
\begin{cases}
\Gamma\big(U;\sHom\big(\bigotimes_{0<i\le l} {}^{(s_1s_0)^i}{\cal L}_0,
\bigotimes_{0\le i\le l} {}^{(s_1s_0)^i}{\cal L}_1\big)\big), & l\ge 0,
\\[.5ex]
\Gamma\big(U;\sHom\big(\bigotimes_{l<i<0} {}^{(s_1s_0)^i}{\cal L}_1,
\bigotimes_{l<i\le 0} {}^{(s_1s_0)^i}{\cal L}_0\big)\big), &l<0.
\end{cases}
\end{gather*}
(Of course, the condition is independent of~$f$ as long as $f$ is a unit in
the local ring at every component of the reflection hypersurface.)

Now, suppose $s_1s_0$ has finite order $m$, so that the image of~$\langle
s_0,s_1\rangle$ in~$\Aut(C)$ is actually the dihedral group of order $2m$.
Thus as automorphisms of~$C$, they satisfy the braid relation~$(s_0s_1)^{m/2}=(s_1s_0)^{m/2}$ (if $m$ is even) or
$(s_0s_1)^{(m-1)/2}s_0=(s_1s_0)^{(m-1)/2}s_1$ (if $m$ is odd). In~the
dihedral group, both sides represent the longest element, and thus the
corresponding Bruhat interval is the entire group. Something similar
happens in~$\tilde{A}_1$: each side generates a Bruhat interval which is
not only faithful, but a setwise section of the map to the dihedral group,
and the two sections agree except on the longest element. Even better, the
residue conditions in the two Bruhat intervals are very nearly the same;
indeed, below the top, the conditions are necessarily the same, while the
conditions involving the top element differ only mildly. We thus obtain
the following, where $\pi_q\colon C\to C'$ is the quotient by $\langle
s_1s_0\rangle$ (which is translation by some torsion point $q$), an
isogenous curve with induced hyperelliptic involution~$s$ (the common
action of~$s_0$ and $s_1$). (We write $(s_0s_1)^{m/2}$ for~$(s_0s_1)^{(m-1)/2}s_0$ when~$m$ is odd, so that we can express the result
in a uniform way for either parity.)

\begin{prop}
 Let $U\subset C'$ be an $s$-invariant open subset on which there is an
 invertible section~$g\in \Gamma\big(U;N_{\pi_q}\big({\cal L}_1\otimes {\cal L}_0^{-1}\big)^*\big)$. Then
 \begin{gather*}
c_{(s_0s_1)^{m/2}} (s_0s_1)^{m/2}+\sum_{w<(s_0s_1)^{m/2}} c_w w
\in \Gamma\big(\pi_q^*U;{\cal H}_{\tilde{A}_1;\gamma}(C)\big)
 \end{gather*}
 iff
 \begin{gather*}
 \pi_q^*(g/{}^sg) c_{(s_0s_1)^{m/2}} (s_1s_0)^{m/2} + \sum_{w<(s_1s_0)^{m/2}} c_w w
 \in
 \Gamma\big(\pi_q^*U;{\cal H}_{\tilde{A}_1;\gamma}(C)\big).
 \end{gather*}
\end{prop}

\begin{proof}
 This reduces to a straightforward verification that the transformation
 respects the residue conditions on every reflection hypersurface.
\end{proof}

\begin{cor}
 With $U$, $g$ as above, let $f_0,f_1\in \Gamma\big(U;\sO_X(\pi_q^*[C^{\prime
 s}])\big)$ be such that $f_1-\pi_q^*(g/{}^sg)f_0$ is holomorphic.
 Then
 \begin{gather*}
 f_0+f_1(s_1s_0)^m\in \Gamma\big(U;{\cal H}_{\tilde{A}_1;\gamma}(C)\big).
 \end{gather*}
\end{cor}

\begin{proof}
 Take a general section of~${\cal H}_{\tilde{A}_1;\gamma}(C)[{\le}
 (s_0s_1)^{m/2}]$, apply the proposition to get a section supported on~$[{\le}(s_1s_0)^{m/2}]$, take the difference, and then left-multiply by a
 suitable operator $h (s_1s_0)^{m/2}$. This gives a general section~$f_0+f_1(s_1s_0)^m$ such that $ f_1 = \pi_q^*(g/{}^sg)f_0$, to which we~may add any $f'_0\in \Gamma(U;\sO_X)$.
\end{proof}

\begin{rem}
 Note that $f_1$ is indeed supposed to be a meromorphic section of the
 same bundle of~which $\pi_q^*(g/{}^sg)$ is a local trivialization, namely
 $\bigotimes_{w\in G} {}^w\big({\cal L}_1\otimes {\cal L}_0^{-1}\big)^{(-1)^{\ell(w)}}$,
 where $G$ is the image of~$\tilde{A}_1$ in~$\Aut(C)$.
\end{rem}

\begin{rem}
 This points out that in the affine case (in contrast to the situation for
 more general Coxeter groups where there is no parameter $q$), there is
 something special about Bruhat intervals: the subsheaves of operators
 supported on more general finite subsets may fail to be flat.
\end{rem}

The restriction of~$\sO_C[\tilde{A}_1]_\gamma$ to the kernel
$\langle (s_1s_0)^m\rangle$ of the action has a natural geometric
description. Indeed, it may be given in the form
$\bigoplus_{j\in \Z} {\cal Z}_{(s_1s_0)^m}^{\otimes j} (s_1s_0)^{jm}$,
which is easily recognized as the relative structure sheaf of the
$\G_m$-bundle over $C$ corresponding to the invertible sheaf ${\cal
 Z}_{(s_1s_0)^m}$. The corollary tells us that ${\cal
 H}_{\tilde{A}_1;\gamma}(C)$ contains a larger ring, which is easily
recognized as coming from an affine blowup. To~be precise, the
restrictions of the various local sections $\pi_q^*(g/{}^sg)$ to the union
of reflection hypersurfaces induces a section of the $\G_m$-bundle over
that union. Enlarging the algebra has the effect of blowing up that union
of sections and then removing the strict transforms of the fibers over the
reflection hypersurfaces.

More generally, let $\tW$ act on~$X$ with kernel $\Lambda_q$. The~restriction of the gerbe to~$\Lambda_q$ in particular induces a
homomorphism $\Lambda_q\to \Pic(X/S)$, or equivalently a class in~$H^1(X;\Hom(\Lambda_q,\G_m))$, and thus a principal
$\Hom(\Lambda_q,\G_m)$-bundle $P$ over $X$. Each ${\cal Z}_\lambda$ for
$\lambda\in \Lambda_q$ comes with an induced trivialization on the
pullback, and this in fact trivializes that portion of the gerbe. Indeed,
this compatibility is clear for any twisting datum specified as a cocycle
in Cartier divisors, and (per our standing assumption) any twisting datum
is fppf locally of this form. We thus obtain a~natural
$\tW_q$-equivariant gerbe on the torus bundle, on which $\tW_q$
acts faithfully, and there is a~natural isomorphism $\sO_X\big[\tW\big]_{\cal
 Z}\cong \pi_*\sO_P\big[\tW\big]_{\cal Z}$.

The twisting datum itself does not directly lift, for~the simple reason
that $\tW$ acts nontrivially on~$\Lambda_q$, and thus in particular
the fixed subschemes of the various reflections are codimension 2; the
fixed subscheme of~$r$ is a principal $\Hom(\Lambda_q,\G_m)^{\langle
 r\rangle}$-bundle over $[X^r]$. (Note that this is either a
$\G_m^{n-1}$-bundle or a $\G_m^{n-1}\times \mu_2$-bundle, with the
latter only arising when~$W\ltimes \Lambda_q\cong W(\tilde{C}_n)$.)
However, each reflection in~$\tW_q$ is the image of an infinite
collection of reflections in~$W\ltimes \Lambda_q$, the set of reflections
of a copy of~$\tilde{A}_1$. Moreover, for~any reflection~$r$, the
subalgebra of operators supported on~$\langle 1,r\rangle$ may be computed
by conjugating the corresponding algebra for some simple reflection.
We thus find that each such copy of~$\tilde{A}_1$ actually gives rise to a
subalgebra of the form~${\cal H}_{\tilde{A}_1;\gamma'}(X)$. We thus obtain
additional elements as above.

We may think of each such element as a section of~$\sO_P([X^r])$, and find
that the ideal sheaf generated by all such sections cuts out a
\smash{$\Hom\big(\Lambda_q^{\langle r\rangle},\G_m\big)$}-bundle over $[X^r]$. (In the
untwisted case, this is the subbundle of the $r$-fixed scheme containing
$([X^r],1)$, and in general is uniquely determined by the requirement that
it contain a section determined by the twisting datum.) Adjoining all such
elements to the relative coordinate ring gives a new scheme $X^+$, still
affine over $X$. Geometrically, $X^+$ is obtained by blowing up the union
of all such bundles then removing the strict transforms of the fibers of~$P$ over the reflection hypersurfaces. This operation is
$\tW_q$-equivariant, and thus $X^+$ inherits an action of~$\tW_q$. Now, not only does the gerbe lift, but the twisting datum
as well: on~$X$, it is specified by sections of line bundles on the various
reflection hypersurfaces, and each reflection hypersurface of~$X^+$ lies
over a reflection hypersurface of~$X$, so we can simply pull back the
section. (Again, any potential issues with compatibility may be reduced to
the untwisted case by expressing the twisting datum via Cartier divisors,
pulling back the Cartier divisors to~$X^+$ and twisting. In~the untwisted
case, the possible incompatibilities coming from the nontriviality of the
gerbe maps go away, since the blowup makes those functions congruent to 1
on the relevant reflection hypersurface.)

We then have the following.

\begin{thm}\label{thm:daha_for_torsion_q}
 Let $X^+$ be constructed as above, with associated projection~$\pi\colon X^+\to
 X$ and induced twisting datum $\gamma^+$ for the action of~$\tW_q$
 on~$X^+$. Then there is a natural isomorphism
 \begin{gather*}
 {\cal H}_{\tW;\gamma}(X)\cong \pi_* {\cal H}_{\tW_q;\gamma^+}(X^+).
 \end{gather*}
\end{thm}

\begin{proof}
 The map $\sO_X\big[\tW\big]_{\gamma}\to \pi_*\sO_{X+}\big[\tW_q\big]_{\gamma^+}$ extends
 to a map on meromorphic operators, and we find that the restriction to
 any rank~1 subalgebra of~${\cal H}_{\tW;\gamma}$ is contained in~$\pi_*{\cal H}_{\tW_q;\gamma^+}(X^+)$. Since these generate the full
 algebra, we~obtain a homomorphism \smash{${\cal H}_{\tW;\gamma}(X)\to \pi_*{\cal
 H}_{\tW_q;\gamma^+}(X^+)$}, and it remains only to show that this is
 surjective.

 By construction, the image contains $\sO_{X^+}$, and thus for any
 reflection~$r$, the image contains any element of the form $g \pi^*f
 (1-\pi^*h r)$, where $g\in \Gamma(\pi^*U;\sO_{X^+})$, $f\in
 \Gamma\big(U;\sO_X([X^r])\big)$, and~$h\in \Gamma(U;{\cal Z}_r^*)$ restricts to
 give the twisting datum along $r$. For~$U$ sufficiently small (but
 containing any chosen codimension~1 point), the elements of the form
 $g\pi^*f$ span $\Gamma\big(U;\sO_X([(X^+)^r])\big)$, and thus we obtain any
 element in~$\Gamma(U;\sO_{X^+}) + \Gamma\big(U;\sO_{X^+}([X^r])\big)(1-\pi^*h
 r)$. But this is precisely the algebra corresponding to~$\langle
 1,r\rangle$ in~${\cal H}_{\tW_q;\gamma^+}(X^+)$, and these subalgebras
 again generate the full algebra.
\end{proof}

\begin{cor}
 The center of~${\cal H}_{\tW;\gamma}(X)$ as a sheaf of algebras
 over $X/\tW_q$ is naturally isomorphic to the relative coordinate
 ring of~$X^+/\tW_q$.
\end{cor}

\begin{proof}
 Since $\tW_q$ acts faithfully on~$X^+$, the centralizer of~$\sO_{X^+}$ in~${\cal H}_{\tW_q,\gamma^+}(X^+)$ is precisely
 $\sO_{X^+}$; furthermore, for~an element of~$\sO_{X^+}$ to commute with
 sections of~${\cal Z}_g g$ for all $g\in \tW_q$, it must be~$\tW_q$-invariant. We thus find that the center is contained in~$\sO_{X^+}^{\tW_q}=\sO_{X^+/\tW_q}$. This subalgebra is
 clearly central, and thus the claim holds.
\end{proof}

{\sloppy
We would also, of course, like to understand the center of the spherical
algebra. When \mbox{$X\to X/W$} is flat, a symmetric idempotent already
establishes a Morita equivalence bet\-ween~${\cal H}_W(X)$ and $\sO_{X/W}$, and
thus forces the center of the spherical algebra to agree with that of the
DAHA itself. Of course, as~we noted above, this map is essentially never
flat (with the notable exceptions of the action of~$W(A_n)$ on the sum zero
subscheme and the action of~$W(C_n)$ on~$E^n$ by signed permutations), but
it turns out to be close enough to let us prove the analogous result.

}

\begin{prop}
 Suppose $\gamma$ is trivial on~$W$ and the root kernel of~$W$ is
 diagonalizable. Then there is a natural isomorphism $Z\big({\cal
 H}_{\tW;\gamma}(X)\big)\cong Z\big({\cal H}_{\tW,W;\gamma}(X)\big)$,
 and both algebras are strongly flat.
\end{prop}

\begin{proof}
 Let $U$ be the largest $\tW_q$-invariant open subset of~$X$ such
 that every point of~$U$ lies over a regular point of~$X/W$. The~morphism
 $U\to U/W$ is a morphism of regular schemes with~$0$-dimensional fibers,
 so is flat, and thus in particular $\pi_{W*}\sO_U$ is locally free. It
 follows that a covering of~$U$ by symmetric idempotents induces a Morita
 equivalence between~${\cal H}_W(U)$ and~$\sO_{U/W}$, and thus between the
 restrictions of~${\cal H}_{\tW;\gamma}(X)$ and ${\cal
 H}_{\tW,W;\gamma}(X)$ $\big($viewed as sheaf algebras on~$X/\tW_q\big)$ to~$U/\tW_q$. In~particular, we~find that the
 natural morphism between the centers is an~isomorphism on~$U/\tW_q$.

 For any other $\tW_q$-invariant open subset $V$, given any element
 ${\cal D}\in \Gamma\big(V;Z({\cal H}_{\tW,W;\gamma}(X))\big)$, we~may
 restrict it to~$V\cap U$ and transport it through the isomorphism of
 centers to obtain an element ${\cal D}'\in \Gamma\big(V\cap U;Z({\cal
 H}_{\tW;\gamma}(X))\big)$. By normality of~$X/W$, $U$ contains every
 codimension~1 point of~$X$, and thus
 $\Gamma(V\cap U;{\cal H}_{\tW;\gamma}(X))
 =
 \Gamma(V;{\cal H}_{\tW;\gamma}(X))$,
 so that ${\cal D}'$ is actually holomorphic on~$V$ as a section of the
 DAHA. Since an operator is in the center of the DAHA iff it is in the
 center of the meromorphic twisted group algebra, we~conclude that
 ${\cal D}'\in \Gamma\big(V;Z({\cal H}_{\tW;\gamma}(X))\big)$,
 and thus the natural map from the center of the DAHA to the center of the
 spherical algebra is indeed surjective.
\end{proof}

\begin{cor}
 If $\Lambda$ acts trivially on~$X$, $\gamma$ is trivial on~$W$, and the
 root kernel of~$W$ is dia\-gonalizable, then
 there is a natural isomorphism $Z\big({\cal H}_{\tW;\gamma}(X)\big)\cong
 {\cal H}_{\tW,W;\gamma}(X)$.
\end{cor}

\begin{rem}
 More generally, if $\gamma=\gamma'\partial{\cal L}$ with~$\gamma'$
 trivial on~$W$, then we can identify $Z\big({\cal H}_{\tW;\gamma}(X)\big)$
 with the twist of~${\cal H}_{\tW,W;\gamma'}(X)$ by ${\cal L}$. In~particular, if the root kernel is trivial, we~can always do this fppf
 locally on~$S$.
\end{rem}

\begin{rem}
 One consequence is that there is a natural Poisson structure on~$X^+/\tW_q$, since as a~spherical algebra, its relative coordinate
 ring is the commutative fiber of a family of noncommutative algebras.
 Presumably this Poisson structure has a natural description in terms of
 the geometry of~$X^+$. Note that $\sO_{X^+}$ is an elliptic analogue of
 the subalgebra of the usual DAHA generated by the commutative
 subalgebras, so it would not be unreasonable to expect it to have a flat
 deformation to non-torsion~$q$, corresponding to a pulled back Poisson
 structure on~$X^+$. This fails, however: the pullback of the Poisson
 structure on~$\sO_{X^+}^{\tW_q}$ is only meromorphic on~$\sO_{X^+}$. Indeed, it agrees with the natural Poisson bracked on~$\sO_X[\Lambda]_\gamma$, but the bracket of one of the additional
 generators with a function with nonzero partial derivative on the
 reflection hypersurface will fail to be holomorphic. Note that
 $\tW_q$-invariance forces this to vanish, so the restriction of
 this meromorphic Poisson structure to the $\tW_q$-invariants is
 holomorphic.
\end{rem}

{\sloppy
Unfortunately, the trick with~$\tilde{A}_1$-subalgebras breaks down
completely when the system of~para\-meters is nontrivial. As~a result, we~cannot give such a simple description of the center. It~turns out,
however, that we {\em can} reduce to the case $q=0$.

}

Return for the moment to the case $\vec{T}=0$. The~sublattice $\Lambda_q$
is isomorphic as a group with~$W$ action to one of the root lattices of~$W$, and thus the semidirect product $W\ltimes \Lambda_q$ is itself an
affine Weyl group. (Note that when~$\tW$ is of type $B$ or $C$, this
new affine Weyl group may be of the other type.) The quotient
$\Lambda/\Lambda_q$ acts on both $X$ and $X^+$, and we observe that
$X^+/(\Lambda/\Lambda_q)$ is generically a torus bundle over
$Y=X/(\Lambda/\Lambda_q)$, and is more precisely an affine blowup of~$Y$ in
the corresponding bundles over the reflection hypersurfaces. We thus
obtain the following.

\begin{prop}
 There is an induced twisting datum $N_{X/Y}(\gamma)$ on~$Y$ such that
 the center of~${\cal H}_{\tW;\gamma}(X)$ is canonically isomorphic
 to the center of~${\cal H}_{W\ltimes\Lambda_q;N_{X/Y}(\gamma)}(Y)$.
\end{prop}

\begin{proof}
 When $\gamma$ is trivial, this holds with~$N_{X/Y}(\gamma)$ trivial.
 More generally, if $\gamma$ is represented by~a~cocycle in Cartier
 divisors, we~may obtain~$N_{X/Y}(\gamma)$ by taking the images of the
 dif\-fe\-rent Cartier divisors, and the resulting twisting datum is
 independent of the choice of cocycle representation.
\end{proof}

We can then say the following in general, where $N_{X/Y}\vec{T}$ is again
obtained by taking the image of each $T_\alpha$ under the induced isogeny
of coroot curves.

\begin{thm}\label{thm:center_for_general_T}
 If the root kernel of~$W$ is diagonalizable, then the center of~${\cal
 H}_{\tW;\vec{T};\gamma}(X)$ is canonically isomorphic to the
 center of~${\cal H}_{W\ltimes
 \Lambda_q;N_{X/Y}\vec{T};N_{X/Y}\gamma}(Y)$.
\end{thm}

\begin{proof}
 The claim respects twisting by line bundles, so we may assume that
 $\gamma$ is trivial on~$W$, and replace the algebra on~$Y$ with the
 spherical algebra ${\cal H}_{W\ltimes
 \Lambda_q,W;N_{X/Y}\vec{T};N_{X/Y}\gamma}(Y)$. (Both cen\-ters will then
 be isomorphic to the same spherical algebra.)

 Imposing a system of parameters does not change the generic fiber, and
 thus we have
 \begin{gather*}
 Z\big({\cal H}_{\tW;\vec{T};\gamma}(X)\big) =
 {\cal H}_{\tW;\vec{T};\gamma}(X) \cap
 Z\big({\cal H}_{\tW;\gamma}(X)\big).
 \end{gather*}
 When $\vec{T}$ is sufficiently general, we~may apply Corollary~\ref{cor:affine_vanishing_conditions} to identify the precise conditions
 for an element of~$Z({\cal H}_{\tW;\gamma}(X))$ to be contained in
 the smaller Hecke algebra. (Note that determining whether a negative
 root of~\smash{$\tW$} becomes positive under a translation is quite
 straightforward.) In~particular, we~find as in the discussion above
 regarding reflection hypersurfaces that the divisor along which each
 coefficient must vanish is $\Lambda/\Lambda_q$-invariant. $\big($To be
 precise, it is a nonnegative linear combination of sums of the form
 $\sum_{x\in \langle q_\alpha\rangle} {}^x T_\alpha$, where $q_\alpha$
 generates the group of translations acting on the corresponding coroot
 curve.$\big)$ Thus imposing the condition that the coefficient is~$\Lambda/\Lambda_q$-invariant has no additional effect on the vanishing
 condition, so that the algebras agree under the given constraint on~$\vec{T}$.

 Since the spherical algebra is strongly flat, any section for special
 $\vec{T}$ extends to a neighborhood in parameter space. The~corresponding operator in~$Z\big({\cal H}_{\tW;\gamma}(X)\big)$ is thus
 contained in~${\cal H}_{\tW;\vec{T};\gamma}(X)$ for generic
 parameters, and thus for all parameters. It follows in particular
 that the image of~the~spherical algebra saturates the Bruhat filtration,
 and is thus surjective as required.
\end{proof}

\begin{rem}
 In the $\vec{T}=0$ case, we~were able to use $\sO_{X^+}$ as an
 intermediate step to understanding the center. Is there an analogous
 geometric description of~$\sO_{X^+}\cap {\cal
 H}_{\tW;\vec{T};\gamma}(X)$? For $\vec{T}$ transverse to
 reflection hypersurfaces, we~can again describe this intersection via
 vanishing conditions, but it does not follow from the above discussion
 that the intersection is flat, and the obvious comparison to the
 spherical module fails, as~the latter is not $\tW_q$-invariant. It
 is likely that a~description in the $\tilde{A}_1$ case could be extended
 to general type.
\end{rem}

The same argument gives the following.

\begin{cor}
 If the root kernel of~$W$ is diagonalizable and $\gamma$ is trivial on~$W$, then the center of~${\cal H}_{\tW;\vec{T};\gamma}(X)$ is
 canonically isomorphic to the center of~${\cal
 H}_{\tW,W;\vec{T};\gamma}(X)$.
\end{cor}

Note that $Z\big({\cal H}_{\tW;\vec{T};\gamma}(X)\big)$ is always a domain
(it~is contained in the structure sheaf of the integral scheme $X^+$). We
conjecture that it is also Noetherian, and moreover that ${\cal
 H}_{\tW;\vec{T};\gamma}(X)$ is finite over its center, just as in
the $\vec{T}=0$ case. Assuming the technique of
\cite{ArtinM/TateJ/VandenBerghM:1990} (reducing to finite fields, where $q$
is always torsion) could be adapted to the case of sheaf algebras, this
would be enough to prove ${\cal H}_{\tW;\vec{T};\gamma}(X)$
Noetherian even when~$q$ was not torsion.

The above description of the DAHA for~$q$ torsion can in principle be used
to give explicit degenerations of the residue conditions, though in
practice it seems simpler to check the analogous conditions on~${\cal
 H}_{\tW_q;\gamma^+}(X^+)$. For~that, the following general
reduction (which we implicitly used above) may be useful. Recall that any
reflection~$r\in R(W)$ induces a corresponding subgroup of type
$\tilde{A}_1$ of~$\tW$ (generated by the reflections in~$r\Lambda$),
and that this gives rise to a corresponding subalgebra which we denote by
${\cal H}_{\tilde{A}_1(r);\gamma}(X)$.

\begin{prop}\sloppy
 An element $\sum_w c_w w\in k(X)\big[\tW\big]_\gamma$ is a local section of~${\cal H}_{\tW}(X)$ iff its coefficients are holomorphic away from
 the reflection hypersurfaces and for any reflection~$r\in R(W)$ and any
 $w_0\in W$, the operator $\sum_{w\in \tilde{A}_1(r)w_0} c_w w$ is a section
 of the localization of~${\cal H}_{\tilde{A}_1(r);\gamma}(X) {\cal
 Z}_{w_0} w_0$ to the corresponding union of reflection hypersurfaces.
\end{prop}

\begin{proof}
 An element is in the DAHA if it is holomorphic away from the reflection
 hypersurfaces and is in the DAHA over the localization to the union of
 reflection hypersurfaces. Since the reflection hypersurfaces of~$\tilde{A}_1(r)$ for distinct $r$ remain transverse to each other even
 after specializing $q$, the corresponding conditions are independent. The~conditions along the reflection hypersurfaces of~$\tilde{A}_1(r)$ for
 special $q$ are the limit from the case for general $q$, where they agree
 with the residue conditions on~${\cal H}_{\tilde{A}_1(r);\gamma}(X)$.
 That algebra is itself a DAHA, and thus the corresponding conditions
 remain flat for special $q$ as well.
\end{proof}

One possible application of this reduction is to the construction of
degenerations: a limit of~${\cal H}_{\tW;\gamma}(X)$ should be
well-behaved as long as the limits of the various ${\cal
 H}_{\tilde{A}_1(r);\gamma}(X)$ are well-behaved and there is no further
coalescence of reflection hypersurfaces. (There are, however, technical
issues, in that most natural degenerations will break the normality of~$X$,
but one may be able to finesse this by working over $X/W$ instead, as~the
quotient will tend to remain a~weighted projective space.)

One such limit of interest is the other natural $q\to 0$ (or $q$ torsion)
limit: rather than take the limit to a twisted group algebra in which the
group does not act faithfully, one might instead take a limit to an algebra
of differential-reflection operators; i.e., take the limit in such a way as
to consider how the operators themselves actually act. This would
presumably be the correct way to interpret the algebra of
holomorphy-preserving operators if one does not suppress the poles for~$q$
torsion, but is not directly accessible via our techniques. In~particular,
the resulting algebra would not be generated by the rank~1 subalgebras, and
the corresponding spherical algebra is almost certainly not a domain when~$q$ is a nontrivial torsion element.

\section[The $C^\vee C_n$ case]{The $\boldsymbol{C^\vee C_n}$ case}\label{section7}

We now restrict our attention to the case that the affine Weyl group is of
type $C$. This has a~natu\-ral action on the family ${\cal E}^{n+1}$ given
as follows:
\begin{gather*}
\begin{array}{l}
 s_0(z_1,\dots,z_n,q/2) = (q-z_1,z_2,\dots,z_n,q/2),
 \\[.5ex]
 s_n(z_1,\dots,z_n,q/2) = (z_1,\dots,z_{n-1},-z_n,q/2),
\end{array}
\end{gather*}
while for~$1\le i\le n-1$, $s_i$ swaps $z_i$ and $z_{i+1}$. Here we denote
the last coordinate in~${\cal E}^{n+1}$ by~$q/2$, so that $q$ is twice that
coordinate. We use this notation since the corresponding family of actions
of~$\tilde{C}_n$ on~${\cal E}^n$ (and thus the resulting Hecke algebras)
depends only on~$q$, but it will be convenient (and more symmetric) to be
able to divide $q$ by $2$. We find for this action that the simple coroot
morphisms are $q/2-z_1,z_1-z_2,\dots,z_{n-1}-z_n,z_n$, and thus that this
is in fact an~action of coroot type, as~required for our theory. We also
have an action corresponding to the diagram automorphism:
\begin{gather*}
\omega(z_1,\dots,z_n) = (q/2-z_n,\dots,q/2-z_1),
\end{gather*}
which clearly permutes the simple coroot morphisms as expected.

The root curves are all isomorphic to~${\cal E}$, and the simple root
morphisms are given by
\begin{gather*}
(-1,0,\dots,0),\,
(1,-1,0,\dots,0),\,(0,1,-1,0,\dots,0),\,\dots,\,(0,\dots,0,1,-1,0),\,(0,\dots,0,1,0).
\end{gather*}
It follows that the root kernel for any finite parabolic subgroup is
trivial, so that there will be no difficulties with invariants and
(local) idempotents.

In fact, we~have the following.

\begin{prop}
 For any finite parabolic subgroup $W_I\subset \tilde{C}_n$, the quotient
 ${\cal E}^n/W_I$ is smooth over ${\cal M}_{1,1}$.
\end{prop}

\begin{proof}
 Any quotient ${\cal E}^n/W_I$ is a product of symmetric powers of~${\cal
 E}$ and a quotient or two of the form ${\cal E}^m/C_m$ (using the
 diagram automorphism to identify parabolic subgroups involving $s_0$ with
 the usual hyperoctahedral group). Symmetric powers of a curve are
 smooth, so there is no difficulty. For~the quotient by the full
 hyperoctahedral group, we~first observe that the quotient by the normal
 subgroup of order $2^m$ is just the product of~$n$ copies of the quotient
 of~${\cal E}$ by $[-1]$, a.k.a.~$\P^1$. We thus find that ${\cal
 E}^m/C_m\cong \big(\P^1\big)^m/S_m\cong \P^m$, giving smoothness as required.
\end{proof}

``Miracle flatness'' immediately gives the following, which in particular
ensures that the sphe\-ri\-cal algebras we consider will be locally free in a
suitable sense (i.e., that their direct images in~either copy of~${\cal
 E}^n/W$ are locally free).

\begin{cor}
 For any parabolic subgroup $W_I\subset C_n$, the quotient maps
 ${\cal E}^n\to {\cal E}^n/W_I$ and ${\cal E}^n/W_I\to {\cal E}^n/C_n$
 are flat.
\end{cor}

Suppose for the moment that $\vec{T}_0={}^\omega\vec{T}_n$. Then we may
consider the extended double affine Hecke algebra obtained by adjoining
$\sO_X\omega$ to~${\cal H}_{\tilde{C}_n;\vec{T}}({\cal E}^n)$. This has a
natural $\Z/2\Z$-grading, and it will be useful to think of it as
corresponding to a (sheaf) category with two objects rather than a sheaf
algebra. That is, we~have objects $0$ and $1$ with endomorphisms given by
the sheaf algebra ${\cal H}_{\tilde{C}_n;\vec{T}}({\cal E}^n)$ and the
remaining morphisms given in either direction by the sheaf bimodule ${\cal
 H}_{\tilde{C}_n;\vec{T}}({\cal E}^n)\omega$. Doing this actually lets us
remove the constraint on the parameters: we may always obtain a sheaf
category with two objects by taking
\begin{alignat*}{3}
&\Hom(0,0)= {\cal H}_{\tilde{C}_n;\vec{T}}({\cal E}^n),\qquad &&
\Hom(0,1) = {\cal H}_{\tilde{C}_n;{}^{\omega}\vec{T}}({\cal E}^n)\omega,&\\
&\Hom(1,0)= {\cal H}_{\tilde{C}_n;\vec{T}}({\cal E}^n)\omega,\qquad &&
\Hom(1,1)= {\cal H}_{\tilde{C}_n;{}^{\omega}\vec{T}}({\cal E}^n),&
\end{alignat*}
where ${\cal H}_{\tilde{C}_n;\vec{T}}({\cal E}^n)\omega$ denotes the
bimodule obtained by twisting the regular bimodule of the DAHA by
the isomorphism induced by $\omega$.

The Bruhat order on the extended affine Weyl group induces a Bruhat
filtration on this category (which agrees with the usual Bruhat filtration
coming from viewing each $\Hom$ bimodule as a regular module of some DAHA).
Although of course this filtration is useful in itself, we~will also make
great use of a filtration obtained from a much coarser order. The~category
is certainly generated by the rank~1 subalgebras along with~$\omega$, but
since $\omega$ permutes the rank~1 subalgebras, we~can actually omit the
subalgebras corresponding to~$s_0$ from the generators. As~a result, we~find that the category is generated by the two morphisms corresponding to~$\omega$ along with the Hecke algebras in each degree corresponding to the
finite Weyl group. We may then define a~filtration on each Hom bimodule by
the number of times $\omega$ was used; e.g., the degree $\le d$ piece of~$\Hom(0,d\bmod 2)$ is the image of
\begin{gather*}
 {\cal H}_{C_n;\vec{T}}(X)\omega{\cal H}_{C_n;{}^\omega\vec{T}}(X)
 \omega{\cal H}_{C_n;\vec{T}}(X) \cdots
\end{gather*}
(with $d$ copies of~$\omega$ and $d+1$ finite Hecke algebras as tensor
factors). We can express each finite Hecke algebra factor as the image of
a product of rank~1 algebras (corresponding to the longest element of~$C_n$), move all of the copies of~$\omega$ to the end, and then use
Corollary~\ref{cor:infin_hecke_twisted_interval} to express this in terms
of a Bruhat interval in the appropriate Hecke algebra. Moreover, that
Bruhat interval is clearly a union of~$(W,W)$ double cosets, so is
determined by the corresponding set of~dominant weights; we find that the
condition to be degree $\le d$ is simply that the first coefficient of the
dominant weight is $\le d/2$.

This leads to a ``graded'' (or ``compactified'') version of the extended
DAHA: take the sheaf category with objects $\Z$ generated by elements
$\omega\in \Hom(k,k+1)$ and algebras ${\cal H}_{C_n;\vec{T}}(X)\subset
\Hom(2k,2k)$, ${\cal H}_{C_n;{}^{\omega}\vec{T}}(X)\subset
\Hom(2k+1,2k+1)$. (We may think of this as a sort of Rees algebra
corresponding to the filtration, taking into account parity.)

Note that when asking whether two (small) sheaf categories are isomorphic,
there are actually two natural notions. The~issue here is that for any
automorphism of an object of a category, there is a corresponding inner
automorphism of the category as a whole (and more generally for any
assignment of an automorphism to each object). In~our case, any unit on
parameter space gives such an inner automorphism, and since we are
primarily interested in the individual fibers, we~should extend that to
allow {\em local} units. This leads to a notion of twisting objects by
line bundles on the base; note that the resulting sheaf category will still
be locally isomorphic to the original sheaf category. In~general, we~will
often only state that given sheaf categories are isomorphic locally on the
base; this is mainly to save the bookkeeping effort of determining
precisely which line bundles one needs to twist by to make the isomorphism
global. In~particular, when specifying line bundles and equivariant
gerbes, the $z$-independent terms of the polarization and weight will be
largely irrelevant, so we can make the simplest consistent choice without
having to worry too much about which choice would make later polarizations
simpler.

Of course, we~have yet to incorporate a twisting datum. We first note that
the underlying equivariant gerbe induces a cocycle valued in pairs of
polarizations and weights. We can embed the $\tilde{C}_n$-module of
polarizations in the degree 2 subspace of~$(1/q)\Q[\vec{z},q,\vec{\pi}]$
(where $\vec{\pi}$ corresponds to additional factors of~${\cal E}$ which we
include to allow some room for further parameters), and similarly for the
$\tilde{C}_n$-module of weights. Since $q$ is $\tilde{C}_n$-invariant,
both rational $\tilde{C}_n$-modules are isomorphic to the corresponding
module $\Q[\vec{z},\vec{\pi}]$ obtained by specializing $q=1$. We can
compute~$H^1$ of this module by restriction to the (finite index)
translation subgroup, where the filtration by degree makes it easy to
verify that $H^1$ is trivial. It follows that the given cocycle is the
coboundary of a pair
$(p_3(\vec{z},q,\vec{\pi})/q,p_1(\vec{z},q,\vec{\pi})/q)$, where $p_3$ and
$p_1$ are homogeneous polynomials of degree $3$ and $1$ respectively. Of
course, $p_3$ and $p_1$ are only determined modulo $0$-cocycles, but those
are again easily seen to be just the polynomials independent of~$\vec{z}$.

In our case, since our primary interest is in the spherical algebras, we~want the twisting datum to be trivial on~$C_n$. The~same must in
particular hold for the equivariant gerbe, so that~$p_3$ and $p_1$ must be
$C_n$-invariant polynomials. (More precisely, they must be $C_n$-invariant
modulo $0$-cocycles, but we can then average over $C_n$ without changing
the coboundary.) Since we are allowed to ignore terms independent of~$\vec{z}$, we~see that we may as well take $p_1=0$ and
$p_3=\frac{\lambda(q/2,\vec{\pi})}{q} \sum_i \frac{z_i^2}{2}$ for some
linear polynomial $\lambda$ with rational coefficients. Imposing the
condition that the coboundary consist of actual polarizations then forces
$\lambda$ to have integer coefficients. That is, the value of the
coboundary at $s_0$ is $p_3(\vec{z})-{}^{s_0}p_3(\vec{z}) =
\lambda(q/2,\vec{\pi})z_1 - \lambda(q/2,\vec{\pi})\frac{q}{2}$, which is
integral iff $\lambda$ is integral since $q/2$ is a variable.

To extend this to a twisting datum, it suffices to choose a meromorphic
section of the restriction to~$\langle s_0\rangle$ and represent the
corresponding Cartier divisor as $D_0-{}^{s_0}D_0$ with~$D_0$ transverse to~$[X^{s_0}]$. Since $D_0$ is only determined modulo $s_0$-invariant
divisors and linear equivalence, it is equivalent to specify the polarization\vspace{-.5ex}
\begin{gather*}
k\frac{z_1(z_1-q)}{2} + \frac{\lambda(q/2,\vec{\pi}) z_1}{2},
\end{gather*}
with $k\in \{0,1\}$. (Again, the constant term is irrelevant for our
purposes.) Note that this imposes a stronger integrality constraint on~$\lambda$, which must now have even coefficients.

Of course, we~saw above when discussing the elliptic Gamma ``function''
that not every suitably integral cocycle has meromorphic sections, even for
the translation subgroup. Luckily, in our case, we~can easily write down
explicit products of elliptic Gamma functions that do the trick. To~be
precise, consider the product
\begin{gather*}
\prod_{1\le i\le n} \Gamq(a\pm z_i),
\end{gather*}
where $a$ is linear in~$q/2$ and $\pi$. Here and below, we~have used the
shorthand notation that multiple arguments to~$\Gamq$ or $\vartheta$
represent a product and the appearance of~$\pm$ in the argument means that
{\em both} signs should be used; thus $\Gamq(a\pm
z_i)=\Gamq(a+z_i)\Gamq(a-z_i)$.

This is $C_n$-invariant, and has polarization (ignoring the $z$-independent
term)
\begin{gather*}
(2a-q)\sum_i \frac{z_i^2}{2}.
\end{gather*}
Moreover, we~find
\begin{gather*}
\frac{{}^{s_0}\prod_{1\le i\le n} \Gamq(a\pm z_i)}
 {\prod_{1\le i\le n} \Gamq(a\pm z_i)}
=
\frac{\vartheta(a-z_1)}{\vartheta(a-q+z_1)},
\end{gather*}
and thus this corresponds to the twisting datum with~$D_0=[z_1=a]$.
This gives the general degree 1 case, and we may obtain the general degree
0 case by taking a ratio of two such products.

Such a choice of product induces an embedding of the twisted Hecke algebra
in the algebra of~meromorphic difference-reflection operators;
equivalently, meromorphic sections of the twisted Hecke algebra act on
formal functions of the form $f\gamma$, where $\gamma$ is the product of~$\Gamq$ symbols. This lets us extend the holomorphy-preserving property to
the twisted case: if $f$ is locally holomorphic away from the poles of the
sections of the cocycle corresponding to~$\gamma$, then so its image.
This leads to issues where $\gamma$ has poles, but we have enough choice in
how we represent things that the corresponding invariant localizations give
a finite covering. $\big($I.e., in the degree~0 case, we~may take
$\gamma=\prod_{1\le i\le n} \Gamq(v+a\pm z_i)/\Gamq(v\pm z_i)$ with~$v$
varying; in degree 1, we~take $\gamma=\prod_{1\le i\le n} \Gamq(v\pm
z_i,w\pm z_i)/\Gamq(v+w-a\pm z_i)$ with~$v$, $w$ varying.$\big)$

To include the diagram automorphism, we~note that because we extended to a
category above, we~may choose a different product of~$\Gamq$ symbols for
each object, and need only have a~line bundle ${\cal L}_i$ in each degree
such that sections of~${\cal L}_i\omega$ take $k(X)\gamma_i$ to~$k(X)\gamma_{i+1}$. This is easy enough, since
\begin{gather*}
\frac{{}^\omega\prod_{1\le i\le n} \Gamq(a\pm z_i)}
 {\prod_{1\le i\le n} \Gamq(a-q/2\pm z_i)}=
\prod_{1\le i\le n} \vartheta(a-q/2-z_i),
\\
\frac{{}^\omega\prod_{1\le i\le n} \Gamq(a\pm z_i)}
 {\prod_{1\le i\le n} \Gamq(a+q/2\pm z_i)}=
\prod_{1\le i\le n} \frac{1}{\vartheta(a-q/2+z_i)}.
\end{gather*}

In the degree 0 case, this ends up not changing the twisting datum, but in
degree 1, the polarization of~${\cal L}_i$ ends up alternating between
positive and negative. This would increase the number of cases we would
need to consider, so it will be convenient to enlarge the category even
further by replacing each object by a sequence of objects, one for each
(isomorphism class of) line bundle on~$\P^n$. The~$\Hom$ bimodule between
two objects in the enlarged category is then just the twist of the original
$\Hom$ bimodule by the pair of line bundles (inverting the one on the
domain side). The~benefit of this is that we can move between the two
degree~1 cases by~twisting by $\sO_{\P^n}(\pm 1)$, and thus in the enlarged
category, there is only one case.

It turns out that even the above category is not quite general enough to
include everything we want to do for the spherical algebras. As~a result, we~will first focus our attention on the case in which the endpoints do not
have any parameters assigned. We also specialize to the usual
Macdonald-ish case, in which $T_i$ for~$1\le i\le n-1$ is the divisor
$t+z_i-z_{i+1}=0$. (The~definitions work more generally; in particular, we~note that everything works mutatis mutandum for the case $T_i=0$ for $1\le
i\le n-1$, though oddly enough this ``$t$-free'' case turns out to be
harder to understand in a number of respects.) It also turns out that most
of the symmetries of the algebra do not preserve the untwisted case, but
{\em do} preserve a particular twist; this leads us to make a somewhat
odd-appearing choice in parametrizing twists. The~specific basis we use
for the $\Z^2$ of objects is inspired by~\cite{noncomm1} (where it in turn
came from the geometry of rational surfaces); this also informs the choice
of parameters (in~\cite{noncomm1}, there was a parameter $\eta$ which we~have arranged to be 0).

\begin{defn}
 The even elliptic DAHA ${\cal H}^{(n)}_{\eta';q,t}$ (of type $C$) is the
 smallest sheaf category on~${\cal E}^{n+3}$ with objects $\Z\langle
 s,f\rangle$ such that
 \begin{gather*}
 {\cal H}^{(n)}_{\eta';q,t}(ds+d'_1f,ds+d'_2f) = \sO_{\P^n}(d'_2)\otimes {\cal
 H}_{C_n;\vec{T}_t}\big({\cal E}^{n+3}\big)\otimes \sO_{\P^n}(-d'_1)
 \end{gather*}
 and
 \begin{gather*}
 {\cal H}^{(n)}_{\eta';q,t}(ds+d'_1f,(d+1)s+d'_2f)\supset {\cal
 L}\omega,
 \end{gather*}
 where ${\cal L}$ is the line bundle with polarization
 \begin{gather*}
 -(d'_1-d'_2+1)\sum_i z_i^2
 -((n-1)t+\eta'+(d-d'_1+1)q)\sum_i z_i
 +(d-d'_1)q^2/4.
 \end{gather*}
 The odd elliptic DAHA ${\cal H}^{\prime(n)}_{x_0;q,t}$ is
 defined similarly, except that ${\cal L}$ has polarization
 \begin{gather*}
 -(2d'_1-2d'_2+3)\sum_i \frac{z_i^2}{2}
 -((n-1)t+x_0+(3d-2d'_1+2)q/2)\sum_i z_i
 +(3d-2d'_1)q^2/8.
 \end{gather*}
\end{defn}

\begin{rem}
 In this case, we~were able to choose the $z$-independent part of the
 polarization to make everything globally consistent, where $\sO_{\P^n}(1)$
 is chosen so as to pull back to the line bundle with polarization~$\sum_i
 z_i^2$. We can recover the usual (compactified) elliptic DAHA by
 restricting the algebra ${\cal H}^{(n)}_{-(n-1)t-q;q,t}$ to the subset
 $\Z(s+f)$ of objects; if we restrict to even multiples and then invert
 the sections $1\in \Hom(d(s+f),(d+2)(s+f))$, we~recover the
 uncompactified elliptic DAHA. We should note that while ${\cal
 H}^{(n)}_{\eta';q,t}(0,2s+2f)$ always contains 1, this is only true
 locally on the base for~${\cal H}^{\prime(n)}_{x_0;q,t}(0,2s+3f)$, and it
 is not possible to fix this without also changing the particular
 representative of~$\sO_{\P^n}(1)$.
\end{rem}

We similarly let ${\cal S}^{(n)}_{\eta';q,t}$ and ${\cal
 S}^{\prime(n)}_{x_0;q,t}$ denote the corresponding spherical categories
(i.e., replacing each $\Hom$ bimodule by the appropriate subquotient).
Each $\Hom$ bimodule in one of these categories is a sheaf bimodule on the
quotient $\P^n={\cal E}^n/C_n$, every local section of which is a~meromorphic difference operator on~${\cal E}^n$.

\begin{prop}
 The subsheaf corresponding to any Bruhat order ideal in either ${\cal
 S}^{(n)}_{\eta';q,t}$ or~${\cal S}^{\prime(n)}_{x_0;q,t}$ is a coherent sheaf
 bimodule on~$\P^n\times_{{\cal E}^3} \P^n$, and the direct image in
 either $\P^n$ is locally free.
\end{prop}

\begin{proof}
 This reduces to the corresponding statement for the subquotients in the
 Bruhat filtration, each of which comes from a line
 bundle on the quotient by a parabolic subgroup of~$C_n$, and is therefore
 flat on the quotient $\P^n$.
\end{proof}

Note that just as for more general Coxeter groups, we~can describe the
$t$-independent conditions on sections of these sheaf categories in terms
of residue conditions (at least generically). This does not {\em quite}
follow from that case, as~we are working in an extended affine Weyl group
rather than an affine Weyl group, and of course the twisting makes things
trickier.

For ${\cal H}^{(n)}_{-q-(n-1)t;q;t}(0,ds+d'f)$, ${\cal
 H}^{\prime(n)}_{-q-(n-1)t;q,t}(0,ds+d'f)$, the local sections are just
local sections of the untwisted Hecke algebra up to twisting on the left by
some $\sO_{\P^n}(l)$, and thus satisfy the usual residue conditions: the
coefficients are all meromorphic sections of the same bundle
$\sO_{\P^n}(1)$, and for any two elements of the extended affine Weyl group
related by a reflection, the sum of the corresponding coefficients must be
holomorphic along the reflection hypersurface. (Moreover, the polar
divisors of the coefficients must be sums of reflection hypersurfaces, with
a given reflection hypersurface appearing only if it appears in a residue
condition). For~the spherical algebra, the condition is analogous, giving
a residue condition for any pair of distinct translations which are
conjugates by a reflection.

In our case, the twisting is simple enough that we can express the residue
conditions explicitly even in the presence of twisting. For~the even case, we~find that for sufficiently general $u$, $v$,
\begin{gather*}
\bigg(\prod_{1\le i\le n}\vartheta(v\pm z_i)^{d'}
 \Gamq(-dq/2-u\pm z_i,-dq/2-(n-1)t-\eta'+u\pm z_i)\bigg)^{-1}
 \\ \hphantom{\bigg(\prod_{1\le i\le n}}
{}\times{\cal H}^{(n)}_{\eta';q,t}(0,ds+d'f)
\prod_{1\le i\le n} \Gamq(-u\pm z_i,-(n-1)t-\eta'+u\pm z_i)
\end{gather*}
consists of operators with elliptic coefficients and preserving holomorphy
(away from divisors depending on~$u$ and $v$). (This is essentially just
performing an elementary transformation as considered below.) Thus the
residue conditions on~${\cal H}^{(n)}_{\eta';q,t}(0,ds+d'f)$ may be
obtained by gauging the untwisted residue conditions. We may then use
Propositions~\ref{prop:pullback_simplification_i} and
\ref{prop:pullback_simplification_ii} to~simplify the~result. We give the
conditions for the spherical category for simplicity; for the DAHA, the
correction is the same except for those reflections preserving the
corresponding coset, when the~correction factor is necessarily trivial.

We find that sections of~${\cal S}^{(n)}_{\eta';q,t}(0,ds+d'f)$ satisfy the
following $t$-independent residue conditions (in addition to~$C_n$-invariance and the corresponding bounds on poles). Write such a
local section as $\sum_{\vec{k}} c_{\vec{k}}(\vec{z}) \prod_i T_i^{k_i}$,
where each $k_i$ is congruent to~$d/2$ modulo $\Z$.
First, for~$m\ne k_1+k_2$,
\begin{gather*}
c_{m-k_2,m-k_1,k_3,\dots,k_n}(\vec{z})\,{+}
\biggl[\frac{\vartheta(u,q\,{+}\,(n-1)t\,{+}\,\eta'\,{+}\,u\,{+}\,z_1\,{+}\,z_2\,{+}\,mq)}
 {\vartheta(u\,{+}\,z_1\,{+}\,z_2\,{+}\,mq,q\,{+}\,(n-1)t\,{+}\,\eta'\,{+}\,u)}\biggr]^{2(m-k_1-k_2)}
\!\!\!\!\!c_{k_1,k_2,k_3,\dots,k_n}(\vec{z})
\end{gather*}
is holomorphic along $z_1+z_2+mq=0$, where $u$ is arbitrary subject to~$u\ne 0$ and $u\ne -q-\allowbreak(n-1)t-\eta'$ and otherwise has no effect on the
condition, and similarly when~$m\ne 2k_1$,
\begin{gather*}
c_{m-k_1,k_2,\dots,k_n}(\vec{z})
+
\biggl[
\frac{\vartheta(u,q+(n-1)t+\eta'+u+2z_1+mq)}
 {\vartheta(u+2z_1+mq,q+(n-1)t+\eta'+u)}
\biggr]^{m-2k_1}
c_{k_1,k_2,\dots,k_n}(\vec{z})
\end{gather*}
is holomorphic along $2z_1+mq=0$. (The conditions along other reflection
hyperplanes follow by~$C_n$-invariance.) Moreover, for~generic parameters,
any operator satisfying these conditions and the appropriate $t$-dependent
vanishing conditions will be a section of the given $\Hom$ sheaf. (To be
precise, we~must assume that $q$ is not torsion (or we are in a faithful
Bruhat interval) and that $t$ is not a multiple of~$q$; of course, by
flatness, to check that a family of operators is a section of the
corresponding family of categories ${\cal S}$, it suffices to verify that
this holds generically.)

Similarly, for~${\cal S}^{\prime(n)}_{x_0;q,t}(0,ds+d'f)$, the correction
factors in the residue conditions are
\begin{gather*}
\biggl[
\frac{\vartheta(u,q+(n-1)t+x_0+u+z_1+z_2+mq)}
 {\vartheta(u+z_1+z_2+mq,q+(n-1)t+x_0+u)}
\biggr]^{2(m-k_1-k_2)}
\end{gather*}
and
\begin{gather*}
\biggl[
{\mathfrak q}(z_1+mq/2)
\frac{\vartheta(u,q+(n-1)t+x_0+u+2z_1+mq)}
 {\vartheta(u+2z_1+mq,q+(n-1)t+x_0+u)}
\biggr]^{m-2k_1},
\end{gather*}
where we recall that in characteristic not 2, ${\mathfrak q}\in
\mu_2(E[2])$ is the function taking $0$ to~$1$ and nontrivial $2$-torsion
points to~$-1$. (Note that without the appearance of~${\mathfrak q}$ in
the residue conditions, the subcategory ${\cal S}'$ with objects $2\Z s+\Z
f$ would be a simple reparametrization of the corresponding subcategory of~${\cal S}$.)

By comparison with the univariate case (which we will discuss in more
detail shortly), we~are led to define generalizations (``blowups'') of
these algebras with even more parameters. Recall that local sections of
the spherical algebras are difference operators. We let $T_i$ denote the
operator that pulls back through~$z_i\mapsto z_i+q$; note that in terms of
our convention for how group elements act, $T_i$ is the same as the action
of the translation~$z_i\mapsto z_i-q$. We extend this to half-integer
powers of~$T_i$ by using the chosen $q/2$. Every local section is then a
left linear combination of monomials $T^{\vec{k}}:=\prod_i T_i^{k_i}$ in
which all $k_i$ are half-integers in the same coset of~$\Z$ (determined by
the coefficient of~$s$ in the degree of the operator in the category).

\begin{defn}
 The sheaf category ${\cal S}^{(n)}_{\eta',x_1,\dots,x_m;q,t}$ is the sheaf
 category on~$\P^n/{\cal E}^{m+3}$ with objects $\Z\langle
 s,f,e_1,\dots,e_m\rangle$ defined by taking ${\cal
 S}^{(n)}_{\eta',x_1,\dots,x_m;q,t}(d_1s+d'_1f-r_{11}e_1-\cdots-r_{1m}e_m,d_2s+d'_2f-r_{21}e_1-\cdots-r_{2m}e_m)$
 to be the subsheaf of~${\cal S}^{(n)}_{\eta';q,t}(d_1s+d'_1f,d_2s+d'_2f)$
 consisting (locally) of operators ${\cal D}$ such that the left
 coefficient of~$\prod_{1\le i\le n} T_i^{k_i}$ vanishes on
 the divisors $z_i=x_j-(2l-d_2+1)q/2$ for~$1\le i\le n$, $1\le j\le m$,
 $k_i+(d_2-d_1)/2+r_{1j}\le l<r_{2j}$ and the divisors
 $z_i=-x_j+(2l-d_2+1)q/2$ for~$1\le i\le n$, $1\le j\le m$,
 $-k_i+(d_2-d_1)/2+r_{1j}\le l<r_{2j}$. Similarly, the sheaf category ${\cal
 S}^{\prime(n)}_{x_0,x_1,\dots,x_m;q,t}$ is the subsheaf category of~${\cal
 S}^{\prime(n)}_{x_0;q,t}$ satisfying the same vanishing conditions.
\end{defn}

\begin{rem}
 Note that by symmetry, it would have sufficed to impose the first set of
 vanishing conditions.
\end{rem}

Note that in this definition, we~need to impose the vanishing conditions on
the family as a~whole; on individual fibers, the condition along individual
divisors may be too strong (when the divisors are reflection hypersurfaces)
or too weak (when the divisors are not distinct). As~a result, it is a
nontrivial question whether the family is flat, and even when it is flat,
it is conceivable that the individual fibers might fail to inject in the
category of meromorphic operators (which could in turn allow the fibers to
acquire zero divisors).

Luckily, since we defined these using the spherical algebra of an elliptic
DAHA, there is an~obvi\-ous approach to studying these categories: construct
them as spherical subquotients of a~suitable subcategory of~${\cal
 H}^{(n)}_{\eta';q,t}$ or ${\cal H}^{\prime(n)}_{x_0;q,t}$. There are
actually multiple choices one might make here, as~one can view a divisor
$z_i=x_j-kq/2$ as a pullback from the coroot curve associated to either
endpoint. Although it might seem natural to choose the endpoint that
matches the parity of~$k$, it will be easier for present purposes to
consistently use $s_0$. Other choices may lead to more natural Hecke
algebras, however; for instance the classical double affine Hecke algebra
of type $C^\vee C_n$ corresponds to assigning two parameters to~$s_n$ and
two parameters to~$s_0$.

Since we are keeping the parameters away from the finite Weyl group, the
vanishing conditions should be appropriately $C_n$-equivariant, and thus
the vanishing condition for the left coefficient of~$w$ should depend only
on~$w C_n$ and be equivariant on the left. Each coset of~$wC_n$ has a
unique representative $\prod_i T_i^{k_i}$, and we readily verify that the
vanishing conditions we imposed above transform well under $C_n$. We may
thus define ${\cal H}^{(n)}_{\eta',x_1,\dots,x_m;q,t}$ by imposing the
resulting vanishing conditions, and similarly for~${\cal
 H}^{\prime(n)}_{x_0,x_1,\dots,x_m;q,t}$.

We note that in addition to the functoriality on fibers implied by the fact
that this is defined over the moduli stack, we~also have functoriality with
respect to translation by $2$-torsion.

\begin{prop}
 If $\tau$ is an fppf-local section of~${\cal E}[2]$, then there are
 isomorphisms
 \begin{gather*}
 {\cal H}^{(n)}_{\eta',x_1,\dots,x_m;q,t} \cong {\cal H}^{(n)}_{\eta',x_1+\tau,\dots,x_m+\tau;q,t},
 \\
 {\cal H}^{(n)}_{x_0,x_1,\dots,x_m;q,t} \cong {\cal H}^{(n)}_{x_0+\tau,x_1+\tau,\dots,x_m+\tau;q,t}.
 \end{gather*}
\end{prop}

\begin{proof}
 Conjugating by the involution~$(z_1,\dots,z_n)\mapsto
 (z_1+\tau,\dots,z_n+\tau)$ induces such an isomorphism on the generators
 for~$m=0$ and acts as described on the additional vanishing conditions
 for~$m>0$.
\end{proof}

\begin{rems}
 In fact, we~could have defined these categories over the moduli stack of
 hyperelliptic curves of genus 1, at the cost of making the action of~$C_n$ slightly more complicated (with $s_n$ acting by $z_n\mapsto
 s(z_n)$, where $s$ is the hyperelliptic involution), in which case these
 isomorphisms follow by functoriality. This would have made a number of
 later formulas more complicated, as~well as making it more difficult to
 discuss line bundles and $\Gamq$ symbols.
\end{rems}

\begin{rems}
 The above isomorphism involves translation by $\tau$ in every degree. If~one instead only translates in the odd degrees, none of the parameters
 visible in the notation change, but~$q/2$ is replaced by $q/2+\tau$. In~particular, the algebra is indeed independent of the choice of~$q/2$.
\end{rems}

We have the following useful ``elementary transformation'' symmetry. Here
and below, we~simplify things by observing that our sheaf categories
satisfy a natural translation symmetry in which translation in the group of
objects corresponds to translations of the parameters by multiples of~$q/2$, and thus it suffices to consider $\Hom$ bimodules starting at the
$0$ object.

\begin{prop}
 There are $($locally on the base$)$ natural isomorphisms
 \begin{gather*}
 {\cal H}^{\prime(n)}_{x_0,x_1,x_2,\dots,x_m;q,t}(0,ds+d'f-r_1e_1-\cdots-r_me_m)\\
 \qquad{} \cong \prod_{1\le i\le n} \Gamq((r_1+(1-d)/2)q-x_1\pm z_i)
 \\ \qquad\quad
 {}\times{\cal H}^{(n)}_{x_0-x_1,-x_1,x_2,\dots,x_m;q,t}(0,ds+(d'-r_1)f-(d-r_1)e_1-r_2e_2-\cdots-r_me_m)
 \\ \qquad\quad
 {}\times \prod_{1\le i\le n} \Gamq(q/2-x_1\pm z_i)^{-1}
 \end{gather*}
 and
 \begin{gather*}
 {\cal H}^{(n)}_{\eta',x_1,x_2,\dots,x_m;q,t}(0,ds+d'f-r_1e_1-\cdots-r_me_m)\\
 \qquad{}
 \cong \prod_{1\le i\le n} \Gamq((r_1+(1-d)/2)q-x_1\pm z_i)
 \\ \qquad\quad
 {}\times {\cal H}^{\prime(n)}_{\eta'-x_1,-x_1,x_2,\dots,x_m;q,t}(0,ds+(d+d'-r_1)f-(d-r_1)e_1-r_2e_2-\cdots-r_me_m)
 \\ \qquad\quad
 {}\times
 \prod_{1\le i\le n} \Gamq(q/2-x_1\pm z_i)^{-1}.
 \end{gather*}
\end{prop}

\begin{proof}
 It is easy to see that the corresponding categories of meromorphic
 operators are isomorphic (locally on the base), and the
 pseudo-conjugation by $\Gamq$ symbols respects the twisting data, so the
 image of the right-hand side in the meromorphic category corresponding to
 the left-hand side satisfies the same conditions on the reflection
 hypersurfaces. The~conditions along the divisors corresponding to~$x_i$
 for~$2\le i\le m$ are clearly the same, and it is straightforward to
 check that the same holds for~$i=1$.
\end{proof}

\begin{rem}
 Since the definition is also clearly invariant under permutations of~$x_1,\dots,x_m$, we~may apply this symmetry in any $x_i$, and then by
 composition in any subset of the $x_i$. If~the subset has odd size, then
 the isomorphism switches between ${\cal H}^{(n)}$ and ${\cal
 H}^{\prime(n)}$, while even subsets induce isomorphisms between
 extended DAHAs of the same parity. In~particular, for~each of the two
 families, there is an action of~$W(D_m)$ on the parameter space that
 extends to an action on the family of sheaf categories.
\end{rem}

If we can show that these sheaf categories have well-behaved Bruhat
filtrations (i.e., in which the subquotients are obtained from the $m=0$
case by imposing the vanishing conditions), then the same will immediately
hold for their spherical subquotients. Unfortunately, we~cannot simply
copy the arguments we used in the usual elliptic Hecke algebra case;
although the upper bound on the Bruhat subquotients works the same way, the
category structure means there is no longer a canonical way to associate a
multiplication map to a reduced word.

There are some special cases which are easy, however. As~before, we~may
restrict our attention to~$\Hom$ bimodules starting from the $0$ object.

The following is a trivial consequence of the definition.

\begin{prop}
 If $r_m\le 0$, then
 \begin{gather*}
 {\cal H}^{(n)}_{\eta',x_1,\dots,x_m;q,t}(0,ds+d'f-r_1e_1-\cdots-r_me_m)
 \\ \phantom{{\cal H}^{(n)}_{\eta',x_1,\dots,x_m;q,t}}
 {}= {\cal H}^{(n)}_{\eta',x_1,\dots,x_{m-1};q,t}(0,ds+d'f-r_1e_1-\cdots-r_{m-1}e_{m-1}),
 \\
 {\cal H}^{\prime(n)}_{x_0,x_1,\dots,x_m;q,t}(0,ds+d'f-r_1e_1-\cdots-r_me_m)
 \\ \phantom{ {\cal H}^{\prime(n)}_{x_0,x_1,\dots,x_m;q,t}}
 {}= {\cal H}^{\prime(n)}_{x_0,x_1,\dots,x_{m-1};q,t}(0,ds+d'f-r_1e_1-\cdots-r_{m-1}e_m).
 \end{gather*}
\end{prop}

Applying the elementary transformation symmetry in~$x_m$ means that the
case $r_m\ge d$ is also straightforward to deal with. There is one more
case which is nice.

\begin{prop}
 There is a system of parameters $\vec{T}$ and a twisting datum $\gamma$
 such that each $\Hom$ bimodule
 \begin{gather*}
 {\cal H}^{(n)}_{\eta',x_1,\dots,x_m;q,t}(0,d(2s+2f-e_1-\cdots-e_m))
 \end{gather*}
 is equal to the corresponding Bruhat interval in~${\cal
 H}_{\tilde{C}_n;\vec{T};\gamma}\big({\cal E}^{n+m+3}\big)$,
 and similarly for
 \begin{gather*}
 {\cal H}^{\prime(n)}_{x_0,x_1,\dots,x_m;q,t}(0,d(2s+3f-e_1-\cdots-e_m)).
 \end{gather*}
\end{prop}

\begin{proof}
 We can rephrase the vanishing conditions for generic parameters as
 stating that
 \begin{gather*}
 \prod_{\substack{1\le i\le n\\1\le j\le m}} \Gamq(q/2-x_j\pm z_i)
 {\cal D}
 \prod_{\substack{1\le i\le n\\1\le j\le m}} \Gamq(q/2-x_j\pm z_i)^{-1}
 \end{gather*}
 has holomorphic coefficients. The~cocycle in Cartier divisors associated
 to this product of~$\Gamq$ symbols is precisely the right form to come
 from a system of parameters (associated to the orbit of~$s_0$).
\end{proof}

It turns out that it suffices to understand these cases.

\begin{thm}\label{thm:CCn_DAHA_strongly_flat}
 For any vector $v=ds+d'f-r_1e_1-\cdots-r_me_m$, the bounds on the Bruhat
 subquotients of~${\cal H}^{(n)}_{\eta',x_1,\dots,x_m;q,t}(0,v)$ and
 ${\cal H}^{\prime(n)}_{x_0,x_1,\dots,x_m;q,t}(0,v)$ coming from the
 vanishing conditions are saturated. In~particular, both sheaf categories
 are locally free, and the map from any $\Hom$ bimodule to the sheaf
 bimodule of meromorphic operators is injective on fibers.
\end{thm}

\begin{proof}
 We first use the elementary transformation symmetry to replace all the
 cases with~$r_i\ge d$ with cases with~$r_i\le 0$ (possibly changing the
 parity), and thus with~$r_i=0$. In~particular, we~observe that the cases
 $d\in \{0,1\}$ of the theorem reduce to the subcase $r_1=\cdots=r_m=0$,
 and thus to the original elliptic DAHA, where we certainly have
 saturation. For~$d>1$, if some $r_i=0$, we~can simply omit that
 parameter and thus reduce to a case with smaller $m$. We thus find that
 it suffices to prove saturation when $0<r_1,\dots,r_m<d$.

 We consider the even case ${\cal H}^{(n)}$, with the odd case ${\cal
 H}^{\prime(n)}$ being entirely analogous. Let
 $C_m:=2s+2f-e_1-\cdots-e_m$, and suppose by induction that we have
 saturation for~$v-C_m$. It will then suffice to show that the image of the
 multiplication map
 \begin{gather*}
 {\cal H}^{(n)}_{\eta',x_1,\dots,x_m;q,t}(v-C_m,v)
 \otimes
 {\cal H}^{(n)}_{\eta',x_1,\dots,x_m;q,t}(0,v-C_m)
 \to
 {\cal H}^{(n)}_{\eta',x_1,\dots,x_m;q,t}(0,v)
 \end{gather*}
 saturates the bound. Since the first factor is a Bruhat interval in an
 honest elliptic DAHA, it in particular has a global section~$1$, giving
 an inclusion
 \begin{gather*}
 {\cal H}^{(n)}_{\eta',x_1,\dots,x_m;q,t}(0,v-C_m)
 \subset
 {\cal H}^{(n)}_{\eta',x_1,\dots,x_m;q,t}(0,v),
 \end{gather*}
 identifying the former with the Bruhat interval of the latter
 corresponding to the dominant weight $(d/2-1,d/2-1,\dots,d/2-1)$. The~bounds on the subquotients are clearly the same, and thus we have
 saturation for any $w$ in this interval.

 For the rest of the module, we~note that both sides are bimodules over
 the finite Hecke algebra, and since the $x_i$ parameters have no effect
 on this algebra, we~may reduce to the case of a Bruhat interval $[\le w]$
 with~$w$ a minimal representative of~$C_n\setminus \tilde{C}_n$. Let
 $\lambda(w)$ be the corresponding dominant weight. We have already shown
 that the leading coefficient map is saturated when $\lambda(w)_1\le d/2-1$,
 so without loss of generality may assume $\lambda(w)_1=d/2$. It~then follows from the structure of minimal coset representatives
 that not only is $s_0w<w$, but $\lambda(s_0w)$ is obtained from
 $\lambda(w)$ by reducing some coefficient from $d/2$ to~$d/2-1$.
 This changes the bound on the leading coefficient bundle by precisely the
 leading coefficient divisor of~$s_0$ in the relevant DAHA, and thus
 gives saturation as required.
\end{proof}

Passing to the spherical subquotient gives us the following.

\begin{cor}
 For any vector $v=ds+d'f-r_1e_1-\cdots-r_me_m$, the bounds on the Bruhat
 subquotients of~${\cal S}^{(n)}_{\eta',x_1,\dots,x_m;q,t}(0,v)$ and
 ${\cal S}^{\prime(n)}_{x_0,x_1,\dots,x_m;q,t}(0,v)$ coming from the
 vanishing conditions are saturated. In~particular, both sheaf categories
 are locally free, and the map from any $\Hom$ bimodule to the sheaf
 bimodule of meromorphic difference operators is injective on fibers.
\end{cor}

As in the general DAHA situation, we~would like to understand the centers
of the extended DAHAs and their spherical algebras. There is a technical
issue, however, in that the notion of center is not quite well-defined for
general categories. More precisely, one might be inclined to~define the
center of a category to be the endomorphism algebra of the identity
functor, but this is too small in our cases. We could fix the ``too
small'' problem by instead taking the center to be the category with
objects the automorphisms of the original category and morphisms given by
natural transformations, but this is too large to be useful. Luckily,
there is a happy medium: we take a suitable action of~$\Z^{m+2}$ on the
category such that the induced action on objects is just the usual action
by translation.

This action is mostly straightforward: since every $\Hom$ space of degree
$2s+2f$ (or $2s+3f$ in the odd case, at least locally on the base) contains
the operator $1$, and similarly for any $\Hom$ space of degree $e_i$, any
element of the sublattice $\langle 2s+2f,e_1,\dots,e_m\rangle$ induces an
isomorphism between corresponding categories. Translation by $f$ is
slightly more subtle, but corresponds to twisting by the equivariant bundle
$\pi^*\sO_{\P^n}(1)$, so again induces an isomorphism of categories. The~result is index $2$ in~$\Z^{m+2}$, which is good enough for many purposes,
but since we would like to be able to identify the center of the extended
DAHA with one of our spherical algebras, we~need to fill out the lattice.
This is not too difficult in the extended DAHA, which contains a unique
``element'' of the form ${}^w{\cal L} w\omega$ with~$w\in C_n$ and
$w\omega$ a translation. This is invertible, and again induces appropriate
isomorphisms of categories. (Note, however, that its square is not
actually compatible with the lattice of automorphisms already described.)

Now, suppose $q$ is torsion, of order $l$. If~we take the $l$th powers of
the above isomorphisms, then all of the categories that arise can be
identified with the original category, and we thus obtain the desired
family of automorphisms (which became consistent after taking $l$th powers, as~the error was translation by $q$). Moreover, the automorphism
corresponding to~$ls$ of the extended DAHA is $W$-invariant, and thus
extends to an automorphism of the spherical algebra. We define the
``center'' of the extended DAHA or its spherical algebra to be the category
with objects $l\Z^{m+2}$ arising from the above construction.

The following is straightforward.

\begin{lem}
 If $q=0$, then the categories ${\cal S}^{(n)}$ and ${\cal S}^{\prime(n)}$
 are naturally isomorphic to their centers.
\end{lem}

More generally, let $\pi_q\colon E\to E'$ be the $l$-isogeny with kernel
generated by $q$.

\begin{thm}\label{thm:CCn_center}
 The centers of~${\cal H}^{(n)}_{\eta',x_1,\dots,x_m;q,t}$ and ${\cal
 S}^{(n)}_{\eta',x_1,\dots,x_m;q,t}$ are naturally isomorphic to the
 pullback of~${\cal
 S}^{(n)}_{\pi_q(\eta'),\pi_q(x_1),\dots,\pi_q(x_m);0,\pi_q(t)}$
 from $E'$.
\end{thm}

\begin{proof}
 The argument of Theorem~\ref{thm:center_for_general_T} together with the
 flatness results for the spherical alge\-bra allows us to reduce to the
 parameter-free case (defined in the obvious way). We may as well
 localize the central section~$1$ of degree $2s+2f$, as~we can then
 recover the center of the category from the induced filtration on the
 center of the localization. Any invariant section of~$\pi^*\sO_{\P^n}(l)$ is clearly central, and we can locally choose an
 invertible such section to deal with the automorphisms of degree $lf$.
 We obtain a description of the resulting sheaf of categories with objects
 $\Z/2\Z$ as a sheaf algebra over the corresponding $\sO_{X^+}$. The~even part of the center is precisely the center of the corresponding
 DAHA, while the odd part of the center may be identified with a submodule
 of the form
 $Z({\cal H}_{\tilde{C}_n,\gamma}(X)) {\cal Z}_{w_l} {w_l}$,
 where $w_l$ is given by $w_l\omega^l = (w\omega)^l$, where $w\omega$ is a~translation with $w\in W$. Thus the center satisfies strong flatness and saturates the
 appropriate Bruhat filtration, so agrees with the spherical algebra.
\end{proof}

To understand the significance of our construction, we~will need to
understand some special cases. The~case $t=0$ is of particular interest,
due to the following. Note that the tensor and symmetric power
constructions on modules extend to sheaf bimodules, and thus carry over to
analogous constructions for algebras and categories. (Of course, in the
latter cases, one should take the symmetric subobject, not the quotient
object.) The following is an immediate consequence of the fact that the
corresponding $A_{n-1}$ Hecke algebra is just the usual twisted group
algebra.

\begin{prop}\label{prop:spherical_as_symmetric_power}
 One has the following isomorphisms $($locally on the base$)$:
 \begin{gather*}
 {\cal S}^{(n)}_{\eta',x_1,\dots,x_m;q,0}
 \cong \Sym^n\big({\cal S}^{(1)}_{\eta',x_1,\dots,x_m;q,0}\big),
 \\
 {\cal S}^{\prime(n)}_{x_0,x_1,\dots,x_m;q,0}
 \cong \Sym^n\big({\cal S}^{\prime(1)}_{x_0,x_1,\dots,x_m;q,0}\big).
 \end{gather*}
\end{prop}

\begin{rem}
 Note that the tensor product of~$n$ univariate difference operators is
 described as follows. The~$i$th operator acts on~$k({\cal E}^n)$ by pulling
 back the coefficients from the $i$th factor and only translating the
 $i$th coordinate. These actions commute, and thus we may compose them to
 obtain the tensor product difference operator on~$k({\cal E}^n)$. The~tensor product of the algebras is then the image of the tensor product of
 sheaves under this operation, and the symmetric power consists of those
 operators in the tensor product that commute with~$S_n$.
\end{rem}

There is a similar description for~$t=q$ coming from the following symmetry
(essentially Corollary~\ref{cor:spherical_t_symmetry}, combined with the
fact that the conditions associated to~$x_1,\dots,x_m$ are unaffected).

\begin{prop}\label{prop:t_q-t_symmetry}
 One has the following isomorphisms $($locally on the base$)$:
 \begin{gather*}
 {\cal S}^{(n)}_{\eta',x_1,\dots,x_m;q,q-t}
 \cong \prod_{1\le i<j\le n} \Gampq(t\pm z_i\pm z_j)
 {\cal S}^{(n)}_{\eta',x_1,\dots,x_m;q,t}
 \prod_{1\le i<j\le n} \Gampq(t\pm z_i\pm z_j)^{-1},
 \\
 {\cal S}^{(n)}_{x_0,x_1,\dots,x_m;q,q-t}
 \cong \prod_{1\le i<j\le n} \Gampq(t\pm z_i\pm z_j)
 {\cal S}^{(n)}_{x_0,x_1,\dots,x_m;q,t}
 \prod_{1\le i<j\le n} \Gampq(t\pm z_i\pm z_j)^{-1}.
 \end{gather*}
\end{prop}

We also note the following version of the adjoint symmetry. It will be
convenient to express the adjoint in terms of a formal density; in
particular, ${\rm d}T$ simply represents a formal $\tilde{C}_n$-invariant
measure.

\begin{prop}
 The adjoint with respect to the formal density
 \begin{gather*}
 \prod_{1\le i<j\le n} \frac{\Gamq(t\pm z_i\pm z_j)}
 {\Gamq(\pm z_i\pm z_j)}
 \prod_{1\le i\le n} \frac{1}{\Gamq(\pm 2z_i)}\, {\rm d}T
 \end{gather*}
 induces $($locally on the base$)$ contravariant isomorphisms
 \begin{gather*}
 {\cal S}^{(n)}_{\eta',x_1,\dots,x_m;q,t}\cong
 {\cal S}^{(n)}_{-\eta',-x_1,\dots,-x_m;q,t} ,
 \\
 {\cal S}^{\prime(n)}_{x_0,x_1,\dots,x_m;q,t}\cong
 {\cal S}^{\prime(n)}_{-x_0,-x_1,\dots,-x_m;q,t}
 \end{gather*}
 acting on objects as $v\mapsto -v$.
\end{prop}

\begin{rem}
 We will refer to this formal adjoint as the ``Selberg'' adjoint, as~the
 formal density consists of the interaction terms in the elliptic Selberg
 integral. More generally, composing the Selberg adjoint with a sequence
 of elementary transformations gives an adjunction involving densities of
 the form
 \begin{gather*}
 \prod_{1\le i\le n} \frac{\prod_{1\le j\le k} \Gamq(u_j\pm z_i)}
 {\Gamq(\pm 2z_i)} \prod_{1\le i<j\le n} \frac{\Gamq(t\pm z_i\pm z_j)}
 {\Gamq(\pm z_i\pm z_j)}\, {\rm d}T,
 \end{gather*}
 where the $u_j$ depend on the $x_j$ and the specific domain and codomain
 objects. Here we should think of the operators as mapping between two
 different inner product spaces, so that the two formal integrals are
 against (slightly) different densities.
\end{rem}

We were somewhat vague in our descriptions of the subquotients of the
Bruhat filtration above, as~the specific divisors that are forced
into a given coefficient are somewhat complicated to describe in general. For~the most part, though, the important information about the subquotients
is not how they are built up out of divisors, but simply which line bundle
one ends up with in the end. (Recall that the subquotients are obtained
from line bundles which are invariant under some parabolic subgroup $W_I$
by descending to~$X/W_I$ then taking the direct image to~$X/C_n$.)

Propositions~\ref{prop:spherical_as_symmetric_power} and
\ref{prop:t_q-t_symmetry} make this information relatively straightforward
to determine. Any filtered isomorphism preserves the Bruhat subquotients,
and thus the associated polarizations must be invariant under $t\mapsto
q-t$ (modulo line bundles on the base, that is). Modulo line bundles on
the base, the $t$-dependent contribution to the polarization is linear, and
thus the $t\mapsto q-t$ symmetry forces it to be trivial. In~other words,
the {\em subquotients} are (up to isomorphism local on the base)
independent of~$t$ and thus by the first proposition are determined by the
subquotients for~$n=1$. More precisely, for~$t=0$, the line bundle on~${\cal E}^n$ associated to a given dominant weight is the outer tensor
product of the univariate subquotients associated to the parts of the
weight; one then descends to the quotient by the stabilizer in~$C_n$ of the
weight.

Helpfully, it turns out that the univariate case has already been studied.
Let $\Gamma{\cal S}^{(n)}$, $\Gamma{\cal S}^{\prime(n)}$ denote the
associated ``global section'' categories; more precisely, these are sheaves
of categories on~${\cal E}^{m+3}$ in which each $\Hom$ sheaf is the direct
image of the corresponding $\Hom$ bimodule. Since we included twisting by
$\sO_{\P^n}(1)$ in the definition of the category, one can recover~${\cal
 S}^{(n)}$ and~${\cal S}^{\prime(n)}$ from their global section
categories: for each $\Hom$ bimodule $M$ in the sheaf category, the
corresponding graded module (relative to the Segre embedding of~$\P^n\times
\P^n$) can be extracted from the global section category.

In~\cite{noncomm1}, two families of categories ${\cal
 S}_{\eta,\eta',x_1,\dots,x_m;q,p}$ and ${\cal
 S}'_{\eta,x_0,x_1,\dots,x_m;q,p}$ were constructed on an analytic curve
$\C^*/\langle p\rangle$. These have the same group of objects as our
categories, and an interpretation of the local sections of the $\Hom$
sheaves as difference operators. Moreover, the categories for general $m$
are cut out from the categories for~$m=0$ by suitable vanishing conditions,
while the categories for~$m=0$ are described via explicit generators given
in terms of theta functions. Switching from multiplicative to additive
notation and replacing $\theta$ by $\vartheta$ then extends this to
arbitrary curves. (In fact,~\cite{noncomm1} gave such an extension by
specifying an explicit gauging by products of Gamma functions that makes
everything elliptic, and observing that the elliptic functions extend. But
of course one could do the same gauging in terms of~$\vartheta$ and $\Gamq$
symbols, so the resulting categories are the same.)

Although those operators are not quite $C_1$-symmetric in our sense, they
are close: indeed, each operator formally takes functions invariant under
$z\mapsto (1-d_1)q+\eta-z$ to functions invariant under $z\mapsto
(1-d_2)q+\eta-z$. This, of course, is easy enough to fix: if we base
change to have an element $\eta/2$ (and recall that we already have an
element $q/2$), then we can compose on both sides by a suitable translation
to make the operator honestly $C_1$-symmetric.

\begin{prop}
 Locally on the base, the global section categories are isomorphic to the
 $C_1$-symmetric versions of the categories ${\cal S}$, ${\cal S}'$
 constructed in~{\rm \cite{noncomm1}}. More precisely, if $v,w$ are arbitrary
 elements of the object group $\Z\langle s,f,e_1,\dots,e_m\rangle$, then
 \begin{gather*}
 \Gamma{\cal S}^{(1)}_{\eta',x_1,\dots,x_m;q,t}(v,w)
 \cong {\cal S}_{2c,2c+\eta',c+x_1,\dots,c+x_m;q;{\cal E}}(v,w),
 \\
 \Gamma{\cal S}^{\prime(1)}_{x_0,x_1,\dots,x_m;q,t}(v,w)
 \cong {\cal S}'_{2c,c+x_0,c+x_1,\dots,c+x_m;q;{\cal E}}(v,w),
 \end{gather*}
 where $c$ is a free parameter.
\end{prop}

\begin{proof}
 For $m=0$, ${\cal S}_{2c,2c+\eta';q;{\cal E}}$ and ${\cal
 S}'_{2c,c+x_0;q;{\cal E}}$ are generated in degrees $f$, $s$, $s+f$. The~degree $f$ operators are clearly elements of the corresponding global
 section category, and the $C_1$-symmetry along with the fact that the
 only poles are $[X^{s_1}]$ implies the same for the degree $s$ and $s+f$
 operators. Since the subcategory generated in this way saturates the
 Bruhat filtration, the categories are actually isomorphic.

 For each of the four categories, every $\Hom$ sheaf for~$m>0$ is
 contained in the appropriate $\Hom$ sheaf for~$m=0$, and the image of a
 local section of~${\cal S}$ or ${\cal S}'$ satisfies the correct
 vanishing conditions to be a local section of~$\Gamma{\cal S}^{(1)}$ or
 $\Gamma{\cal S}^{\prime(1)}$. Moreover, the Bruhat filtration and the
 analogous filtration by order tells us that both $\Hom$ sheaves are
 direct images of vector bund\-les~on~$\P^1$ with the same Hilbert
 polynomials, and must therefore be identified by the isomorphism.
\end{proof}

This leads to a particularly nice interpretation of our categories. If~the
rational surface $X_m$ is obtained from a Hirzebruch surface $X_0$ by
blowing up $m$ points of a smooth anticanonical curve, then the line
bundles on~$X_m$ are parametrized by the group $\Z\langle
s,f,e_1,\dots,e_m\rangle$. This then gives rise to a category on this
group of objects by taking the full subcategory of~$\coh(X_m)$ in~which the
objects consist of one line bundle of each isomorphism class. We can, of
course, do~this over the entire moduli stack of such surfaces, which turns
out (at least for~$m>0$) to be isomorphic to~${\cal E}^{m+1}$. We then
obtain a sheaf of categories on this base by taking the appropriate sheaf
version of~$\Hom$ between line bundles. It was shown in~\cite{noncomm1}
that this sheaf of categories is precisely the specialization to~$q=0$ of~${\cal S}$ or ${\cal S}'$, depending on whether $X_0$ comes from a vector
bundle of even or odd degree. (One caveat here is that the fibers in this
category can be slightly different from the categories associated to
individual surfaces; $\Hom$ spaces in the latter may jump in the presence
of~$-2$-curves, while the global category is flat.)

The same, therefore, applies to our categories, and thus for all $n$,
${\cal S}^{(n)}_{\eta',x_1,\dots,x_m;0,0}$ and ${\cal
 S}^{\prime(n)}_{x_0,\dots,x_m;0,0}$ can be interpreted as symmetric
powers of rational surfaces. (To be precise, each fiber is equi\-va\-lent to a
subcategory of the subcategory of line bundles on such a symmetric power,
which is full whenever the ratio of line bundles is acyclic.) The
categories with~$q=0$ and general $t$ are thus commutative deformations of
such powers (some sort of compactified discrete elliptic Calogero--Moser
spaces), while the categories with general $q$, $t$ are further
noncommutative deformations.

We can also obtain analogous deformations for~$\P^2$, though in that case
only the global section category makes sense. If~we restrict $\Gamma{\cal
 S}^{\prime(n)}_{x_0;q,t}$ to the objects in~$\Z(s+f)$, then for~$n=1$,
$q=t=0$, we~can identify consecutive $\Hom$ spaces in such a way as to
obtain the polynomial algebra in three generators. (For $n=1$, $q\ne 0$, we~instead get the three-generator Sklyanin algebra of~\cite{ArtinM/TateJ/VandenBerghM:1990,BondalAI/PolishchukAE:1993}, see
\cite{noncomm1}.) Thus for general $n$, $q=0$, we~again obtain a family of
commutative deformations of~$\Sym^n(\P^2)$, and further noncommutative
deformations for general parameters.

There are some caveats to the above discussion. One is that since we are
including {\em all} line bundles in the construction, there is no canonical
way to associate a projective variety for~$q=0$ or a noncommutative
analogue in general: in general, we~would need to make an explicit choice
of ample divisor, or make some other choice of what it means for a module
over the category to be torsion (i.e., map to the $0$ sheaf). For~$n=1$,
it was shown in~\cite{noncomm1} that any reasonable choice of ample divisor
(in particular, any divisor which is ample on every $X_m$) gives rise to
the same quotient category, and thus there is no difficulty.
Unfortunately, the argument there relied heavily on showing that various
product maps are surjective, and the analogous surjectivity fails for~$n>1$
even for~$q=t=0$. We therefore leave this as an open question.

There is also an issue here that, due to some difficulties in applying the
Hecke algebra ideas to the $\P^2$ case, we~cannot always prove flatness for
general ample divisors. For~any surface other than $\P^2$, this is not a
significant issue, as~there will always be a nonempty subcone of the ample
cone for which everything does work as expected. For~$\P^2$, or in general
outside this subcone, we~will at least be able to show that each $\Hom$
space is flat outside some finite (and most likely empty) set of bad pairs
$(q,t)$.

We should also note that since the data in each case includes an explicit
morphism to a~Hir\-zebruch surface, that a priori the category might depend
on this map (and, when $X_0\cong \P^1\times \P^1$, on the choice of ruling)
and not just on the surface $X_m$. We will show in the following section
that (just as for~$n=1$) this is not an issue, but the argument is
decidedly nontrivial.

Before proceeding to studying flatness for the global section category, we~should note that there are also interpretations of~${\cal H}^{(1)}$ and
${\cal H}^{\prime(1)}$ in terms of the categories constructed in
\cite{noncomm1}. The~point is that for~$n=1$, we~are taking a spherical
algebra relative to a master Hecke algebra. Since this is the endomorphism
algebra of a vector bundle on~$\P^1$, we~immediately find that the
spherical algebra and the DAHA are Morita equivalent. (Note that even for
generic parameters, this does not quite follow from Proposition~\ref{prop:spherical_is_Morita_equivalent}, as~this new Morita equivalence
applies to the compactified versions of the algebras.) Making this Morita
equivalence explicit gives the following.

\begin{prop}
 Let $v,w\in \Z\langle s,f,e_1,\dots,e_m\rangle$. Then
 \begin{gather*}
 \Gamma{\cal H}^{(1)}_{\eta',x_1,\dots,x_m;q}(v,w)
 \cong
 \begin{pmatrix}
 \Gamma{\cal S}^{(1)}_{\eta',x_1,\dots,x_m;q}(v,w)
 & \Gamma{\cal S}^{(1)}_{\eta',x_1,\dots,x_m;q}(v-2f,w)
 \\
 \Gamma{\cal S}^{(1)}_{\eta',x_1,\dots,x_m;q}(v,w-2f)
 & \Gamma{\cal S}^{(1)}_{\eta',x_1,\dots,x_m;q}(v-2f,w-2f)
 \end{pmatrix}\!,
 \end{gather*}
 and similarly for~$\Gamma{\cal H}^{\prime(1)}$.
\end{prop}

\begin{rem}
 One can also apply this at the level of sheaf categories on~$\P^1$.
 There one finds (per the analogous statement of~\cite{noncomm1}) that the
 spherical sheaf category is the sheaf ``$\Z$-algebra'' associated to a
 noncommutative $\P^1$-bundle on~$\P^1$~\cite{VandenBerghM:2012}. There
 is, of course, no reason why we could not apply the Morita equivalence
 associated to~$\sO_{\P^1}\oplus \sO_{\P^1}(-2)$ to any such
 noncommutative \mbox{$\P^1$-bundle} on~$\P^1$ and thus obtain an associated DAHA
 (which will always be a degeneration of the elliptic DAHA).
\end{rem}

In~\cite{noncomm1}, it was shown that the algebras ${\cal S}$ satisfy a
``Fourier transform'' symmetry swapping~$\eta$ and $\eta'$ and swapping $s$
and $f$, which in turn induces a symmetry
\begin{gather*}
\Gamma{\cal S}^{(1)}_{\eta',x_1,\dots,x_m;q,0}
\cong
\Gamma{\cal S}^{(1)}_{-\eta',x_1-\eta'/2,\dots,x_m-\eta'/2;q,0},
\end{gather*}
again swapping $s$ and $f$. (We will show in the next section how to
extend this to general $n$.) Since the description of~$\Gamma{\cal
 H}^{(1)}$ in terms $\Gamma{\cal S}^{(1)}$ is not invariant under swapping
$s$ and $f$, this symmetry does not actually extend to the DAHA itself. It
turns out that, at least for~$m=0$, this is a~consequence of the
compactification we performed. Each $\Hom$ space of degree $2s+2f$
contains an element $1$ (in a Fourier-invariant way!); if we localize with
respect to those elements, then the objects $v$ and $v+2s+2f$ of the
category become isomorphic, and thus we may replace $v-2f$ by $v+2s$ in the
above description. Using the translation symmetry, we~can then subtract~$2s$ from $v$ and $v+2s$ at the cost of changing the parameters slightly.
But then swapping $s$ and $f$ recovers the above description of~$\Gamma{\cal H}^{(1)}$. In~other words, the localized DAHAs actually {\em
 do} satisfy a Fourier transformation symmetry (though it is not clear how
to describe it in terms of explicit operators). It is likely that
something similar holds in general (including $n>1$), but this will require
a better understanding of the relevant Morita equivalences.

In the univariate setting, one can gain some insight from the results of
\cite{OblomkovA:2004} on the traditio\-nal~$C^\vee C_1$ Hecke algebra. This
suggests in general that replacing $-2f$ above by $-2f+e_1+\cdots+e_k$
would have the effect of moving the parameters $x_1,\dots,x_k$ from $s_0$
to~$s_n$. If~so, then the Fourier transform would continue to extend to
the noncompact elliptic DAHA in the presence of non-$t$ parameters, but
would effectively swap the roles of the two roots vis-\`a-vis the
parameters.

This degeneration also gives strong evidence that the full $\SL_2(\Z)$
action will {\em not} extend to the elliptic DAHA (compact or not).
Indeed, if one looks at the action of~$\SL_2(\Z)$ on the corresponding
surfaces in the case that the spherical algebra is commutative, one finds that
it relies on the fact that the anticanonical curve at infinity is singular.
Each generator of~$\SL_2(\Z)$ blows up a singular point and then blows down
a different component of the anticanonical curve, so that the anticanonical
curve has the same shape and its complement has not changed, but the actual
projective surface has. Blowing up a smooth point of the anticanonical
curve in general {\em does} change the complement of the anticanonical
curve, and thus we cannot expect this operation to survive to the elliptic
level.

As with sheaves in general, when we take global sections in a family of
sheaf categories, the fibers of the global sections can differ considerably
from the global sections of the fibers. That is, there is a natural
morphism from each fiber of the global section category to the
global section category of the corresponding fiber, but this morphism can
fail to be either injective or surjective. The~failure of injectivity is
particularly bad when we consider that the kernel of the map does not
inherit an interpretation in terms of difference operators. In~particular,
if we have such a failure of injectivity, then we can no longer be
confident that the fiber is a domain.

Each $\Hom$ sheaf in the global section category is the direct image of the
corresponding $\Hom$ bimodule, and we can factor the direct image through
one of the projections $\P^n\times \P^n\to \P^n$ to find that the $\Hom$
sheaf is the direct image of a vector bundle on~$\P^n$. If~every fiber of
that vector bundle is acyclic, then Grauert's theorem tells us that taking
the direct image actually {\em does} commute with passing to fibers. It
turns out that this holds (modulo some genericity assumptions in some
cases) for a sufficiently large class of degrees to allow us to prove in
general that the map is always injective and that the global section
category is flat.

There are two ways to show acyclicity. One is to show that every
subquotient of the Bruhat filtration is acyclic; the other is to use the
symmetric power description for~$t=0$ to deduce acyclicity for~$t=0$ and
thus in a neighborhood of~$t=0$ by semicontinuity. In~either case,
the~bundle is either itself a symmetric power or is built up out of symmetric
powers, and thus we need to understand when such a bundle is acyclic.

Given a sheaf $M$ on a scheme $X$, we~may define a sheaf $\Sym^n(M)$ on the
symmetric power $\Sym^n(X)$ by descending $M^{\boxtimes n}$ through the
quotient by $S_n$.

\begin{lem}
 Let $X$ be a projective scheme over a field $k$, and let $M$ be an
 acyclic sheaf on~$X$. Then $\Sym^n(M)$ is an acyclic sheaf on~$\Sym^n(X)$ for all $n\ge 1$.
\end{lem}

\begin{proof}
 First, note that if $k$ has characteristic $p>n$ or $0$, then this is
 immediate, since $\Sym^n(M)$ is a direct summand of an acyclic sheaf,
 namely the direct image of the acyclic sheaf $M^{\boxtimes n}$.

 In general, we~proceed by induction on the pair $(n,\dim X)$ relative to
 the product partial order. Let $\sO_X(1)$ be a very ample divisor on~$X$, and note that $\Sym^n(\sO_X(1))$ is at least ample on~$\Sym^n(X)$.
 (It can fail to be very ample!) In particular, there exists $l>0$ so
 that
 \begin{gather*}
 \Sym^n(M)\otimes \Sym^n(\sO_X(1))^l\cong \Sym^n(M(l))
 \end{gather*}
 is acyclic. Choose a nonzero section of~$\sO_X(l)$, and use it to embed
 $M$ as a subsheaf of~$M(l)$. This one-step filtration of~$M(l)$ induces
 a symmetric power filtration~$F_\bullet$ of~$\Sym^n(M(l))$ such that
 $F_{m+1}/F_m$ is the direct image on~$\Sym^n(X)$ of the sheaf
 $\Sym^m(M(l)/M)\boxtimes \Sym^{n-m}(M)$ on~$\Sym^m(X)\times
 \Sym^{n-m}(X)$. By induction each subquotient $F_{m+1}/F_m$ for~$m>0$ is
 acyclic, as~each factor is either a symmetric power of lower degree or
 supported on a lower-dimensional projective scheme. Since
 $F_1=\Sym^n(M)$, it follows that $\Sym^n(M(l))/\Sym^n(M)$ is acyclic,
 and thus that $H^p(\Sym^n(M))=0$ for~$p>1$.

 For all $m \ge 0$, we~have $H^0(\Sym^n(M(m))) \cong \Sym^n\big(H^0(M(m))\big)$, and
 thus $h^0(\Sym^n(M(m)))\allowbreak = \binom{h^0(M(m))+n-1}{n}$. Applying this to~$m\gg 0$ lets us compute the Hilbert polynomial of~$\Sym^n(M)$, and then
 setting $m=0$ gives
 \begin{gather*}
 \chi\big(\Sym^n(M)\big)=\binom{\chi(M)+n-1}{n}=\binom{h^0(M)+n-1}{n} =
 h^0\big(\Sym^n(M)\big).
 \end{gather*}
 Since we have already shown that the higher cohomology spaces vanish,
 this implies that $h^1$ also vanishes, and the claim follows.
\end{proof}

By considering subquotients for the Bruhat filtration, we~obtain
the following.

\begin{lem}
 Suppose $d'\ge d-1$. Then every fiber of~${\cal
 S}^{(n)}_{\eta';q,t}(0,ds+d'f)$ is acyclic for the morphism to
 parameter space.
\end{lem}

\begin{lem}
 Suppose $2d'/3\ge d-1$. Then every fiber of~${\cal
 S}^{\prime(n)}_{x_0;q,t}(0,ds+d'f)$ is acyclic for the
 morphism to parameter space.
\end{lem}

\begin{lem}
 Suppose $d'\ge \max(d,d/2+r_1)$ and $0\le r_1\le d$. Then every fiber of~${\cal S}^{(n)}_{\eta',x_1;q,t}(0,ds+d'f-r_1e_1)$ is acyclic for the morphism
 to parameter space.
\end{lem}

\begin{proof} The first two lemmas are straightforward. For~the third,
 note that if $r_1\le d/2$, then imposing the vanishing conditions
 subtracts $r_1$ from the degree of the top subquotient, $r_1-1$ from the
 next, etc., until we reach 0, and in each case there is sufficient degree
 to do this without becoming negative (or 0, apart from the subquotient
 supported on~$\P^1$). For~$r_1>d/2$, we~apply the elementary
 transformation symmetry to reduce to a $r_1\le d/2$ case with the
 opposite parity. Each case gives a convex cone in which we are
 guaranteed acyclicity, and combining the cones gives the desired result.
\end{proof}

\begin{prop}
 Suppose $v=ds+d'f-r_1e_1-\cdots-r_me_m$ satisfies the inequalities $d'\ge
 \max(d,d/2+r_1)$, $d\ge r_1+r_2$ and $r_1\ge r_2\ge\cdots\ge r_m\ge 0$,
 and either $v=0$ or $2d+2d'-r_1-r_2-\cdots-r_m>0$. Then every fiber of~${\cal S}^{(n)}_{\eta',x_1,\dots,x_m;q,t}(0,v)$ is acyclic for the morphism
 to parameter space.
\end{prop}

\begin{proof}
 Every subquotient is the direct image under a finite morphism of an outer
 tensor product of symmetric powers of subquotients of the $n=1$ case. The~constraints on~$v$ ensure that every univariate subquotient is
 acyclic: either (the direct image of) a line bundle of positive degree on~${\cal E}$ or a line bundle of nonnegative degree on~$\P^1$. It follows
 that every Bruhat subquotient for general $n$ is acyclic, and thus the
 same holds for the full $\Hom$ space.

 To check the univariate assertion, note that for~$d=0$, the vector bundle
 is simply $\sO_{\P^1}(d')$, so there is no problem. For~$d>0$, the
 leading subquotient in the filtration comes from a line bundle on~${\cal
 E}$ of degree $2d+2d'-r_1-\cdots-r_m$, so is positive, and the
 corresponding subsheaf is the same as the $\Hom$ sheaf obtained by
 subtracting $2s+2f-e_1-\cdots-e_m$ from $v$ (unless $d=1$, when the
 subsheaf is trivial and there is nothing further to discuss). If~$m>1$
 (the $m=1$ case already having been dealt with), then this subtraction
 preserves all of the inequalities except possibly $r_m\ge 0$ and
 $2d+2d'-r_1-\cdots-r_m>0$. The~first inequality could only be violated
 if we had $r_m=0$, in which case we might as well have omitted that
 parameter. For~the other inequality, we~are adding $m-8$ to the
 left-hand side, so there is no problem if $m\ge 8$. But if $m<8$, then
 the inequality is implied by the other inequalities.
\end{proof}

Note that this proposition is already stronger than it seems, as~we can
always arrange to have the inequalities $d\ge r_1+r_2$, $r_1\ge
r_2\ge\cdots\ge r_m\ge 0$ by applying a suitable combination of elementary
transformations and setting negative $r_i$ to 0. Indeed, with the
exception of the final inequality, this is just stating that the vector $v$
is in the fundamental chamber for the corresponding action of~$W(D_m)$. In~particular, for~any vector $v$, we~can use this proposition to find an
explicit $d'$ such that $v+d'f$ satisfies acyclicity. (The existence of
such a $d'$ was of course already guaranteed by Serre vanishing.) If
$r_1\le d/2$, then this bound is pretty close to tight (based on what we
know about the $n=1$ case), but for~$r_1\ge d/2$, the following result
suggests that there is considerable room for improvement.

\begin{prop}
 Suppose $v=ds+d'f-r_1e_1-\cdots-r_me_m$ satisfies the inequalities $d'\ge
 d\ge r_1+r_2$; $r_1\ge r_2\ge\cdots\ge r_m\ge 0$; and either $v=0$ or
 $2d+2d'-r_1-r_2-\cdots-r_m>0$. Then there is a codimension~$\ge 2$
 subscheme of parameter space, not meeting the subschemes $t=0$ or~$t=q$,
 such that every fiber of~${\cal S}^{(n)}_{\eta',x_1,\dots,x_m;q,t}(0,v)$
 on the complement is acyclic for the morphism to parameter space.
\end{prop}

\begin{proof}
 If $n=1$, then these inequalities are enough to guarantee acyclicity, per~\cite{noncomm1}. It follows that any fiber with~$t=0$ satisfies
 acyclicity, and thus the open subscheme on which acyclicity holds
 contains the divisor $t=0$. By the $t\mapsto q-t$ symmetry, the acyclic
 locus also contains the divisor $t=q$. This pair of divisors is
 relatively ample over ${\cal M}_{1,1}$ for the ${\cal E}^2$ parametrizing
 $q$ and~$t$, and thus their complement contains no closed subscheme of
 codimension~$\le 1$.
\end{proof}

We of course conjecture that the codimension~$\ge 2$ subscheme is always
empty.

\begin{cor}\label{cor:universally_ample_is_flat}
 Suppose $v=ds+d'f-r_1e_1-\cdots-r_me_m$ satisfies the inequalities $d'\ge
 d\ge r_1+r_2$ and $r_1\ge r_2\ge\cdots\ge r_m\ge 0$. Then there is a
 codimension~$\ge 2$ subscheme of parameter space $($empty if $d'\ge
 d/2+r_1$ and never meeting $t=0$ or $t=q)$ on the complement of which
 $\Gamma{\cal S}^{(n)}_{\eta',x_1,\dots,x_m;q,t}(0,v)$ is flat and the map
 to meromorphic difference operators is injective on fibers.
\end{cor}

\begin{proof}
 If $v=0$ or $2d+2d'-r_1-\cdots-r_m>0$, then this follows from acyclicity,
 so suppose $2d+2d'-r_1-\cdots-r_m\le 0$. This is the degree of the
 leading subquotient of the univariate filtration; if it is negative, then
 this leading subquotient never has a global section, while if it is~0,
 the line bundle depends nontrivially on the parameters, and thus {\em
 generically} does not have a global section. Either way, the direct
 image of a nontrivial symmetric power of the leading univariate
 subquotient will always be 0, and the same holds for an outer tensor
 product with such a power.

 Consider a Bruhat order ideal (i.e., an order ideal in the poset of
 dominant weights) contained in the interval $[\le (d/2,\dots,d/2)]$. If~this order ideal contains a dominant weight with~\mbox{$\lambda_1=d/2$}, then
 there is such a weight which is a maximal element of the order ideal.
 Since the subquotient corresponding to that maximal element has no direct
 image, removing it has no effect on the direct image. We thus find that
 the direct image of the interval $[\le (d/2,\dots,d/2)]$ is the same as
 the direct image of the interval $[\le (d/2-1,\dots,d/2-1)]$,
 and thus we reduce to~$v-(2s+2f-e_1-\cdots-e_m)$ as before.
\end{proof}

\begin{rem}
 That the bad subscheme has codimension~$\ge 2$ follows from general
 considerations; a failure of injectivity in codimension~1 comes from a
 local family of operators vanishing on a hypersurface, and we can always
 locally divide such a family by a function cutting out the hypersurface.
 Thus the true content is that the locus does not meet $t=0$ or $t=q$, and
 further has codimension~$\ge 2$ image in the surface parametrizing $q$
 and $t$. In~particular, we~have injectivity over the local ring at any
 point with~$t=0$ or $t=q$, including the point $q=t=0$ corresponding to
 the undeformed case.
\end{rem}

For $0\le m\le 7$ (or for~$m=-1$, i.e., $\P^2$), the corresponding
commutative surface is a~(pos\-sibly singular) del Pezzo surface with a~choice of smooth anticanonical curve and a sequence of blowdowns to a
Hirzebruch surface. The~anticanonical embedding of this surface is given
by the graded algebra
\begin{gather*}
\bigoplus_{d\ge 0} \Gamma{\cal
 S}^{(1)}_{\eta',x_1,\dots,x_m;q,t}(0,d(2s+2f-e_1-\cdots-e_m));
\end{gather*}
we can interpret this as a graded algebra by using the fact that $1$ is in
\begin{gather*}
\Gamma{\cal
 S}^{(1)}_{\eta',x_1,\dots,x_m;q,t}(d(2s+2f-e_1-\cdots-e_m),(d+1)(2s+2f-e_1-\cdots-e_m))
\end{gather*}
for any $d$. This is the Rees algebra of the natural filtration on the
spherical algebra
\begin{gather*}
\bigcup_{d\ge 0} \Gamma{\cal
 S}^{(1)}_{\eta',x_1,\dots,x_m;q,t}(0,d(2s+2f-e_1-\cdots-e_m)),
\end{gather*}
the coordinate ring of the complement of the chosen smooth anticanonical
curve. Taking the multivariate versions thus gives deformations of the
(anticanonically embedded) symmetric powers of~$X_m$ and $X_m\setminus E$,
and, apart from possible codimension~$\ge 2$ exceptions for~$\P^2$, these
deformations are always flat, and every fiber is a domain.

If $X_m$ has $-2$-curves or $m>7$, then the anticanonical divisor is no
longer ample, and thus one can no longer expect to obtain a deformation of~$\Sym^n(X_m)$ as a graded algebra, or of~$\Sym^n(X_m\setminus E)$ as a
filtered algebra. This is why we generalized the spherical algebra
construction to the above categories: one needs to work with {\em some}
non-(pluri)anticanonical divisor, and it is then easier to include all
divisors.

There is an interesting phenomenon that arises for the spherical algebra in
the $m=8$ case. The~global section algebra in this case is trivial
(consisting only of the global section~$1$), but this merely reflects the
fact that the generic fiber has no nontrivial global sections. The~univariate subquotients in this case are all multiples of~$x_1+\cdots+x_8-2\eta'$, and thus if this value is \mbox{$r$-tor}sion, then any
subquotient of weight a multiple of~$r$ will be trivial. (In terms of
surfaces, this corresponds to the case that $X_8$ is an elliptic surface,
in which one fiber consists of~$r$ copies of the chosen anticanonical
curve.) As~a result, the dimension of global sections of such a fiber in a
given Bruhat interval can in principle be as large as the number of such
weights contained in the interval, or (a priori) as small as 1.

It turns out that at least for~$r=1$ (i.e., when the elliptic surface has a
section), this upper bound is attained (i.e., the dimension of global
sections in a Bruhat interval is equal to the size of the Bruhat interval),
and furthermore those global sections satisfy a surprising property. Note
that it suffices to find $n+1$ global sections of degree
$2s+2f-e_1-\cdots-e_8$, as~we can then obtain global sections with
arbitrary dominant weight by taking products.

\begin{thm}\label{thm:vandiejen}
 On any fiber such that $2\eta'=x_1+\cdots+x_8$, the space of global
 sections of~${\cal
 S}^{(n)}_{\eta',x_1,\dots,x_8;q,t}(0,2s+2f-e_1-\cdots-e_8)$ is
 $n+1$-dimensional, and any two global sections commute.
\end{thm}

\begin{proof}
 Certainly, $n+1$ is an upper bound on the number of global sections,
 since there are $n+1$ subquotients, each of which has a unique global
 section. The~given $\Hom$ bimodule is contained in the $\Hom$ bimodule
 ${\cal S}^{(n)}_{\eta',x_1,\dots,x_7;q,t}(0,2s+2f-e_1-\cdots-e_7)$, and
 the latter $\Hom$ bimodule satisfies acyclicity. Since the $m=7$
 bimodule has $2$ global sections when~$n=1$, it has
 $n+1=\binom{2+n-1}{n}$ global sections when~$t=0$ and thus (by flatness)
 in general. We thus need to show that those global sections are actually
 global sections of the subsheaf we want.

 Let ${\cal D}$ be such a global section. This is determined by the left
 coefficients $c_m$ of~$\prod_{1\le i\le m} T_i^{-1}$ for~$0\le m\le n+1$,
 where $c_m$ is $S_m\times C_{n-m}$-invariant. Each $c_m$ is a section of
 a line bundle~${\cal L}_m(D_m)$, where ${\cal L}_m$ comes
 from the equivariant gerbe and $D_m$ comes from the allowed poles and
 forced zeros. The~allowed poles are somewhat complicated, since we are
 not assuming that~$c_m$ is a leading coefficient, but the forced zeros
 are the same as they would have been if it were a leading coefficient. The~symmetric power property then tells us that when~$t=0$, the~forced
 zeros associated to~$t$ must all cancel allowed poles, and any allowed
 pole associated to~a~root of type $D_n$ must be cancelled in this way.

 The remaining zeros and poles can be deduced from the univariate case,
 and we thus find that $c_m$ is a multiple of
 \begin{gather*}
 \prod_{1\le i<j\le m}
 \frac{\vartheta(t-z_i-z_j,q+t-z_i-z_j)}
 {\vartheta(-z_i-z_j,q-z_i-z_j)}
 \prod_{\substack{1\le i\le m\\m<j\le n}}
 \frac{\vartheta(t-z_i\pm z_j)}
 {\vartheta(-z_i\pm z_j)}
 \\ \qquad
 {}\times\prod_{1\le i\le m}
 \frac{\prod_{1\le j\le 7} \vartheta(q/2+x_j-z_i)}
 {\vartheta(-2z_i,q-2z_i)}
 \prod_{m<j\le n} \frac{1}{\vartheta(-q-2z_j,q-2z_j)},
 \end{gather*}
 in the sense that the ratio is a holomorphic section of the line bundle
 with polarization
 \begin{gather*}
 \sum_{1\le i\le m} (z_i^2/2-(q/2+x_8)z_i)
 +\sum_{m<j\le n} 4z_i^2,
 \end{gather*}
 modulo line bundles on the base. (The only difference between this and
 the leading coefficient of the corresponding Bruhat interval are the
 factors $\vartheta(-q-2z_i,q-2z_i)$ for~$m<j\le n$.) This line bundle
 has degree 1 in each $z_i$ for~$1\le i\le m$, so every holomorphic
 section has the same dependence on those variables, which we can read off
 from the polarization.

 We thus conclude that $c_m/\prod_{1\le i\le m} \vartheta(q/2+x_8-z_i)$ is
 independent of~$z_1$ through~$z_m$, so is still holomorphic. As~a
 result, we~find that every global section of the $m=7$ bimodule is also
 a~global section for~$m=8$; more precisely, the holomorphy gives it
 generically, but the condition is closed, so it holds in general.

 It remains to show commutativity. We first note as a sanity check that
 the $n+1$ leading term operators
 \begin{gather*}
 \prod_{1\le i<j\le m}\!\!\! \!\frac{\vartheta(t-z_i-z_j,q+t-z_i-z_j)}
 {\vartheta(-z_i-z_j,q-z_i-z_j)}\!\!
 \prod_{\substack{1\le i\le m\\m<j\le n}}\!\!\! \frac{\vartheta(t-z_i\pm z_j)}
 {\vartheta(-z_i\pm z_j)}\!\!
 \prod_{1\le i\le m}\!\!
 \frac{\prod_{1\le j\le 8} \vartheta(q/2+x_j-z_i)}
 {\vartheta(-2z_i,q-2z_i)}\,
 T_i^{-1}
 \end{gather*}
 commute. It follows that on any fiber with~$x_1+\cdots+x_8=2\eta'$, the
 global section algebra of the spherical algebra
 \begin{gather*}
 \bigcup_{d\ge 0}
 {\cal S}^{(n)}_{\eta',x_1,\dots,x_8;q,t}(0,d(2s+2f-e_1-\cdots-e_8))
 \end{gather*}
 has abelian associated graded algebra; the above leading term operators
 give one element for~each fundamental weight, so generate the associated
 graded algebra.

 This global section algebra has a particularly nice symmetry: it is
 preserved by the formal adjoint with respect to the density
 \begin{gather*}
 \prod_{1\le i\le n} \frac{\prod_{1\le j\le 8} \Gamq(q/2+x_j\pm z_i)}
 {\Gamq(\pm 2z_i)}
 \prod_{1\le i<j\le n} \frac{\Gamq(t\pm z_i\pm z_j)}
 {\Gamq(\pm z_i\pm z_j)}\,
 {\rm d}T.
 \end{gather*}
 Indeed, this is the composition of the Selberg adjoint and all 8
 elementary transformations, with the total effect on the parameters being
 $\eta'\mapsto x_1+\cdots+x_8-\eta'=\eta'$. Note that although such
 isomorphisms are usually only defined up to a unit, we~can eliminate that
 freedom by insisting that the adjoint of 1 be 1. This is triangular with
 respect to Bruhat order, and is trivial on~the~associated graded algebra,
 since it fixes the generators and the associated graded algebra is
 abelian. Since a triangular involution which is 1 on the diagonal is 1, we~find that this formal adjoint acts trivially on the entire global
 section algebra. Since an algebra consisting entirely of~self-adjoint
 operators is abelian, we~conclude that the generators commute as
 required.
\end{proof}

Since the $n+1$ operators are filtered by Bruhat order, it is natural from
an integrable systems perspective to designate the first nontrivial
operator (with leading term $\propto T_1^{-1}$) as the Hamiltonian. This
has leading coefficient
\begin{gather*}
 \frac{\prod_{1\le j\le 8} \vartheta(q/2+x_j-z_1)}
 {\vartheta(-2z_1,q-2z_1)}
 \frac{\prod_{2\le j\le n} \vartheta(t-z_1\pm z_j)}
 {\prod_{2\le j\le n} \vartheta(-z_1\pm z_j)},
\end{gather*}
which turns out to be a mild reparametrization of the leading coefficient
of the Hamiltonian proposed by van Diejen in~\cite{vanDiejenJF:1994} (see
also~\cite[equations~(3.12)--(3.14)]{vanDiejenJF:1995}, with the caveat that one must
gauge the operator), and later shown to be integrable in
\cite{KomoriY/HikamiK:1997}. In~fact, one can verify (we omit the details)
that van Diejen's operator satisfies the appropriate residue and vanishing
conditions to be a global section of~${\cal
 S}^{(n)}_{\eta',x_1,\dots,x_8;q,t}(0,2s+2f-e_1-\cdots-e_8)$, and thus we
have given a new proof that van Diejen's Hamiltonian is integrable.

\begin{rem}
 Since our Hecke algebra methods gave a new proof of the existence of the
 commuting operators which were constructed in~\cite{KomoriY/HikamiK:1997}, it is natural to wonder whether there might
 be applications in the other direction; that is, using their $R$-matrix
 based approach to construct global sections of other $\Hom$ sheaves in
 our spherical DAHA categories. Such a construction might make it
 possible to prove flatness in general without having to exclude a
 codimension~$\ge 2$ subscheme; if a given $\Hom$ bimodule generically has
 $N$ global sections, then to prove flatness and injectivity in a
 neighborhood of a given fiber, it suffices to construct $N$ local
 sections on a neighborhood of the fiber such that the restrictions to the
 fiber are linearly independent.
\end{rem}

The connection to elliptic surfaces suggests a possible generalization of
this integrable system. If~$x_1+\cdots+x_8-2\eta'$, instead of being 0, is
a torsion point of order $r$, then we again find that there are many
trivial Bruhat subquotients, and thus it becomes nontrivial to determine
how many global sections the spherical algebra has. We cannot answer this
in general, but we can, at least, show that the $r$-torsion condition
forces there to be {\em some} nontrivial global sections.

\begin{prop}
 Let $E$ be an elliptic curve and $\eta'$, $x_1,\dots,x_8$, $q$, $t$ be
 points of~$E$ such that $x_1+\cdots+x_8-2\eta'$ is a torsion point of
 order $r$. Then the corresponding fiber of the spherical algebra
 $\bigcup_d {\cal
 S}^{(n)}_{\eta',x_1,\dots,x_8;q,t}(0,d(2s+2f-e_1-\cdots-e_8))$ has a
 global section of dominant weight $(r,0,\dots,0)$ with nonzero leading
 term.
\end{prop}

\begin{proof}
 Indeed, every subquotient in the Bruhat filtration for the order ideal
 $[{<}(r,0,\dots,0)]$ is acyclic: the bottom subquotient is $\sO_{P^n}$,
 while the remaining subquotients are nontrivial elements of~$E^n[r]$.
\end{proof}

It is then natural to conjecture that the resulting Hamiltonian is
integrable, or more precisely the following.

\begin{conj}\label{conj:torsion_van_diejen}
 Under the same hypotheses, the fiber of
 \begin{gather*}
 {\cal S}^{(n)}_{\eta',x_1,\dots,x_8;q,t}(0,r(2s+2f-e_1-\cdots-e_8))
 \end{gather*}
 has $n+1$ global sections, all of which commute.
\end{conj}

Both parts of the proof for~$r=1$ fail here: the $m=7$ surface has too many
global sections, and the adjoint no longer gives an element of the same
spherical algebra. There is some experimental evidence for this
conjecture, however: for~$n=r=2$, the analogous statement for a suitable
degeneration to a nodal curve holds by a computer calculation. (We will
briefly discuss how to construct such degenerations at the end of the next
section.) This statement is, of course, trivial for~$n=1$ (given the
proposition), but it is worth noting there that the global sections of~${\cal S}^{(1)}_{\eta',x_1,\dots,x_8;0,0}(0,r(2s+2f-e_1-\cdots-e_8))$ are
just the pullback of the global sections of~$\sO_{\P^1}(1)$ from the base
of the elliptic fibration.

\section[The (spherical) $C^\vee C_n$ Fourier transform]{The (spherical) $\boldsymbol{C^\vee C_n}$ Fourier transform}\label{section8}

Our objective in the present section is to prove the following result.

\begin{thm}
 There is, locally on the base, an isomorphism
 \begin{gather*}
 \Gamma{\cal S}^{(n)}_{2c,x_1,\dots,x_m;q,t}
 \cong
 \Gamma{\cal S}^{(n)}_{-2c,x_1-c,\dots,x_m-c;q,t}
 \end{gather*}
 acting on objects as $ds+d'f-r_1e_1-\cdots-r_me_m\mapsto
 d's+df-r_1e_1-\cdots-r_me_m$ and triangular with respect to the Bruhat
 filtration. Moreover, this isomorphism commutes $($up to local units$)$ with
 the Selberg adjoint.
\end{thm}

We refer to this isomorphism as the ``Fourier transform'': in particular,
note that it takes multiplication operators (of degree $d'f$) to difference
operators (of degree $d's$) and (at least on the parameters) is an
involution. (In addition, though we will not be using this fact, the
Fourier transform can be represented in the analytic setting by a formal
integral operator~\cite{quadxforms}.)

Before constructing the Fourier transform, we~give some consequences. The~simplest is that we can conjugate the symmetry by an elementary
transformation.

\begin{cor}
 There is, locally on the base, an isomorphism
 \begin{gather*}
 \Gamma{\cal S}^{\prime(n)}_{x_1+2c,x_1,x_2,\dots,x_m;q,t}
 \cong
 \Gamma{\cal S}^{\prime(n)}_{x_1-c,x_1+c,x_2-c,\dots,x_m-c;q,t}
 \end{gather*}
 acting on objects as $ds+d'f-r_1e_1-r_2e_2-\cdots-r_me_m\mapsto
 (d'-r_1)s+d'f-(d'-d)e_1-r_2e_2-\cdots-r_me_m$.
\end{cor}

This also tells us that the deformations of~$\P^2$ we constructed are
independent of~$x_0$ (as one would expect).

\begin{cor}
 The restriction to~$\Z(s+f)$ of~$\Gamma{\cal S}^{\prime(n)}_{x_0;q,t}$ is
 $($fppf locally$)$ independent of~$x_0$.
\end{cor}

\begin{proof}
 The previous corollary gives (locally) an isomorphism $\Gamma{\cal
 S}^{\prime(n)}_{x_0,x_0-2c;q,t} \cong \Gamma{\cal
 S}^{\prime(n)}_{x_0-3c,x_0-c;q,t}$. The~action on objects takes
 $d(s+f)$ to~$d(s+f)$, so this local isomorphism induces a local
 isomorphism
 $\Gamma{\cal S}^{\prime(n)}_{x_0;q,t}|_{\Z(s+f)}
 \cong
 \Gamma{\cal S}^{\prime(n)}_{x_0-3c;q,t}|_{\Z(s+f)}$
 for any $x_0$ and $c$. It follows that any two geometric fibers with the
 same values of~$q$, $t$ are isomorphic.
\end{proof}

\begin{rems}
 More generally, if $(E,x_0,x_1,q,t)$ is a~point of~${\cal E}^4$ over some
 scheme $S$, then we have an isomorphism $\Gamma{\cal
 S}^{\prime(n)}_{x_0;q,t}|_{\Z(s+f)}\cong \Gamma{\cal
 S}^{\prime(n)}_{x_1;q,t}|_{\Z(s+f)}$ defined Zariski locally on~$S$ as long as
 $x_1-x_0\in 3E(S)$. Without this assumption, there may very well be no
 such isomorphism; indeed for~$n=1$, $q=0$, these are essentially the
 homogeneous coordinate rings of the embeddings of~$C$ via $[x_0]+2[0]$
 and $[x_1]+2[0]$.
\end{rems}

\begin{rems}
 When $c\in E[3]$, this isomorphism becomes an automorphism, but is quite
 nontrivial.
\end{rems}

The most significant consequence is the following. Note here that we are, as~usual, taking global sections {\em before} passing to fibers.

\begin{thm}
 For any $v\in \Z\langle s,f,e_1,\dots,e_m\rangle$, there is a codimension~$\ge 2$ subscheme of para\-me\-ter space on the complement of which
 $\Gamma{\cal S}^{(n)}_{\eta',x_1,\dots,x_m;q,t}(0,v)$ is flat and the map
 to meromorphic difference operators is injective on fibers.
\end{thm}

\begin{proof}
 Applying the Fourier transform has no effect on flatness (since it is an
 isomorphism), and the Fourier transform will be constructed via an action
 on meromorphic difference operators, and thus injectivity on fibers is
 also preserved. This allows us to reduce to Corollary~\ref{cor:universally_ample_is_flat}, as~in~\cite{noncomm1}. To~be
 precise, let $v=ds+d'f-r_1e_1-\cdots-r_me_m$. We may apply a permutation
 and an even number of elementary transformations to put $v$ into the
 fundamental chamber for~$W(D_m)$. Moreover, if $r_m<0$, then we may set it
 to 0 without changing the sheaf of global sections, and in this way may
 arrange to have $d\ge r_1+r_2$ and $r_1\ge \cdots\ge r_m\ge 0$. If~$d'\ge d$, then we may apply Corollary~\ref{cor:universally_ample_is_flat}. If~$d'<0$, then we observe that
 $\Gamma{\cal S}^{(n)}_{\eta',x_1,\dots,x_m;q,t}(0,v)=0$, so the result
 again follows. Otherwise, we~apply the Fourier transform. Since this
 strictly decreases $d$ but keeps it nonnegative, the claim follows by
 induction.
\end{proof}

\begin{rem}
 Of course, the codimension~$\ge 2$ subscheme is the same as that of the
 appropriate special case of Corollary~\ref{cor:universally_ample_is_flat}.
\end{rem}

Just as elliptic pencils gave rise to integrable systems above, there is
something analogous (if slightly weaker) for rational pencils. Call a
small category with object set $\Z$ ``quasi-abelian'' if there is a
commutative graded algebra $A$ such that $\Hom(j,k)\cong A[k-j]$ for all
$j$, $k$, with composition given by multiplication.

\begin{cor}
 Suppose $v\in \Z\langle s,f,e_1,\dots,e_m\rangle$ is the class of a
 rational pencil on the rational surface $X_m$. Then $\Gamma{\cal
 S}^{(n)}_{\eta',x_1,\dots,x_m;q,t}|_{\Z v}$ is quasi-abelian, and the
 corresponding graded algebra is a~free polynomial algebra in~$n+1$
 generators.
\end{cor}

\begin{proof}
 If $v=f$, this is easy: the global sections of degree $df$ are just
 multiplication by $C_n$-inva\-riant sections of the bundle with
 polarization~$d\sum_i z_i^2$, and this is precisely the pullback of~$\sO_{\P^n}(d)$. More generally, it follows from the theory of rational
 surfaces (see~\cite{rat_Hitchin}) that $v$ represents a rational pencil
 iff it is in the orbit of~$f$ under the group $W(E_{m+1})$ generated by
 the Fourier transform and $W(D_m)$.
\end{proof}

Just as integrable systems lead to natural eigenvalue equations, such
``quasi-integrable'' systems lead to generalized eigenvalue problems. A
{\em generalized eigenfunction} of a space ${\cal D}$ of operators is a
function~$f$ such that the image ${\cal D}f$ is $1$-dimensional. (We then
obtain an associated ``generalized eigenvalue'', namely the point in~$\P({\cal D})$ associated to the kernel of the map $D\mapsto Df$ on~${\cal
 D}$.)

Given a quasi-integrable system associated to a rational pencil, we~have
for each $d$ a map~$\phi_d$ from $A[1]$ to the space of operators, such
that $\phi_{d+1}(y)\phi_d(x)=\phi_{d+1}(x)\phi_d(y)$. We may then consider
for each $d$ the generalized eigenvalue problem associated to~$\phi_d(A[1])$. For~any generalized eigenfunction~$f_d$ for~$\phi_d(A[1])$,
let $f_{d+1}$ be a nonzero representative of~$\phi_d(A[1])f_d$. Then
for suitable $y\in A[1]$, we~have
\begin{gather*}
\phi_{d+1}(x) f_{d+1}
=
\phi_{d+1}(x) \phi_d(y) f_d
=
\phi_{d+1}(y) \phi_d(x) f_d
=
\lambda_d(x) \phi_{d+1}(y) f_{d+1}
\end{gather*}
for all $x\in A[1]$, and thus $f_{d+1}$ is a generalized eigenfunction for~$\phi_{d+1}(A[1])$. More generally, if $V\subset A[1]$ is such that the
corresponding generalized eigenvalue problem for~$\phi_{d+1}(V)$ is
nondegenerate (i.e., for~each point of projective space, the corresponding
problem has at most $1$-dimensional solution space), then any generalized
eigenfunction~$f_d$ for~$\phi_d(V)$ is a generalized eigenfunction for~$\phi_d(A[1])$, since then $\phi_d(y)f_d$ is a generalized eigenfunction
for~$\phi_{d+1}(V)$.

There are two cases of particular interest. In~the case $v=s+f-e_1-e_2$,
$\eta'=-(n-1)t-q$, the generators of the quasi-integrable system are
operators of the form considered in~\cite{bctheta}, and the quasi-abelian
property turns into the quasi-commutation relation used there. The~corresponding generalized eigenvalue problem is precisely the difference
equation~\cite[Proposition~3.9]{bctheta} satisfied by the elliptic interpolation
functions. (The interpolation kernel of~\cite{quadxforms} is also a
generalized eigenfunction for essentially the same space of operators,
\cite[Proposition~3.12]{quadxforms}.)

The biorthogonal functions of~\cite{bctheta,xforms} are also generalized
eigenfunctions of such a quasi-integrable system, corresponding to~$v=2s+2f-e_1-e_2-e_3-e_4-2e_5$ and $\eta'=-(n-1)t-q$. Indeed, the
first-order difference operators considered in~\cite{xforms} correspond to
products of operators of degrees $s-e_5$, $s+2f-e_1-e_2-e_3-e_4-e_5$ and
$s+f-e_i-e_j-e_5$, $1\le i<j\le 4$, giving rise to~8 operators of degree
$2s+2f-e_1-e_2-e_3-e_4-2e_5$. The~biorthogonal functions are generalized
eigenfunctions of the span of these 8 operators, and the generalized
eigenvalues are all distinct points of~$\P^7$. It thus follows that any
biorthogonal function is a generalized eigenfunction for the full space of
operators. (With more effort, one can in fact verify that the generalized
eigenvalues are given by suitable specializations of the leading
coefficients; this is a consequence in general of the fact that the Fourier
transform respects leading coefficients.)

We now turn to constructing the Fourier transform. The~traditional
approach would be to construct a Fourier transform on the DAHA and then
observe that it restricts to a transform on the spherical algebra. One
significant issue that arises here is that although we have a~reasonable
facsimile of a presentation, it is at the level of sheaves, not at the
level of global sections, while the Fourier transform does not make sense
in terms of sheaves (since it does not preserve multiplication).
Furthermore, most of the rank~1 subalgebras we used to generate the DAHA do
not have {\em any} nontrivial global sections (the leading Bruhat
subquotient is a generically nontrivial line bundle of degree 0 in every
variable). As~a result, it seems unlikely that the Fourier transform on
the DAHA (assuming it exists) would have a construction that was
significantly simpler than the construction we give in the spherical case.
Beyond that, there is another issue: as we discussed above, the description
of the rank~1 DAHA via a Morita equivalence to the spherical algebra
strongly suggests that the Fourier transform only exists for the noncompact
version of the DAHA. In~other words, the Fourier transform on the DAHA
would not respect the filtration by degree; since this filtration comes
from the Bruhat filtration, the latter also could not be preserved. As~a
result, even having a Fourier transform for the DAHA would not be enough to
prove the theorem; one also needs to understand why the spherical version
is triangular!

We thus wish an approach that works directly with the spherical algebra.
Note that since the action of the Fourier transform on objects preserves
$\Z\langle s,f\rangle$, the Fourier transform for~$m>0$ restricts to a
transform of the same sort for~$m=0$. Moreover, since every $\Hom$ sheaf
is contained in one of degree in~$\Z\langle s,f\rangle$, it suffices to
specify how the transform acts on such sheaves and show that it preserves
the various subsheaves of interest. We thus focus our initial attention on
the case $m=0$.

In the univariate setting, the Fourier transform was easy to construct: for
generic parameters, one can give an explicit presentation for the category
(with generators of degrees $s$ and $f$) and this presentation has an
obvious symmetry. Moreover, a slightly larger set of elements generates
the category even without the genericity condition, and one can determine
how the transform must act on those elements by taking a suitable limit.

Although we have analogues of those generators (and will indeed be able to
describe their Fourier transforms explicitly), this approach founders in
the multivariate setting for two reasons. The~first is that the operators
of degree $s$ and $f$ do not even come close to generating the category for~$n>1$ in general: for~$q=0$, $t=0$, the full category is the bihomogeneous
coordinate ring of~$\Sym^n\big(\P^1\times \P^1\big)$, while the elements of degrees
$s$ and $f$ lie in the subring corresponding to the quotient $\P^n\times
\P^n$. If~we include elements of degree $s+f$, the situation is somewhat
better (we will see that these come close enough to generating to be
useful), but this only forces us to confront the fact that we have
absolutely no understanding of the {\em relations} satisfied by these
elements.

As a result, we~will need some way to construct the Fourier transform which
is explicitly a homomorphism. We will do this by constructing a transform
on a much larger algebra of operators, and then show that it preserves the
particular subspace we care about. The~simplest way to construct a
homomorphism on a category of operators is to apply a gauge transformation:
assign an operator to each object and apply the associated
quasi-conjugation.

The first step in constructing such operators is to determine on what
spaces they act, and thus we need to think a bit about where our existing
operators act. Define a family of (gerbe) polarizations
\begin{gather*}
P_d(\eta';q,t):=-((n-1)t-(d-1)q+\eta')\sum_i z_i^2/q.
\end{gather*}
If $F$ is the product of a $\Gamq$ symbol with polarization~$P_{d'_1-d_1}(\eta';q,t)$ and a rational function on~${\cal E}^n$, then we
can apply any global section of a fiber of~${\cal
 S}^{(n)}_{\eta';q,t}(d_1s+d'_1f,d_2s+d'_2f)$ and the result will be a
rational function times a $\Gamq$ symbol with polarization~$P_{d'_2-d_2}(\eta';q,t)$.

Thus the Fourier transform should be given by operators that take functions
with polarization~$P_d(\eta';q,t)$ to functions with polarization~$P_{-d}(-\eta';q,t)$, or equivalently take $P_0(\eta'+dq;q,t)$ to~$P_0(-\eta'-dq;q,t)$. There are issues in general, but there is one
important case in which ope\-ra\-tors of this form do indeed exist. Indeed,
the simplest way to obtain an operator mapping $P_d(0;q,t)$ to~$P_{-d}(0;q,t)$ would be to take a global section of~${\cal
 S}^{(n)}_{0;q,t}(df,ds)$, assuming such a global section exists.

For $d=1$, this is not too difficult to control, and indeed we can understand
global sections of order 1 in general.

\begin{lem}\label{lem:first_order}
 For any point of~${\cal E}^3$, the corresponding fiber of~${\cal
 S}^{(n)}_{\eta';q,t}(0,s+d'f)$ is spanned by operators of the form
 \begin{gather*}
 D^{(n)}_q(u_0,u_1,\dots,u_{2d'+1};t)
 \\ \qquad
 {}=
 \sum_{\sigma\in \{\pm 1\}^n}
 \prod_{1\le i\le n}
 \frac{\prod_{0\le r<2d'+2} \vartheta(u_r+\sigma_i z_i)}
 {\vartheta(2\sigma_i z_i)}
 \prod_{1\le i<j\le n}
 \frac{\vartheta(t+\sigma_i z_i+\sigma_j z_j)}
 {\vartheta(\sigma_i z_i+\sigma_j z_j)}
 \prod_{1\le i\le n}T_i^{\sigma_i/2},
 \end{gather*}
 with~$u_0+\cdots+u_{2d'+1}=q+\eta'$.
\end{lem}

\begin{proof}
 The interval $[\le (1/2,\dots,1/2)]$ in the Bruhat order consists of a
 single double coset, and thus the space of global sections is (up to
 multiplication by an explicit product of~$\vartheta$ functions) the space
 of~$S_n$-invariant sections of the appropriate line bundle. That space
 is spanned by products of the form $\prod_{1\le i\le n} f(z_i)$, where $f$
 is a section of the corresponding line bundle on~${\cal E}$, and any such
 section can be factored into~$\vartheta$ functions.
\end{proof}

\begin{prop}
 For any $d\ge 0$, the space of global sections of~${\cal
 S}^{(n)}_{0;q,t}(df,ds)$ is $1$-dimensional.
\end{prop}

\begin{proof}
 We proceed by induction in~$d$, with the case $d=0$ being obvious.
 Suppose we are given a nonzero global section~$D^{(n)}_{q,t}(d)\in {\cal
 S}^{(n)}_{0;q,t}(df,ds)$. Then for any $u$, $v$, the operator
 \begin{gather*}
 D^{(n)}_q((d+1)q/2\pm u;t) D^{(n)}_{q,t}(d)
 \prod_{1\le i\le n} \vartheta(z_i\pm v)
 \\ \phantom{D^{(n)}_q((d+1)q/2\pm u}
 {}- D^{(n)}_q((d+1)q/2\pm v;t) D^{(n)}_{q,t}(d)
 \prod_{1\le i\le n} \vartheta(z_i\pm u)
 \end{gather*}
 is a section of~${\cal S}^{(n)}_{0;q,t}((d-1)f,(d+1)s)$. Moreover, we~know the leading coefficient of~$D^{(n)}_{q,t}(d)$ up to a scalar multiple,
 and may therefore verify that both operators have the same leading
 coefficient. Since every subquotient below the top of the corresponding
 univariate vector bundle has negative degree, none of the multivariate
 subquotients below the top have polarizations represented by positive
 semidefinite matrices. Thus none of those subquotients have any global
 sections, let alone symmetric ones. It follows that a section of~${\cal
 S}^{(n)}_{0;q,t}((d-1)f,(d+1)s)$ with vanishing leading coefficient
 must in fact be 0, and thus
 \begin{gather*}
 D^{(n)}_q((d+1)q/2\pm u;t)
 D^{(n)}_{q,t}(d)\!\!\!
 \prod_{1\le i\le n}\!\!\! \vartheta(z_i\pm v)
 =
 D^{(n)}_q((d+1)q/2\pm v;t)
 D^{(n)}_{q,t}(d)\!\!\!
 \prod_{1\le i\le n} \!\!\!\vartheta(z_i\pm u).
 \end{gather*}
 Equivalently,
 \begin{gather*}
 D^{(n)}_q((d+1)q/2\pm u;t)
 D^{(n)}_{q,t}(d)
 \prod_{1\le i\le n} \vartheta(z_i\pm u)^{-1}
 \end{gather*}
 is independent of~$u$. In~particular, the apparent $u$-dependent poles
 of this product of operators are not, in fact, singularities, and thus
 this gives a section of~${\cal S}^{(n)}_{0;q,t}((d+1)f,(d+1)s)$ as
 required. That this is the only global section up to scalar multiples
 follows by observing that again all Bruhat subquotients below the top
 have indefinite polarizations, while the top subquotient is trivial.
\end{proof}

Following the above proof, we~define $D^{(n)}_{q,t}(d)$ by the recurrence
\begin{gather*}
D^{(n)}_{q,t}(d+1)
=
D^{(n)}_q((d+1)q/2\pm u;t)
D^{(n)}_{q,t}(d)
\prod_{1\le i\le n} \vartheta(z_i\pm u)^{-1},
\end{gather*}
with base case $D^{(n)}_{q,t}(0)=1$. (Note that
$D^{(n)}_{q,t}(1)=D^{(n)}_q(;t)$.) Equivalently, $D^{(n)}_{q,t}(d)$ is the
unique global section of~${\cal S}^{(n)}_{0;q,t}(df,ds)$ with leading term
\begin{gather*}
\prod_{1\le i<j\le n}
 \frac{\Gamq(dq+t-z_i-z_j)}
 {\Gamq(t-z_i-z_j)}
\prod_{1\le i\le j\le n}
 \frac{\Gamq(-z_i-z_j)}
 {\Gamq(dq-z_i-z_j)}
\prod_{1\le i\le n} T_i^{-d/2}.
\end{gather*}
Since translation by $s+f$ does not change the parameters, this also gives
a global section of~${\cal S}^{(n)}_{0;q,t}(d_0(s+f)+df,d_0(s+f)+ds)$ for
any $d_0$.

\begin{cor}
 If $q$ has exact order $d$, then
 \begin{gather*}
 D^{(n)}_{q,t}(d)
 =
\sum_{\sigma\in \{\pm 1\}^n}
\prod_{1\le i\le n}
 \frac{1}
 {\vartheta(2\sigma_i z_i;q)_d}
\prod_{1\le i<j\le n}
 \frac{\vartheta(t+\sigma_i z_i+\sigma_j z_j;q)_d}
 {\vartheta(\sigma_i z_i+\sigma_i z_j;q)_d}
\prod_{1\le i\le n} T_i^{-d/2}.
\end{gather*}
\end{cor}

\begin{proof}
 By Theorem~\ref{thm:CCn_center}, the center of the given spherical
 algebra is itself a spherical algebra with~$q=0$. In~particular, the
 center contains a nonzero element mapping $f$ to~$s$, which becomes an
 element of~${\cal S}^{(n)}_{0;q,t}(df,ds)$ under the isomorphism. By
 uniqueness, this element is proportional to~$D^{(n)}_{q,t}(d)$. For~this
 element to be central, every shift that appears must be congruent to~$-d/2$ modulo $d$. There is only one $C_n$-orbit of shifts that
 survives, so that we may recover all coefficients from the leading
 coefficient, and obtain the stated formula.
\end{proof}

\begin{rem}
 If we gauge by a product of gamma functions before specializing $q$ (\`a
 la an elementary transformation), the interior coefficients will still
 vanish, giving central sections of degree \mbox{$d(s+d'f)$} for any $d'$, and
 establishing that any element of the center with degree of the given form
 is obtained in this way. Since the Fourier transform respects leading
 coefficients, we~can compute how it acts on central elements of degree
 $ds$, $df$, $d(s+f)$, and thus conclude (following the argument below)
 that the Fourier transform respects the isomorphism of Theorem~\ref{thm:CCn_center}.
\end{rem}

The following result shows that these operators indeed behave like Fourier
transforms.

\begin{prop}
 We have the operator relations
 \begin{gather*}
 D^{(n)}_q((d+1)q/2\pm u;t) D^{(n)}_{q,t}(d)=
 D^{(n)}_{q,t}(d+1)
 \prod_{1\le i\le n} \vartheta(z_i\pm u),
 \\
 D^{(n)}_{q,t}(d) D^{(n)}_q(-dq/2\pm u;t)=
 \prod_{1\le i\le n} \vartheta(z_i\pm u)
 D^{(n)}_{q,t}(d+1),
 \end{gather*}
 and, if $u_0+u_1+u_2+u_3=(d+1)q$,
 \begin{gather*}
 D^{(n)}_{q,t}(d)
 D^{(n)}_q(u_0,u_1,u_2,u_3;t)
 =
 D^{(n)}_q(u_0+dq/2,u_1+dq/2,u_2+dq/2,u_3+dq/2;t)
 D^{(n)}_{q,t}(d).
 \end{gather*}
\end{prop}

\begin{proof}
 In each case, both sides are sections of the same $\Hom$ sheaf of~${\cal
 S}^{(n)}_{0;q,t}$ with the same leading coefficient, and only the top
 subquotient has positive semidefinite polarization.
\end{proof}

It turns out that if we adjoined formal inverses of the operators
$D^{(n)}_{q,t}(d)$ and declared them to be $D^{(n)}_{q,t}(-d)$, then the result
would indeed define a Fourier transform on a certain subcategory of the
category with~$\eta'=0$ (in which the $\Hom$ sheaves of degree $ds+d'f$ for~$d>d'$ are replaced by the images under the Fourier transform of the $\Hom$
sheaves of degree $d's+df$). Proving this directly is somewhat tricky,
however, as~unlike in the univariate setting, there does not appear to be a
readily accessible test for right divisibility by $D^{(n)}_{q,t}(d)$. And,
of course, even using the translation symmetry, this would at best give us
a transform for~$\eta'\in \Z q$, which is especially weak when~$q$ is
torsion.

The key idea for proceeding further is that the relation
\begin{gather*}
D^{(n)}_{q,t}(d)
D^{(n)}_q(-dq/2\pm u;t)
=
\prod_{1\le i\le n} \vartheta(z_i\pm u)
D^{(n)}_{q,t}(d+1)
\end{gather*}
gives us a system of recurrences that we can use to solve for the
coefficients of~$D^{(n)}_{q,t}(d)$. Indeed, it follows from this relation
that
\begin{gather*}
D^{(n)}_{q,t}(d)
D^{(n)}_q(-dq/2\pm u;t)
|_{u=z_i}
=
0
\end{gather*}
for~$1\le i\le n$. Since the operators are symmetric, let us consider the
specialization~$u=z_n$. The~coefficient of~$\prod_i T_i^{k_i-(d+1)/2}$ in
this specialized operator is a linear combination of the left coefficients
of~$\prod_i T_i^{l_i-d/2}$ in~$D^{(n)}_{q,t}(d)$ for~$\max(k_i-1,0)\le l_i\le
k_i$. The~coefficient in this linear combination for~$\vec{l}=\vec{k}$ is
\begin{gather*}
\frac{\prod_{1\le i\le n} \vartheta(-k_iq+z_n-z_i,k_iq+z_i+z_n)
 \prod_{1\le i<j\le n} \vartheta(t+dq-(k_i+k_j)q-z_i-z_j)}
 {\prod_{1\le i\le j\le n} \vartheta(dq-(k_i+k_j)q-z_i-z_j)},
\end{gather*}
and thus we can solve for the coefficient of~$\prod_i T_i^{k_i-d/2}$ in~$D^{(n)}_{q,t}(d)$, at least generically. In~fact, we~find the only
difficulty arises when $\vartheta(-k_nq)=0$, so if $q$ is not torsion and
$\vec{k}\ne 0$, there will always be one of the $n$ specializations that
allows us to solve for the coefficient of~$\prod_i T_i^{k_i-d/2}$ in terms
of coefficients of terms which are smaller in dominance order. In~other
words, $D^{(n)}_{q,t}(d)$ is determined by the given relation along with the
choice of leading coefficient.

The fact that we can control coefficients near the leading coefficients
suggests a way to proceed further: take an appropriate completion! Define
a nonarchimedean metric on the $\Z/2\Z$-graded algebra
$k(X)\big[T_1,\dots,T_n,\prod_i T_i^{-1/2}\big]$ by
\begin{gather*}
\bigg|\sum_{\vec{k}} c_k \prod_i T_i^{k_i}\bigg|
:=
\max_{\vec{k}:c_k\ne 0} \exp\bigg({-}\sum_i k_i\bigg).
\end{gather*}
We call an element of the corresponding completion a {\em formal difference
 operator}. We in particular denote the completion of the subalgebra
$k(X)[T_1,\dots,T_n]$ by $k(X)[[T_1,\dots,T_n]]$. This construction of
course applies equally well to the case of twisted difference operators, or
even to~the~corresponding category in which the objects are polarizations
$P_d(\eta';q,t)$ with fixed $\eta'$. We~will mostly suppress the twisting
from the notation.

The major advantage of formal difference operators is that the ring has a
large number of units. Indeed, the usual argument for inverting a
commutative formal power series with invertible constant term applies
equally well in the noncommutative setting to give the following.

\begin{prop}
 If $D\in k(X)[[T_1,\dots,T_n]]$ has nonzero constant term, then $D$ is a unit.
\end{prop}

Since $\prod_i T_i^{-1/2}$ is also clearly invertible, we~find that any of
the operators $D^{(n)}_{q,t}(d)$ are inver\-tible as formal operators. In~fact, in the ring of formal operators we can solve for~$D^{(n)}_{q,t}(d)$ in
terms of~$D^{(n)}_{q,t}(d+1)$ and in this way define $D^{(n)}_{q,t}(d)$ for~$d<0$. We then find by an easy induction that
$D^{(n)}_{q,t}(-d)=D^{(n)}_{q,t}(d)^{-1}$.

In addition to these inner automorphisms, we~also have automorphisms coming
from gauging by $\Gamq$ symbols and translations on~$E$. Let $T_\omega(c)$
denote the translation of all variables by $c$, so that $T_\omega(q/2) =
\prod_{1\le i\le n} T_i^{1/2}$. Then for any $\Gamq$ symbol $\Gamma$
of polarization
\begin{gather*}
(\eta'-\eta'')\sum_i z_i^2/q
+2((n-1)t+q+\eta')c\sum_i z_i/q,
\end{gather*}
there is an induced isomorphism
\begin{gather*}
D\mapsto \Gamma T_\omega(c) D T_\omega(-c) \Gamma^{-1}
\end{gather*}
from $\End(P_0(\eta';q,t))^0$ (the subspace involving only integer powers
of~$T_i$) to~$\End(P_0(\eta'';q,t))^0$. With this in mind, we~define a
``formal gauging operator'' from $P_0(\eta';q,t)$ to~$P_0(\eta'';q,t)$ to
be an object of the form
\begin{gather*}
\Gamma T_\omega(c) D,
\end{gather*}
where $D$ is a unit in the endomorphism ring. The~``leading term'' of such
an operator is the formal symbol $\Gamma T_\omega(c) f$, where $f$ is the
constant term of~$D$. The~formal gauging operators form a group, with a
natural subgroup consisting of elements of the form $\Theta T_\omega(kq/2)
D$, where $\Theta$ is a product of~$\vartheta$ symbols (i.e., of invertible
formal difference operators). If~$G_1$, $G_2$ are formal gauging operators
such that $G_1G_2^{-1}$ lies in the subgroup of formal difference
operators, then for any formal difference operator $D$ with only integer
shifts (and with coefficients having appropriate polarizations), $G_1 D
G_2^{-1}:=(G_1 G_2^{-1}) G_2 D G_2^{-1}$ will again be a formal difference
operator. This extends to the half-integer case by writing
$D=T_\omega(q/2) D'$ and $G_1 D G_2^{-1}:=(G_1 T_\omega(q/2) G_2^{-1}) G_2
D' G_2^{-1}$. In~either case, the operation clearly respects
multiplication as long as the gauging operators match up.

\begin{prop}\label{prop:fourier_xform_exists}
 There is a unique family of formal gauging operators ${\cal
 D}^{(n)}_{q,t}(c)$ from $P_0(2c;q,t)$ to~$P_0(-2c;q,t)$
 with leading term
 \begin{gather*}
 \prod_{1\le i\le j\le n}
 \frac{\Gamq(-z_i-z_j)}
 {\Gamq(-2c-z_i-z_j)}
 \prod_{1\le i<j\le n}
 \frac{\Gamq(t-2c-z_i-z_j)}
 {\Gamq(t-z_i-z_j)}
 T_\omega(c)
 \end{gather*}
 such that ${\cal D}^{(n)}_{q,t}(-dq/2) = D^{(n)}_{q,t}(d)$ for all $d\in
 \Z$. Moreover, if one divides any coefficient of~${\cal D}^{(n)}_{q,t}(c)$
 by the leading term, then the only $z$-independent poles of the resulting
 meromorphic section on~${\cal E}^{n+3}$ are along hypersurfaces for
 which $q$ is torsion.
\end{prop}

\begin{proof}
 If such a family of operators exists, then it must satisfy
 \begin{gather*}
 {\cal D}^{(n)}_{q,t}(c)
 D^{(n)}_q(c\pm u;t)
 |_{u=z_i}
 =
 0
 \end{gather*}
 for~$1\le i\le n$. This gives an algebraic (and triangular) system of
 equations for the coefficients of~${\cal D}^{(n)}_{q,t}(c)$ which we have
 already seen has at most one solution (and if it has a solution, the only
 $z$-independent poles are where $q$ is torsion). Since it has a solution
 on the Zariski dense set of divisors $c\in \Z q/2$, it must have a solution
 in general.
\end{proof}

\begin{rem}
 For an analytic approach to constructing such operators, see~\cite{quadxforms}.
\end{rem}

Since we understand a Zariski dense subset of these operators, we~can
immediately deduce some relations.

\begin{prop}
 The operators ${\cal D}^{(n)}_{q,t}(c)$ satisfy the operator identities
 \begin{gather*}
 D^{(n)}_q(-c\pm u;t) {\cal D}^{(n)}_{q,t}(c+q/2)
 = {\cal D}^{(n)}_{q,t}(c) \prod_{1\le i\le n} \vartheta(z_i\pm u),
 \\
 {\cal D}^{(n)}_{q,t}(c) D^{(n)}_q(c\pm u;t)
 = \prod_{1\le i\le n} \vartheta(z_i\pm u) {\cal D}^{(n)}_{q,t}(c-q/2),
 \end{gather*}
 and, if $u_0+u_1+u_2+u_3=q+2c$,
 \begin{gather*}
 {\cal D}^{(n)}_{q,t}(c) D^{(n)}_q(u_0,u_1,u_2,u_3;t)
 =
 D^{(n)}_q(u_0-c,u_1-c,u_2-c,u_3-c;t) {\cal D}^{(n)}_{q,t}(c).
 \end{gather*}
\end{prop}

We also note the following fact, generalizing the first two identities.

\begin{prop}\label{prop:fourier_braid_relation}
 The operators ${\cal D}^{(n)}_{q,t}(c)$ satisfy the operator identity
 \begin{gather*}
 \prod_{1\le i\le n}\!\! \frac{\Gamq(t_0\,{-}\,d\pm z_i)} {\Gamq(t_0\,{+}\,d\pm z_i)}
 {\cal D}^{(n)}_{q,t}(c\,{+}\,d)\!\!\! \prod_{1\le i\le n}\!\!
 \frac{\Gamq(t_0\,{-}\,c\pm z_i)} {\Gamq(t_0\,{+}\,c\pm z_i)}
 \,{=}\, {\cal D}^{(n)}_{q,t}(c)\!\!\! \prod_{1\le i\le n}\!\!
 \frac{\Gamq(t_0\,{-}\,c\,{-}\,d\pm z_i)} {\Gamq(t_0\,{+}\,c\,{+}\,d\pm z_i)} {\cal D}^{(n)}_{q,t}(d).
 \end{gather*}
 In particular, ${\cal D}^{(n)}_{q,t}(c)^{-1}={\cal D}^{(n)}_{q,t}(-c)$.
\end{prop}

\begin{proof}
 Consider the composition
 \begin{gather*}
 \prod_{1\le i\le n} \frac{\Gamq(t_0+d\pm z_i)} {\Gamq(t_0-d\pm z_i)}
 {\cal D}^{(n)}_{q,t}(c) \prod_{1\le i\le n}
 \frac{\Gamq(t_0-c-d\pm z_i)} {\Gamq(t_0+c+d\pm z_i)}
 {\cal D}^{(n)}_{q,t}(d) \prod_{1\le i\le n}
 \frac{\Gamq(t_0+c\pm z_i)} {\Gamq(t_0-c\pm z_i)}.
 \end{gather*}
 If we substitute
 \begin{gather*}
 {\cal D}^{(n)}_{q,t}(d) =
 D^{(n)}_q(-d\pm (t_0-c);t) {\cal D}^{(n)}_{q,t}(d+q/2)
 \prod_{1\le i\le n} \vartheta(z_i\pm (t_0-c))^{-1},
 \end{gather*}
 then apply the easy relation
 \begin{gather*}
 \prod_{1\le i\le n} \frac{\Gamq(t_0-c-d\pm z_i)} {\Gamq(t_0+c+d\pm z_i)}
 \,D^{(n)}_q(c-d-t_0,t_0-c-d;t)
 \\ \phantom{\prod_{1} \frac{\Gamq(t_0-c-d\pm z_i)} {\Gamq(t_0+c+d\pm z_i)}}\!\!
 {}=
 D^{(n)}_q(c-d-t_0,t_0+c+d;t)
 \prod_{1\le i\le n}
 \frac{\Gamq(t_0+q/2-c-d\pm z_i)}
 {\Gamq(t_0+q/2+c+d\pm z_i)},
 \end{gather*}
 we can combine the two operators:
 \begin{gather*}
 {\cal D}^{(n)}_{q,t}(c)
 D^{(n)}_q(c\pm (t_0+d);t)
 =
 \prod_{1\le i\le n} \vartheta(z_i\pm (t_0+d))
 {\cal D}^{(n)}_{q,t}(c-q/2),
 \end{gather*}
 and find that the result simplifies to the case $(c,d,t_0)\mapsto
 (c-q/2,d+q/2,t_0+q/2)$ of the above composition. In~other words, the
 given operator is invariant under such translations, so by density is
 invariant under any translation~$(c,d,t_0)\mapsto (c-u,d+u,t_0+u)$.
 Taking $u=c$ gives
 \begin{gather*}
 \prod_{1\le i\le n}
 \frac{\Gamq(t_0+2c+d\pm z_i)}
 {\Gamq(t_0-d\pm z_i)}
 {\cal D}^{(n)}_{q,t}(0)
 \prod_{1\le i\le n}
 \frac{\Gamq(t_0-d\pm z_i)}
 {\Gamq(t_0+2c+d\pm z_i)}
 {\cal D}^{(n)}_{q,t}(c+d)
 =
 {\cal D}^{(n)}_{q,t}(c+d),
 \end{gather*}
 since ${\cal D}^{(n)}_{q,t}(0)=1$.
\end{proof}

\begin{rems} Compare the proof of~\cite[Theorem~4.1]{bctheta}. The~similarity
 in arguments is not at all a coincidence: The analytic construction of~${\cal D}^{(n)}_{q,t}(c)$ in terms of the interpolation kernel of~\cite{quadxforms} implies that one can obtain the elliptic binomial
 coefficients of~\cite{bctheta} as specializations of the coefficients of~${\cal D}^{(n)}_{q,t}(c)$, making~\cite[Theorem~4.1]{bctheta} a (Zariski
 dense) special case of the above relation.
\end{rems}

\begin{rems} The action of the Fourier transform on objects is a reflection
 in an appropriate inner product (the intersection form of the surface!), as~are the generators of the $W(D_m)$ action. Each generator has a
 certain action on operators. If~one takes into account the action on
 parameters, the generators are involutions, and all relevant braid
 relations are satisfied, so that this gives an action of a~Coxeter group
 $W(E_{m+1})$ on the base of the family. Showing that this lifts to an
 action of~$W(E_{m+1})$ on the actual sheaf categories reduces to
 verifying lifts of each braid relation, and the only nontrivial such lift
 reduces to the above identity.
\end{rems}

We thus define a Fourier transform on formal difference operators in the
following way. If~the formal operator $D$ maps the polarization~$P_0(2c;q,t)$
to the polarization~$P_0(2c';q,t)$, then its Fourier transform $\hat{D}$
is the operator
\begin{gather*}
 {\cal D}^{(n)}_{q,t}(c') D {\cal D}^{(n)}_{q,t}(-c)
\end{gather*}
mapping $P_0(-2c;q,t)$ to~$P_0(-2c';q,t)$. There is some choice here (since
the polarizations only depend on~$2c$, $2c'$), but luckily it is not
particularly serious.

\begin{lem}
 If $\tau$ is a $2$-torsion point, then
 \begin{gather*}
 {\cal D}^{(n)}_{q,t}(c+\tau)
 =
 T_\omega(\tau)
 {\cal D}^{(n)}_{q,t}(c)
 =
 {\cal D}^{(n)}_{q,t}(c)
 T_\omega(\tau).
 \end{gather*}
\end{lem}

\begin{proof}
 Indeed, the recurrence we used to solve for the coefficients of~${\cal
 D}^{(n)}_{q,t}(c)$ is equivariant under translation by $2$-torsion.
\end{proof}

For our purposes, we~will always be working in the subcategory with objects
$P_0(-dq+\eta';q,t)$, and will take $c$, $c'$ in the Fourier transform to
be the appropriate linear combination of~$q/2$ and some fixed $\eta'/2$.
We have, of course, already computed some instances of the Fourier
transform:
\begin{gather*}
\prod_{1\le i\le n} \vartheta(z_i\pm u)\mapsto D^{(n)}_q(q/2-c\pm u;t),
\\
D^{(n)}_q(c+q/2\pm u;t)\mapsto \prod_{1\le i\le n} \vartheta(z_i\pm u),
\end{gather*}
and
\begin{gather*}
D^{(n)}_q(u_0,u_1,u_2,q\!+\!2c\!-\!u_0\!-\!u_1\!-\!u_2;t)
\mapsto D^{(n)}_q(u_0\!-\!c,u_1\!-\!c,u_2\!-\!c,q\!+\!c\!-\!u_0\!-\!u_1\!-\!u_2;t),
\end{gather*}
where in each case the input is a (general) section of~${\cal
 S}^{(n)}_{2c;q,t}$ starting from the $0$ object, of deg\-ree~$f$,~$s$, and
$s+f$ respectively.

\begin{thm}
 If $c'-c$ is an integer multiple of~$q/2$, then the Fourier transform
 is holomorphic; that is, the Fourier transform of any holomorphic family
 of operators is a holomorphic family of operators.
\end{thm}

\begin{proof}
 The only issue is when~$q$ is torsion, as~otherwise both
 ${\cal D}^{(n)}_{q,t}(-c)$ and ${\cal D}^{(n)}_{q,t}(c')$ are
 holomorphic (in the sense that they have no $z$-independent poles other
 than those for~$q$ torsion).

 Consider a multiplication operator $h$. If~this is $C_n$-invariant, we~can express it as a ratio of~holomorphic $C_n$-invariant theta functions. The~algebra of such theta functions is generated
 by~functions~$\prod_{1\le i\le n} \vartheta(u\pm z_i)=(-1)^n\prod_{1\le i\le n}
 \vartheta(z_i\pm u)$, and thus any holomorphic family of~\mbox{$C_n$-inva}riant
 functions $h$ has holomorphic Fourier transform. (The leading term of
 the Fourier transform of an operator is easy to determine, so we find
 that the Fourier transform of the denominator is indeed invertible.)

 Now, let $h$ be a general multiplication operator. To~show that
 $\hat{h}$ is holomorphic, we~need to show that every coefficient is
 holomorphic. The~coefficient of~$\prod_i T_i^{k_i}$ has denominator
 dividing $\prod_{1\le j\le \max(k_1,\dots,k_n)} \vartheta(jq)$, and by
 Hartog's lemma it suffices to prove that the coefficient is holomorphic
 at the generic point of every component of the corresponding divisor.
 Each coefficient is a finite linear combination of shifts of~$h$, and we
 are evaluating it at a~point with generic $(z_1,\dots,z_n)$. In~particular, none of the points where we are evaluating $h$ are in the
 same $C_n$ orbit (though we may be hitting the same point multiple
 times). It follows that there exists a $C_n$-invariant function~$g$
 such that the corresponding sum for~$h-g$ is holomorphic: simply take $g$
 to be a very good approximation near the points where $h$ is being
 evaluated. Since $\hat{g}$ is holomorphic and this coefficient of the
 Fourier transform of~$h-g$ is holomorphic, it follows that the given
 coefficient of~$\hat{h}$ is holomorphic as required.

 Now, let $D$ be an operator of the form $D^{(n)}_q(c+q/2\pm u;t)$, which
 again has a holomorphic Fourier transform. If~$q\ne 0$, then the space
 of operators $k(X) D^{(n)}_q(c+q/2\pm u;t) k(X)$ is a $2^n$-dimensional
 vector space on the left. Indeed, each of the $2^n$ shifts that appear
 induce different automorphisms of~$k(X)$. It follows that any element of
 that space has holomorphic Fourier transform (except possibly where
 $q=0$). Since that space contains elements $\propto \prod_i T_i^{\pm
 1/2}$ for every combination of signs, we~have proved holomorphy of the
 Fourier transform on a set of (topological) generators of the ring of
 twisted formal difference operators. The~Fourier transform is continuous
 with respect to the nonarchimedean metric, so the result follows in
 general.

 It remains to consider the case $q=0$. This splits into two components,
 depending on whether $q/2=0$ or $q/2$ is nontrivial $2$-torsion. The~latter case reduces to the first, however, since everything is invariant
 under translation by $2$-torsion. We may thus restrict our attention to
 the local ring at the generic point with~$q/2=0$. In~that case, the
 special fiber of the ring of twisted formal difference operators is
 abelian, since all shifts are trivial. As~a result, the algebra over the
 local ring picks up an additional (Poisson bracket) operation on
 operators: $(D_1,D_2)\mapsto (D_1D_2-D_2D_1)/\pi$, where $\pi$ is a
 uniformizer. This takes any pair of holomorphic families of operators to
 a holomorphic family of operators, and the Fourier transform respects
 this operation. We may thus use this operation to construct operators
 with known holomorphic Fourier transform. It turns out that the usual
 proof of independence of automorphisms of fields can be expressed in
 terms of this operation, and thus we still obtain the full
 $2^n$-dimensional space of operators.
\end{proof}

\begin{rem}
 The proof for~$q=0$ is of course based on the standard fact that an
 automorphism of a~family of noncommutative algebras preserves the induced
 Poisson structure on any commutative fiber.
\end{rem}

Of course, the algebra of formal difference operators is far too large, and
doesn't even have an action of~$C_n$ (as it preserves neither the metric
nor the topology). So we need to show that the operators we care about map
to operators which not only have finite support, but have $C_n$ symmetry.
Luckily, this is a closed condition, so it suffices to prove it
generically.

\begin{lem}
 On the generic fiber, the $\Z/2\Z$-graded algebra
 $
 \bigcup_d {\cal S}^{(n)}_{2c;0,0}(0,d(s+f))
$
 is generated by ${\cal S}^{(n)}_{2c;0,0}(0,s+f)$.
\end{lem}

\begin{proof}
 In fact, we~claim that for~$d\gg 0$, ${\cal S}^{(n)}_{2c;0,0}(0,d(s+f))$
 is spanned by products of~$d$ elements of~${\cal
 S}^{(n)}_{2c;0,0}(0,s+f)$. Since this contains the spaces for all
 smaller $d$ of the same parity, the result will immediately follow.

 Since this graded algebra is the homogeneous coordinate ring of~$\Sym^n(\P^1\times \P^1)$, what we are in fact claiming is that the ample
 bundle $\Sym^n(\sO_{\P^1\times \P^1}(1))$ is very ample. In~general, it
 follows from~\cite[Section~1.3]{BrionM:1993} that over any field of
 characteristic 0, $\Sym^n(\sO_{\P^m}(1))$ is very ample on~$\Sym^n(\P^m)$, and thus the same holds for the symmetric power of any
 closed sub\-scheme~of~$\P^m$.
\end{proof}

\begin{rem}
 It is likely that this fails in small characteristic. It is certainly
 the case that the line bundle $\Sym^n(\sO_{\P^m}(1))$ can fail to be very
 ample on~$\Sym^n(\P^m)$; indeed this already happens for~$\Sym^3(\P^2)$
 in characteristic 3. In~addition, even in characteristic 0, the
 $\Z$-{\em graded} algebra is not generated in degree 1 if $n$ is
 sufficiently large. Indeed, one has $h^0(\Sym^n(\sO_{\P^1\times
 \P^1}(1))) = \binom{n+3}{3}$, while $h^0(\Sym^n(\sO_{\P^1\times
 \P^1}(2))) = \binom{n+8}{8}$. So for~$n\gg 0$, even if we take into
 account noncommutativity, there are simply not enough sections of degree
 1 for their products to account for every section of degree 2!
\end{rem}

\begin{cor}
 For any $d$, the Fourier transform induces an isomorphism of stalks
 \begin{gather*}
 \Gamma{\cal S}^{(n)}_{2c;q,t}(0,d(s+f))_{q=t=0}
 \cong
 \Gamma{\cal S}^{(n)}_{-2c;q,t}(0,d(s+f))_{q=t=0}.
 \end{gather*}
\end{cor}

\begin{proof}
 Fix a basis of the global sections of the fiber over the generic point
 with~$q=t=0$. Each such global section can be expressed as a polynomial
 in sections of degree 1; if we choose an~extension to the stalk for each
 degree 1 operator that appears, then the result will be a basis of~the
 stalk of degree $d$ operators in which every element is a polynomial in
 first-order operators. It~follows that every element of the basis has
 Fourier transform in~$\bigcup_{e\ge 0} \Gamma{\cal S}^{(n)}_{-2c;q,t}(0,e(s+f))_{q=t=0}$,
 but the Fourier transform clearly preserves the space of operators
 $\prod_i T_i^{-l/2} k(X)[[T_1,\dots,T_n]]$ for each $l$, and thus
 the Fourier transform is actually in~$\Gamma{\cal
 S}^{(n)}_{-2c;q,t}(0,d(s+f))_{q=t=0}$ as required. The~inverse operation is of course just the Fourier transform again.
\end{proof}

\begin{cor}
 For any $d$, the Fourier transform induces a $($local$)$ isomorphism of
 sheaves of categories $\Gamma{\cal S}^{(n)}_{2c;q,t}|_{\Z(s+f)} \cong
 \Gamma{\cal S}^{(n)}_{-2c;q,t}|_{\Z(s+f)}$.
\end{cor}

\begin{proof}
 The given $\Hom$ sheaves of the global section category are flat and the
 map to difference operators is injective on fibers. We may thus identify
 sections with holomorphic families of difference operators and apply the
 Fourier transform to obtain a holomorphic family of formal difference
 operators. The~generic point of this family is a section of the other
 global section category, and thus the family itself is a section.
\end{proof}

\begin{cor}
 For $d\le d'$, the Fourier transform induces a morphism
 \begin{gather*}
 \Gamma{\cal S}^{(n)}_{2c;q,t}(0,ds+d'f)
 \to
 \Gamma{\cal S}^{(n)}_{-2c;q,t}(0,d's+df).
 \end{gather*}
\end{cor}

\begin{proof}
 If $d=0$, this is easy, as~the algebra is generated in degree 1, and we
 know the result there. More generally, given a section~$D$ of~$\Gamma{\cal S}^{(n)}_{2c;q,t}(0,ds+d'f)$ and any section~$g$ of~$\Gamma{\cal S}^{(n)}_{-2c;q,t}(d's+df,d's+d'f)$, consider the
 composition~$\hat{g} D\in \Gamma{\cal S}^{(n)}_{2c;q,t}(0,d's+d'f)$,
 which makes sense since $g$ has degree $(d'-d)f$. Since the Fourier
 transform is a covariant involution, we~find that $\hat{g} D$ has Fourier
 transform $g \hat{D}$, so that $g\hat{D}$ is a section of~$\Gamma{\cal
 S}^{(n)}_{-2c;q,t}(0,d's+d'f)$ for any $g$. But this implies that
 $\hat{D}$ is actually a section of~$\Gamma{\cal
 S}^{(n)}_{-2c;q,t}(0,d's+df)$ as required.
\end{proof}

\begin{cor}
 For any $d$, $d'$, the sheaf $\Gamma{\cal S}^{(n)}_{2c;q,t}(0,ds+d'f)$ is
 flat and the map to difference operators is injective on fibers.
 Moreover, the Fourier transform induces an isomorphism
 \begin{gather*}
 \Gamma{\cal S}^{(n)}_{2c;q,t}(0,ds+d'f)
 \cong
 \Gamma{\cal S}^{(n)}_{-2c;q,t}(0,d's+df)
 \end{gather*}
 for all $d$, $d'$.
\end{cor}

\begin{proof}
 We already have flatness if $d<0$ or $d\le d'$, so suppose $d\ge d'$. It
 suffices to show injectivity on fibers, as~it implies that any $\Tor_1$
 of the cokernel is 0. Thus, let $D\in \Gamma{\cal
 S}^{(n)}_{2c;q,t}(0,ds+d'f)$ be a local section such that the
 corresponding difference operator vanishes on some fiber. Consider the
 Fourier transform
 \begin{gather*}
 \Gamma{\cal S}^{(n)}_{-2c;q,t}(0,d's+df)\to \Gamma{\cal S}^{(n)}_{2c;q,t}(0,ds+d'f).
 \end{gather*}
 The domain is locally free and injective on fibers, and the codomain is
 at least {\em generically} free of the same rank. Since the Fourier
 transform is invertible at the level of operators, this map is injective,
 and thus generically an isomorphism. It follows that there is a section~$D'\in \Gamma{\cal S}^{(n)}_{-2c;q,t}(0,d's+df)$ such that $D-\hat{D'}$
 is generically 0. But this, of course, implies that $D'=\hat{D}$. In~particular, the corresponding fiber of~$D'$ vanishes, which means that in
 a suitable local basis we have $ D' = \sum_i c_i D_i$, in which each
 $c_i$ vanishes on a divisor passing through that fiber. We~then have $D
 = \sum_i c_i \hat{D}_i$ with each $\hat{D}_i$ a local section of~$\Gamma{\cal S}^{(n)}_{2c;q,t}(0,ds+d'f)$. It follows that the section
 corresponding to~$D$ vanishes at the fiber, so that injectivity holds.

 In particular, the Fourier transform induces a morphism in both
 directions, and thus gives an isomorphism as required.
\end{proof}

To finish the proof of the theorem, we~need to show that the transform
respects Bruhat order, that it respects the vanishing conditions associated
to~$x_1,\dots,x_m$, and that it commutes with the Selberg adjoint. Each
of these have analogous statements for general formal difference operators,
and in the first two cases reduce to the fact that (due to continuity) the
Fourier transform affects leading coefficients in easy to control ways.

For the Bruhat order, we~actually obtain a finer (inclusion) partial order
in the formal setting.

\begin{prop}
 Let $D$ be a holomorphic family of formal difference operators from
 $P_0(2c;q,t)$ to~$P_0(2c+lq;q,t)$. Let $S\subset \Z^n\cup
 (1/2,\dots,1/2)\Z^n$ be the set of vectors $\vec{v}$ such that for some
 $\vec{k}\in \N^n$, the left coefficient of~$\prod_i T_i^{v_i-k_i}$ is
 nonzero, and let $\hat{S}$ be the corresponding set for~$\hat{D}$.
 Then $\hat{S}=(l/2,\dots,l/2)+S$.
\end{prop}

\begin{proof}
 Conjugating by $T_\omega(c)$ or a $\Gamq$ symbol has no effect on the
 support of an operator, and multiplication by $T_\omega(lq/2)$ shifts the
 support by $(l/2,\dots,l/2)$. The~remaining operation consists of left-
 and right-multiplication by units in~$k(X)[[T_1,\dots,T_n]]$, and this
 clearly preserves the set~$S$.
\end{proof}

For the vanishing conditions, we~have the following. Note that we only
consider half of the vanishing conditions, as~in the formal setting it only
makes sense to consider conditions on the leading few terms. Also, for~convenience, we~only consider the generic case.

\begin{prop}
 Over the generic point $(E,x,c,q,t)\in {\cal E}^4$ and for integers $r$,
 $l$, consider the space of formal difference operators $D$ mapping
 $P_0(2c;q,t)$ to~$P_0(2c+lq;q,t)$ such that $D \prod_i T_i^{r/2}$
 involves only integer shifts. If~the left coefficients of both $D$ and
 \begin{gather*}
 \prod_{1\le i\le n} \Gamq(x+rq/2-z_i)^{-1}
 D
 \prod_{1\le i\le n} \Gamq(x-z_i)
 \end{gather*}
 are holomorphic along all hypersurfaces of the form $z_i\in x+rq/2+kq$,
 $k\in \Z$, then the left coefficients of both $\hat{D}$ and
 \begin{gather*}
 \prod_{1\le i\le n} \Gamq(x-c+(r-l)q/2-z_i)^{-1}
 \hat{D}
 \prod_{1\le i\le n} \Gamq(x-c-z_i)
 \end{gather*}
 are holomorphic along all hyperplanes of the form $z_i\in
 x-c+(r-l)q/2+kq$, $k\in \Z$.
\end{prop}

\begin{proof}
 By definition, we~have
 \begin{gather*}
 \hat{D} = {\cal D}^{(n)}_{q,t}(c+lq/2) D {\cal D}^{(n)}_{q,t}(-c).
 \end{gather*}
 There are only countably many hypersurfaces of the form $z_i=y$ on which
 some left coefficient of~${\cal D}^{(n)}_{q,t}(-c)$ and ${\cal
 D}^{(n)}_{q,t}(c+lq/2)$ has a pole (including poles of the meromorphic
 sections of equivariant gerbes corresponding to the leading
 coefficients). Since $x$ is generic, it follows that all three factors
 on the right are holomorphic on the given orbits of hypersurfaces, and
 thus so is the product.

 The claim for
 \begin{gather*}
 \prod_{1\le i\le n} \Gamq(x-c+(r-l)q/2-z_i)^{-1} \hat{D} \prod_{1\le i\le
 n} \Gamq(x-c-z_i)
 \end{gather*}
 analogously reduces to checking possible poles of
 \begin{gather*}
 \prod_{1\le i\le n} \Gamq(x-z_i)^{-1}
 {\cal D}^{(n)}_{q,t}(-c)
 \prod_{1\le i\le n} \Gamq(x-c-z_i)
 \end{gather*}
 and
 \begin{gather*}
 \prod_{1\le i\le n} \Gamq(x-c+(r-l)q/2-z_i)^{-1}
 {\cal D}^{(n)}_{q,t}(c+lq/2)
 \prod_{1\le i\le n} \Gamq(x+rq/2-z_i).
 \end{gather*}
 In each case, the gauging only multiplies the coefficients by holomorphic
 theta functions, so cannot introduce any new poles.
\end{proof}

To get an analogue for the Selberg adjoint, there is a mild difficulty
coming from the fact that the Selberg adjoint was only defined for~$C_n$-symmetric operators, and the obvious extension does not make sense
for formal operators. Luckily, the formal adjoint with respect to the
inner product
\begin{gather*}
\int f(z_1,\dots,z_n)g(-z_1,\dots,-z_n)
\prod_{1\le i<j\le n} \frac{\Gamq(t\pm z_i\pm z_j)} {\Gamq(\pm z_i\pm z_j)}
\prod_{1\le i\le n} \frac{1} {\Gamq(\pm 2z_i)}\,{\rm d}T
\end{gather*}
{\em does} make sense for formal difference operators and formal gauging
operators and agrees with the Selberg adjoint in the $C_n$-symmetric case.
Using this as the definition of the Selberg adjoint for formal operators
gives the following, which immediately implies consistency of the Fourier
transform with the Selberg adjoint.

\begin{prop}
 The operators ${\cal D}^{(n)}_{q,t}(c)$ are self-adjoint under the Selberg
 adjoint.
\end{prop}

\begin{proof}
 The Selberg adjoint has the correct leading term, so it suffices to show
 that
 \begin{gather*}
 {\cal D}^{(n)}_{q,t}(c)^{\ad_t}
 D^{(n)}_q(c\pm u;t)
 =
 \prod_{1\le i\le n} \vartheta(z_i\pm u)
 {\cal D}^{(n)}_{q,t}(c-q/2)^{\ad_t}.
 \end{gather*}
 Since $\prod_{1\le i\le n} \vartheta(z_i\pm u)$ is self-adjoint,
 this reduces to checking that
 \begin{gather*}
 D^{(n)}_q(c\pm u;t)^{\ad_t}
 =
 D^{(n)}_q(q/2-c\pm u;t),
 \end{gather*}
 an easy verification.
\end{proof}

\begin{rem}
 In fact, one has in general
 \begin{gather*}
 D^{(n)}_q(u_0,\dots,u_{2d'+1};t)^{\ad_t}
 =
 D^{(n)}_q(q/2-u_0,\dots,q/2-u_{2d'+1};t),
 \end{gather*}
 either by a direct computation or by using the fact that both are
 sections of the same $\Hom$ sheaf, and with the same leading coefficient.
\end{rem}

We mention a couple of further consequences of the proof. First, the fact
that the Fourier transform is determined by its values where we know it
explicitly has consequences in the analytic setting. Indeed, in
\cite{quadxforms}, a kernel function~${\cal
 K}^{(n)}_c(\vec{x};\vec{y};q,t)$ was constructed, with the property that
for~$D$ of degree $s$, $f$, or $s+f$, one had
\begin{gather*}
D_{\vec{x}}{\cal K}^{(n)}_c(\vec{x};\vec{y};q,t)
=
\hat{D}_{\vec{y}}^{\ad_t}{\cal K}^{(n)}_c(\vec{x};\vec{y};q,t).
\end{gather*}
It follows from the above proof and continuity that this holds for {\em
 all} operators which are global sections of the appropriate $\Hom$
spaces. In~particular, this applies to operators of degree
$2s+2f-e_1-\cdots-e_8$ (i.e., the van Diejen/Komori--Hikami integrable
system considered in Theorem~\ref{thm:vandiejen}), showing that the
associated formal integral operator takes eigenvalue equations of this form
to~eigenvalue equations of the same form.

Also, we~have already mentioned the consequence that the resulting
deformations of~$\Sym^n(\P^2)$ only depend (geometrically) on~$E$, $q$, and
$t$. It is worth mentioning the specific form that the given isomorphisms
take. The~isomorphism $\Gamma{\cal S}^{\prime(n)}_{x_0;q,t}|_{\Z(s+f)}
\cong \Gamma{\cal S}^{\prime(n)}_{x_1;q,t}|_{\Z(s+f)}$ is given (up to a~choice of element $(x_0-x_1)/3$) by gauging by the operator
\begin{gather*}
G_d(x_0,x_1):=
\prod_{1\le i\le n} \Gamq\bigg({-}\frac{(d-1)q}{2}-\frac{2x_0+x_1}{3}\pm z_i\bigg)
{\cal D}^{(n)}_{q,t}((x_0-x_1)/3)
\\ \phantom{G_d(x_0,x_1):=}
{}\times\prod_{1\le i\le n} \Gamq\bigg({-}\frac{(d-1)q}{2}-\frac{x_0+2x_1}{3}\pm z_i\bigg)^{-1}
\end{gather*}
in degree $d$; i.e., the action on morphisms from $d_1(s+f)$ to~$d_2(s+f)$
is given by
\begin{gather*}
D \mapsto G_{d_2}(x_0,x_1) D G_{d_1}(x_0,x_1)^{-1}.
\end{gather*}
In particular, we~see that when $(x_0-x_1)/3$ is $3$-torsion, the resulting
automorphism is still quite nontrivial. In~addition, the translation
symmetry of the category involves changing $x_0$, and thus although one
{\em can} identify it with a graded algebra at the cost of choosing an
element $q/3$, the resulting graded algebra does not actually have a
representation in (finite) difference operators. If~we restrict to the
``anticanonical'' model, i.e., to~$\Z(3s+3f)$, then the $\Hom$ space of
degree $3s+3f$ contains the $1$-dimensional subspace of operators of degree
$s$, spanned by
\begin{gather*}
G_0(x_0,x_0-3q/2)
=
\prod_{1\le i\le n} \Gamq(-x_0\pm z_i)
{\cal D}^{(n)}_{q,t}(-q/2)
\prod_{1\le i\le n} \Gamq(-q/2-x_0\pm z_i)^{-1}.
\end{gather*}
If we adjoin the formal inverse of such an operator, then the result in
degree 0 may be identified with a filtered algebra of {\em formal}
difference operators. We can include elements of degree not a~multiple of
3 at the cost of choosing $q/3$ and allowing some $\Gamq$ factors and
shifts by multiples of~$q/3$. Indeed, Proposition~\ref{prop:fourier_braid_relation} tells us (assuming compatible choices
when dividing by $3$) that $G_d(x_1,x_2)G_d(x_0,x_2)=G_d(x_0,x_2)$, and
thus the various isomorphisms between categories with parameter $x_0+kq/2$
are all compatible. It follows that if we compose an operator mapping
$d_1(s+f)$ to~$d_2(s+f)$ with the formal operators giving isomorphisms
$x_0\mapsto x_0+d_1q/2$ and $x_0+d_2q/2\mapsto x_0$, then the result will
be compatible with compositions and will be the same as if we only used the
spaces with~$d_1=0$. Of course, even in the univariate setting, the
resulting algebra is not likely to be easy to describe in any direct
fashion!

One final thing to mention is that our description of first-order operators
in Lemma~\ref{lem:first_order} as well as our description of the operators
${\cal D}^{(n)}_{q,t}(c)$ are both quite well suited to considering
degenerations of the tuple $(E,c,q,t)$. In~light of the fact that
operators of degree $s+f$ are generically very ample, we~can give at least
indirect descriptions of the limiting algebras $\bigcup_d {\cal
 S}^{(n)}_{\eta';q,t}(0,d(s+f))$ by specifying their elements of degree 1,
and understanding the extension to the whole category simply requires
keeping track of the elements of degree $f$ as well.

Taking the limit can be somewhat tricky in general, as~it may be necessary
to gauge by suitable functions before the limit is well-defined. The~simplest approach is to choose a suitable gauge transformation to make the
operators elliptic before taking the limit; this introduces additional
parameters which we can then eliminate by a further limit. Indeed, we~find
that for any operator $D\in \Gamma{\cal S}^{(n)}_{\eta';q,t}(0,ds+d'f)$,
the gauge transformation
\begin{gather*}
\prod_{1\le i\le n}
 \frac{\Gamq(\eta'-dq/2-d'q+(n-1)t+v_0+v_1+v_2\pm z_i)}
 {\Gamq(-dq/2+v_0\pm z_i,-dq/2+v_1\pm z_i,-dq/2+v_2\pm z_i)}
 \\ \hphantom{\prod_{1\le i\le n}}
\times {} D
\prod_{1\le i\le n} \frac{\Gamq(v_0\pm z_i,v_1\pm z_i,v_2\pm z_i)}
 {\Gamq(\eta'+(n-1)t+v_0+v_1+v_2\pm z_i)}
\end{gather*}
is elliptic (for fixed $v_0$, $v_1$, $v_2$). We can then take the limit as
$p\to 0$ and remove the parameters by gauging back by an appropriate
product of~$q$-Pochhammer symbols $(x;q)_\infty:=\prod_{0\le j} (1-q^j x)$.
We obtain a limit $\Gamma{\cal S}^{(n)}_{\eta';q,t;*}(0,s+d'f)$ (with
$q$, $t$, $\eta'$ and $\vec{z}$ in the multiplicative group) consisting of
operators of the form
\begin{gather*}
\sum_{\sigma\in \{\pm 1\}^n}
\prod_{1\le i\le n}
 \frac{z_i^{\sigma_i}f(z_i^{\sigma_i})}
 {1-z_i^{2\sigma_i}}
\prod_{1\le i<j\le n}
 \frac{1-t z_i^{\sigma_i} z_j^{\sigma_j}}
 {1-z_i^{\sigma_i} z_j^{\sigma_j}}
\prod_{1\le i\le n}T_i^{\sigma_i/2},
\end{gather*}
where $f(z)$ is a univariate Laurent polynomial with exponents ranging from
$1-d'$ to~$d'-1$ satisfying the condition~$
[z^{d'-1}]f(z) = q\eta'[z^{1-d'}]f(z)
$
on its extreme coefficients. The~Fourier transformation has a
corresponding limit, represented by operators
$
{\cal D}^{(n)}_{q,t:*}(c)
$
satisfying
\begin{gather}
{\cal D}^{(n)}_{q,t:*}(c)\prod_{1\le i\le n}
 \frac{((v cd) z_i^{\pm 1};q)_\infty}
 {((v/cd) z_i^{\pm 1};q)_\infty}
{\cal D}^{(n)}_{q,t:*}(d)\nonumber
\\ \phantom{{\cal D}^{(n)}_{q,t:*}}
{}=
\!\prod_{1\le i\le n}
 \frac{((v d) z_i^{\pm 1};q)_\infty}
 {((v/d) z_i^{\pm 1};q)_\infty}
{\cal D}^{(n)}_{q,t:*}(cd)
\prod_{1\le i\le n}
 \frac{((v c) z_i^{\pm 1};q)_\infty}
 {((v/c) z_i^{\pm 1};q)_\infty}
\label{eq:braid_for_*_case}
\end{gather}
with leading term
\begin{gather*}
\prod_{1\le i\le n}
 \frac{\theta_{q^{1/2}}(-1/z_i)}
 {\theta_{q^{1/2}}(-1/cz_i)}
\prod_{1\le i\le j\le n}
 \frac{(1/c^2z_iz_j;q)_\infty}
 {(1/z_iz_j;q)_\infty}
\prod_{1\le i<j\le n}
 \frac{(t/z_iz_j;q)_\infty}
 {(t/c^2z_iz_j;q)_\infty}
T_\omega(c)
\end{gather*}
and special case
\begin{gather*}
 {\cal D}^{(n)}_{q,t:*}(q^{-1/2})
 =
 \sum_{\sigma\in \{\pm 1\}^n}
\prod_{1\le i\le n}
 \frac{z_i^{\sigma_i}}
 {1-z_i^{2\sigma_i}}
\prod_{1\le i<j\le n}
 \frac{1-t z_i^{\sigma_i} z_j^{\sigma_j}}
 {1-z_i^{\sigma_i} z_j^{\sigma_j}}
\prod_{1\le i\le n}T_i^{\sigma_i/2}.
\end{gather*}
Note that taking $v=0$ in \eqref{eq:braid_for_*_case} gives $ {\cal
 D}^{(n)}_{q,t:*}(c) {\cal D}^{(n)}_{q,t:*}(d) = {\cal
 D}^{(n)}_{q,t:*}(cd)$, so that we may interpret ${\cal
 D}^{(n)}_{q,t:*}(c)$ as a fractional power of the operator for~$c=q^{-1/2}$, which in turn is a lowering operator appearing in the theory
of Koornwinder polynomials. We can further extend this limit to the case
of blowups in sufficiently general position (i.e., with all $x_i$ finite)
by imposing the appropriate conditions on the leading coefficients; this is
how we tested the $n=r=2$ case of Conjecture~\ref{conj:torsion_van_diejen}.

Another noteworthy limit involves gauging by a translation so as to break
the $z\mapsto 1/z$ symmetry, and taking a limit in the resulting
parameter. This gives a Fourier transform represented by operators
satisfying
\begin{gather*}
{\cal D}^{(n)}_{q,t:**}(c)
\!\!\prod_{1\le i\le n}\!\!
 \frac{((v cd) z_i;q)_\infty}
 {((v/cd) z_i;q)_\infty}\,
{\cal D}^{(n)}_{q,t:**}(d)
=
\!\!\prod_{1\le i\le n}\!\!
 \frac{((v d) z_i;q)_\infty}
 {((v/d) z_i;q)_\infty}\,
{\cal D}^{(n)}_{q,t:**}(cd)
\!\!\prod_{1\le i\le n}\!\!
 \frac{((v c) z_i;q)_\infty}
 {((v/c) z_i;q)_\infty}
\end{gather*}
with leading term
\begin{gather*}
\prod_{1\le i\le n}
 \frac{\theta_{q^{1/2}}(-1/z_i)}
 {\theta_{q^{1/2}}(-1/cz_i)}
T_\omega(c)
\end{gather*}
and special case
\begin{gather*}
{\cal D}^{(n)}_{q,t:**}(q^{-1/2})
=
\sum_{I\subset \{1,\dots,n\}}
(-1)^{|I|} t^{|I|(|I|-1)/2}
\prod_{1\le i\le n} z_i^{-1}
\prod_{i\in I,j\notin I}
 \frac{z_j-t z_i}
 {z_j-z_i}
\prod_{i\in I}T_i^{1/2}
\prod_{i\notin I}T_i^{-1/2},
\end{gather*}
a.k.a.\ the lowering operator for~$\GL_n$-type Macdonald polynomials. The~$\Hom$ spaces of deg\-ree~$s+d'f$ have similar, if somewhat more complicated
forms, obtained by gauging the $q^{-1/2}$ case of the Fourier transform
operator by suitable products of Pochhammer symbols. We omit the details,
except to note that the results again look like operators arising in
Macdonald theory.

There are some other symmetry breaking limits (e.g., the image of~$\eta'\to
0$ under the Fourier transform); we omit the details. Of course, such
symmetry-breaking limits have unfortunate effects on the Bruhat ordering;
for instance, the ``leading term'' must now incorporate all $\sim 2^n$
$S_n$-orbits corresponding to the given $C_n$-orbit of weights. As~a
result, in more degenerate cases, it can be difficult to figure out the
correct way to compactify the algebra. This can be fixed in some cases by
realizing that the $S_n$-symmetric operator is actually a shadow of a
$C_n$-symmetric operator acting on a power of a reducible curve (a
hyperelliptic curve of {\em arithmetic} genus 1). Similarly, there are
differential limits living on a power of the nonreduced curve $y^2=0$.

\section{Deformations of Hilbert schemes}\label{section9}

One potential issue with studying the category of sheaves on~${\cal S}$,
even if one can resolve the uncertainty in the definition, is that the
commutative projective scheme it deforms, the symmetric power of a surface,
is singular. This suggests that one should look for an analogous
deformation in which the symmetric power is replaced by its natural
resolution, namely the Hilbert scheme of points.

The Picard group of the $n$-point Hilbert scheme of~$X$ is still discrete
(and isomorphic to the N\'eron--Severi group), with~$\Pic(\Hilb^n(X))\cong
\Pic(X)\oplus \Z$ for~$n>1$. The~copy of~$\Pic(X)$ in~$\Pic(\Hilb^n(X))$
is the pullback of~$\Pic(\Sym^n(X))\cong \Pic(X)$, and since the map
$\Hilb^n(X)\to \Sym^n(X)$ is a birational morphism, the global sections of
any such bundle will be the same on~either $2n$-fold. Thus if we deform
the category of line bundles on~$\Hilb^n(X)$, the subcategory corresponding
to~$\Sym^n(X)$ should be precisely ${\cal S}$. As~we saw for the symmetric
power, twisting by line bundles cannot be expected to give an endofunctor,
and thus we expect that twisting by the additional generator of~$\Pic(\Hilb^n(X))$ should also change the parameters.

There is, of course, only one remaining parameter we could reasonably
shift, and thus we~should consider what happens as we change $t$. Consider
for the moment the $\P^1\times \P^1$ case. The~sections of~${\cal S}$ are
(for generic $t$) cut out by the condition that both ${\cal D}$ and
\begin{gather*}
 \prod_{1\le i<j\le n} \Gamq(t\pm z_i\pm z_j)
 {\cal D}
 \prod_{1\le i<j\le n} \Gamq(t\pm z_i\pm z_j)^{-1}
\end{gather*}
are sections of the corresponding parameter-free category. Incorporating a
shift in~$t$ here is then mostly straightforward: we should be considering
operators such that both ${\cal D}$ and
\begin{gather*}
 \prod_{1\le i<j\le n} \Gamq(t+a_2q\pm z_i\pm z_j)
 {\cal D}
 \prod_{1\le i<j\le n} \Gamq(t+a_1q\pm z_i\pm z_j)^{-1}
\end{gather*}
are holomorphic for some integers $a_1$ and $a_2$. We should also note the
effect on twisting, which boils down to noting that shifting $t$ acts on
the standard polarization as $P_d(\eta';q,t+aq) = P_{d-(n-1)a}(\eta';q,t)$.

With this in mind (and incorporating the conditions corresponding to~$x_1,\dots,x_m$), define
a~product of~$\Gamq$ symbols for any element of~$\Pic(\Hilb^n(X))=\Z\langle \delta,s,f,e_1,\dots,e_m\rangle$ as follows:
\begin{gather*}
\Delta_{x_1,\dots,x_m;q,t}(a\delta+ds+d'f-r_1e_1-\cdots-r_me_m;\vec{z})
\\[.5ex] \qquad
{}=
\prod_{1\le i<j\le n} \Gamq(t+aq\pm z_i\pm z_j)
\prod_{\substack{1\le i\le n\\1\le j\le m}} \Gamq((r_j+(1-d)/2)q-x_j\pm z_i).
\end{gather*}
Also, let $\pi\colon\Pic(\Hilb^n(X))\to \Z\langle s,f\rangle$ be given by
$\pi(s)=s$, $\pi(f)=f$, $\pi(e_i)=0$ and $\pi(\delta)=-(n-1)f$. Then (for~$n>1$) ${\cal S}^{(n)}_{\eta',x_1,\dots,x_m;q,t}(v,w)$ is defined to be the
subsheaf of~${\cal S}^{(n)}_{\eta';q}(\pi(v),\pi(w))$ (the spherical
algebra of the construction corresponding to the master DAHA) consisting
for generic parameters of the operators ${\cal D}$ such that
$\Delta_{x_1,\dots,x_m;q,t}(w,\vec{z}) {\cal D}
\Delta_{x_1,\dots,x_m;q,t}(v,\vec{z})^{-1}$ is also a local section of the
correspondingly twisted master spherical algebra. It is easy to see that
when~$v$ and $w$ have the same coefficient of~$\delta$, this imposes the
same vanishing conditions as the original definition of~${\cal S}$, and
thus this indeed extends the category to the larger group of objects. The~elementary transformation symmetries automatically extend (using the
obvious definition for~${\cal S}^{\prime(n)}$), as~does the $t\mapsto q-t$
symmetry, which now acts nontrivially on the objects (negating~$\delta$);
for the Fourier transform, see below.

We should note in passing that this is closely related to a construction
valid for general affine Weyl groups. We have already mentioned that
translating part of the system of parameters of such a DAHA or spherical
algebra by $q$ gives a Morita equivalent algebra, and thus in particular a
natural bimodule. More generally, we~may translate any point of any $T_i$
by any multiple of~$q$ without affecting generic Morita equivalence, and
the associated bimodule can be constructed as the tensor product of the
bimodules for the atomic shifts. As~a result, we~may construct a sheaf
category in which the objects are the lattice of different ways we may
shift the parameters (i.e., with rank equal to the total number of
components of the independent $T_i$), and each $\Hom$ bimodule is the
relevant Morita equivalence. Each of the original bimodules embeds in~${\cal H}_{\tW,W;\gamma}(X)$ (or ${\cal H}_{\tW;\gamma}(X)$ in the DAHA
version) as operators satisfying appropriate vanishing conditions and thus
the construction has a natural extension to general parameters. (Note that
this construction only allows us to shift those parameters which are
associated to roots of~$W$ itself, so to obtain the $C^\vee C_n$ version in
this way we need to associate the $\vec{x}$ parameters to~$s_n$ rather than
$s_0$.) This may be thought of as a generalization (to the elliptic level
and arbitrary numbers of parameters) of the construction of
\cite{GordonI/StaffordJT:2005}.

As in the symmetric power case, we~would like to show that the above
generic conditions extend to give a strongly flat family of categories
(i.e., not only flat but such that each fiber injects in the algebra of
meromorphic difference operators). The~first step of the construction
carries over: replacing the $t$-dependent factor of~$\Delta$ by the
appropriate nonsymmetric version gives us a natural corresponding extension
of~${\cal H}$. Unfortunately, the argument of Theorem~\ref{thm:CCn_DAHA_strongly_flat} fails in this case to show flatness of the
extended ${\cal H}$. The~key difference here is that the roots involved in
the conditions corresponding to~$x_1,\dots,x_m$ form a root system of
type $A_1^n$. If~we make the analogous construction for any other root
system, then we will find that there are some leading terms for which the
only vanishing condition involves a non-simple root, and such that any of
the terms covered by that leading term do not have the corresponding
vanishing condition. In~particular, since the $t$-dependent vanishing
conditions live on a root system of type~$D_n$, this argument fails, with
the notable exception of~$n=2$.

\begin{prop}
 The sheaf category ${\cal S}^{(2)}_{\eta',x_1,\dots,x_m;q,t}$ defined
 above is locally free, and the map from any $\Hom$ bimodule to the sheaf
 bimodule of meromorphic operators is injective on fibers.
\end{prop}

\begin{proof}
 We argue essentially as in the symmetric power case. We first note that
 by using elementary transformations and the analogue for~$t$, that we may
 arrange for the vanishing condition on the identity to be trivial. Then
 for any $w$ in the appropriate Bruhat interval, if there is any
 nontrivial vanishing condition associated to~$t$ or $x_i$ in the
 interval, then there is such a condition at $w$ and associated to a
 simple reflection that makes $w$ shorter. Thus we may obtain the
 restriction of the given $\Hom$ sheaf to~$[\le w]$ by composing the
 appropriate rank~1 Hecke algebra (using only those parameter still
 active) to the restriction to~$[\le sw]$.
\end{proof}

\begin{rem}
 For each simple root of~$W$ and each associated parameter $t$ or $x_i$,
 the existence of a~nontrivial vanishing condition of~$c_w$ along
 $t+\alpha$ or $x_i+\alpha$ is equivalent to the alcove associated to~$w$
 being on the appropriate side of a hyperplane orthogonal to~$\alpha$. For~$n=2$, there is thus a square (itself a Bruhat order ideal) such that
 any alcove in the square has no vanishing condition at $t$ (so we may
 ignore the parameter) and similarly for each $x_i$. For~$n>2$, there is
 a corresponding shape in which no {\em simple} root (or its negative) has
 a vanishing condition for~$t$; if one could prove strong flatness for the
 corresponding Bruhat order ideal, this would imply strong flatness in
 general.
\end{rem}

There is, however, one more special case in which we can prove strong
flatness. This has to do with the additional symmetry of the spherical
algebra in the DAHA case. In~particular, we~noted that both the algebra
for~$t$ and the algebra for~$t+q$ could be obtained from the Hecke algebra
via an analogue of the construction of the spherical algebra, and that
moreover there were analogous constructions of intertwining bimodules. But
these intertwiners are precisely the $\Hom$ sheaves we want.

\begin{prop}
 Suppose the coefficient of~$\delta$ in~$v$ is $-1$, $0$, or $1$. Then
 ${\cal S}^{(n)}_{\eta',x_1,\dots,x_m;q,t}(0,v)$ is locally free and
 injects on fibers into the sheaf bimodules of meromorphic operators.
\end{prop}

\begin{rem}
 Actually, there are two more cases in which we can prove strong flatness. If~$a\ge d$, then the $t$-dependent vanishing conditions become vacuous,
 and thus the claim reduces to strong flatness in the version without a
 $t$ parameter, where the usual argument works. By the $t\mapsto q-t$
 symmetry, this implies strong flatness when~$a\le -d$.
\end{rem}

Not only is this further evidence that this construction is well-behaved,
but this also has seve\-ral useful consequences. The~most significant of
these is that the construction of the intertwining bimodules allows us to
understand the case $t=0$. In~particular, let $v\in \Z\langle
s,f,e_1,\dots,e_m\rangle$, and consider the $\Hom$ sheaf
\begin{gather*}
{\cal S}^{(n)}_{\eta',x_1,\dots,x_m;q,0}(0,\delta+v).
\end{gather*}
By the infinite analogue of Proposition~\ref{prop:shift_operators_finite},
these operators are all of the form $\sum_{w\in C_n} w {\cal D}$, where
${\cal D}$ is a section of the left twist by $\sO_X(D_{w_{C_n}})$ of the
corresponding spherical {\em module}. Since $t=0$, the spherical module is
just the tensor product of~$n$ copies of the univariate spherical module
$M^{(1)}_{\eta',x_1,\dots,x_m;q}(v)$, and thus the $\Hom$ sheaf is spanned
by elements of the form
\begin{gather*}
\bigg(\sum_{w\in C_n} w\bigg)
\frac{1}{\prod_{1\le i\le n} \vartheta(2z_i)
 \prod_{1\le i<j\le n} \vartheta(z_i\pm z_j)}
D_1(z_1)\cdots D_n(z_n)
\end{gather*}
with each $D_i$ in the univariate spherical module. Summing over the
normal subgroup of order~$2^n$ turns each $D_i$ into a general element of
the corresponding spherical algebra, giving
\begin{gather*}
\sum_{\pi\in S_n} \pi
\frac{1}{\prod_{1\le i<j\le n} \vartheta(z_i\pm z_j)}
D'_1(z_1)\cdots D'_n(z_n)
=
\frac{1}{\prod_{1\le i<j\le n} \vartheta(z_i\pm z_j)}
\det_{1\le i,j\le n} D'_i(z_j)
\end{gather*}
with each $D'_i$ in the univariate spherical algebra. (Note that since
these are univariate operators in distinct variables, the determinant makes
sense.) We thus conclude the following.

\begin{prop}
 For any $v\in \langle s,f,e_1,\dots,e_m\rangle$, there is a natural
 isomorphism of sheaf bimo\-dules
 \begin{gather*}
 {\cal S}^{(n)}_{\eta',x_1,\dots,x_m;q,0}(0,\delta+v)
 \cong
 \wedge^n {\cal S}^{(1)}_{\eta',x_1,\dots,x_m;q,0}(0,v).
 \end{gather*}
\end{prop}

If $q=0$ and $v$ corresponds to an acyclic line bundle on the corresponding
rational projective surface $X_m$, then the exterior power is also acyclic
and we deduce
\begin{gather*}
\Gamma{\cal S}^{(n)}_{\eta',x_1,\dots,x_m;0,0}(0,\delta+v)
\cong
\wedge^n \Gamma(X_m;\sO_{X_m}(v)).
\end{gather*}
Any element of either side (which differ only in the product of theta
functions corresponding to the $D_n$ roots) gives a rational function on~$X_m^n$, and the description as an exterior power tells us that the
corresponding map to projective space is in fact a map to a~Grassmannian:
each point in~$X_m$ determines a linear functional on~$\Gamma(X_m;\sO_{X_m}(v))$, and the map takes $X_m^n$ to the
$n$-dimensional span of the corresponding linear functionals.

This is clearly invariant under permutations of the $n$-tuple, so gives a
rational map on the symmetric power, and thus on the Hilbert scheme. Since
we may also view this as the Grassmannian of~$(h^0(\sO_{X_m}(v))-n)$-dimensional subspaces of~$\Gamma(X_m;\sO_{X_m}(v))$, we~see that the map on the Hilbert scheme takes
a given ideal sheaf $I$ to the subspace $\Gamma(X_m;I(v))$, assuming this
has the correct dimension. In~particular, for~any sufficiently ample $v$,
$I(v)$ will always be acyclic and globally generated, and thus we obtain an
embedding of the Hilbert scheme in the Grassmannian. From known facts
about acyclicity and Hilbert polynomials of line bundles on the Hilbert
scheme~\cite{EllingsrudG/GoettscheL/LehnM:2001}, the Pl\"ucker coordinates
in fact form all global sections of the given line bundle on the Hilbert
scheme. (This line bundle has the form $-\Delta/2+v$, where $\Delta$ is
the discriminant, i.e., the divisor where the corresponding $n$-point
subscheme is nonreduced.)

We thus conclude that for~$q=t=0$ and $v$ sufficiently (relatively) ample
on~$X_m$, the (left) vector bundle ${\cal
 S}^{(n)}_{\eta',x_1,\dots,x_m;0,0}(0,\delta+v)$ on~$\P^n$ may be
identified with the direct image of~$\sO_{X_m}(v)$ under the map
$\Hilb^n(X_m)\to \Sym^n(\P^1)\cong \P^n$. Moreover, since we obtained this
identification by using the actual values of the ``operators'' in~${\cal
 S}$, these elements satisfy precisely the same relations as their
counterparts on the Hilbert scheme.

Note that although this indicates that there is a strong relation between
our construction and the Hilbert scheme, it is not quite enough to tell us
that ${\cal S}$ is a true deformation of the Hilbert scheme: the problem is
that there could conceivably be additional global sections involving larger
multiples of~$\delta$. For~$n=2$, we~can resolve this issue: it is
straightforward (if tedious) to use the Bruhat filtration to compute the
Euler characteristics of~$\Hom$ sheaves of~${\cal S}$, and we find in
particular that when $a(-\Delta/2)+v$ is acyclic over $\P^n$ on the Hilbert
scheme, the Euler characteristic of the $\Hom$ sheaf agrees with the Euler
characteristic of the corresponding line bundle on the Hilbert scheme.
Since we've shown an isomorphism for a large class of very ample divisors,
it follows that there is a cone in which the direct image of the line
bundle on the Hilbert scheme injects in the corresponding $\Hom$ sheaf, and
must therefore be isomorphic. We~thus conclude that (in a somewhat vague
sense) ${\cal S}^{(2)}$ is indeed the desired deformation of the Hilbert
scheme. (Note that this argument does not apply to the obvious category
associated to~$\Hilb^2\big(\P^2\big)$ in the absence of additional acyclicity
results.)

Unfortunately, the calculation of Euler characteristics of~$\Hom$ sheaves
is not only tedious but subject to combinatorial explosion: the
contribution from each subquotient depends in a~nontrivial and apparently
nonuniform way on the corresponding parabolic subgroup, and thus the
general computation requires computing a~subsum for each parabolic
subgroup, a total of~$2^n$ individual sums. Moreover, the condition under
which the divisor on the Hilbert scheme is acyclic involves an upper bound
on~$a$ depending on~$d$, $d'$, etc., and that bound is not preserved if we
subtract $2s+2f-e_1-\cdots-e_m$; as a result, we~cannot expect to simplify
the calculation by using an induction on~$d$ and $n$. A further issue
arises in the final step of comparing to the Hilbert scheme: the formula of
\cite{EllingsrudG/GoettscheL/LehnM:2001} is not {\em quite} explicit; it
depends on a certain universal power series which is only known in low
degree, and thus the Euler characteristics are only known for small $n$.

Still, it seems reasonable on the above evidence to conjecture that ${\cal
 S}$ indeed provides a flat deformation of~$\Hilb^n(X_m)$ for every $n$.

Assuming that this is true, we~would like to know that the deformation
depends only on the original rational surface rather than on the way in
which we obtained it as a blowup. As~in the symmetric power case,
elementary transformations are easy to deal with, and it is only the
analogue of the Fourier transform that we must consider. There is a clear
guess as to how the Fourier transformation should be defined, at least when~$q$ is not torsion: use the same operators~${\cal D}_{q,t}(c)$ and simply
adjust $t$ as appropriate. This leads to an obvious question, namely
whether this extension of the Fourier transform is still well-defined when~$q$ is torsion. As~in the symmetric power case, we~can first ask for this
to hold for formal difference operators, where it takes the~form
\begin{gather*}
 D
 \mapsto
 {\cal D}^{(n)}_{q,t+aq}(c+bq/2)
 D
 {\cal D}^{(n)}_{q,t}(-c)
\end{gather*}
for~$a,b\in \Z$. Since we know the usual Fourier transform
\begin{gather*}
\hat{D} = {\cal D}^{(n)}_{q,t}(c+(b+(n-1)a)q/2)
 D
 {\cal D}^{(n)}_{q,t}(-c)
\end{gather*}
is well-defined, we~may factor this as
\begin{gather*}
 {\cal D}^{(n)}_{q,t+aq}(c+bq/2)
 {\cal D}^{(n)}_{q,t}(-c-(b+(n-1)a)q/2)
 \hat{D},
\end{gather*}
and thus reduce to showing that the formal operator
\begin{gather*}
 {\cal D}^{(n)}_{q,t+aq}(c+bq/2)
 {\cal D}^{(n)}_{q,t}(-c-(b-(n-1)a)q/2)
\end{gather*}
remains holomorphic when~$q$ is torsion. This then easily reduces to the
analogous statement for
\begin{gather*}
 {\cal D}^{(n)}_{q,t+q}(-c-(n-1)q/2)
 {\cal D}^{(n)}_{q,t}(c).
\end{gather*}
Indeed, for~$a>0$, we~can factor the operator in question into~$a$
operators of the given form, while for~$a<0$ we have a similar
factorization of the inverse, and the leading term is clearly nonzero.

\begin{prop}
 The formal operator
 \begin{gather*}
 {\cal D}^{(n)}_{q,t+q}(-c-(n-1)q/2)
 {\cal D}^{(n)}_{q,t}(c)
 \end{gather*}
 is a global section of~${\cal S}^{(n)}_{2c;q,t}(0,\delta+(n-1)s)$.
\end{prop}

\begin{proof}
 The $\Hom$ sheaf ${\cal S}^{(n)}_{2c;q,t}(0,\delta+(n-1)s)$ is a vector
 bundle, and the description as an alternating power for~$t=0$ tells us
 that it has Euler characteristic $1$ and is acyclic away from a finite
 set of hypersurfaces, not containing any component on which $q$ is
 torsion. Similarly, ${\cal S}^{(n)}_{2c;q,t}(-s,\delta+(n-1)s)$ has
 Euler characteristic $n+1$ and is also acyclic on a suitable open subset.
 Moreover, a section of either vector bundle with vanishing leading
 coefficient is 0; by~semicontinuity (and flatness for Bruhat intervals),
 it suffices to check this for~$t=0$, where it again reduces to facts
 about the univariate case. Let $D(c)$ be any nonzero local section of~${\cal S}^{(n)}_{2c;q,t}(0,\delta+(n-1)s)$. Then as $u$ varies,
 \begin{gather*}
 D(c) D^{(n)}_q(c\pm u;t)
 \end{gather*}
 spans an $n+1$-dimensional family of sections of~${\cal
 S}^{(n)}_{2c;q,t}(-s,\delta+(n-1)s)$, which must therefore be
 everything. Applying the same argument on the left and comparing the
 dependence of the leading coefficient on~$u$, we~conclude that
 \begin{gather*}
 D^{(n)}_q((n-1)q/2+c\pm u;t+q)^{-1} D(c) D^{(n)}_q(c\pm u;t)
 \end{gather*}
 is a section of~${\cal S}^{(n)}_{2c;q,t}(-s,\delta+(n-2)s)={\cal
 S}^{(n)}_{2c-q;q,t}(0,\delta+(n-1)s)$ and thus
 \begin{gather*}
 D^{(n)}_q((n-1)q/2+c\pm u;t+q)^{-1} D(c) D^{(n)}_q(c\pm u;t)
 \propto
 D(c-q/2),
 \end{gather*}
 with coefficient independent of~$z_i$ and $u$.

 Now, $D(c)$ has leading term
 \begin{gather*}
 F(\vec{z};c;q,t)
 \prod_{1\le i\le j\le n}
 \frac{1}
 {\vartheta(-z_i-z_j;q)_{n-1}}
 \prod_{1\le i<j\le n}
 \vartheta(q+t-z_i-z_j;q)_{n-2}
 T_\omega(-(n-1)q/2),
 \end{gather*}
 with~$F(\vec{z};c;q,t)$ holomorphic and $S_n$-invariant. The~relation
 between $D(c)$ and $D(c-q/2)$ gives a weak recurrence for the leading
 coefficient, namely that
 \begin{gather*}
 F(z_1-q/2,\dots,z_n-q/2;c-q/2;q,t)
 F(z_1,\dots,z_n;c;q,t)^{-1}
 \end{gather*}
 is independent of~$\vec{z}$, and thus $F$ factors as
 \begin{gather*}
 F(\vec{z};c;q,t) = G(c-\vec{z};q,t) H(c;q,t),
 \end{gather*}
 where we may as well absorb any $c$-dependence of~$H$ into~$D(c;q,t)$
 and thus make the recurrence exact.

 If we specialize $c=-(n-1)q/2$, then we still find that ${\cal
 S}^{(n)}_{-(n-1)q;q,t}(0,\delta+(n-1)s)$ is generically 1-dimensional
 by reference to the $t=0$ case. Since the operator $D^{(n)}_{q,t}(n-1)$
 satisfies all of the requisite vanishing conditions, we~conclude that
 $D(-(n-1)q/2)$ is proportional to~$D^{(n)}_{q,t}(n-1)$, and thus that
 \begin{gather*}
 G(-(n-1)q/2-\vec{z};q,t) \propto \prod_{1\le i<j\le n} \vartheta(t-z_i-z_j),
 \end{gather*}
 so that (after rescaling) $D(c)$ has leading term
 \begin{gather*}
 \prod_{1\le i\le j\le n} \frac{1} {\vartheta(-z_i-z_j;q)_{n-1}}
 \prod_{1\le i<j\le n}
 \vartheta(q+t-z_i-z_j;q)_{n-2}
 \vartheta(2c+t+(n-1)q-z_i-z_j)
 \\ \hphantom{\prod_{1\le i\le j\le n} \frac{1} {\vartheta(-z_i-z_j;q)_{n-1}} \prod_{1\le i<j\le n}}
 {}\times T_\omega(-(n-1)q/2).
 \end{gather*}
 We then find that
 \begin{gather*}
 {\cal D}^{(n)}_{q,t+q}(c+(n-1)q/2)D(c)
 \end{gather*}
 has the same leading term and satisfies the same defining recurrence as
 ${\cal D}^{(n)}_{q,t}(c)$, and thus by the proof of Proposition~\ref{prop:fourier_xform_exists},
 \begin{gather*}
 D(c)
 =
 {\cal D}^{(n)}_{q,t+q}(-c-(n-1)q/2){\cal D}^{(n)}_{q,t}(c),
 \end{gather*}
 from which the desired claim immediately follows.
\end{proof}

\begin{rem}
 When $n=2$, this is a first-order operator, and thus one can easily
 deduce an explicit formula from the leading coefficient. This turns out
 to be an operator we have already seen, albeit with a rather odd change of
 parameters:
 ${\cal D}^{(2)}_{q,t+q}(-c-q/2){\cal D}^{(2)}_{q,t}(c)
 =
 D^{(2)}_{q,t+q+2c}(1),$
 a special case of the curious identity
 $
 {\cal D}^{(2)}_{q,2u_1}(u_2-u_3)
 {\cal D}^{(2)}_{q,2u_2}(u_3-u_1)
 {\cal D}^{(2)}_{q,2u_3}(u_1-u_2)
 =
 1$. The~latter can be proved by Zariski closure from the case $u_2=kq/2+u_3$,
 which in turn follows by induction in~$k$ from the known special case.
\end{rem}

The operator considered in the proposition is the Fourier transform of the
global section
\begin{gather*}
1\in \Gamma{\cal S}^{(n)}_{-2c;q,t}(0,\delta+(n-1)f).
\end{gather*}
As mentioned above, the usual $t\mapsto q-t$ symmetry extends to a symmetry
that negates $\delta$, and thus we also find that the Fourier transform of
the global section
\begin{gather*}
\prod_{1\le i<j\le n} \vartheta(t\pm z_i\pm z_j)
\in
\Gamma{\cal S}^{(n)}_{-2c;q,t}(0,-\delta+(n-1)f)
\end{gather*}
is a section of~$\Gamma{\cal S}^{(n)}_{2c;q,t}(0,-\delta+(n-1)s)$.
An easy induction then shows that
\begin{gather*}
 {\cal D}^{(n)}_{q,t}(-c-a(n-1)q/2)
 \prod_{1\le i<j\le n} \vartheta(t\pm z_i\pm z_j;q)_a
 {\cal D}^{(n)}_{q,t+aq}(c)
\end{gather*}
is a section of~${\cal S}^{(n)}_{2c;q,t}(a\delta,a(n-1)s)$ for~$a\ge 0$. For~purposes of the following proof, denote this operator by
$S_-^{(n)}(a;c;q,t)$.

\begin{prop}
 The Fourier transform extends to the Hilbert scheme category.
\end{prop}

\begin{proof}
 As~in the symmetric power case, this reduces to showing that the Fourier
 transform takes $\Gamma{\cal S}^{(n)}_{2c;q,t}(0,a\delta+ds+d'f)$
 to~$\Gamma{\cal S}^{(n)}_{-2c;q,t}(0,a\delta+d's+df)$. By the $t\mapsto
 q-t$ symmetry, it suffices to show this for~$a\ge 0$. Let $D$ be a
 local section of this category on some open subset of parameter space.
 Then
 \begin{gather*}
 \prod_{1\le i<j\le n} \vartheta(t\pm z_i\pm z_j;q)_a
 D \in \Gamma{\cal S}^{(n)}_{2c;q,t}(0,ds+(d'+a(n-1))f),
 \\
 S_-^{(n)}(a;c-(d'-d)q/2;q,t)
 D\in \Gamma{\cal S}^{(n)}_{2c;q,t}(0,(d+a(n-1))s+d'f),
 \end{gather*}
 since each prefactor is itself a section of~$\Gamma{\cal S}^{(n)}$.
 Both of these operators have well-behaved Fourier transforms, and we thus
 conclude that
 \begin{gather*}
 S_-^{(n)}(a;-c+(d'-d)q/2;q,t)\hat{D}
 \in \Gamma{\cal S}^{(n)}_{2c;q,t}(0,(d'+a(n-1))s+df)
 \\
 \prod_{1\le i<j\le n} \vartheta(t\pm z_i\pm z_j;q)_a
 \hat{D} \in \Gamma{\cal S}^{(n)}_{2c;q,t}(0,d's+(d+a(n-1))f).
 \end{gather*}
 We claim that this implies $\hat{D}\in \Gamma{\cal
 S}^{(n)}_{2c;q,t}(0,a\delta+d's+df)$ as required.

 It suffices to show this on an open subset, and thus we may assume that
 $c$ and $t$ are in general position. The~second claim shows that
 \begin{gather*}
 \prod_{1\le i<j\le n} \Gamq(aq+t\pm z_i\pm z_j) \hat{D}
 \prod_{1\le i<j\le n} \Gamq(t\pm z_i\pm z_j)^{-1}
 \\ \phantom{ \prod_{1\le i<j\le n} }
 {}= \prod_{1\le i<j\le n} \Gamq(t\pm z_i\pm z_j)
 \prod_{1\le i<j\le n} \vartheta(t\pm z_i\pm z_j;q)_a
 \hat{D} \prod_{1\le i<j\le n} \Gamq(t\pm z_i\pm z_j)^{-1}
 \end{gather*}
 is a section of the $t$-free version of the category, and thus it remains
 only to show that $\hat{D}$ itself is such a section. In~other words, we~need to show that dividing by $\vartheta(t\pm z_i\pm z_j;q)_a$ does not
 introduce any poles. Now, the operator $S_-^{(n)}(a;-c+(d'-d)q/2;q,t)$
 is holomorphic away from the reflection hypersurfaces, and has leading
 (left) coefficient
 \begin{gather*}
 \frac{\prod_{1\le i<j\le n}
 \vartheta(t\,{-}\,z_i\,{-}\,z_j;q)_{an}
 \vartheta(t\pm (z_i\,{-}\,z_j);q)_a
 \vartheta(2c\,{+}\,t\,{+}\,z_i\,{+}\,z_j\,{-}\,(d'\,{-}\,d\,{+}\,a(n\,{-}\,1))q;q)_a}
 {\prod_{1\le i\le j\le n} \vartheta(\,{-}\,z_i\,{-}\,z_j;q)_{a(n\,{-}\,1)}}.
 \end{gather*}
 If we take its inverse as a formal operator, the only poles that can
 arise are translates of those originally present and translates of the
 divisors on which the leading coefficient vanishes. We~find in
 particular that the inverse remains (since $t$ and $c$ are generic)
 holomorphic on any divisor of the form $t+z_i+z_j=kq$. Since $
 S_-^{(n)}(a;-c+(d'-d)q/2;q,t)\hat{D}$ is holomorphic away from the
 reflection hypersurfaces, it follows that $\hat{D}$ itself is holomorphic
 on the divisors $t+z_i+z_j=kq$, and thus by symmetry on any divisor $t\pm
 z_i\pm z_j=kq$. But this is precisely what we needed to~show.
\end{proof}

\begin{rem}
 As~in the symmetric power case, this tells us that the subcategory
 corresponding to~$\P^2$ is indeed (up to fppf local isomorphism)
 independent of~$c$.
\end{rem}

When $q=0$, the category is equal to its center (taking the natural
extension of the definition used in the symmetric power case), and thus we
expect to obtain a family of commutative deformations of the Hilbert scheme
as $t$ varies. (We similarly expect the center for~$q$ torsion to be the
pullback of this family through a suitable base change; this does not {\em
 quite} follow from the usual arguments, since those depended on strong
flatness for the spherical algebra.) Such a family was constructed in
\cite{noncomm1} (see also~\cite{NevinsTA/StaffordJT:2007} for the case of~$\P^2$), and thus we naturally conjecture that they agree under a suitable
reparametrization. The~significance of this is that the construction of
\cite{NevinsTA/StaffordJT:2007,noncomm1} is as a moduli space of sheaves on
a noncommutative surface (essentially $\Proj {\cal S}^{(1)}$), and there
are some natural birational transformations between such moduli spaces. In~particular, under certain conditions, the deformed Hilbert scheme is
birational to a moduli space of sheaves with~$1$-dimensional support, which
in turn may be interpreted as a moduli space of elliptic difference
equations. We thus expect that our deformations are similarly closely
related to {\em noncommutative} deformations of such moduli spaces.

When $q=t=0$ (i.e., the usual Hilbert scheme), this birational map may be
described as follows. Let $X$ be a rational surface with a chosen
anticanonical curve $C_a$, and let $D$ be a divisor class such that
$\sO_{C_a}(D)$ is trivial and the linear system $|D|$ is generically
integral. The~curves in this linear system have genus $g=D^2/2+1$, and the
linear system is itself a $\P^g$. There is a natural incidence relation
between $\Hilb^g(X)$ and the linear system, i.e., whether the curve
contains the \mbox{$g$-point} subscheme $Z$. Each divisor in~$|D|$ is incident
with a $g$-dimensional subscheme of~$\Hilb^g(X)$, and thus $\Hilb^g(X)$ is
birational to the relative $\Hilb^g$ of the linear system. For~$C\in |D|$,
there is a natural map from $\Hilb^g(C)$ to a compactification of~$\Pic^g(C)$ (i.e., torsion-free sheaves on~$C$), and again this is
generically invertible. In~other words, the incidence relation is the
graph of a birational map as desired.

There are two things that may go wrong with the map from $\Hilb^g(X)$ to
the relative $\Pic^g$: there may be more than one curve in the linear
system containing the given subscheme, and the curve may be reducible (so
that the resulting torsion-free sheaf may fail to be stable). The~first
issue happens when $h^0({\cal I}(D))>1$ (where ${\cal I}$ is the ideal
sheaf of~$Z$); since this sheaf has Euler characteristic $0$, we~expect
this to happen only in codimension~$\ge 3$. (In fact, given a particular
curve $C$ containing $Z$, the deformations of~$C$ that still contain~$Z$
are given by a~subsheaf of the cotangent sheaf corresponding to~$\sO_C(K-Z)$, and thus $C$ is moveable iff $Z$ is moveable in~$\Pic^g(C)$,
so this is codimension~$\ge 2$ in the graph of the birational map.)
For~generic parameters, the only reducible curves of~$D$ will be those meeting
$C_a$ and thus having it as a~component. Since $|D-C_a|\cong \P^{g-1}$,
there are two divisors in~$\Hilb^g(X)$ compatible with such a reducible
curve: either all $g$ points lie on a curve of~$|D-C_a|$ or at least one
point lies on~$C_a$. (These correspond to the divisor classes on~$\Hilb^g(X)$ denoted by $\delta+D-C_a$ and $C_a$ respectively relative to
the above basis.) For an appropriate choice of stability conditions, the
map is well-defined on the locus where one point lies on~$C_a$, and
contracts the other divisor to a subscheme of codimension~$\ge 2$.

We thus find that if we remove the unique divisor of class $\delta+D-C_a$
from $\Hilb^g(X)$, then the result is not only birational to the
compactified relative $\Pic^g(C)$, but the birational map is an isomorphism
in codimension~1. The~compactified relative $\Pic^g(C)$ is an abelian
fibration over $\P^g\cong |D|$ (with integral fibers), and the derived
autoequivalences of the fibers should extend to derived autoequivalences of~$\Pic^g(C)$. One also expects that birational maps which are isomorphisms
in codimension~1 should induce derived equivalences, and thus we expect
that contracting the given divisor on~$\Hilb^g(X)$ should give a projective
scheme with a large family of derived autoequivalences. Moreover, since
the noncommutative deformations of a scheme are functions of the derived
category alone, we~should expect these derived autoequivalences to act on
the corresponding formal neighborhood in the family corresponding to~${\cal
 S}^{(n)}$, and thus one hopes on the family itself. The~result would be
an analogue of the derived equivalences of geometric Langlands, with
(symmetric) elliptic difference equations replacing connections.
(For~$n=1$, such derived equivalences in fact exist, see~\cite[Section~12]{noncomm1}.)

One part of the above line of reasoning is the strong suggestion that the
divisor $\delta+D-C_a$ should be contractible. Something along these lines
holds for~${\cal S}$.

\begin{prop} For any $v\in \langle \delta,s,f,e_1,\dots,e_m\rangle$, the union
 \begin{gather*}
 \bigcup_{k\in \Z} {\cal
 S}^{(n)}_{\eta',x_1,\dots,x_m;q,t}(0,v+k(\delta+(n-1)f))
 \end{gather*}
 is coherent and equal to
 \begin{gather*}
 {\cal S}^{(n)}_{\eta',x_1,\dots,x_m;q,t}(0,v+(d-a)(\delta+(n-1)f))
 \\ \phantom{{\cal S}^{(n)}_{\eta',x_1,\dots,x_m;q,t}}
{} = {\cal S}^{(n)}_{\eta'+(n-1)t,x_1,\dots,x_m;q,0}(0,v+(d-a)(\delta+(n-1)f))
 \end{gather*}
 if $v=a\delta+ds+\cdots$.
\end{prop}

\begin{proof}
 Indeed, this is a nested sequence of sheaf bimodules as $k$ increases,
 and once $k\ge d-a$, there are no longer any $t$-dependent vanishing
 conditions, and thus the sequence stabilizes. The~only remaining
 dependence on~$t$ is in the twisting datum, and thus we may make $t=0$
 as long as we adjust $\eta'$ accordingly.
\end{proof}

\begin{rem}
 For the commutative Hilbert scheme, what this is suggesting is that if
 $D$ is the unique divisor of class $\delta+(n-1)f$, then for any line
 bundle ${\cal L}$, $\Gamma(\Hilb^n(X)\setminus D;{\cal L})$ is
 finite-dimensional (and isomorphic to the space of global sections of
 some ${\cal L}'$). Applying this to powers of an ample bundle gives a
 graded algebra the $\Proj$ of which has the same homogeneous coordinate
 ring as $\Hilb^n(X)\setminus D$, and thus the corresponding map contracts
 $D$ to a subscheme of codimension~$\ge 2$. Note that we cannot expect
 this to be a blow down, but it should be similar: if we take $e_m$ rather
 than $\delta+(n-1)f$, then the same construction gives the Hilbert scheme
 of the blow down of~$X$.
\end{rem}

\begin{rem}
 The case $D-C_a=(n-1)f$ (and thus $m=2n+6$) corresponds to a moduli space
 of second-order equations with~$2n+6$ singularities.
\end{rem}

As this holds even without the constraint that $D\cdot C_a=0$, it seems likely that something ana\-lo\-gous should hold for any divisor. To~be precise,
if $w\in \Z\langle s,f,e_1,\dots,e_m\rangle$ is such that~$\sO_{X_m}(w)$ is
generically acyclic with~$n$ global sections, then we conjecture that the
union
\begin{gather*}
\bigcup_{k\in \Z} {\cal S}^{(n)}_{\eta',x_1,\dots,x_m;q,t}(0,v+k(\delta+w))
\end{gather*}
stabilizes (and is thus coherent) for large $k$, and moreover that the
smallest $k$ for which it stabilizes is a linear function of~$v$, so that
the resulting category may be identified with a~subcategory of~${\cal S}$. The~result would then be invariant under the shift in parameters
corresponding to~$\delta+w$, and thus equivalent to a category with~$t=0$.
Something along these lines appears to work for~$n=2$ with~$w=2s+2f-e_1-\cdots-e_7$, in that the Euler characteristic of the line
bundle reaches a maximum at $k$ depending linearly on~$v$. (This case
corresponds to the unique other case (modulo the Fourier transform and
elementary transformations) of a divisor $D$ such that $(D+C_a)\cap
C_a=\varnothing$ and $(D+C_a)^2=2$, corresponding to a moduli space of line
bundles on genus~2 curves.)

\subsection*{Acknowledgements} The author would particularly like to thank
P.~Etingof both for asking the original seed question (with an important
assist from A.~Okounkov!) and hosting the author's sabbatical at~MIT (which
the author would also like to thank, naturally) where much of the basic
approach was worked out, with great assistance from conversations with not
only Etingof but also (regarding various geometrical issues) B.~Poonen.
Thanks also go to T.~Graber and E.~Mantovan for helpful and encouraging
conversations regarding the constructions of Section~\ref{section2} as well as various
general algebraic geometric questions, and to O.~Chalykh for asking some
fruitful questions about residue conditions. The~author's work presented
here was supported in part by grants from the National Science Foundation,
DMS-1001645 and DMS-1500806.

\LastPageEnding

\end{document}